\documentclass{article}

\usepackage{iclr2022_conference,times}

\PassOptionsToPackage{numbers}{natbib}

\usepackage[utf8]{inputenc} 
\usepackage[T1]{fontenc}    
\usepackage{hyperref}       
\usepackage{url}            
\usepackage{booktabs}       
\usepackage{amsfonts}       
\usepackage{nicefrac}       
\usepackage{microtype}      

\usepackage{amsmath,amsthm,graphicx,multirow}
\usepackage{amssymb}
\usepackage{mathtools}
\usepackage{mathrsfs}
\usepackage{array}
\usepackage{xcolor}
\usepackage{float, subcaption}
\usepackage{algorithm,algorithmicx,algpseudocode}
\usepackage{xparse}
\makeatletter
\NewDocumentCommand{\LeftComment}{s m}{%
  \Statex \IfBooleanF{#1}{\hspace*{\ALG@thistlm}}\(\triangleright\) #2}
\makeatother

\newtheorem{theorem}{Theorem}

\newtheorem{prop}{Proposition}
\newtheorem{definition}{Definition}
\newtheorem{lem}{Lemma}[section]
\newtheorem{lemma}{Lemma}

\newtheorem{remark}{Remark}

\newcommand{\R}{\mathbb{R}}
\newcommand{\defeq}{\stackrel{\mathrm{def}}{=}}

\newcommand{\FP}{\ensuremath{\mathrm{FP}}}
\newcommand{\FN}{\ensuremath{\mathrm{FN}}}
\newcommand{\vt}{\vartheta}

\newcommand{\beq}{\begin{equation}}
\newcommand{\eeq}{\end{equation}}

\DeclareMathOperator*{\argmax}{arg\,max}
\DeclareMathOperator*{\argmin}{arg\,min}
\DeclareMathOperator\dif{d\!}
\DeclareMathOperator{\sgn}{sgn}
\DeclarePairedDelimiter\abs{\lvert}{\rvert}
\DeclarePairedDelimiter\norm{\lVert}{\rVert}

\newcommand{\FPtwo}{\vt+f_1(\sqrt{r},\lambda')}
\newcommand{\FNone}{\vt+f_2(\sqrt{r},\lambda')}
\newcommand{\FNtwo}{2\vt+f_3(\sqrt{r},\lambda')}

\newcommand{\tfpone}{\frac{1}{1 -\rho^2}\left[(1 +|\rho|) \lambda' + (1 -\rho^2)t'\right]^2}
\newcommand{\tfptwo}{(\lambda' + t')^2}
\newcommand{\tfpthree}{\vt+\frac{1}{1 -\rho^2}\left[\lambda'(1 -|\rho|) + t'(1 -\rho^2)\right]^2}
\newcommand{\tfnone}{\vt+(\sqrt{r} -\lambda' - t')_ +^2}
\newcommand{\tfntwo}{\vt+\frac{1}{1 -\rho^2}\left[(1 -\rho^2)\sqrt{r} - t'(1 -\rho^2) -\lambda'(1 -|\rho|)\right]^2}
\newcommand{\tfnthree}{2\vt+\frac{1}{1 -\rho^2}\left[(1 -\rho^2)\sqrt{r} - t'(1 -\rho^2) -\lambda'(1 +|\rho|)\right]^2}

\title{A Comparison of Hamming Errors of Representative Variable Selection Methods}

\author{Zheng Tracy Ke \\
Department of Statistics\\
Harvard University\\
Cambridge, MA 02138, USA \\
\texttt{zke@fas.harvard.edu} \\
\And
Longlin Wang \\
Department of Statistics \\
Harvard University \\
Cambridge, MA 02138, USA \\
\texttt{lwang2@fas.harvard.edu} 
}

\iclrfinalcopy 

\begin{document}

\maketitle

\begin{abstract}
Lasso is a celebrated method for variable selection in linear models, but it faces challenges when the variables are moderately or strongly correlated. This motivates alternative approaches such as using a non-convex penalty, adding a ridge regularization, or conducting a post-Lasso thresholding. In this paper, we compare Lasso with 5 other methods: Elastic net, SCAD, forward selection, thresholded Lasso, and forward backward selection. We measure their performances theoretically by the expected Hamming error, assuming that the regression coefficients are {\it iid} drawn from a two-point mixture and that the Gram matrix is block-wise diagonal. By deriving the rates of convergence of Hamming errors and the phase diagrams, we obtain useful conclusions about the pros and cons of different methods. 
\end{abstract}

\section{Introduction} \label{sec:Intro}

Variable selection is one of the core problems in high-dimensional data analysis. Consider a linear regression, where  the response $y\in\mathbb{R}^n$ and the design matrix $X=[x_1,\ldots,x_p]\in\mathbb{R}^{n\times p}$ satisfy that  
\beq \label{linearM}
y=X\beta+z, \qquad \|x_j\|=1,\qquad  z\sim {\cal N}(0, \sigma^2I_n). 
\eeq
The goal is  estimating the support of $\beta$ ($\mathrm{Supp}(\beta)$). Lasso \citep{tibshirani1996regression} is a popular method: 
\beq \label{Lasso}
\hat{\beta}^{\mathrm{lasso}} =\mathrm{argmin}_{\beta}\bigl\{ \norm*{y-X\beta}^2/2 + \lambda \norm*{\beta}_1 \bigr\}. 
\eeq
Lasso has good rates of convergence on the $L_q$-estimation error or prediction error \citep{bickel2009simultaneous}. However, it can be unsatisfactory for variable selection, especially when the columns in the design matrix are moderately or strongly correlated.  \cite{zhao2006model} showed that an {\it irrepresentable condition} on $X$ is necessary for Lasso to recover $\mathrm{Supp}(\beta)$ with high probability, and such a condition is restrictive when $p$ is large \citep{fan2010selective}. \cite{ji2012ups} studied the Hamming error of Lasso and revealed Lasso's non-optimality by lower-bounding its Hamming error rate.
Many alternative strategies were proposed for variable selection, such as using non-convex penalties \citep{fan2001variable, zhang2010nearly, shen2012likelihood}, 
adding a ridge regularization \citep{zou2005regularization}, post-processing on the Lasso estimator \citep{zou2006adaptive,zhou2009thresholding}, and iterative algorithms \citep{zhang2011adaptive, donoho2012sparse}. In this paper, our main interest is to theoretically compare these different strategies. 

Existing theoretical studies focused on `model selection consistency' (e.g., \cite{fan2001variable,zhao2006model, zou2006adaptive, meinshausen2010stability, loh2017support}), which uses $\mathbb{P}(\mathrm{Supp}(\hat{\beta})= \mathrm{Supp}(\beta))$ to measure the performance of variable selection. 
However, for many real applications, the study of the Hamming error (i.e., total number of false positives and false negatives) is in urgent need. 
For example, in genome-wide association studies (GWAS) or Genetic Regulatory Network,  the goal is to identify the genes or SNPs that are truly associated with a given phenotype, and we hope to find a multiple testing procedure that simultaneously 
controls the FDR and maximizes the power (for multiple testing). This problem can be re-cast as minimizing the Hamming error in a special regression setting \citep{efron2004large,jin2012comment,sun2007oracle}. This motivates us to study the {\it Hamming errors} of variable selection methods, which were rarely considered in the literature.


We adopt the {\it rare and weak signal model} \citep{donoho2004higher,arias2011global, jin2016rare}, which is often used in theoretical analysis of sparse linear models. Let $p$ be the asymptotic parameter. 
Given constants $\vartheta\in (0,1)$ and $r>0$,  we assume that $\beta_j$'s are iid generated such that 
\beq \label{model-beta}
\beta_j=\begin{cases} \tau_p, &\mbox{with probability }\epsilon_p,\\
0, & \mbox{with probability }1-\epsilon_p,
\end{cases} \qquad\mbox{where}\qquad \epsilon_p=p^{-\vartheta}, \quad \tau_p=\sqrt{2r\log(p)}.  
\eeq 
As $p\to\infty$, $\|\beta\|_0\approx p^{1-\vartheta}$, and a nonzero $\beta_j$ is at the critical order $\sqrt{\log(p)}$. \footnote{In \eqref{linearM}, we assume that each column of $X$ is standardized to have a unit $\ell^2$-norm and that the order of nonzero $\beta_j$ is $\sqrt{\log(n)}$. Alternatively, many works assume that each column of $X$ is standardized to have an $\ell^2$-norm of $\sqrt{n}$ and that the order for nonzero $\beta_j$ is $n^{-1/2}\sqrt{\log(p)}$. These are two equivalent parameterizations.} The two parameters $(\vartheta, r)$ capture the sparsity level and signal strength, respectively. We may generalize \eqref{model-beta} to let nonzero $\beta_j$'s take different values in $[\tau_p,\infty)$, but the current form is more convenient for presentation.

The {\it blockwise covariance structure} is frequently observed in real applications. In genetic data, there may exist strong correlations between nearby genetic markers, but the long-range dependence is usually negligible; as a result, the sample covariance matrix is approximately blockwise diagonal \citep{dehman2015performance}. In financial data, the sample covariance matrix of stock returns (after common factors are removed) is also approximately blockwise diagonal, where each block corresponds to an industry group \citep{fan2015risks}. Motivated by these examples, we consider an idealized setting, where the Gram matrix $G=X'X$ is block-wise diagonal consisting of $2\times 2$ blocks:
\beq \label{model-X}
G=\mathrm{diag}(B,B,\ldots,B,B_0), \qquad \mbox{where}\quad B=\begin{bmatrix}
1 & \rho\\
\rho & 1
\end{bmatrix}\mbox{ and  }B_0=\begin{cases}
B, & \mbox{if $p$ is even}, \\
1, & \mbox{if $p$ is odd}. 
\end{cases}
\eeq
This is an idealization of the blockwise covariance structures in real applications. We may generalize \eqref{model-X} to allow unequal block sizes and unequal off-diagonal entries, but we keep the current form for convenience of presentation. Model~\eqref{model-X} is also closely connected to the random designs in compressed sensing \citep{donoho2006compressed}. Write $X=[X_1,X_2,\ldots,X_n]'$. Suppose $X_1, X_2,\ldots,X_n$ are iid generated from ${\cal N}\bigl(0,\; n^{-1}\Sigma)$, where $\Sigma$ has the same form as $G$ in \eqref{model-X}. In a high-dimensional sparse setting, we have $\|\beta\|_0\ll n\ll p$. Then, $G=X'X\approx \Sigma$, and due to the blessing of sparsity of $\beta$, $G\beta\approx \Sigma\beta$. As a result, $X'y$ (sufficient statistic of $\beta$) satisfies that $X'y = G\beta + {\cal N}(0, G) \approx \Sigma\beta + {\cal N}(0, \Sigma)$, and the right hand side reduces to Model~\eqref{model-X} \citep{genovese2012comparison}. In Section~\ref{subsec:random}, we formally show that this random design setting is asymptotically equivalent to Model~\eqref{model-X}.  

Now, under model \eqref{model-beta} and model \eqref{model-X}, we have three parameters $(\vartheta,r,\rho)$. They capture the sparsity level, signal strength and design correlations, respectively. Our main results are the explicit convergence rates of Hamming error, as a function of $(\vartheta,r,\rho)$, for different methods. 
We will study six methods: (i) Lasso as in \eqref{Lasso}; (ii) Elastic net \citep{zou2005regularization}, which adds an additional $L^2$-penalty to \eqref{Lasso}, (iii) smoothly clipped absolute deviation (SCAD) \citep{fan2001variable}, which replaces the $L^1$-penalty by a non-convex penalty, (iv) thresholded Lasso \citep{zhou2009thresholding}, which further thresholds the Lasso solution, and two iterative algorithms, (v) forward selection and (vi) forward backward selection  \citep{huang2016partial}; see Section~\ref{sec:main} for a precise description of each method.
To our best knowledge, our results are the first that directly compare Hamming errors of these methods.


\section{A preview of main results and some discussions} \label{sec:Preview}

For any $\hat{\beta}$, its Hamming error is $H(\hat{\beta},\beta)=\sum_{j=1}^p 1\{\hat{\beta}_j\neq 0, \beta_j=0\}+\sum_{j=1}^p 1\{\hat{\beta}_j= 0, \beta_j\neq 0\}$. As we shall show, for any of the six methods studied here, there exists a function $h(\vartheta, r, \rho)\in [0,1]$ such that $\mathbb{E}[H(\hat{\beta},\beta)]=L_p p^{1-h(\vartheta,r,\rho)}$, where $L_p$ is a {\it multi-$\log(p)$} term. (A multi-$\log(p)$ term is such that \( L_p \cdot p^{\epsilon}\to\infty \) and \( L_p \cdot p^{ -\epsilon}\to 0 \) for any \( \epsilon > 0 \).) Since the expected number of true relevant variables is $p^{1-\vartheta}$, we are interested in three cases: 
\begin{itemize}
\item {\it Exact recovery: $h(\vartheta,r,\rho)>1$.} In this case, the expected Hamming error is $o(1)$ as $p\to\infty$. It follows that 
model selection consistency holds. 
\item {\it Almost full recovery: $\vartheta< h(\vartheta,r,\rho)<1$.} In this case, the expected Hamming error does not vanish as $p\to\infty$, but it is much smaller than 
the total number of true relevant variables. Variable selection is still satisfactory (although model selection consistency no longer holds). 
\item {\it No recovery: $h(\vartheta,r,\rho)\leq \vartheta$.} In this case, the expected Hamming error is comparable with or much larger than the total number of true relevant variables. Variable selection fails. 
\end{itemize}
We call the two-dimensional space $(\vartheta,r)$ the {\it phase space}. For each fixed $\rho$, the phase space is divided into three regions: {\it Region of Exact Recovery (ER)}, which is the subset $\{(\vartheta,r): h(\vartheta,r,\rho)>1\}$, and {\it Region of Almost Full Recovery (AFR)} and {\it Region of No Recovery (NR)} defined similarly. This gives rise to a {\it phase diagram} for each method. We denote the curve separating ER region and AFR region by $r= U(\vartheta)$ and the curve separating AFR region and NR region by $r=L(\vartheta)$; they are called the {\it upper} and {\it lower phase curves}, respectively. 
The phase diagram and phase curves are convenient ways to visualize the convergence rates of the Hamming error.

\begin{figure}[tb!]
  \centering
  \includegraphics[height=.38\textwidth]{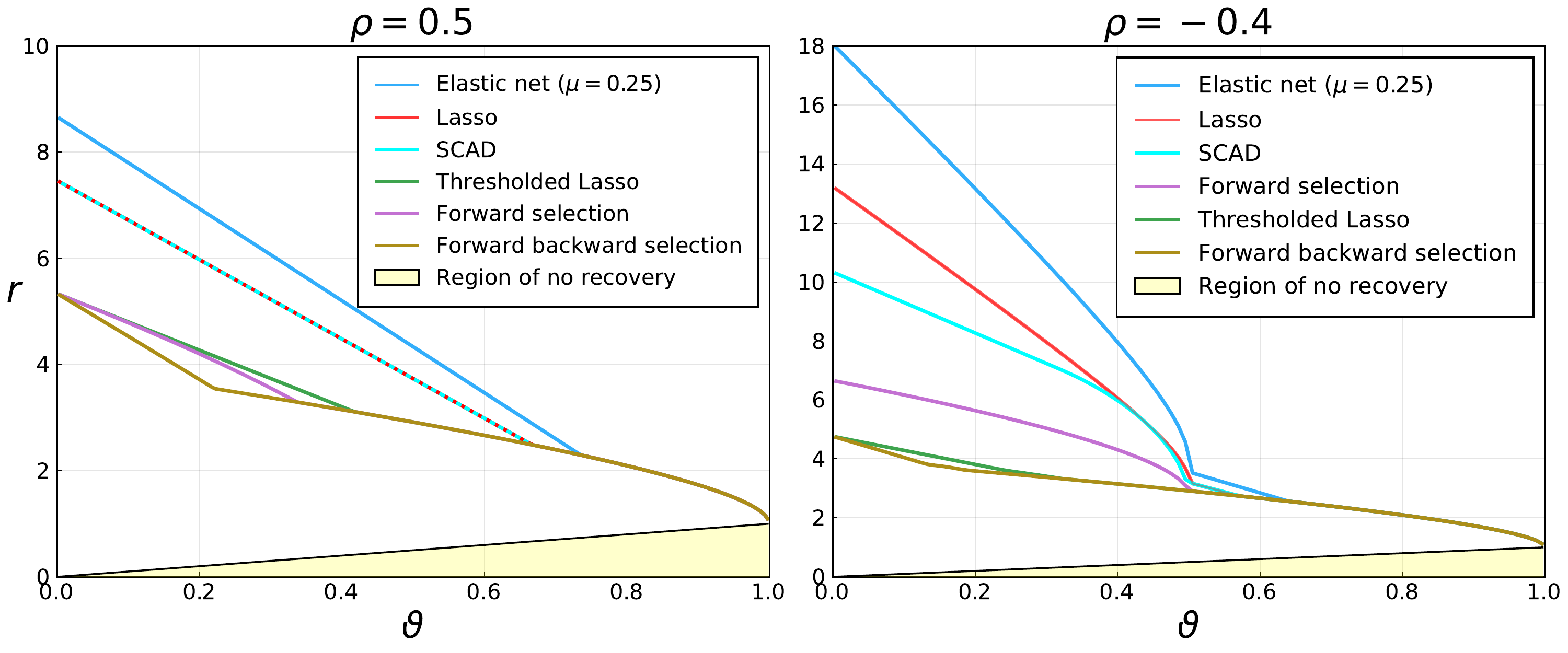} 
  \caption{Phase diagrams of six variable selection methods for a block-wise diagonal design. The parameters $(\vartheta,r,\rho)$ characterize the sparsity, signal strength, and design correlations, respectively. For each method, we plot the curve $r=U(\vartheta)$ which separates Region of Almost Full Recovery and Region of Exact Recovery (the lower this curve, the better). 
Explicit forms of $U(\vartheta)$ are in Section~\ref{sec:main}. 
On the left panel, the curves for Lasso and SCAD overlap and are displayed as a dashed line. 
How to interpret these phase curves are discussed in Section~\ref{sec:Preview}.
}\label{fig:overview}
  \end{figure}

Figure~\ref{fig:overview} shows the phase curves for the six methods (with explicit expressions in the theorems in Section~\ref{sec:main}). 
These phase curves depend on the correlation parameter $\rho$. 
Under our model, for each diagonal block $(j, j+1)$, it holds that $\mathbb{E}[x_j'y|\beta]=\beta_j+\rho \beta_{j+1}$, where $\beta_j, \beta_{j+1}\in \{0,\tau_p\}$. Therefore, a {\it positive} $\rho$ boosts the signal at each individual site (i.e., $\mathbb{E}[x_j'y|\beta]\geq \beta_j$), while a {\it negative} $\rho$ leads to potential `signal cancellation' (i.e., $\mathbb{E}[x_j'y|\beta]\leq \beta_j$). This is why the phase curves have different shapes for positive and negative $\rho$. In Figure~\ref{fig:overview}, we plot the phase curves for  $\rho=0.5$ and $\rho=-0.4$. For other positive/negative value of $\rho$, the patterns are similar.

{\bf Discussion of SCAD}. SCAD is a representative of non-convex penalization methods. There have been inspiring works that demonstrate the advantages of using a non-convex penalty (e.g., \cite{fan2004nonconcave,loh2017support}). Our results support their insights from a different angle: The phase curve of SCAD is strictly better than that of Lasso, when $\vartheta<0.5$ and $\rho<0$. 
Furthermore, our results illustrate where the advantage of SCAD comes from --- compared with Lasso, it handles `signal cancellation' better. 
To see this, we recall that `signal cancellation' only happens for $\rho<0$. Moreover, under our model \eqref{model-beta}, the expected number of signal pairs (a signal pair is  a diagonal block $\{j,j+1\}$ where both $\beta_j$ and $\beta_{j+1}$ are nonzero) is $\asymp p\epsilon_p^2=p^{1-2\vartheta}$. Therefore, `signal cancellation' becomes problematic only when $\vartheta<0.5$ and $\rho<0$ both hold. This explains why the phase curves of SCAD and Lasso are the same for the other values of $\vartheta$ and $\rho$.  
We note that in the previous studies (e.g., \cite{loh2017support}), the advantages of a non-convex penalty in handling `signal cancellation' are reflected in the weaker conditions of $(X,\beta)$ for achieving model selection consistency. 
Our results support the advantage of using a non-convex penalty by directly studying the Hamming errors and phase diagrams.  

The performance of SCAD can be further improved by adding an entry-wise thresholding on $\hat{\beta}$. 
We believe that the phase diagrams of \textit{thresholded SCAD} are better than those of SCAD itself, although the extremely tedious analysis impedes us from specific results for now.
Also, we are cautious about what to conclude from comparing SCAD and thresholded Lasso. In our settings, Lasso has no model selection consistency mainly because the signals are too weak (i.e., $r$ is not sufficiently large). In such settings, thresholded Lasso outperforms SCAD in terms of the Hamming error. However, there are cases where Lasso has no model selection consistency no matter how large the signal strength is \citep{zhao2006model}. For those cases, it is possible that SCAD is better than thresholded Lasso (see \cite{wainwright2009sharp} for a related study).

{\bf Discussion of Elastic net}.  The phase curve of Elastic net is worse than that of Lasso. As we will explain in Section~\ref{subsec:en}, Elastic net is a `bridge' between Lasso and marginal regression in our case. Since the phase curve of marginal regression is always worse than that of Lasso for the blockwise diagonal design,  we do not benefit from using Elastic net in the current setting. We must note that Elastic net is motivated by genetic applications where several correlated variables are competing as predictors, and where it is implicitly assumed that groups of correlated variables tend to be all relevant or all irrelevant \citep{zou2005regularization}. 
This is not captured by our model \eqref{model-beta}. Therefore, our results do not go against the benefits of Elastic net known in the literature, but rather our results support that the advantages of Elastic net come from `group' appearance of signal variables. 

{\bf Discussion of thresholded Lasso}.  Thresholded Lasso is a representative of improving Lasso by post-processing. There have been inspiring works that demonstrate the advantages of such a post-processing \citep{van2011adaptive, wang2017bridge,weinstein2020power}. Our results support these insights from a different angle. It is surprising (and very encouraging) that the improvement by post-Lasso thresholding is so significant. We note that Lasso is a 1-stage method, which solves a single optimization to obtain $\hat{\beta}$. By comparison, thresholded Lasso is a 2-stage method. Lasso has only one algorithm parameter $\lambda$, while thresholded Lasso has two algorithm parameters $\lambda$ and $t$ (the threshold). In Lasso, we control false positives and false negatives with the same algorithm parameter $\lambda$, and it is sometimes hard to find a value of $\lambda$ that simultaneously controls the two types of errors well. In thresholded Lasso, the two types of errors can be controlled separately by two algorithm parameters. This explains why thresholded Lasso enjoys such a big improvement upon Lasso. It inspires us to modify other 1-stage methods, such as SCAD, by adding a post-processing step of thresholding. For example, we conjecture that thresholded SCAD also has a strictly better phase diagram than that of Lasso, even for a positive $\rho$. On the other hand, thresholding is no free lunch. It leaves one more tuning parameter to be decided in practice. Our theoretical results are based on ideal tuning parameters. How to properly choose these tuning parameters in a data-driven way is an interesting question. \cite{weinstein2020power} proposes a promising approach, where they use cross-validation to select $\lambda$ and FDR control by knockoff to select $t$. We leave it for future work to study the phase diagrams with data-driven tuning parameters.

{\bf Discussion of the two iterative algorithms}. We consider two iterative algorithms, forward selection (`Forward') and forward backward selection (`FB'). The FB algorithm we analyze is a simplified version in \citet{huang2016partial}, which has only one  backward step (after all the forward steps have finished) by thresholding the refitted least-squares solution. 
Our results show that both methods outperform Lasso, and between these two methods, FB is strictly better than Forward. In the literature, there are very interesting theoretical works showing the advantages of iterative algorithms for variable selection \citep{donoho2012sparse,zhang2011adaptive}. Our results support their insights from a different angle. We discover that, for a wide range of $\rho$, FB has the best phase diagram among all the six methods. This is a very encouraging result. Of course, it is as important to note that the performance of an iterative algorithm tends to be more sensitive to the form of the design, due to its sequential nature.

\section{Main Results} \label{sec:main}
Consider model \eqref{linearM}, \eqref{model-beta}, and \eqref{model-X}, where we set $\sigma^2=1$ without loss of generality. 
Let $\mathbb{E}[H(\hat{\beta},\beta)]$ be the expected Hamming error, where 
the expectation is with respect to the randomness of  $\beta$ and $z$. 
Let $L_p$ denote a generic multi-$\log(p)$ term such that $L_pp^\epsilon\to \infty$ and $L_pp^{-\epsilon}\to 0$ for any $\epsilon>0$. 

\begin{theorem}
Under Models \eqref{linearM}, \eqref{model-beta}, and \eqref{model-X}, for each of the methods considered in this paper (Lasso, Elastic net, SCAD, thresholded Lasso, forward selection, forward backward selection, as well as marginal regression in Section~\ref{subsec:en}), there exists a function $h(\vartheta,r,\rho)$ such that $\mathbb{E}[H(\hat{\beta},\beta)]=L_pp^{1-h(\vartheta,r,\rho)}$. The explicit expressions of $h(\vartheta,r,\rho)$, which may depend on the tuning parameters of a method, are given in Theorems B.1, C.1, D.1-D.3, F.1, G.1, H.1-H.4 of the supplement. 
\end{theorem}
In the main article, to save space, we only present the expressions of the upper phase curve $U(\vartheta)=U(\vartheta;\rho)$ and the lower phase curve $L(\vartheta)=L(\vartheta;\rho)$ for each method, which are defined as follows: 
\beq \label{phase-curves}
U(\vartheta;\rho)= \inf\{r>0:  h(\vartheta,r,\rho)>1\}, \qquad L(\vartheta;\rho)=\inf\{r>0: h(\vartheta,r,\rho)>\vartheta\}.   
\eeq
These two curves describe the phase diagram: The upper phase curve $U(\vartheta)$ separates the ER region and AFR region, and the lower phase curve $L(\vartheta)$ separates the AFR region and NR region.    

\subsection{Elastic net and Lasso} \label{subsec:en}
The Elastic net \citep{zou2005regularization} is a method that estimates $\beta$ by 
\beq \label{Elastic-net}
\hat\beta^{\mathrm{EN}} =\mathrm{argmin}_{\beta} \bigl\{ \norm*{y-X\beta}^2/2 + \lambda \norm*{\beta}_1 + (\mu/2) \norm*{\beta}^2 \bigr\}. 
\eeq
Compared with Lasso, it adds an additional $L^2$-penalty to the objective function. 
Below, we fix $\mu>0$ and re-parametrize $\lambda=\sqrt{2q\log(p)}$, for some constant $q>0$. The choice of $q$ affects the exponent, $1-h(\vartheta,r,\rho)$, in the expression of $\mathbb{E}[H(\hat{\beta},\beta)]$. We choose the ideal $q$ that minimizes $1-h(\vartheta,r,\rho)$. The next theorem is proved in the supplement.  

\begin{theorem}[Elastic Net] \label{thm:elastic-net}
Under Models \eqref{linearM}, \eqref{model-beta}, and \eqref{model-X}, let $\hat{\beta}^{\mathrm{EN}}$ be the Elastic net estimator in \eqref{Elastic-net}. Fix $\mu$ and write $\eta=\rho/(1+\mu)$. Let $\lambda=\sqrt{2q\log(p)}$ with an ideal choice of $q$ that minimizes the exponent of $\mathbb{E}[H(\hat{\beta},\beta)]$. 
The phase curves are given by $L(\vartheta)=\vartheta$, and 
  \begin{equation*}
    U(\vt)= \begin{cases} 
    \max \left\{ h_1(\vt),h_2(\vt) \right\}, & \text{ when }\rho \geq 0, \\
    \max \left\{ h_1(\vt),h_2(\vt),h_3(\vt),h_4(\vt) \right\}, &\text{ when }\rho < 0,
    \end{cases} 
  \end{equation*}
where $h_1(\vt) =(1 + \sqrt{1 -\vt})^2$, $h_2(\vt) = \bigl( \frac{1 -\abs*{\eta}}{1 -\abs*{\rho}} + \frac{\sqrt{1 +\eta^2 - 2\rho\eta}}{1 -\abs*{\rho}} \bigr)^2 (1 -\vt)$, $h_3(\vt) =\frac{1}{(1 -\abs*{\rho})^2} \bigl( 1 + \frac{\sqrt{1 +\eta^2 - 2\rho\eta}}{1 +\abs*{\eta}}\sqrt{1 - 2\vt}\bigr)^2$, and $h_4(\vt) = \frac{1 +\eta^2 - 2\rho\eta}{(1 - 2\abs{\rho} + \rho \eta)_ + ^2} \bigl( \sqrt{1 -\vt} + \frac{1 -\abs{\eta}}{1 +\abs{\eta}}\sqrt{1 - 2\vt} \bigr)^2$. 
\end{theorem}

\begin{figure}[tb!]
  \centering
  \includegraphics[height=.31\textwidth]{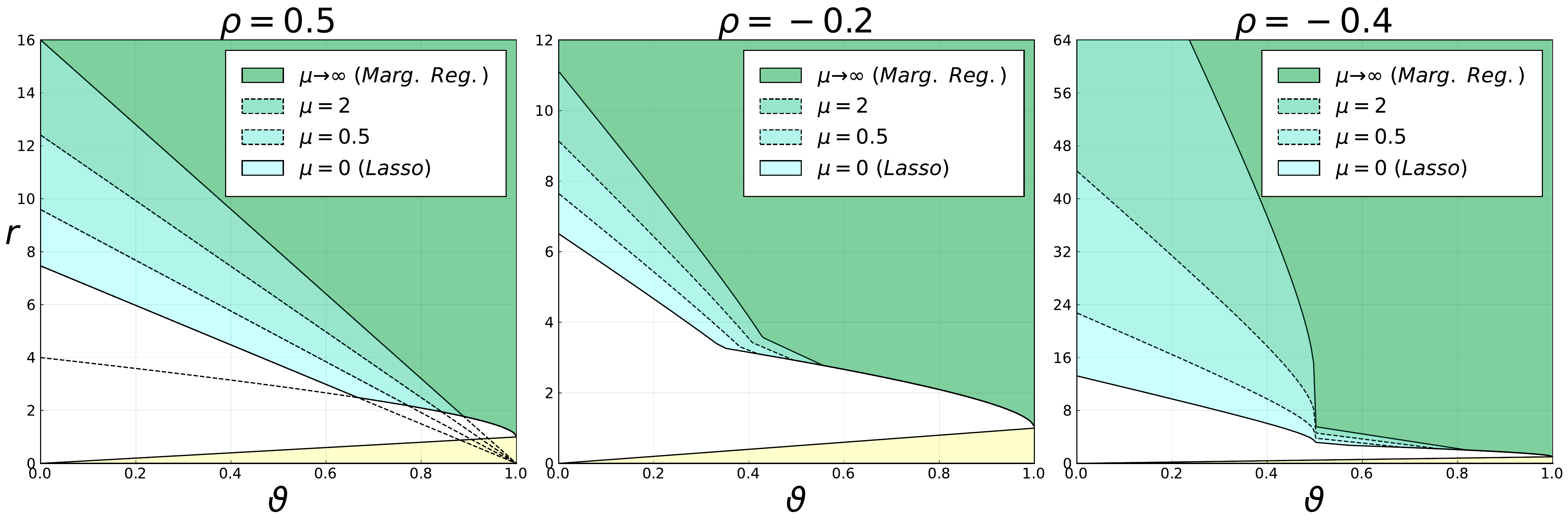} 
  \vspace{-2mm}
  \caption{The phase diagrams of Elastic net and its comparison with Lasso (notation: $\eta=\rho/(1+\mu)$).}\label{fig:elastic-net}
  \vspace{-2mm}
  \end{figure}

Lasso is a special case with $\mu=0$. By setting $\mu=0$ (equivalently, $\eta=\rho$) in   
Theorem~\ref{thm:elastic-net}, we obtain the phase curves for Lasso. They agree with the results in  \cite{ji2012ups} (but \cite{ji2012ups} does not cover Elastic net).  

To see the effect of the $L^2$-penalty, we consider an extreme case where $\mu\to\infty$. Some elementary algebra shows that $(1+\mu)\hat{\beta}^{\mathrm{EN}}$ converges to the soft-thresholding of $X'y$ at the threshold $\lambda$. In other words, as $\mu\to\infty$, Elastic net converges to marginal screening (i.e., select variables by thresholding the marginal regression coefficients). At the same time, when $\mu=0$, $(1+\mu)\hat{\beta}^{\mathrm{EN}}$ equals the Lasso estimate. Hence, Elastic net serves as a bridge between Lasso and marginal regression. 
In the setting here, the phase diagram of marginal regression is inferior to that of Lasso, and so the phase diagram of Elastic net is also inferior to that of Lasso. See the proposition below and Figure~\ref{fig:elastic-net}:

\begin{prop} \label{prop:bridge}
In Theorem~\ref{thm:elastic-net}, for each fixed $\vartheta\in (0,1)$, 
as \( \mu\to 0 \), $U(\vartheta)$ is monotone decreasing and converges to $U_{\mathrm{Lasso}}(\vt)$, which is the upper phase curve of Lasso; as $\mu\to\infty$, $U(\vartheta)$ is monotone increasing and converges to $U_{\mathrm{MR}}(\vartheta)$, which is the upper phase curve of marginal regression. Furthermore, when \( \rho \leq  - \frac{1}{2} \), $U_{\mathrm{MR}}(\theta)=\infty$ for all \( 0 < \vt \leq \frac{1}{2} \) (i.e. exact recovery is impossible to achieve no matter how large $r$ is).
\end{prop}

\subsection{Smoothly clipped absolute deviation penalty (SCAD)}   
SCAD \citep{fan2001variable} is a non-convex penalization method. For any $a>2$, it defines a penalty function $q_{\lambda}(\theta)$ on $(0,\infty)$ by $q_{\lambda}(\theta)=\int_0^\theta q_{\lambda}'(t)dt$, where $
q_{\lambda}'(\theta)=\lambda\bigl\{I(\theta \leq \lambda)+\frac{(a \lambda-\theta)_{+}}{(a-1) \lambda} I(\theta>\lambda)\bigr\}$. 
The resulting penalty function $q_{\lambda}(\cdot)$ coincides with the $L^1$-penalty in $(0,\lambda]$ and becomes a constant in $[a\lambda,\infty)$. 
Let $Q_{\lambda}(\beta)=\sum_{j=1}^pq_{\lambda}(|\beta_j|)$. Then, SCAD estimates $\beta$ by  
\beq \label{SCAD}
\hat\beta^{\mathrm{SCAD}} =\mathrm{argmin}_{\beta}\bigl\{\norm*{y-X\beta}^2/2 + Q_{\lambda}(\beta) \bigr\}.  
\eeq
The following theorem is proved in the supplemental material (see Figure~\ref{fig:scad}, left panel): 
\begin{theorem}[SCAD]   \label{thm:SCAD}
Under Models \eqref{linearM}, \eqref{model-beta}, and \eqref{model-X}, let $\hat{\beta}^{\mathrm{SCAD}}$ be the SCAD estimator in \eqref{SCAD}. Fix \( a\in (2, \frac{2}{1 -\abs*{\rho}} )\). Let $\lambda=\sqrt{2q\log(p)}$ with an ideal choice of $q$ that minimizes the rates of convergence of the expected Hamming error. The phase curves are given by $L(\vartheta)=\vartheta$, and 
  \begin{equation*}
    U(\vt)= \begin{cases} 
    \max \left\{ h_1(\vt),h_2(\vt), h_3(\vt) \right\}, & \text{ when }\rho \geq 0, \\
    \max \left\{ h_1(\vt),h_2(\vt),h_4(\vt),h_5(\vt) \right\}, &\text{ when }\rho < 0,
    \end{cases} 
  \end{equation*}
 where $h_1(\vt)=(1 + \sqrt{1 -\vt})^2$,  and $h_2(\vt)=\bigl( 1 + \sqrt{\frac{1 +\abs{\rho}}{1 -\abs{\rho}}} \bigr)^2 (1 -\vt)$,  $h_4(\vt) = \bigl( \sqrt{\frac{1-2 \vartheta}{1 -|\rho|^{2}}}+\frac{1}{1-|\rho|} \bigr)^2$, $h_3(\vt)=\bigl( \frac{3 +\rho}{2(1 -\rho^2)}\sqrt{\frac{1 +\rho}{1 -\rho}}\sqrt{1 -\vt} + \frac{1}{2}\sqrt{\frac{2(1 - 2\vt)}{1 +\rho} - \frac{(1 -\vt)}{(1 -\rho)^2}} \bigr)^2$, and 
   \begin{align*}
    h_5(\vt) = 
      \begin{cases} 
        \bigl( \frac{5 + 3|\rho|}{1 -|\rho|} \bigr)(1 -\vt), &  \text{ if }\sqrt{\frac{1-2 \vartheta}{1-\vartheta}} \geq \frac{3-4|\rho|-3\rho^2}{(1-|\rho|)} \sqrt{\frac{1+|\rho|}{5+3|\rho|}},\\
        \frac{1}{(1 -\abs*{\rho})^2} \bigl(\sqrt{\frac{1 +\abs{\rho}}{1 -\abs{\rho}}} \sqrt{1 -\vt} + \sqrt{\frac{1 -\abs{\rho}}{1 +\abs{\rho}}}\sqrt{1 - 2\vt} \bigr)^2, &  \text{ if } \sqrt{\frac{1-2 \vartheta}{1-\vartheta}} \leq \frac{(1+|\rho|)(1-2 |\rho|)}{1-|\rho|}, \\
        h_6(\vt) & \text{ other wise},
      \end{cases} 
  \end{align*}
\begin{equation*}
h_6(\vt) = \left\{ \sqrt{\frac{1-2 \vartheta}{1-\rho^{2}}}+\frac{\frac{1-2 |\rho|}{1-|\rho|} \sqrt{\frac{1-2 \vartheta}{1-\rho^{2}}}+\sqrt{\bigl[\bigl(\frac{1-2 |\rho|}{1-|\rho|}\bigr)^{2}+ \frac{1-|\rho|}{1+|\rho|}\bigr](1-\vartheta)-\frac{1-2 \vartheta}{(1+|\rho|)^{2}}}}{(1-|\rho|)\bigl[\bigl(\frac{1-2 |\rho|}{1-|\rho|}\bigr)^{2}+\frac{1-|\rho|}{1+|\rho|}\bigr]} \right\}^2.
  \end{equation*}
 \end{theorem}
 
\begin{figure}[tb!]
 \centering
  \includegraphics[width=.92\textwidth]{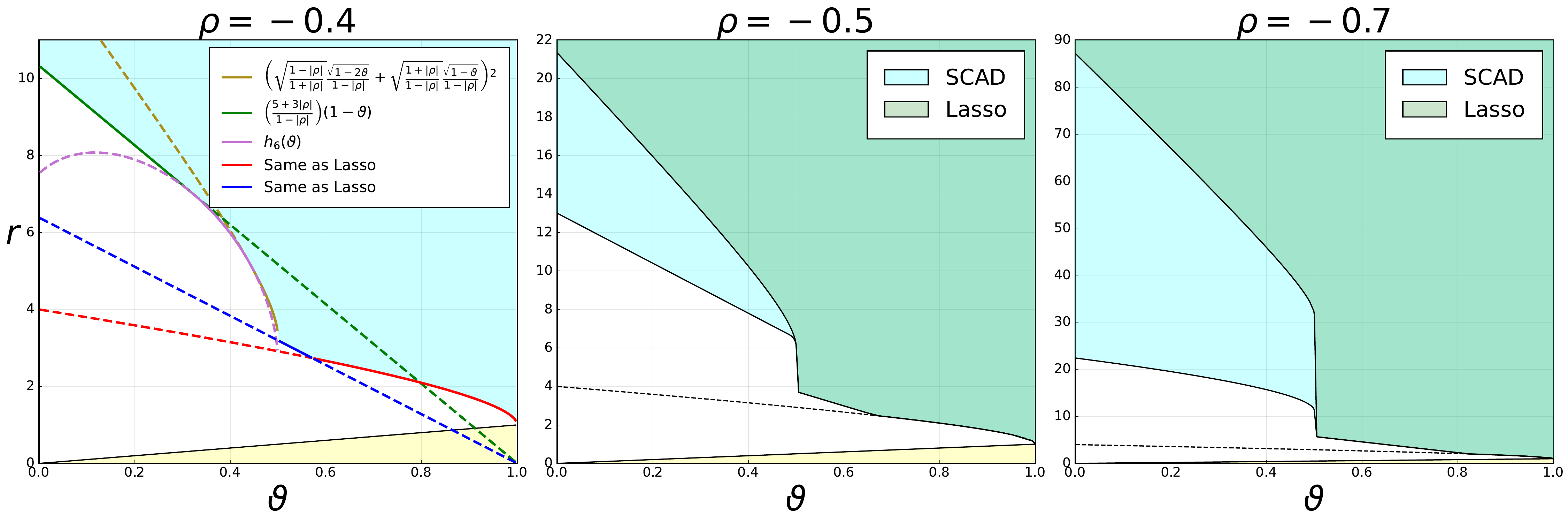} 
  \vspace{-2mm}
  \caption{Left: Phase curves of SCAD. Middle and Right: Comparison of SCAD and Lasso. }\label{fig:scad}
  \vspace{-2mm}
  \end{figure}


Note that the phase curves of Lasso are given in Theorem~\ref{thm:elastic-net} by setting $\eta=\rho$. We compare SCAD with Lasso.  When $\rho< 0$, the upper phase curve in Theorem~\ref{thm:SCAD} is strictly lower than that of Lasso (see Figure~\ref{fig:scad}, middle and right panels). When $\rho\geq 0$, the upper phase curve in  Theorem~\ref{thm:SCAD} is sometimes higher than that of Lasso. Note that we restrict $a<\frac{2}{1-|\rho|}$ in Theorem~\ref{thm:SCAD}. In fact, a larger $a$ may be preferred for $\rho\geq 0$. The next proposition is about using an optimal $a$. 
\begin{prop} \label{prop:SCAD-larger-a}
In the SCAD estimator, we choose $a=a^*$ and $\lambda=\sqrt{2q^* \log(p)}$ such that $(a^*, q^*)=(a^*(\vartheta,r,\rho), q^*(\vartheta,r,\rho))$ minimize the rates of convergence of the expected Hamming error among all choices of $(a,q)$.  Let $U^*(\vt)$ be the resulting upper phase curve for SCAD. 
Then, $U(\vt)=U_{\mathrm{Lasso}}(\vt)$ when $\rho\geq 0$, and $U(\vt)< U_{\mathrm{Lasso}}(\vt)$ when $\rho<0$. 
\end{prop}


The phase curves of SCAD are insensitive to the choice of $a$. When $a<0$, the optimal $a^*$ can be any value in $(2, \frac{2}{1-|\rho|})$. When $\rho\geq 0$, there exists a constant $c=c(\vartheta,\rho)$ such that the optimal $a^*$ is any value in $(c,\infty)$. As $a\to\infty$, the SCAD penalty reduces to the $L^1$-penalty. This explains why the phase curve of SCAD is the same as that of Lasso when $\rho\geq 0$.

\subsection{Thresholded Lasso}
Let $\hat{\beta}^{\mathrm{Lasso}}$ be the Lasso estimator in \eqref{Lasso}. The thresholded Lasso estimator $\hat{\beta}^{\mathrm{TL}}$ is obtained by applying coordinate-wise hard-thresholding to the Lasso estimator:  
\beq \label{thresh-Lasso}
\hat{\beta}^{\mathrm{TL}}_j = \hat{\beta}^{\mathrm{Lasso}}\cdot 1\{|\hat{\beta}^{\mathrm{Lasso}}|>t\}, \qquad 1\leq j\leq p. 
\eeq
%
%

\begin{theorem}[Thresholded Lasso]   \label{thm:thresh-lasso}
Under Models \eqref{linearM}, \eqref{model-beta}, and \eqref{model-X}, let $\hat{\beta}^{\mathrm{TL}}$ be the thresholded Lasso estimator in \eqref{thresh-Lasso}. Let $\lambda=\sqrt{2q\log(p)}$ and $t=\sqrt{2w\log(p)}$ with the ideal $(q, w)$ that minimize the exponent of the expected Hamming error. The phase curves are given by $L(\vartheta)=\vartheta$, and 
\begin{equation*}
    U(\vt)= \begin{cases} 
    \max \left\{ h_1(\vt),h_2(\vt) \right\}, & \text{ when }\rho \geq 0, \\
    \max \left\{ h_1(\vt),h_2(\vt),h_3(\vt) \right\}, &\text{ when }\rho < 0,
    \end{cases} 
\end{equation*}
where $h_1(\vartheta)= (1 + \sqrt{1 -\vt})^2$, $h_2(\vartheta)=\frac{4(1-\vartheta)}{1-\rho^{2}}$, and $h_3(\vartheta)= \bigl(1+\frac{1+|\rho|}{2} \sqrt{\frac{1-\vartheta}{1-\rho^{2}}}+\frac{1-|\rho|}{2} \sqrt{\frac{1-2 \vartheta}{1-\rho^{2}}}\bigr)^2$. 
%
\end{theorem}

\begin{figure}[tb!]
 \centering
  \includegraphics[width=.91\textwidth]{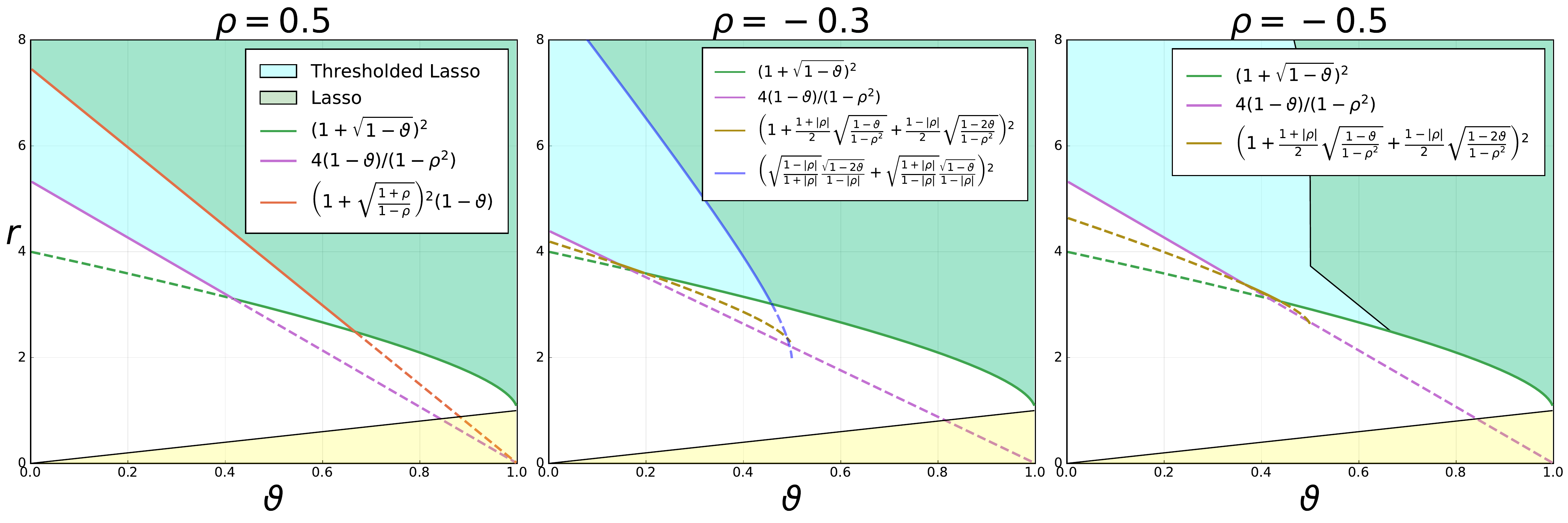} 
  \vspace{-2mm}
  \caption{Comparison of the phase diagrams of thresholded Lasso and Lasso.}\label{fig:thres.lasso}
  \vspace{-3mm}
  \end{figure}

See Figure~\ref{fig:thres.lasso} for a comparison with Lasso (a special case of $t=0$). With the flexibility of using an optimal $t$, the phase diagram of thresholded Lasso is always better than that of Lasso. 

Theorem~\ref{thm:thresh-lasso} also gives other interesting facts about thresholded Lasso. First, the shape of phase curves is much less affected by the sign of $\rho$. This differs from Lasso, Elastic net, and SCAD, for which the shape of phase curves is significantly different for positive and negative $\rho$. 
Second, the optimal $\lambda$ in thresholded Lasso is considerably smaller than the optimal $\lambda$ in Lasso (it can be seen from the proofs of Theorem~\ref{thm:thresh-lasso} and Theorem~\ref{thm:elastic-net}). This is because the $\lambda$ in thresholded Lasso only serves to control false negatives, but the $\lambda$ in Lasso is used to simultaneously control false positives and false negatives, hence, cannot be too small. 
We observe the same phenomenon in simulations; see Section~\ref{sec:simu}.

\subsection{Forward selection and forward backward selection}
Forward selection is a classical textbook method for variable selection. Write $X=[x_1,x_2,\ldots,x_p]$, where $x_i\in \mathbb{R}^n$ for $1\leq i\leq p$. For any subset $A\subset\{1,2,\ldots,p\}$, let $P^{\bot}_{A}$ be the projection  onto the orthogonal complement of the linear space spanned by \( \{ x_i:i\in  A \}  \). 
Given a threshold $t>0$, the forward selection algorithm initializes with $S_0=\emptyset$ and $\hat{r}_0=y$.  At the $k$th iteration, compute
\[
i^* = \mathrm{argmax}_{i\notin S_{k - 1}}\abs{x_i'\hat r_{k - 1}}, \qquad \delta = 
         \abs{x_{i^* }'\hat r_{k - 1} }/\norm{ P^{\bot}_{S_{k - 1}} x_{i^*}}. 
\]
If $\delta>t$, compute $S_k=S_{k-1}\cup \{i^*\}$ and $\hat{r}_{k}=P^{\bot}_{S_k}y$; otherwise, output $\hat{\beta}^{\mathrm{forward}}$ as the least-squares estimator restricted to $S_{k-1}$. 
The stopping rule of $\delta\leq t$ is equivalent to measuring the decrease of the residual sum of squares. 
The following theorem is proved in the supplemental material: 
\begin{theorem}[Forward Selection]  \label{thm:forward}
Under Models \eqref{linearM}, \eqref{model-beta}, and \eqref{model-X}, let $\hat{\beta}^{\mathrm{forward}}$ be the estimator from forward selection. Let $t=\sqrt{2q\log(p)}$ with the ideal $q$ that minimizes the exponent of the expected Hamming error. The phase curves are given by $L(\vartheta)=\vartheta$, and 
\begin{equation*}
    U(\vt)= \begin{cases} 
    \max \left\{ h_1(\vt),h_2(\vt), h_3(\vt) \right\}, & \text{ when }\rho \geq 0, \\
    \max \left\{ h_1(\vt),h_2(\vt),h_3(\vt), h_4(\vt) \right\}, &\text{ when }\rho < 0,
    \end{cases} 
\end{equation*}
with $h_1(\vt)$=$(1+\sqrt{1-\vartheta})^2$, $h_2(\vt)=\frac{2(1-\vartheta)}{1-|\rho|}$, $h_3(\vt)=\frac{(1 +\sqrt{1 - 2\vt})^2}{1 -\rho^2}$, $h_4(\vt)=\bigl( \sqrt{\frac{1-2 \vartheta}{2(1-|\rho|)}}+\frac{1}{1-|\rho|} \bigr)^2$. 
\end{theorem}

Forward backward selection (FB) modifies forward selection by allowing to drop variables. 
We use the FB algorithm in \cite{huang2016partial}, where the backward step is conducted after all the forward steps are finished. For a threshold $v>0$, it applies entry-wise thresholding on $\hat{\beta}^{\mathrm{forward}}$: 
\beq \label{FoBa}
\hat{\beta}^{\mathrm{FB}}_j = \hat{\beta}^{\mathrm{forward}}_j\cdot 1\{ |\hat{\beta}^{\mathrm{forward}}_j|>v \}, \qquad 1\leq j\leq p. 
\eeq
\begin{theorem}[Forward Backward Selection]  \label{thm:forward-backward}
Under Models \eqref{linearM}, \eqref{model-beta}, and \eqref{model-X}, let $\hat{\beta}^{\mathrm{FB}}$ be the estimator from forward selection. Let $t=\sqrt{2q\log(p)}$ and $v=\sqrt{2u\log(p)}$ with the ideal $(q, u)$ that minimize the exponent of the expected Hamming error. When $\rho\geq 0$, the phase curves are given by $L(\vartheta)=\vartheta$, and 
\begin{equation*}
    U(\vt)= 
    \max \left\{ h_1(\vt),h_2(\vt), h^*_3(\vt) \right\}, 
\end{equation*} 
where $h_1(\vt)$ and $h_2(\vt)$ are the same as in Theorem~\ref{thm:forward} and $h_3^*(\vt)=\frac{(\sqrt{1 -\vt} +\sqrt{1 - 2\vt})^2}{1 -\rho^2}$. 
When $\rho<0$, 
\begin{equation*}
  U(\vt) \leq \max \left\{ g_1(\vt),g_2(\vt),g_3(\vt),g_4(\vt) \right\},
\end{equation*}
where \( g_1(\vt) = (v_{\min}(\vt) + \sqrt{1 -\vt})^2 \), $g_2(\vt)=\frac{2(1-\vartheta)}{1-|\rho|}$, $g_3(\vt)=\bigl( \sqrt{\frac{1-2\vt}{1-\rho^2}} + v_{\min}(\vt)\bigr)^2$, $g_4(\vt)=\bigl( \sqrt{\frac{1-2 \vartheta}{2(1-|\rho|)}}+\frac{t_{\min}(\vt)}{1-|\rho|} \bigr)^2$, \( v_{\min}(\vt) =\max \bigl\{ 1,\sqrt{\frac{1 -\vt}{1 -\rho^2}} \bigr\} \), and \( t_{\min}(\vt) =\max \bigl\{ \frac{\sqrt{2}}{2},\frac{v_{\min}(\vt)}{1 +|\rho|/\sqrt{1 -\rho^2}} \bigr\} \). 
\end{theorem}

Theorem~\ref{thm:forward-backward} gives $U(\vartheta)$ for $\rho\geq 0$ and an upper bound of it for $\rho<0$. Combining it with Theorems~\ref{thm:elastic-net} and \ref{thm:forward}, we conclude that the upper phase curve of FB is always better than those of Lasso and forward selection (for $\rho<0$, the upper bound here is already better than $U(\vartheta)$ for the other two methods).   

   


\begin{figure}[tb!]
\centering
\includegraphics[height=.3\textwidth]{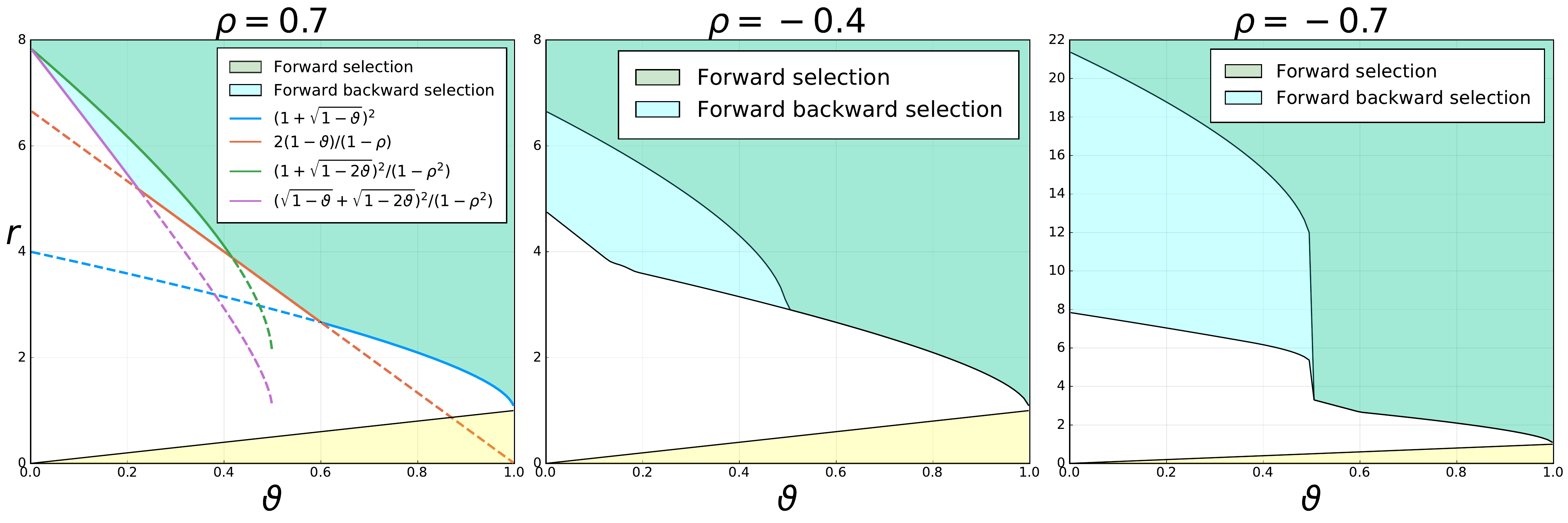}
\vspace{-2mm}
\caption{The phase diagrams of forward selection and forward backward selection.}\label{fig:foba}
\vspace{-3mm}
\end{figure}


We remark that we did obtain the exact phase curve for $\rho<0$ in the proof of Theorem~\ref{thm:forward-backward}. It is just too complicated and space-consuming to present it in the main text.  
However, given specific values of $(\vartheta, \rho)$, we can always plot the exact phase curve using the (complicated) formulas in the supplement. 
In Figures~\ref{fig:overview} and \ref{fig:foba}, 
the phase curves of FB are indeed the exact ones.



\subsection{Connection to the random design model}\label{subsec:random}
Consider the random design as mentioned in Section~\ref{sec:Intro}. The {\it minimax Hamming error} is   $H^*(\vartheta,r,\rho)=\inf_{\hat{\beta}}\mathbb{E}[H(\hat{\beta},\beta)]$, where the infimum is taken over all methods $\hat{\beta}$ and the expectation is with respect to the randomness of $(X,\beta,z)$. We can define $H^*(\vartheta,r,\rho)$ in the same way for our current model \eqref{model-X}. The minimax Hamming error is related to the statistical limit of the model setting, but not any specific method. The next theorem shows that, when $n\gg s_p=p^{1-\vartheta}$ (we allow both $p\leq n$ and $p>n$), the convergence rate of the minimax Hamming error is the same under two models.

\begin{theorem} \label{thm:equivalence}
Under Models \eqref{linearM} and \eqref{model-beta}, suppose $X$ is independent of $(\beta,z)$ and its rows are iid generated from  ${\cal N}(0, n^{-1}\Sigma)$, with $\Sigma$ having the same form as $G$ in \eqref{model-X}. Suppose $n=p^\omega$, with $\omega>1-\vartheta$ (note: this allows $\omega<1$, which corresponds to $n\ll p$). There exists a number $h^{**}(\vartheta,r,\rho)$ such that the minimax Hamming error satisfies that $H^*(\vartheta,r,\rho)=L_pp^{1-h^{**}(\vartheta,r,\rho)}$. Furthermore, if we instead have $X'X=\Sigma$ (i.e., model \eqref{model-X}), then it also holds that $H^*(\vartheta,r,\rho)=L_pp^{1-h^{**}(\vartheta,r,\rho)}$. 
\end{theorem}

\section{Simulations} \label{sec:simu}

In Experiments 1-3, $(n,p)=(1000,300)$. In Experiment 4, $(n,p) = (500,1000)$.

\textbf{Experiment 1} (block-wise diagonal designs).  We generate $(X,\beta)$ as in \eqref{model-beta}-\eqref{model-X}. For each method, we select the ideal tuning parameters that minimize the average Hamming error over 50 repetitions. The averaged Hamming errors and its standard deviations under the ideal tuning parameters over 500 repetitions are reported below. The results are consistent with the theoretical phase diagrams (see Figure~\ref{fig:overview}). E.g., thresholded Lasso and forward backward selection are the two methods that perform the best; Lasso is more unsatisfactory when $\rho<0$; SCAD improves Lasso when $\rho<0$.  

\begin{table}[hbt]
\centering
\begin{center}
\scalebox{0.9}{
\begin{tabular}{ccc|cccccc}
    $\rho$  & $\vartheta$ & $r$   & \multicolumn{1}{l}{Lasso}         & \multicolumn{1}{l}{ThresLasso}    & \multicolumn{1}{l}{ElasticNet}    & \multicolumn{1}{l}{SCAD}          & \multicolumn{1}{l}{Forward}       & \multicolumn{1}{l}{FoBackward}  \\ \hline
    0.5  & 0.1   & 1.5 & 11.57 (3.59) & {\bf 10.48} (3.34) & 11.57 (3.31) & 11.72 (3.33) & 14.88 (4.12) & 13.35 (3.90) \\
    0.5  & 0.1   & 4   & 1.00 (1.00)  & {\bf 0.42} (0.65)  & 1.03 (1.00)  & 1.00 (0.96)  & 0.66 (0.84)  & 0.51 (0.73)  \\
    -0.5 & 0.1   & 1.5 & 35.62 (5.09) & 15.62 (4.06) & 35.48 (5.64) & 25.87 (5.04) & 19.48 (4.61) & {\bf 14.82} (3.82)\\
    \hline
    \end{tabular}}
\end{center}
\caption{Experiment 1 (block-diagonal designs). $(n,p)=(1000,300)$.} 
\end{table}

\noindent \textbf{Experiment 2} (general designs). In the Toeplitz design, we let $(X' X)_{i,j}=0.7^{|i-j|}$ and set $(\vartheta,r)=(0.1,2.5)$.  In the factor model design, we let $X'X=BB'-\text{diag}(BB')+ I_p$, where entries of $B\in \mathbb{R}^{p\times 2}$ are {\it iid} from $\mathrm{Unif}(0,0.6)$, and set $(\vartheta,r)=(0.1,1.5)$. Same as in Experiment 1, we use the ideal tuning parameters. The averaged Hamming errors and its standard deviations are reported below. 
The Toeplitz design is a setting where each variable is only highly correlated with a few other variables. The factor model design is a setting where a variable is (weakly) correlated with all the other variables. 
The results are quite similar to those in Experiment 1. This confirms that the insight gained in the study of the block-wise diagonal design continues to apply to more general designs. 

\begin{table}[hbt]
       \centering
    \scalebox{0.9}{
\begin{tabular}{c|cccccc}
design       & Lasso        & ThresLasso   & ElasticNet   & SCAD         & Forward      & FoBackward \\ \hline
Toeplitz     & 47.15 (6.32) & {\bf 22.02} (5.31) & 47.40 (6.41) & 24.61 (5.70) & 30.77 (6.18) & 22.93 (5.44) \\
Factor model & 21.14 (4.52) & {\bf 15.90} (3.87) & 21.20 (4.45) & 19.68 (4.23) & 20.04 (4.34) & 16.13 (3.76)\\
\hline
\end{tabular}}
\caption{Experiment 2 (general designs). $(n,p)=(1000,300)$.} 
\end{table}

\textbf{Experiment 3} (tuning parameters). Fix $(\vt,r)=(0.1, 1.5)$ and $\rho\in\{\pm 0.5\}$ in the block-wise diagonal design. We study the effect of tuning parameters in Lasso,  thresholded Lasso (ThreshLasso), forward selection (ForwardSelect), and forward backward selection (FB).
In (a)-(b), we show the heatmap of averaged Hamming error (over 50 repetitions) of ThreshLasso for a grid of $(t, \lambda)$; when $t=0$, it reduces to Lasso. In (c)-(d), we show the Hamming error of FB for a grid of $(v,t)$; when $v=0$, it reduces to ForwardSelect.  Cyan points are theoretically optimal tuning parameters (formulas are in proofs of theorems). Red points are empirically optimal tuning parameters that minimize the averaged Hamming error. 
The theoretical tuning parameter values are quite close to the empirically optimal ones. Moreover, the optimal $\lambda$ in ThreshLasso is smaller than the optimal $\lambda$ in Lasso.  

\begin{figure}[htb!]
\hspace*{-10pt}
\subfloat[Lasso and ThreshLasso ($\rho=0.5$)]{\includegraphics[height=3cm, width=.24\textwidth]{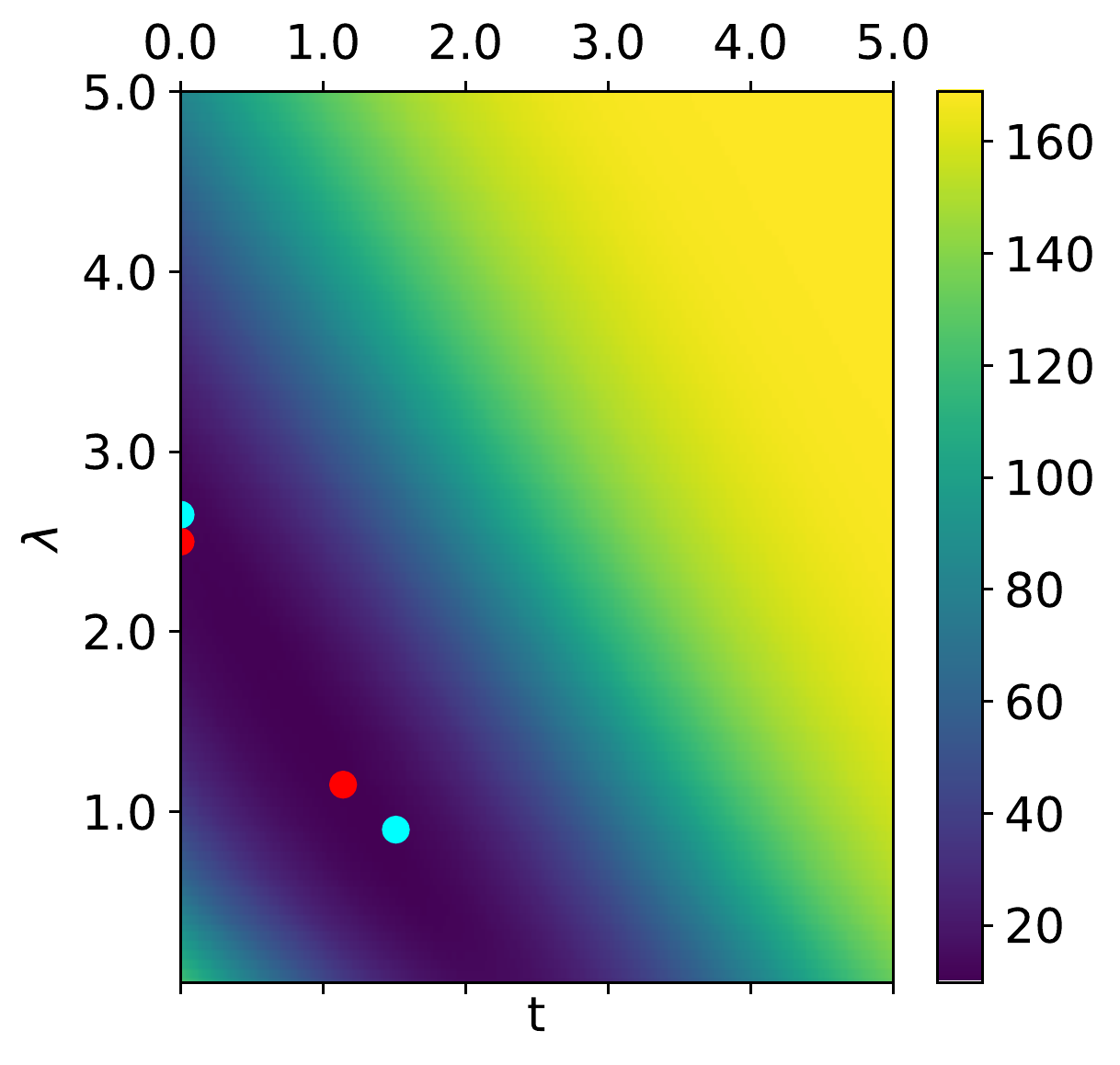}}
\hspace*{5pt}
\subfloat[Lasso and ThreshLasso ($\rho=-0.5$)]{\includegraphics[height=3cm, width=.24\textwidth]{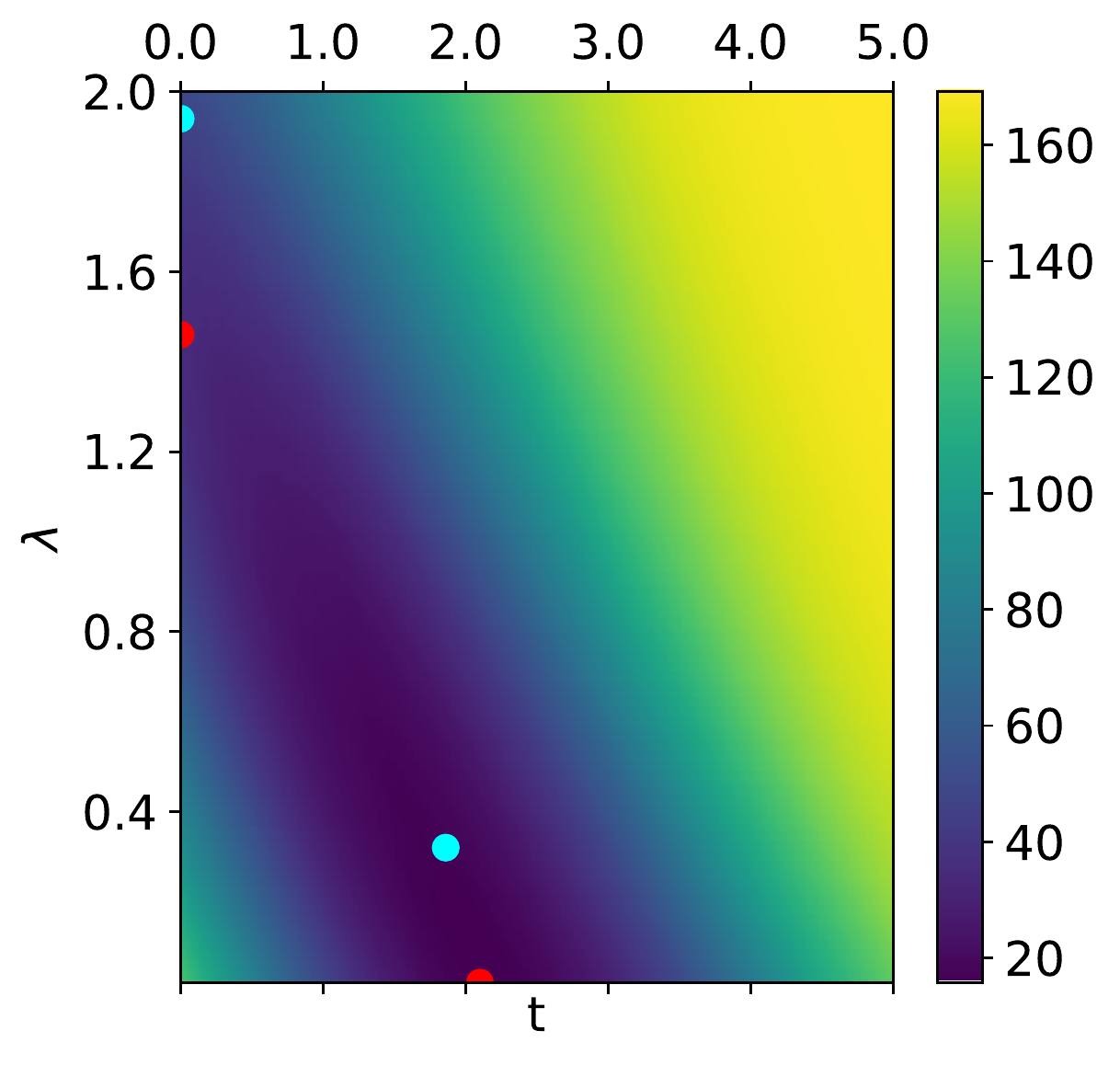}}
\hspace*{5pt}
\subfloat[ForwardSelect and FB ($\rho=0.5$)]{\includegraphics[height=3cm, width=.24\textwidth]{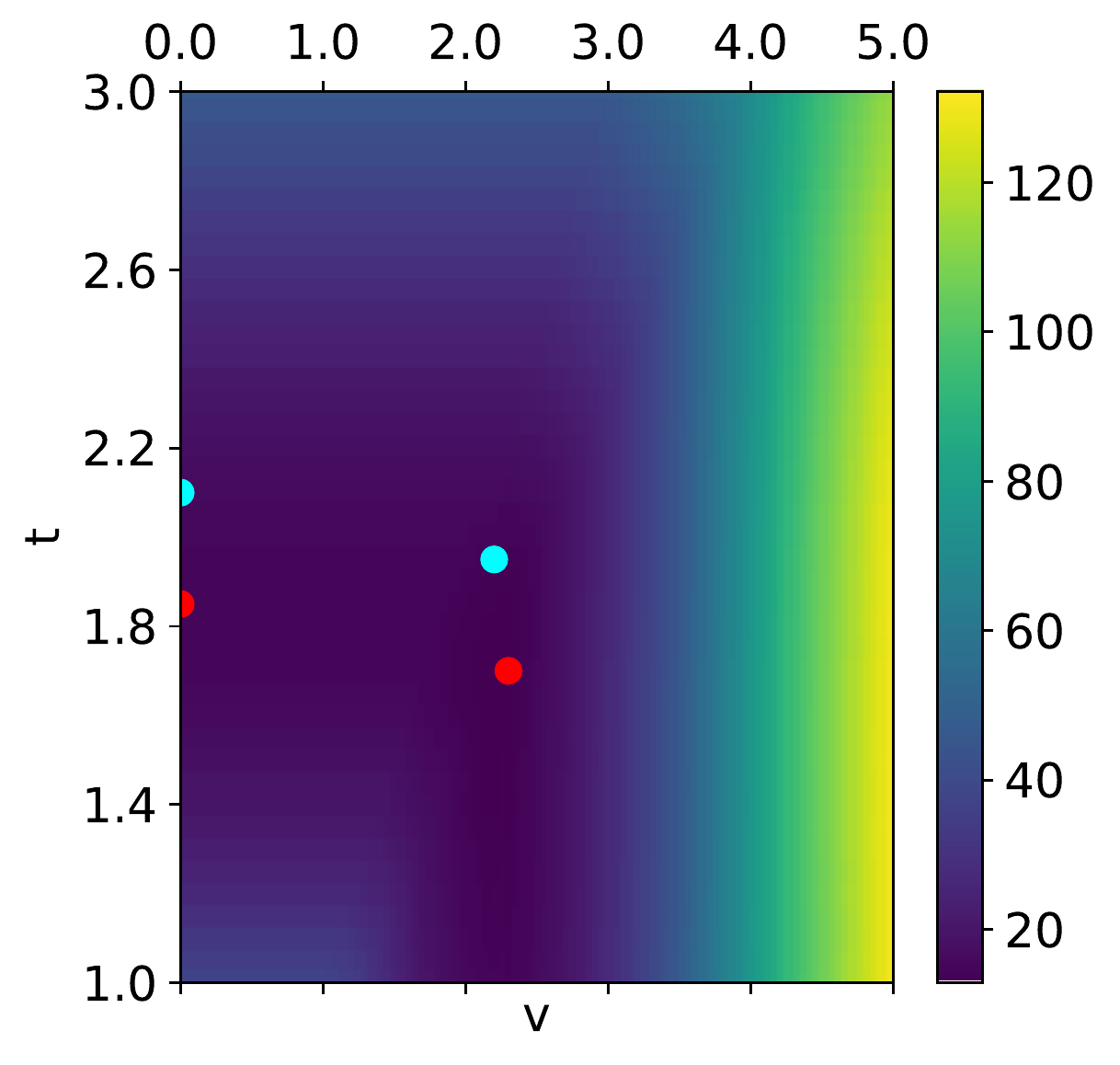}}
\hspace*{5pt}
\subfloat[ForwardSelect and FB ($\rho=-0.5$)]{\includegraphics[height=3cm, width=.24\textwidth]{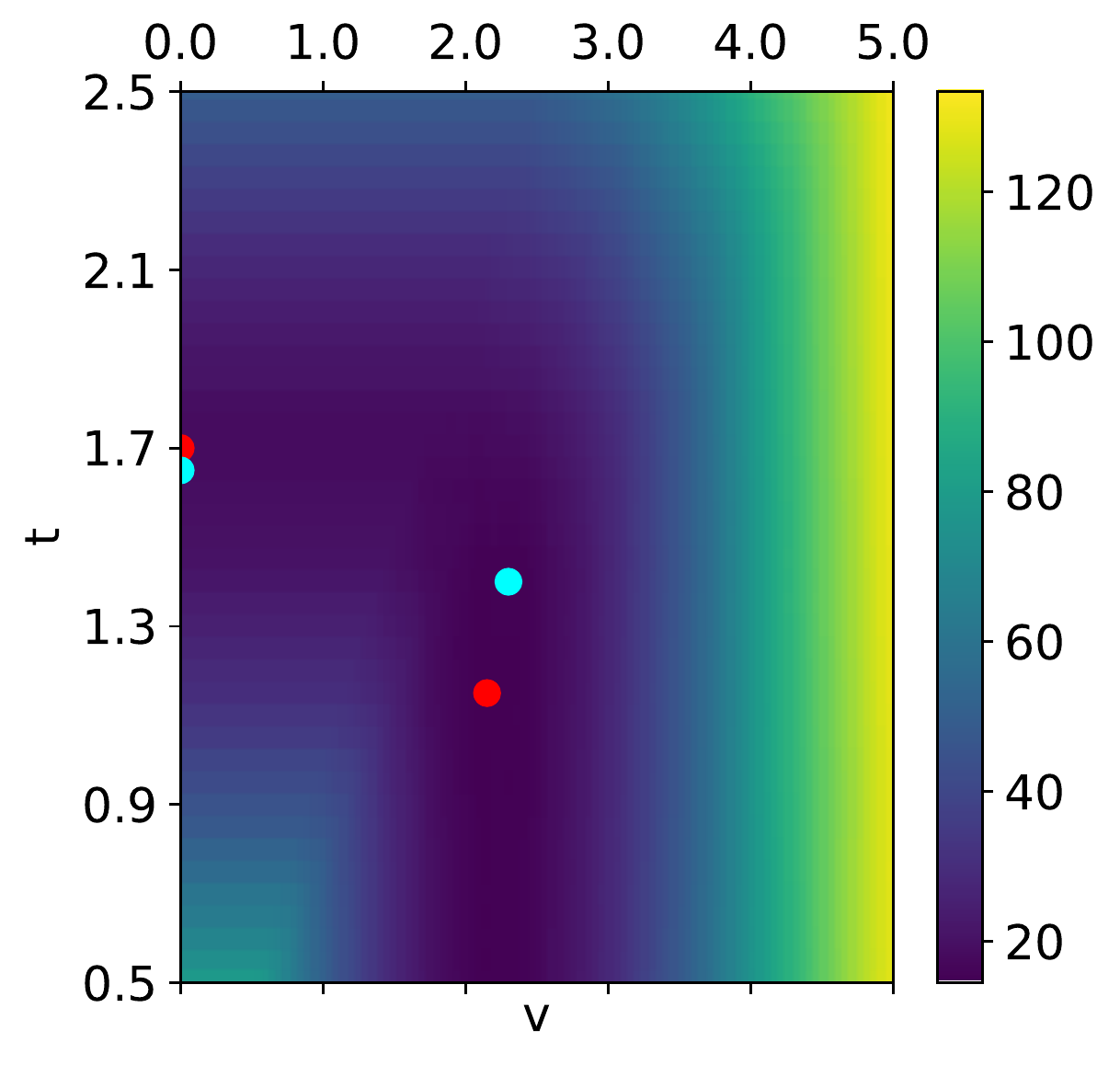}}
\caption{Experiment 3 (effects of tuning parameters). In all plots, cyan points are computed from the formulas in our theory, and red points are the empirically best tuning parameters (they minimize the average Hamming error over 500 repetitions). In (a)-(b), the cyan/red points with $t=0$ correspond to Lasso, and the other two are for thresholded Lasso. In (c)-(d), the cyan/red points with $t=0$ correspond to forward selection, and the other two are for forward backward selection.}
\end{figure}

\textbf{Experiment 4} ($p>n$ and random designs).  Fix \( (n,p,\vartheta,r) = (500,1000,0.5,1.5)\). We simulate data from the random design setting in Theorem~\ref{thm:equivalence}. We study the average Hamming error over 500 repetitions (tuning parameters are set in the same way as in Experiment 1). See Table~\ref{table:experiment.4}. We have some similar observations as before: e.g., ThreshLasso and FoBackward are still the best two, 
\begin{table}[hb!]
  \centering
  \begin{center}
  \scalebox{0.9}{
  \begin{tabular}{ccc|cccccc}
      $\rho$  &$\vartheta$& $r$ &  \multicolumn{1}{l}{Lasso}         & \multicolumn{1}{l}{ThresLasso}    & \multicolumn{1}{l}{ElasticNet}    & \multicolumn{1}{l}{SCAD}          & \multicolumn{1}{l}{Forward}       & \multicolumn{1}{l}{FoBackward}  \\ \hline
      0.5  & 0.5 & 1.5  & 16.02 (5.52) & \textbf{9.83} (4.08) & 13.92 (5.12) & 15.98 (6.28) & 11.74 (5.55) & 9.84 (4.93) \\
      -0.5 & 0.5 & 1.5  & 18.49 (6.03) & 10.50 (4.23) & 15.18 (5.64) & 18.12 (6.00) & 12.00 (5.67) & \textbf{10.41} (5.03) \\
      \hline
      \end{tabular}}
  \end{center}
  \caption{Experiment 4 ($p>n$ and random designs).} 
\label{table:experiment.4}
  \end{table}

\section{Conclusion} \label{sec:discuss}
Most papers on variable selection focus on one method and study its properties in a relatively broad setting.
In contrast, we focus on a relatively narrow setting but study a variety of different methods. 
Our motivation is to facilitate a direct comparison of main-stream approaches for variable selection. 
Although the model we use seems idealized, by varying the parameters, it already accommodates many different combinations of sparsity level, signal strength, and design correlation level. 
Under this model, we derive tractable forms of the Hamming error and phase diagram for each method, and we make notable discoveries out of these theoretical results. 
%

\bibliography{VS}
\bibliographystyle{iclr2022_conference}

\newpage

\appendix


\section{Sketch of the proof ideas} \label{suppsec:sketch}
We use a similar proof idea for every main theorem, which we explain as follows. To obtain the phase diagram, the key is deriving the rate of convergence of the expected Hamming error $\mathbb{E}[H(\hat{\beta},\beta)]$. 
Let 
\beq \label{def:FPFN}
\FP_p=\sum_{j=1}^p\mathbb{P}(\beta_j=0,\hat{\beta}_j\neq 0), \qquad\mbox{and}\qquad \FN_p=\sum_{j=1}^p \mathbb{P}(\beta_j= \tau_p, \hat{\beta}_j=0).
\eeq
By definition, 
\[
\mathbb{E}[H(\hat{\beta},\beta)]  = \FP_p+\FN_p. 
\]
Suppose $j$ is in the diagonal block $\{j,j+1\}$ of the Gram matrix $G$. For most methods (except for forward selection and forward backward selection, which we discuss separately), it is easy to see that $\hat{\beta}_j$ does not depend on any other $\beta_i$ with $i\notin\{j,j+1\}$. It follows that
\begin{align*}
\mathbb{P}(\beta_j=0, \hat{\beta}_j\neq 0) &= \mathbb{P}(\beta_j=0, \beta_{j+1}=0, \hat{\beta}_j\neq 0)+
\mathbb{P}(\beta_j=0, \beta_{j+1}=\tau_p, \hat{\beta}_j\neq 0)\cr
&= (1-\epsilon_p)^2\, \mathbb{P}\bigl( \hat{\beta}_j\neq 0 \big| \beta_j=0, \beta_{j+1}=0\bigr)\cr
&\qquad  +(1-\epsilon_p)\epsilon_p \cdot\mathbb{P}\bigl(\hat{\beta}_j\neq 0 \big| \beta_j=0, \beta_{j+1}=\tau_p \bigr)\cr
&=  L_p\, \mathbb{P}_{00}(\hat{\beta}_j\neq 0 ) + L_pp^{-\vartheta}\,\mathbb{P}_{01}(\hat{\beta}_j\neq 0), 
\end{align*} 
where $\mathbb{P}_{00}$ is the conditional probability conditioning on $(\beta_j, \beta_{j+1})=(0,0)$ and $\mathbb{P}_{01}$ is the conditional probability conditioning on $(\beta_j, \beta_{j+1})=(0, \tau_p)$. Similarly, we can derive 
\[
\mathbb{P}(\beta_j=\tau_p, \hat{\beta}_j=0)  = L_pp^{-\vartheta}\, \mathbb{P}_{10}(\hat{\beta}_j= 0 ) + L_pp^{-2\vartheta}\,\mathbb{P}_{11}(\hat{\beta}_j= 0), 
\]
where $\mathbb{P}_{10}$ is the conditional probability conditioning on $(\beta_j, \beta_{j+1})=(\tau_p,0)$ and $\mathbb{P}_{11}$ is the conditional probability conditioning on $(\beta_j, \beta_{j+1})=(\tau_p, \tau_p)$. 
When $p$ is even, by symmetry in this design, the above expressions do not change with $j$. When $p$ is odd, this is true except for $j=p$; however, this single $j$ has a negligible effect on the expected Hamming error. We thus have
\begin{eqnarray} \label{proof-sketch}
\mathbb{E}[H(\hat{\beta},\beta)] &=& L_pp\cdot \mathbb{P}_{00}(\hat{\beta}_j\neq 0) + L_pp^{1-\vartheta}\cdot \mathbb{P}_{01}(\hat{\beta}_j\neq 0) \cr
&& + L_pp^{1-\vartheta}\cdot \mathbb{P}_{10}(\hat{\beta}_j= 0) + L_pp^{1-2\vartheta}\cdot \mathbb{P}_{11}(\hat{\beta}_j= 0). 
\end{eqnarray}

It remains to study the probabilities in \eqref{proof-sketch}. Let $\tilde{y}_j=x_j'y/\sqrt{2\log(p)}$ and $\tilde{y}_{j+1}=x_{j+1}'y/\sqrt{2\log(p)}$. For most methods considered in this paper, $\hat{\beta}_j$ is determined by $(\tilde{y}_j, \tilde{y}_{j+1})$ only. Define
\beq \label{def:RejRegion}
{\cal R} = \{ (h_1, h_2)\in\mathbb{R}^2:  \; (\tilde{y}_j, \tilde{y}_{j+1})=(h_1, h_2) \mbox{ implies that }\hat{\beta}_j\neq 0  \}. 
\eeq
Write $\tilde{y}=(\tilde{y}_1, \tilde{y}_2)'$. Then, we can re-write \eqref{proof-sketch} as
\begin{eqnarray} \label{proof-sketch2}
\mathbb{E}[H(\hat{\beta},\beta)] &=& L_pp\cdot \mathbb{P}_{00}(\tilde{y}\in {\cal R}) + L_pp^{1-\vartheta}\cdot \mathbb{P}_{01}(\tilde{y}\in {\cal R}) \cr
&& + L_pp^{1-\vartheta}\cdot \mathbb{P}_{10}(\tilde{y}\notin {\cal R} ) + L_pp^{1-2\vartheta}\cdot \mathbb{P}_{11}(\tilde{y}\notin {\cal R}). 
\end{eqnarray} 
In the settings of interest in this paper, conditioning on each realization of $(\beta_j, \beta_{j+1})$, it can be shown that 
\[
\tilde{y}\;\; \sim\;\; {\cal N}_2\Bigl(\mu,\; \frac{1}{2\log(p)} \Sigma\Bigr), \qquad\mbox{for some fixed }\mu\in\mathbb{R}^2\mbox{ and }\Sigma\in\mathbb{R}^{2\times 2}. 
\]
For any $x\in\mathbb{R}^2$ and $S\subset\mathbb{R}^2$, define
\beq \label{d_Sigma}
d^2_{\Sigma}(x, S)=\inf_{v\in S}\bigl\{ (x-v)'\Sigma^{-1}(x-v)\bigr\}.
\eeq
We apply Lemma 6.1 in \cite{ke2020power} to get that, as $p\to\infty$,  ($L_p$ denotes a multi-$\log(p)$ term; see Section~\ref{sec:main} or the notations below) 
\beq \label{proof-sketch3}
\mathbb{P}(\tilde{y}\in {\cal R})=L_pp^{-d^2_{\Sigma}(\mu, \, {\cal R})}, \qquad \mathbb{P}(\tilde{y}\notin {\cal R})=L_pp^{-d^2_{\Sigma}(\mu, \, {\cal R}^c)}.  
\eeq  
Combining \eqref{proof-sketch3} with \eqref{proof-sketch2}, we can get the rate of convergence of the expected Hamming error, if we calculate the following quantities:
\begin{itemize}
\item The set ${\cal R}$ (we call it ``rejection region''). The rejection region depends on the definition of the method and the choice of tuning parameters. 
\item The distances $d_{\Sigma}(\mu, {\cal R})$ and $d_{\Sigma}(\mu, {\cal R}^c)$. Note that $(\mu, \Sigma)$ depend on the realization of $(\beta_j, \beta_{j+1})$. Therefore, we need to calculate $(\mu,\Sigma)$ for each of the four possible realizations. 
\end{itemize}

In the remaining of this supplemental material, we prove Theorems~\ref{thm:elastic-net}-\ref{thm:forward-backward} and Proposition~\ref{prop:bridge}-\ref{prop:SCAD-larger-a}. For each theorem, the proof can be divided into three parts:
\begin{itemize}
\item[(a)] Derive the rejection region ${\cal R}$. 
\item[(b)] Apply \eqref{proof-sketch2}-\eqref{proof-sketch3} to calculate the rate of convergence of $\mathbb{E}[H(\hat{\beta},\beta)]$. 
\item[(c)] Calculate the phase diagram based on the result from (b). 
\end{itemize}  

\bigskip

Throughout the proof, we use $L_p$ to denote a generic multi-$\log(p)$ term, which satisfies that $L_pp^\epsilon\to \infty$ and $L_pp^{-\epsilon}\to 0$ for any $\epsilon>0$.  We also frequently use the notation: 
\begin{definition} \label{def:EllipsDistance}
For $\rho\in (-1,1)$ and $u,v\in\mathbb{R}^2$, define $d_\rho(u,v)>0$ by $d_\rho^2(u,v) = (u_1 -v_1)^2 + (u_2 - v_2)^2 - 2\rho(u_1 - v_1)(u_2 - v_2)$. 
\end{definition}

In our proofs, we also frequently calculate the infimum of $d_\rho^2(u, v)$, for $v$ a line in $\mathbb{R}^2$. 
The following lemma is very useful. Its proof is elementary and thus omitted. 
\begin{lem}\label{supplem:distance}
Fix $\rho\in (-1,1)$. 
Given real numbers $A,B,C$ such that $AB\neq 0$, consider a constrained optimization over $x=(x_1,x_2)$ that minimizes $d_{\rho}^2(x, (0,0))=x_1^2+x_2^2-2\rho x_1x_2$ subject to the constraint $Ax_1+Bx_2+C=0$. The solution is $x_1^* =  \frac{ - C(A +\rho B)}{A^2 + B^2 + 2\rho AB}$ and $x_2^* =\frac{ - C(B +\rho A)}{A^2 + B^2 + 2\rho AB}$,
and the objective function evaluated at $x^*=(x_1^*, x_2^*)$ is 
\[
d_\rho^2(x^*, (0,0)') = \frac{C^2(1 -\rho^2)}{A^2 + B^2 + 2\rho AB}. 
\]
\end{lem}

\section{Proof of Theorem~\ref{thm:elastic-net} (Elastic net)}\label{suppsec:en}
As described in Section~\ref{suppsec:sketch}, our proof has three parts: (a) deriving the rejection region, (b) obtaining the rate of convergence of $\mathbb{E}[H(\hat{\beta},\beta)]$, and (c) calculating the phase diagram.

\paragraph{Part 1: Deriving the rejection region.} 
Recall that the rejection region ${\cal R}$ is as defined in \eqref{def:RejRegion}.
Write $h_1=x_j'y/\sqrt{2\log(p)}$, $h_2=x_{j+1}'y/\sqrt{2\log(p)}$, and $\lambda=\sqrt{2q\log(p)}$. Consider a bivariate Elastic net problem, where $(\hat{b}_1, \hat{b}_2)$ minimizes
\beq \label{enproof-optimization}
L(b)\equiv \frac{1}{2}b'\begin{bmatrix}1&\rho\\\rho & 1\end{bmatrix}b + b'h+\sqrt{q}\|b\|_1 + \frac{1}{2} \mu\|b\|^2.  
\eeq
It is seen that $(\hat{\beta}_j, \hat{\beta}_{j+1})=\sqrt{2\log(p)}(\hat{b}_1, \hat{b}_2)$. Hence, ${\cal R}$ consists of all values of $h$ such that $\hat{b}_1\neq 0$.

Fix $\rho\geq 0$. The next lemma gives the explicit solution to \eqref{enproof-optimization} in the case of $h_1>|h_2|$. It is proved in Section~\ref{subsec:proof-EN-solution}. 
\begin{lem}[Solution path of Elastic net]\label{suppthm:sol.path.en}
Consider the optimization in \eqref{enproof-optimization}. Suppose \( h_1 > \abs{h_2} \geq 0 \). Write $\eta=\rho/(1+\mu)$. 
\begin{itemize}
  \item When \( \sqrt{q} \geq h_1 \), we have \( \hat{b}_1 = \hat{b}_2 = 0 \).
  \item If $h_2\geq \eta h_1$, when \(  \frac{{h_2 -\eta h_1}}{1 -\eta}\leq \sqrt{q} <h_1 \), we have $\hat{b}_1 =\frac{h_1 -\sqrt{q}}{1 +\mu}$, and $\hat{b}_2 = 0$; 
  
  When $\sqrt{q}<\frac{{h_2 -\eta h_1}}{1 -\eta}$, we have   \[
  \hat{b}_1 =\frac{\frac{h_1 -\sqrt{q} }{1 +\mu} -\eta\frac{h_2 -\sqrt{q} }{1 +\mu}}{1 -\eta^2},\qquad 
    \hat{b}_2 =\frac{\frac{h_2 -\sqrt{q} }{1 +\mu} -\eta\frac{h_1 -\sqrt{q} }{1 +\mu}}{1 - \eta^2}; 
   \]
   \item if $h_2<\eta h_1$, when \(  \frac{{-h_2 +\eta h_1}}{1 +\eta}\leq \sqrt{q} <h_1 \), we have $\hat{b}_1 =\frac{h_1 -\sqrt{q}}{1 +\mu}$, and $\hat{b}_2 = 0$; 
   
   When $\sqrt{q}< \frac{{-h_2 +\eta h_1}}{1 +\eta}$, we have 
\[
    \hat{b}_1 =\frac{\frac{h_1 -\sqrt{q} }{1 +\mu} -\eta \frac{h_2 +\sqrt{q} }{1 +\mu}}{1 - \eta^2},\quad 
  \hat{b}_2 =\frac{\frac{h_2 +\sqrt{q} }{1 +\mu} -\eta\frac{h_1 -\sqrt{q} }{1 +\mu}}{1 - \eta^2}.
\]
\end{itemize}
\end{lem}

We now use Lemma~\ref{suppthm:sol.path.en} to derive ${\cal R}$. 
Partition $\mathbb{R}^2$ into 4 non-overlapping regions: 
\begin{align*}
& M_1=\{(h_1, h_2):\, h_1 > |h_2|\}, \qquad M_2=\{(h_1, h_2):\, h_1< - |h_2|\}, \cr
& M_3=\{(h_1, h_2):\, h_2 >  |h_1|\}, \qquad M_4=\{(h_1, h_2):\, h_2<-|h_1|\}. 
\end{align*}
First, we derive ${\cal R}\cap M_1$. By Lemma~\ref{suppthm:sol.path.en}, as $\sqrt{q}$ decreases from $\infty$ to $0$, $\hat{b}_1$ is initially zero and then becomes positive when $\sqrt{q}$ hits $h_1$ (second bullet point of this lemma). Then, if we further decrease $\sqrt{q}$, the value of $\hat{b}_1$ is always increasing (third bullet point of this lemma) and remains positive. Therefore, $\sqrt{q}<h_1$ is the sufficient and necessary condition for $\hat{b}_1$ to be nonzero. It follows that 
\[
{\cal R}\cap M_1 = M_1\cap \{ (h_1, h_2): h_1>\sqrt{q}\}.
\]
Second, we consider ${\cal R}\cap M_2$. Note that $(h_1,h_2)\in {\cal R}\cap M_2$ if and only if $(-h_1, -h_2)\in {\cal R}\cap M_1$. Additionally,  if we simultaneously flip the sign of $(h_1, h_2, b_1, b_2)$, 
the objective in \eqref{enproof-optimization} is unchanged. It follows that 
\[
{\cal R}\cap M_2 = \{(h_1, h_2): (-h_1, -h_2)\in {\cal R}\cap M_1\}. 
\]
Next, we derive ${\cal R}\cap M_3$. Note that $(h_1, h_2)\in M_3$ if and only if $(h_2, h_1)\in M_1$. Moreover, if we swap $(h_1, b_1)$ with $(h_2, b_2)$, the objective in \eqref{enproof-optimization} is unchanged. Hence, we can obtain ${\cal R}\cap M_3$ as follows: We first find the collection of $(h_1, h_2)\in\mathcal{R}\cap M_1$ such that $\hat{b}_2\neq 0$, and then switch the two coordinates $h_1$ and $h_2$ to get ${\cal R}\cap M_3$. To this end, by Lemma~\ref{suppthm:sol.path.en}, for $(h_1,h_2)\in {\cal R}\cap M_1$, $\hat{b}_2\neq 0$ if either $h_2-\eta h_1>\sqrt{q}(1-\eta)$ or $h_2-\eta h_1<-\sqrt{q}(1+\eta)$. It follows that, for $(h_1,h_2)\in {\cal R}\cap M_3$, $\hat{b}_1\neq 0$ if either $h_1-\eta h_2>\sqrt{q}(1-\eta)$ or $h_1-\eta h_2<-\sqrt{q}(1+\eta)$. It implies that 
\[
{\cal R}\cap M_3 = M_3\cap \bigl( \{ (h_1, h_2): h_1-\eta h_2>\sqrt{q}(1-\eta)\}\cup \{(h_1,h_2): h_1-\eta h_2<-\sqrt{q}(1+\eta)\}\bigr). 
\]
Last, we obtain ${\cal R}\cap M_4$ by 
\[
{\cal R}\cap M_4 = \{(h_1, h_2): (-h_1, -h_2)\in {\cal R}\cap M_3\}. 
\]
Combining the above results gives
\begin{align} \label{proof-en-rjRegion}
{\cal R} &= \{(h_1,h_2): h_1-\eta h_2>\sqrt{q}(1-\eta),\, h_1>\sqrt{q}\}\cr
&\;\; \cup \{(h_1, h_2): h_1-\eta h_2>\sqrt{q}(1+\eta)\} \cup \{(h_1, h_2): h_1-\eta h_2<-\sqrt{q}(1+\eta)\}\cr
&\;\; \cup \{(h_1,h_2): h_1-\eta h_2<-\sqrt{q}(1-\eta),\, h_1<-\sqrt{q}\}. 
\end{align}  
See Figure~\ref{suppfig:rejection.region.en} for a visualization of the rejection region (recall that $\eta=\rho/(1+\mu)$). 

\begin{figure}[tb!]
  \centering
  \includegraphics[width=0.7\textwidth]{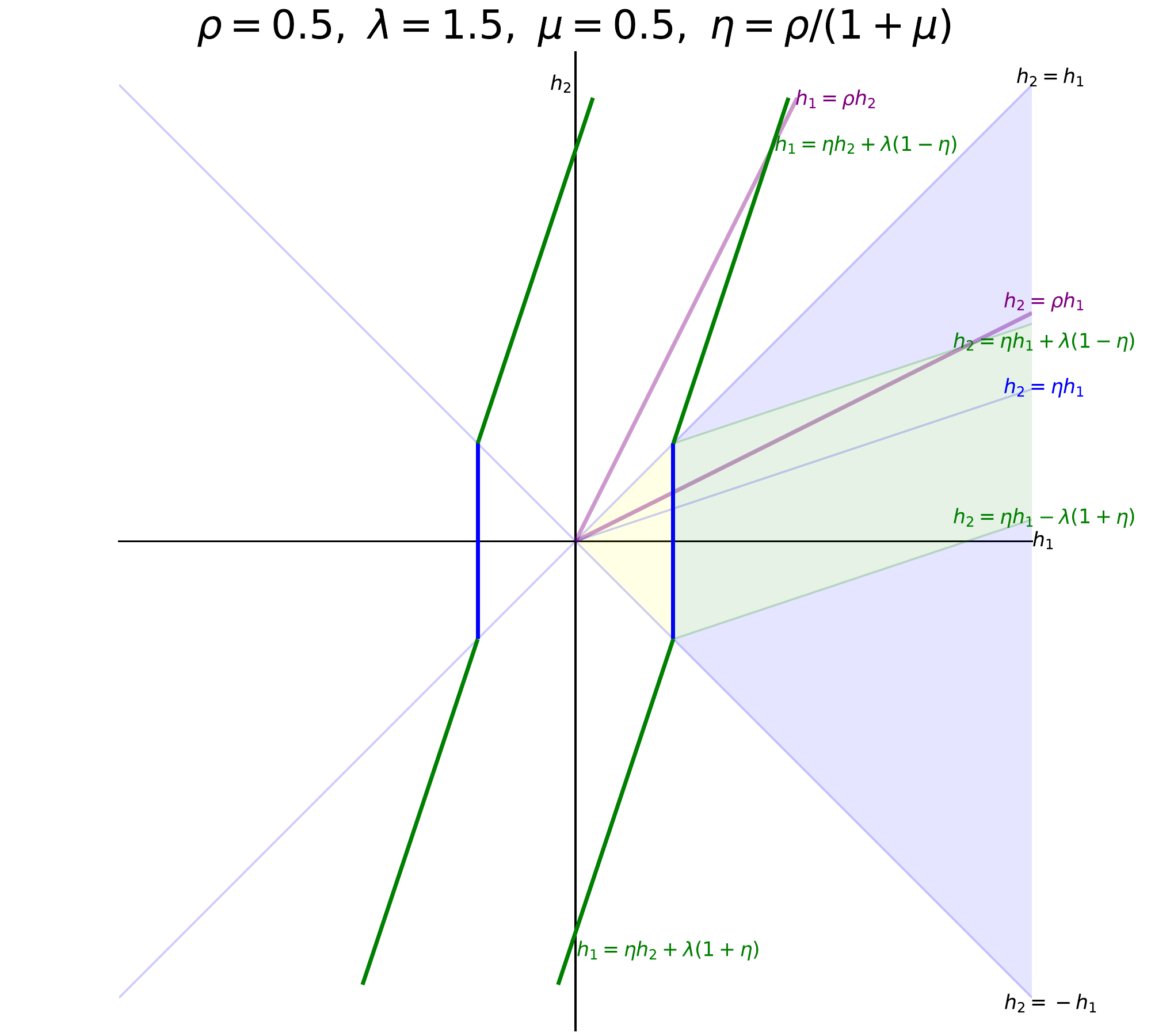}
  \caption{The rejection region of Elastic net for $\rho\geq 0$.}
  \label{suppfig:rejection.region.en}
\end{figure}

Figure~\ref{suppfig:rejection.region.en} only depicts the rejection region for $\rho\geq 0$. For $\rho<0$, we can similarly draw the rejection region, but it is not necessary for the proof of this theorem. In Part 2, we will see that, by carefully utilizing the  symmetry in our problem, we can derive the rate of convergence of the Hamming error for $\rho<0$ without deriving the rejection region directly.   

\paragraph{Part 2. Analyzing the Hamming error.} 
We aim to use \eqref{proof-sketch2}-\eqref{proof-sketch3} to derive the rate of convergence of $\mathbb{E}[H(\hat{\beta},\beta)]$.  
Recall that $\tilde{y}_1=x_j'y/\sqrt{2\log(p)}$ and $\tilde{y}_2=x_{j+1}'y/\sqrt{2\log(p)}$. It is easy to see that  $\tilde{y}\sim {\cal N}_2(\mu,\; \frac{1}{2\log(p)} \Sigma)$, where the covariance matrix $\Sigma$ is the $2\times 2$ matrix with 1 in the diagonal and $\rho$ in the off-diagonal, and the vector $\mu$ is equal to 
\begin{equation*}
\mu_{00}= \begin{bmatrix} 0 \\ 0 \end{bmatrix},\quad 
\mu_{01}=\begin{bmatrix} \rho \sqrt{r} \\ \sqrt{r} \end{bmatrix},\quad 
\mu_{10}=\begin{bmatrix} \sqrt{r} \\ \rho \sqrt{r} \end{bmatrix},\quad 
\mu_{11}=\begin{bmatrix} (1 +\rho) \sqrt{r} \\ (1 +\rho) \sqrt{r} \end{bmatrix}, 
  \end{equation*}
when $(\beta_j, \beta_{j+1})$ takes the value of $(0, 0)$, $(0, \tau_p)$, $(\tau_p, 0)$, and $(\tau_p, \tau_p)$, respectively. By \eqref{proof-sketch2}-\eqref{proof-sketch3}, $\mathbb{E}[H(\hat{\beta},\beta)]=\FP_p+\FN_p$, where
\begin{eqnarray} \label{proof-en-Hamming}
\FP_p &=& L_pp^{1-d^2_{\Sigma}(\mu_{00}, {\cal R})} + L_pp^{1-\vartheta -d^2_{\Sigma}(\mu_{01}, {\cal R}) }, \cr
\FN_p &=& L_pp^{1-\vartheta - d^2_{\Sigma}(\mu_{10}, {\cal R}^c)}+ L_pp^{1-2\vartheta-d^2_{\Sigma}(\mu_{11}, {\cal R}^c)}. 
\end{eqnarray} 
It suffices to calculate $d_{\Sigma}(\mu_{00}, {\cal R})$, $d_{\Sigma}(\mu_{01}, {\cal R})$, $d_{\Sigma}(\mu_{10}, {\cal R}^c)$, and $d_{\Sigma}(\mu_{11}, {\cal R}^c)$.

First, consider the case of $\rho\geq 0$. The expression of ${\cal R}$ is given explicitly in \eqref{proof-en-rjRegion}. 
By the definition in \eqref{d_Sigma} and Definition~\ref{def:EllipsDistance}, for any $S\subset\mathbb{R}^2$ and $\mu\notin S$, 
\beq \label{proof-en-distance}
d^2_{\Sigma}(\mu, S) = \frac{1}{1-\rho^2}\, \inf_{\xi\in S} d^2_\rho(\mu, \xi). 
\eeq
If $S$ can be expressed as the interaction and union of finitely many half-planes, then the point $\xi^*$ that attains the infimum must be on the boundary line of one of these half-planes. We thus only need to calculate: 
\begin{itemize}
\item[(i)] $\inf_{\xi\in {\cal L}}d^2_{\rho}(\mu, \xi)$ for the boundary line ${\cal L}$ of each half-plane in the definition of $S$ (with verification that the tangent point on \( {\cal L} \) is achievable on the boundary); 
\item[(ii)] $d_\rho^2(\mu,\zeta)$ for each point $\zeta$ that is a vertex of $S$ (i.e., the intersection of two boundary lines). 
\end{itemize}
For (i), we apply the formula given in Lemma~\ref{supplem:distance}. For (ii), we apply Definition~\ref{def:EllipsDistance} directly. These calculations give a finite collection of values. In (i), the $\xi^*$ that attains the infimum may not belong to $S$; if that happens, we delete it from the collection. Finally, $\inf_{\xi\in S} d^2_\rho(\mu, \xi)$ is the minimum of the values in this collection.

  

Since ${\cal R}$ and ${\cal R}^c$ can be expressed via the interaction and union of finitely many half-planes, we follow the above routine to calculate the desired quantities. Take $d_{\Sigma}(\mu_{01}, {\cal R})$ for example. 
Recall that \(\mu_{01}= (\rho \sqrt{r},\sqrt{r}) \). By \eqref{proof-en-distance}, it suffices to calculate $\inf_{\xi\in{\cal R}}d^2_\rho(\mu_{01},\xi)$. The region ${\cal R}$ has 6 boundary lines, but since $\rho\geq 0$, the infimum can only be attained in either of three cases:
\begin{itemize}
\item on the line ${\cal L}_1: h_1-\eta h_2=\sqrt{q}(1-\eta)$;
\item  on the line ${\cal L}_2: h_1=\sqrt{q}$;
\item on the the vertex $v^*=(\sqrt{q}, \sqrt{q})'$, which is an intersection of ${\cal L}_1\cap{\cal L}_2 $.  
\end{itemize}
Let $x=(x_1, x_2)'=(h_1-\rho\sqrt{r}, h_2-\sqrt{r})'$. We can re-write ${\cal L}_1$ as a line ${\cal L}_1'$ for $x$, which expression is $x_1-\eta x_2=(1-\eta)\sqrt{q}-(\rho-\eta)\sqrt{r}$. Similarly, we can re-write ${\cal L}_2$ as a line ${\cal L}_2'$: $x_1=\sqrt{q}-\rho\sqrt{r}$.  
We apply Lemma~\ref{supplem:distance} to get
\begin{align*}
\inf_{\xi\in {\cal L}_1} d^2_\rho(\mu_{01}, \xi ) &= \inf_{x\in {\cal L}_1'}d^2_\rho(x, (0,0)') = \frac{\bigl[(1-\eta)\sqrt{q}-(\rho-\eta)\sqrt{r}\bigr]^2(1-\rho^2)}{1+\eta^2-2\rho \eta},\cr
\inf_{\xi\in {\cal L}_2} d^2_\rho(\mu_{01}, \xi ) &= \inf_{x\in {\cal L}_2'}d^2_\rho(x, (0,0)') = (\sqrt{q}-\rho\sqrt{r})^2(1-\rho^2), \cr
d_{\rho}^2(\mu_{01}, v^*) &= d_\rho^2\bigl((\sqrt{q},\sqrt{q})', (\rho\sqrt{r}, \sqrt{r})'\bigr)\cr
& \equiv (\sqrt{q}-\rho\sqrt{r})^2+(\sqrt{q}-\sqrt{r})^2-2\rho (\sqrt{q}-\rho\sqrt{r})(\sqrt{q}-\sqrt{r})^2. 
\end{align*}
The value of $\inf_{\xi\in{\cal R}}d^2_\rho(\mu_{01},\xi)$ is the minimum of the above three values. In fact, the distance $d_\rho$ is related to the size of an ellipsoid that centers at $(\rho\sqrt{r},\sqrt{r})'$ and hits the boundary of ${\cal R}$. As $\sqrt{r}$ increases from zero, the center of this ellipsoid moves upwards on the line of $h_1=\rho h_2$. Consequently, the minimum of the above three values is initially (i) $\inf_{\xi\in {\cal L}_2} d^2_\rho(\mu_{01}, \xi )$ when $\sqrt{r}$ is appropriately small, then (ii) $d_{\rho}^2(\mu_{01}, v^*)$ when $\sqrt{r}$ is moderately large, and finally (iii) $\inf_{\xi\in {\cal L}_1} d^2_\rho(\mu_{01}, \xi )$ when $\sqrt{r}$ is sufficiently large; see Figure~\ref{suppfig:rejection.region.en}. We now figure out the range of $\sqrt{r}$ for each of the three cases. Recall that 
$v^*=(\sqrt{q}, \sqrt{q})'$. 
Let $\xi^*=(\xi_1^*, \xi_2^*)'$ be the vector that 
attains $\inf_{\xi\in {\cal L}_2} d^2_\rho(\mu_{01}, \xi )$. 
We have an explicit expression of $\xi^*$ from Lemma~\ref{supplem:distance}. By equating it with $v^*$, we can solve the critical value of $\sqrt{r}$ at which case (i) transits to case (ii):
\[
\sqrt{q}=v_2^*=\xi_2^*=  \sqrt{r}  + \rho (\sqrt{q}-\rho\sqrt{r}) \qquad\Longrightarrow\qquad \sqrt{r}=\frac{\sqrt{q}}{1 +\rho}. 
\]
Similarly, let $\tilde{\xi}^*=(\tilde{\xi}^*_1, \tilde{\xi}_2^*)'$ be the vector that attains $\inf_{\xi\in {\cal L}_1} d^2_\rho(\mu_{01}, \xi )$. By equating $\tilde{\xi}^*$ with $v^*$, we can solve the critical value of $\sqrt{r}$ at which case (ii) transits to case (iii): 
\[
\sqrt{q}=v_2^*=\tilde{\xi}_2^*=  \sqrt{r}  +  \frac{\bigl[(1-\eta)\sqrt{q}-(\rho-\eta)\sqrt{r}\bigr](\rho-\eta)}{1+\eta^2-2\rho \eta}\qquad\Longrightarrow\qquad \sqrt{r}=\frac{1 +\eta}{1 +\rho}\sqrt{q}. 
\]
We combine the above results to get
\beq \label{proof-en-Hamming-exponent}
\inf_{\xi\in {\cal R}}d^2_\rho(\mu_{01}, \xi) =  
\begin{cases} 
      (1 -\rho^2)(\sqrt{q}- \rho\sqrt{r})^2_ +, & \text{  if  } \sqrt{r} \leq \frac{1}{1+\rho}\sqrt{q}, \\
        d^2_{\rho}\bigl((\sqrt{q},\sqrt{q})',(\rho \sqrt{r},\sqrt{r})'\bigr), & \text{  if  } \frac{1}{1+\rho}\sqrt{q} < \sqrt{r} \leq \frac{1+\eta}{1+\rho}\sqrt{q},\\
       \frac{(1-\rho^2)}{1+\eta^2-2\rho \eta}[(1-\eta)\sqrt{q}-(\rho-\eta)\sqrt{r}]_+^2, & \text{if  } \sqrt{r} > \frac{1+\eta}{1+\rho}q. 
\end{cases} 
\eeq
Recall that $\eta=\rho/(1+\mu)$ is a shorthand notation. 
We plug \eqref{proof-en-Hamming-exponent} into \eqref{proof-en-distance}, and then we insert it into \eqref{proof-en-Hamming}. This gives the second term in $\FP_p$. We can follow the same routine to derive every term in $\FP_p$ and $\FN_p$. We omit the details but summarize the results in Theorem~\ref{suppthm:hamm.en} below.

Next, consider the case of $\rho<0$. We re-parametrize the linear model by replacing $(x_{j+1}, \beta_{j+1})$ with $(-x_{j+1}, -\beta_{j+1})$. After this re-parametrization, the $(j, j+1)$ block of the Gram matrix is a $2\times 2$ matrix $\Sigma$ whose off-diagonal entries are $-\rho=|\rho|$. The rejection region is defined by the solution path of \eqref{enproof-optimization} associated with $|\rho|>0$. This allows us to use the expression in \eqref{proof-en-rjRegion} directly with a simple replacement of $\rho$ by $|\rho|$.  There is no need to re-calculate the rejection region for a negative $\rho$. 

Let $\tilde{y}_1=x_j'y/\sqrt{2\log(p)}$ and $\tilde{y}_2=x_{j+1}'y/\sqrt{2\log(p)}$.  We still have $\tilde{y}\sim {\cal N}_2(\mu,\; \frac{1}{2\log(p)} \Sigma)$. However, the four realizations of $(\beta_j, \beta_{j+1})$ become $(0, 0)$, $(0, -\tau_p)$, $(\tau_p, 0)$, and $(\tau_p, -\tau_p)$. Therefore, the mean vectors $\mu$ have changed to 
\begin{equation*}
\mu_{00}= \begin{bmatrix} 0 \\ 0 \end{bmatrix},\quad 
\mu_{01}=\begin{bmatrix}- |\rho| \sqrt{r} \\ -\sqrt{r} \end{bmatrix},\quad 
\mu_{10}=\begin{bmatrix} -\sqrt{r} \\ -|\rho| \sqrt{r} \end{bmatrix},\quad 
\mu_{11}=\begin{bmatrix} (1 -|\rho|) \sqrt{r} \\ -(1 -|\rho|) \sqrt{r} \end{bmatrix}. 
\end{equation*}
Similar to \eqref{proof-en-Hamming}, it suffices to calculate 
$d_{\Sigma}(\mu_{00}, {\cal R})$, $d_{\Sigma}(\mu_{01}, {\cal R})$, $d_{\Sigma}(\mu_{10}, {\cal R}^c)$, and $d_{\Sigma}(\mu_{11}, {\cal R}^c)$.
Here ${\cal R}$ is the same as in Figure~\ref{suppfig:rejection.region.en}, but the locations of the $\mu$ vectors have changed. Since ${\cal R}$ is centrosymmetric,  $d_{\Sigma}(\mu_{00}, {\cal R})$, $d_{\Sigma}(\mu_{01}, {\cal R})$, and $d_{\Sigma}(\mu_{10}, {\cal R}^c)$ are actually the same as before. We only need to re-calculate $d^2_\Sigma(\mu_{11}, {\cal R}^c)$. The calculation routine is the same as that for \eqref{proof-en-Hamming-exponent}. We omit the details but present the results directly in the theorem below.

To summarize, in this part, we have proved the following theorem: 
\begin{theorem}\label{suppthm:hamm.en}
Suppose the conditions of Theorem~\ref{thm:elastic-net} hold. Let $\lambda=\sqrt{2q\log(p)}$ in Elastic net. Write $\eta=\rho/(1+\mu)$. As $p\to\infty$, 
\[
\FP_p=L_p p^{1- \min\bigl\{ q, \;\; \vt + f_1(\sqrt{r}, \sqrt{q})\bigr\}}, \qquad \FN_p = L_p p^{1-\min\bigl\{\vt + f_2(\sqrt{r}, \sqrt{q}),\;\; 2\vt + f_3(\sqrt{r}, \sqrt{q})\bigr\}}, 
\]
where (below, $d^2_{|\rho|}(u,v)$ is as in Definition~\ref{def:EllipsDistance}) 
\begin{align*}
 f_1(\sqrt{r},\sqrt{q}) & = \begin{cases} 
      (\sqrt{q}- |\rho|\sqrt{r})^2_ +, & \text{  if  } \sqrt{r} \leq \frac{1}{1+|\rho|}\sqrt{q}, \\
      \frac{1}{1-\rho^2}\cdot  d^2_{|\rho|}\bigl((\sqrt{q},\sqrt{q})',(|\rho| \sqrt{r},\sqrt{r})'\bigr), & \text{  if  } \frac{1}{1+|\rho|} \sqrt{q }< \sqrt{r} \leq \frac{1+|\eta|}{1+|\rho|} \sqrt{q},\\
        \frac{[ (1 -|\eta|)\sqrt{q} -(|\rho| -|\eta|)\sqrt{r}]_ +^2}{1 +\eta^2 -2|\rho||\eta|}, & \text{if } \sqrt{r} \geq \frac{1+|\eta|}{1+|\rho|}\sqrt{q},
\end{cases} \cr
f_2(\sqrt{r}, \sqrt{q}) &= \min\Bigl\{ (\sqrt{r} -\sqrt{q})_ +^2, \;\;
       \frac{[ (1 - \rho\eta)\sqrt{r} -(1 -|\eta|)\sqrt{q}]_ +^2}{1 +\eta^2 - 2\rho\eta} \Bigr\},\cr
       f_3(\sqrt{r}, \sqrt{q}) &= \frac{(1 -\eta)^2 [ (1 +\rho)\sqrt{r} -\sqrt{q}]_ +^2}{1 +\eta^2 - 2\rho\eta}. 
\end{align*}
\end{theorem}

\paragraph{Part 3. Calculating the phase diagram.} 
By Theorem~\ref{suppthm:hamm.en}, the Hamming error is $\FP_p+\FN_p=L_p p^{1-h(q;\vartheta,r)}$, where  
\begin{equation}\label{suppeq:h(qvtr)}
  h(q;\vartheta,r) = \min\Bigl\{\min\bigl\{ q, \  \vt + f_1(\sqrt{r}, \sqrt{q})\bigr\},\;\;  \min\bigl\{\vt + f_2(\sqrt{r}, \sqrt{q}),\  2\vt + f_3(\sqrt{r}, \sqrt{q})\bigr\} \Bigr\}.
\end{equation}

To calculate the phase diagram, we need to find $q^*$ that maximizes $h(q;\vartheta,r)$ and then investigate the conditions on $(r,\vartheta)$ such that $h(q^*; \vartheta,r)>1$ or $\vartheta<h(q^*; \vartheta,r)<1$ or $h(q^*; \vartheta,r)\leq \vartheta$. 

We first prove that \( r= \vt \)  is the boundary between the Regions of Almost Full Recovery and No Recovery, i.e., the boundary separating $\vartheta<h(q^*; \vartheta,r)<1$ and  $h(q^*; \vartheta,r)\leq \vartheta$. 

When \( r <\vt \), we need to show \( h(q;\vt,r) \leq \vt,\,\forall\, q \).  If \( q \leq  \vt\), then \( h(q;\vt,r) \leq q \leq \vt \). If \( q > \vt \), then we look at \( f_2(\sqrt{r},\sqrt{q}) \): Now we have \( 0 \leq f_2(\sqrt{r},\sqrt{q}) \leq (\sqrt{r} - \sqrt{q})_ + ^2 = 0\). Thus  \( h(q;\vt,r) \leq \vartheta + f_2(\sqrt{r},\sqrt{q}) =\vt \).

When \( r >\vt \), we can always find suitable \( q \) to make $h(q^*; \vartheta,r) \geq h(q; \vartheta,r) > \vartheta$. It is left for later discussion whether \( h(q^*; \vartheta,r) \) is greater than 1. In fact, such \( q \) can be any value satisfying \( \max \{\vt, (\frac{|\rho|-|\eta|}{1 -|\eta|})^2 r\} < q < r  \), which always exists because \( \frac{|\rho|-|\eta|}{1 -|\eta|} < 1 \). Since \( r> q\), we know \( f_2(\sqrt{r},\sqrt{q}) \) and \( f_3(\sqrt{r},\sqrt{q})  \) are strictly positive from their definition; since \( (1 -|\eta|) \sqrt{q} > (|\rho| -|\eta|) \sqrt{r} \), we also know \( f_1(\sqrt{r},\sqrt{q}) > 0 \). Since all four components of \( h(q;\vt,r) \) in \eqref{suppeq:h(qvtr)} is greater than \( \vt \), we have the desired result. 

To sum up the discussion so far, we have shown that \( r =\vt\) is the curve separating the regions of $\vartheta<h(q^*; \vartheta,r)<1$ and  $h(q^*; \vartheta,r)\leq \vartheta$.

For the rest of Part 3, we try to find the boundary between $h(q^*; \vartheta,r)>1$ and  $\vartheta<h(q^*; \vartheta,r)<1$. 

We need an important fact about such boundary, not only for the proof of Elastic net but also for all other methods. Recall the definition of \( \FP_p \)  and \( \FN_p \) in \eqref{proof-en-Hamming}, and we actually have the general form 
\[
\FP_p=L_p p^{1- \min\bigl\{ d^2_{\Sigma}(\mu_{00},{\cal R}), \;\; \vt + d^2_{\Sigma}(\mu_{01},{\cal R})\bigr\}}, \qquad \FN_p = L_p p^{1-\min\bigl\{\vt + d^2_{\Sigma}(\mu_{10},{\cal R}^c),\;\; 2\vt + d^2_{\Sigma}(\mu_{11},{\cal R}^c)\bigr\}}, 
\]
and 
\begin{equation*}
  h(q;\vartheta,r) = \min\Bigl\{\min\bigl\{ d^2_{\Sigma}(\mu_{00},{\cal R}), \  \vt + d^2_{\Sigma}(\mu_{01},{\cal R})\bigr\},\;\;  \min\bigl\{\vt + d^2_{\Sigma}(\mu_{10},{\cal R}),\  2\vt + d^2_{\Sigma}(\mu_{11},{\cal R})\bigr\} \Bigr\}.
\end{equation*}
As an important fact, we always have the following relationship at the boundary:
\begin{equation}\label{suppeq:important.relationship}
  \min\bigl\{ d^2_{\Sigma}(\mu_{00},{\cal R}), \  \vt + d^2_{\Sigma}(\mu_{01},{\cal R})\bigr\} =  \min\bigl\{\vt + d^2_{\Sigma}(\mu_{10},{\cal R}),\  2\vt + d^2_{\Sigma}(\mu_{11},{\cal R})\bigr\} = 1
\end{equation}
This is because: First, at the boundary \( r =r(\vt) \) , we must always have $h(q^*; \vartheta,r(\vt))=1$; otherwise, since \( h(q; \vartheta,r) \) is continuous in \( (q,r) \) for fixed \( \vt \), it would contradict the definition of the boundary itself. Second, the exponents of \( \FP_p \) and \( FN_p \) has to be equal. This is because if we change the tuning parameters for fixed \( (\vt,r) \), it can only enlarge or shrink the rejection region \( \cal R \), and thus the effects on \( \FP_p \) and \( \FN_p \) would always be in the opposite directions. As a result, if the exponents of \( \FP_p \) and \( \FN_p \) are not equal at the boundary, we can change the tuning parameters to make \( h(q;\vt,r) >1 \).

With the important fact, we can proceed our discussion. By the definition of \( h(q;\vt,r) \) in \eqref{suppeq:h(qvtr)}, it gives us \( 2\times 2 = 4 \) cases respectively for \( \rho > 0 \) and \( \rho <0 \). We discuss them one by one and summarise the results when the full phase curves are complete. For brevity, we also denote \( \lambda' = \sqrt{q} \) for the rest of Part 3.

When \( \rho > 0 \), we have four cases. 

\textit{First}, if \( \lambda'^2 = \vt + f_2(\sqrt{r},\lambda') = 1\) and \( \vt+ f_1(\sqrt{r},\lambda') \geq 1 \), \( 2\vt+ f_4(\sqrt{r},\lambda') \geq 1 \): 
we have 
   \( \lambda' = 1 \). From \(\vt + f_2(\sqrt{r},\lambda') = 1\), we know \[ \sqrt{r} =\max \left\{ 1 + \sqrt{1 -\vt}, \frac{\sqrt{1 +\eta^2 - 2\rho\eta}}{1 -\rho\eta}\sqrt{1 -\vt} +\frac{1 -\eta}{1 -\rho\eta} \right\}. \] We also know from \( \vt+ f_1(\sqrt{r},\lambda') \geq 1 \), that since \( \sqrt{r} > \lambda' = 1 \), \( \sqrt{r} \leq \frac{1 - \eta}{\rho -\eta} - \frac{\sqrt{1 +\eta^2 - 2\rho\eta}}{\rho -\eta} \sqrt{1 -\vt} \); from  \( 2\vt+ f_4(\sqrt{r},\lambda') \geq 1 \), that \(  \sqrt{r} \geq \frac{\sqrt{1 +\eta^2 - 2\rho\eta}}{(1 -\eta)(1 +\rho)} \sqrt{1 - 2\vt} + \frac{1}{1 +\rho} \). 

After roughly interpreting the requirements, we make two points: (i) we always have 
\begin{equation}\label{suppeq:firstcase.intermediate}
  \frac{\sqrt{1 +\eta^2 - 2\rho\eta}}{1 -\rho\eta}\sqrt{1 -\vt} +\frac{1 -\eta}{1 -\rho\eta} \geq \frac{\sqrt{1 +\eta^2 - 2\rho\eta}}{(1 -\eta)(1 +\rho)} \sqrt{1 - 2\vt} + \frac{1}{1 +\rho}
\end{equation}
 and thus the requirement from  \( 2\vt+ f_4(\sqrt{r},\lambda') \geq 1 \) is loose. 
This can be proven by showing \( \frac{1 -\eta}{1 -\rho\eta} \geq  \frac{1}{1 +\rho}\)  and \( \frac{\sqrt{1 +\eta^2 - 2\rho\eta}}{1 -\rho\eta} \geq \frac{\sqrt{1 +\eta^2 - 2\rho\eta}}{(1 -\eta)(1 +\rho)}  \) respectively.    
(ii) Actually, we can eliminate the curve \( \sqrt{r} = \frac{\sqrt{1 +\eta^2 - 2\rho\eta}}{1 -\rho\eta}\sqrt{1 -\vt} +\frac{1 -\eta}{1 -\rho\eta}\) from this step, without using the requirement from \( 2\vt+ f_4(\sqrt{r},\lambda') \geq 1 \). (As a result, the same proof holds for the corresponding case of \( \rho < 0 \).) This is because if we put together
\begin{equation*}
  \begin{cases} 
    \frac{\sqrt{1 +\eta^2 - 2\rho\eta}}{1 -\rho\eta}\sqrt{1 -\vt} +\frac{1 -\eta}{1 -\rho\eta} >&~  1 + \sqrt{1 -\vt} \\
    \frac{\sqrt{1 +\eta^2 - 2\rho\eta}}{1 -\rho\eta}\sqrt{1 -\vt} +\frac{1 -\eta}{1 -\rho\eta} \leq&~ \frac{1 - \eta}{\rho -\eta} - \frac{\sqrt{1 +\eta^2 - 2\rho\eta}}{\rho -\eta} \sqrt{1 -\vt}
  \end{cases} 
\end{equation*}
we will have no solution. To be more specific, the first inequality gives us \( \sqrt{1 -\vt} >\frac{\eta(1 -\rho)}{\sqrt{1 +\eta^2 - 2\rho\eta} - 1 +\rho\eta} \), and the second equation will eventually give us \( \sqrt{1 -\vt} \leq \frac{(1 + \eta)(1 -\rho)}{(1 +\rho)\sqrt{1 +\eta^2 - 2\rho\eta}} \). However, the upper and lower bounds on \( \sqrt{1 -\vt} \) admits no solution, because we can prove \( \frac{\eta(1 -\rho)}{\sqrt{1 +\eta^2 - 2\rho\eta} - 1 +\rho\eta} \geq \frac{(1 + \eta)(1 -\rho)}{(1 +\rho)\sqrt{1 +\eta^2 - 2\rho\eta}}\) jusy by simplifying it for a few steps.

To sum up, the \textit{first} case gives us \( \sqrt{r} = 1 + \sqrt{1 -\vt} \) with the requirement \( \sqrt{r} \leq \frac{1 - \eta}{\rho -\eta} - \frac{\sqrt{1 +\eta^2 - 2\rho\eta}}{\rho -\eta} \sqrt{1 -\vt} \). 

\textit{Second}, if \( q =2\vt+ f_4(\sqrt{r},\lambda') = 1\) and \( \vt + f_1(\sqrt{r},\lambda') \geq  1\), \( \vt+ f_2(\sqrt{r},\lambda') \geq 1 \), we will need 
\begin{equation*}
  \sqrt{r} = \frac{\sqrt{1 +\eta^2 - 2\rho\eta}}{(1 -\eta)(1 +\rho)} \sqrt{1 - 2\vt} + \frac{1}{1 +\rho}
\end{equation*}
while requiring \( \sqrt{r} \geq \max \left\{ 1 + \sqrt{1 -\vt}, \frac{\sqrt{1 +\eta^2 - 2\rho\eta}}{1 -\rho\eta}\sqrt{1 -\vt} +\frac{1 -\eta}{1 -\rho\eta} \right\} \). We know this is impossible from Equation~\eqref{suppeq:firstcase.intermediate}. No curve is produced in this case. 

\textit{Third}, if \( \vt + f_1(\sqrt{r},\lambda') = \vt + f_2(\sqrt{r},\lambda') = 1\), and \( \lambda' \geq  1\), \( 2\vt+ f_3(\sqrt{r},\lambda') \geq 1 \), we will know from  \( \vt + f_2(\sqrt{r},\lambda') = 1 \) that 
\begin{equation*}
  \sqrt{r} =\max \left\{ \lambda' + \sqrt{1 -\vt}, \frac{\sqrt{1 +\eta^2 - 2\rho\eta}}{1 -\rho\eta}\sqrt{1 -\vt} +\lambda'\frac{1 -\eta}{1 -\rho\eta} \right\}.
\end{equation*}
We can use the same method as in the first point of the \textit{first} case to show that \( 2\vt+ f_3(\sqrt{r},\lambda') \geq 1 \) is loose with \( \lambda' \geq 1 \) . 

Now the curve seems to have two choices, but the latter one is actually impossible. When \( \sqrt{r} = \frac{\sqrt{1 +\eta^2 - 2\rho\eta}}{1 -\rho\eta}\sqrt{1 -\vt} +\lambda'\frac{1 -\eta}{1 -\rho\eta} \), we have 
\begin{equation*}
  \begin{cases} 
    \sqrt{r} =&~ \lambda' \frac{1 - \eta}{\rho -\eta} - \frac{\sqrt{1 +\eta^2 - 2\rho\eta}}{\rho -\eta} \sqrt{1 -\vt} = \frac{\sqrt{1 +\eta^2 - 2\rho\eta}}{1 -\rho\eta}\sqrt{1 -\vt} +\lambda'\frac{1 -\eta}{1 -\rho\eta}  \\
    \lambda' \geq&~ 1 
   \end{cases} 
\end{equation*}
which implies \( \lambda' =\frac{(1 +\rho)\sqrt{1 +\eta^2 - 2\rho\eta}\sqrt{1 -\vt}}{(1 -\rho)(1 +\eta)} \geq 1 \). We can eliminate this case now, without considering the requirement of \( 2\vt+ f_3(\sqrt{r},\lambda') \geq 1 \), because with such \(  \lambda'  \geq 1 \), \( \frac{\sqrt{1 +\eta^2 - 2\rho\eta}}{1 -\rho\eta}\sqrt{1 -\vt} +\lambda'\frac{1 -\eta}{1 -\rho\eta} \geq \lambda' + \sqrt{1 -\vt} \) cannot hold. To see this, we can compare \( \frac{\sqrt{1 +\eta^2 - 2\rho\eta}}{1 -\rho\eta}\sqrt{1 -\vt} +\lambda'\frac{1 -\eta}{1 -\rho\eta} \); cancelling out ``\( \sqrt{1 -\vt} \)'', we have 
\begin{equation*}
  \frac{\sqrt{1 +\eta^2 - 2\rho\eta}}{1 -\rho\eta} - 1 \geq \frac{\eta(1 -\rho)}{1 -\rho\eta} \cdot \frac{(1 +\rho)\sqrt{1 +\eta^2 - 2\rho\eta}}{(1 -\rho)(1 +\eta)}
\end{equation*}
Simplifying this for a few steps, and we will arrive at ``\( \sqrt{1 +\eta^2 - 2\rho\eta} \geq 1 +\eta \)'' which gives a contradiction. 

We can only have one case, where   \( \lambda' + \sqrt{1 -\vt} \) is greater:
\begin{equation*}
  \begin{cases} 
    \sqrt{r} =& \lambda' \frac{1 - \eta}{\rho -\eta} - \frac{\sqrt{1 +\eta^2 - 2\rho\eta}}{\rho -\eta} \sqrt{1 -\vt} = \lambda' + \sqrt{1 -\vt} \\
    \lambda' \geq& 1 
  \end{cases} 
\end{equation*}

To sum up, the \textit{third} case gives us the curve \( \sqrt{r} = \left[ \frac{1 -\eta}{1 -\rho} + \frac{\sqrt{1 +\eta^2 - 2\rho\eta}}{1 - \rho} \right]\sqrt{1 -\vt} \) with the requirement \( \lambda' = \left[ \frac{\rho -\eta}{1 -\rho} + \frac{\sqrt{1 +\eta^2 - 2\rho\eta}}{1 - \rho} \right]\sqrt{1 -\vt} \geq 1 \). 

\textit{Fourth}, if \( \vt+ f_2(\sqrt{r},\lambda') = 2\vt+ f_3(\sqrt{r},\lambda') = 1  \) and \( \lambda' \geq 1 \), \( \vt+ f_1(\sqrt{r},\lambda') \geq 1 \), we have the same contradiction as the \textit{second} case, that 
\begin{equation*}
  \frac{\sqrt{1 +\eta^2 - 2\rho\eta}}{(1 -\eta)(1 +\rho)} \sqrt{1 - 2\vt} + \frac{\lambda'}{1 +\rho} \geq \frac{\sqrt{1 +\eta^2 - 2\rho\eta}}{1 -\rho\eta}\sqrt{1 -\vt} +\lambda'\frac{1 -\eta}{1 -\rho\eta}
\end{equation*}
cannot hold. 

\textit{Summarising the cases of positive correlation,}, we have two curves:
\begin{align*}
  \sqrt{r} = &~ 1 + \sqrt{1 -\vt} \\
  \sqrt{r} =&~ \left[ \frac{1 -\eta}{1 -\rho} + \frac{\sqrt{1 +\eta^2 - 2\rho\eta}}{1 - \rho} \right]\sqrt{1 -\vt} 
\end{align*}
and the intersection point of the two curves is exactly at \( \left[ \frac{\rho -\eta}{1 -\rho} + \frac{\sqrt{1 +\eta^2 - 2\rho\eta}}{1 - \rho} \right]\sqrt{1 -\vt} = 1 \) so the two curves can be summarised as taking the maximum.

When \( \rho < 0 \), we also have four cases.

\textit{First}, if \( \lambda'^2 = \vt + f_2(\sqrt{r},\lambda') = 1\) and \( \vt+ f_1(\sqrt{r},\lambda') \geq 1 \), \( 2\vt+ f_4(\sqrt{r},\lambda') \geq 1 \): We already know from the same proof when the correlation is positive, that \( \sqrt{r} = 1 + \sqrt{1 -\vt} \) is the only admissible curve. Now we additionally need it to satisfy \( \sqrt{r} \geq \frac{\sqrt{1 +\eta^2 - 2\rho\eta}}{(1 +|\eta|)(1 -|\rho|)} \sqrt{1 - 2\vt} + \frac{1}{1 -|\rho|} \).

\textit{Second}, if \( \lambda'^2 =2\vt+ f_4(\sqrt{r},\lambda') = 1\) and \( \vt + f_1(\sqrt{r},\lambda') \geq  1\), \( \vt+ f_2(\sqrt{r},\lambda') \geq 1 \), we will need \( \sqrt{r} = \frac{\sqrt{1 +\eta^2 - 2\rho\eta}}{(1 +|\eta|)(1 -|\rho|)} \sqrt{1 - 2\vt} + \frac{1}{1 -|\rho|} \) and we need \( \sqrt{r} \geq 1 + \sqrt{1 -\vt} \) from \( \vt+ f_2(\sqrt{r},\lambda') \geq 1 \).

\textit{To sum up the first two cases,} they give us \[ \sqrt{r} =\max \left\{ 1 + \sqrt{1 -\vt}, \frac{\sqrt{1 +\eta^2 - 2\rho\eta}}{(1 +|\eta|)(1 -|\rho|)} \sqrt{1 - 2\vt} + \frac{1}{1 -|\rho|}\right\}. \]
and we need \( \sqrt{r} \leq \frac{1 - |\eta|}{|\rho| -|\eta|} - \frac{\sqrt{1 +\eta^2 - 2\rho\eta}}{|\rho| -|\eta|} \sqrt{1 -\vt} \) from \( \vt+ f_1(\sqrt{r},\lambda') \geq 1 \).

\textit{Third}, if \( \vt + f_1(\sqrt{r},\lambda') = \vt + f_2(\sqrt{r},\lambda') = 1\), and \( \lambda' \geq  1\), \( 2\vt+ f_3(\sqrt{r},\lambda') \geq 1 \), we already know from the same proof when the correlation is positive,
that we have only one curve \( \sqrt{r} = \left[ \frac{1 -|\eta|}{1 -|\rho|} + \frac{\sqrt{1 +\eta^2 - 2\rho\eta}}{1 - |\rho|} \right]\sqrt{1 -\vt}. \) Now we only need to update the requirement from \( 2\vt+ f_3(\sqrt{r},\lambda') \geq 1 \), which is 
\begin{align*}
  \sqrt{r} \geq \frac{\sqrt{1 +\eta^2 - 2\rho\eta}}{(1 +|\eta|)(1 -|\rho|)} \sqrt{1 - 2\vt} + \frac{\lambda'}{1 -|\rho|}
\end{align*}
Plug in \( \lambda' = \left[ \frac{|\rho| -|\eta|}{1 -|\rho|} + \frac{\sqrt{1 +\eta^2 - 2\rho\eta}}{1 - |\rho|} \right]\sqrt{1 -\vt}\), it is equivalent to 
\begin{equation*}
  \left[ \frac{1 -|\eta|}{1 -|\rho|} + \frac{\sqrt{1 +\eta^2 - 2\rho\eta}}{1 - |\rho|} \right]\sqrt{1 -\vt} \geq \frac{\sqrt{1 +\eta^2 - 2\rho\eta}}{1 - 2|\rho| +\rho\eta} \left[ \sqrt{1 -\vt} +\frac{1 -|\eta|}{1 +|\eta|}\sqrt{1 - 2\vt} \right]
\end{equation*} 
where the RHS is another curve which will show up in the next case.

\textit{Fourth}, if \( \vt+ f_2(\sqrt{r},\lambda') = 2\vt+ f_3(\sqrt{r},\lambda') = 1  \) and \( \lambda' \geq 1 \), \( \vt+ f_1(\sqrt{r},\lambda') \geq 1 \), we will have the most tedious case. 

It can be implied by \( \vt+ f_2(\sqrt{r},\lambda') = 2\vt+ f_3(\sqrt{r},\lambda') = 1  \) that 
\begin{equation*}
  \lambda' \left( \frac{1 -|\eta|}{|\rho| -|\eta|} - \frac{1}{1 -|\rho|} \right) = \sqrt{1 -\vt} \frac{\sqrt{1 +\eta^2 - 2\eta\rho}}{|\rho| -|\eta|} + \sqrt{1 - 2\vt} \frac{\sqrt{1 +\eta^2 - 2\eta\rho}}{(1 +|\eta|)(1 -|\rho|)}
\end{equation*}
When \( 1 +\rho\eta - 2|\rho| \leq 0 \), the equation admits no solution, because the coefficient of \( \lambda' \) is not positive.
In this case, if we look back at the curve in the \textit{second} case \( FP_1 = FN_2 \), we will notice that 
\begin{equation*}
  \frac{\sqrt{1 +\eta^2 - 2\rho\eta}}{(1 +|\eta|)(1 -|\rho|)} \sqrt{1 - 2\vt} + \frac{1}{1 -|\rho|} \leq \frac{1 - |\eta|}{|\rho| -|\eta|} - \frac{\sqrt{1 +\eta^2 - 2\rho\eta}}{|\rho| -|\eta|} \sqrt{1 -\vt}
\end{equation*} 
has no solution either. As a result, when \( 1 +\rho\eta - 2|\rho| \leq 0 \), there is simply no Exact Recovery region in \( \vt \in(0,\frac{1}{2}) \).

When \( 1 +\rho\eta - 2|\rho| > 0  \), we can proceed to solve for \( \lambda' \) and then \( \sqrt{r} \):
\begin{align*}
  \lambda' = &~ \frac{\sqrt{1 +\eta^2 - 2\rho\eta}}{1 - 2|\rho| +\rho\eta} \left[ (1 -|\rho|)\sqrt{1 -\vt} + \frac{|\rho| -|\eta|}{1 +|\eta|}\sqrt{1 - 2\vt} \right] \\
  \sqrt{r} =&~ \frac{\sqrt{1 +\eta^2 - 2\rho\eta}}{1 - 2|\rho| +\rho\eta} \left[ \sqrt{1 -\vt} +\frac{1 -|\eta|}{1 +|\eta|}\sqrt{1 - 2\vt} \right] 
\end{align*}
For all the requirements from \( \lambda' \geq 1 \) and \( \vt+ f_1(\sqrt{r},\lambda') \geq 1 \), we actually need \( \lambda' \geq 1 \) and \( \sqrt{r} \geq \max \left\{ \lambda' + \sqrt{1 -\vt}, \frac{\sqrt{1 +\eta^2 - 2\rho\eta}}{1 -\rho\eta}\sqrt{1 -\vt} +\lambda'\frac{1 -|\eta|}{1 -\rho\eta} \right\} \). 

The requirement \( \lambda' \geq 1 \) is actually \[ \frac{\sqrt{1 +\eta^2 - 2\rho\eta}}{1 - 2|\rho| +\rho\eta} \left[ \sqrt{1 -\vt} +\frac{1 -|\eta|}{1 +|\eta|}\sqrt{1 - 2\vt} \right] \geq  \frac{\sqrt{1 +\eta^2 - 2\rho\eta}}{(1 +|\eta|)(1 -|\rho|)} \sqrt{1 - 2\vt} + \frac{1}{1 -|\rho|} \] 
and
\( \sqrt{r} \geq \lambda' + \sqrt{1 -\vt} \) is actually \[ \frac{\sqrt{1 +\eta^2 - 2\rho\eta}}{1 - 2|\rho| +\rho\eta} \left[ \sqrt{1 -\vt} +\frac{1 -|\eta|}{1 +|\eta|}\sqrt{1 - 2\vt} \right] \geq \left[ \frac{1 -|\eta|}{1 -|\rho|} + \frac{\sqrt{1 +\eta^2 - 2\rho\eta}}{1 - |\rho|} \right]\sqrt{1 -\vt}. \]
So we can already conclude, that the diagram is 
\begin{align*}
  \sqrt{r} =\max \big\{&~ 1 + \sqrt{1 -\vt},\left[ \frac{1 -|\eta|}{1 -|\rho|} + \frac{\sqrt{1 +\eta^2 - 2\rho\eta}}{1 - |\rho|} \right]\sqrt{1 -\vt},\\
   &~\frac{\sqrt{1 +\eta^2 - 2\rho\eta}}{(1 +|\eta|)(1 -|\rho|)} \sqrt{1 - 2\vt} + \frac{1}{1 -|\rho|},\frac{\sqrt{1 +\eta^2 - 2\rho\eta}}{(1 - 2|\rho| +\rho\eta)_ +} \left[ \sqrt{1 -\vt} +\frac{1 -|\eta|}{1 +|\eta|}\sqrt{1 - 2\vt} \right]\big\}
\end{align*}

However, we are left with one last constraint  \( \frac{\sqrt{1 +\eta^2 - 2\rho\eta}}{1 - 2|\rho| +\rho\eta} \left[ \sqrt{1 -\vt} +\frac{1 -|\eta|}{1 +|\eta|}\sqrt{1 - 2\vt} \right] \geq \frac{\sqrt{1 +\eta^2 - 2\rho\eta}}{1 -\rho\eta}\sqrt{1 -\vt} +\lambda'\frac{1 -|\eta|}{1 -\rho\eta} \). It is actually loose, but the proof may be tedious (and unimportant). We just need to prove that when \( 1 - 2|\rho| +\rho\eta > 0 \) and 
\begin{equation*}
  \lambda' =  \frac{\sqrt{1 +\eta^2 - 2\rho\eta}}{1 - 2|\rho| +\rho\eta} \left[ (1 -|\rho|)\sqrt{1 -\vt} + \frac{|\rho| -|\eta|}{1 +|\eta|}\sqrt{1 - 2\vt} \right] \geq 1,
\end{equation*} 
we always have 
\( \lambda' + \sqrt{1 -\vt} \geq  \frac{\sqrt{1 +\eta^2 - 2\rho\eta}}{1 -\rho\eta}\sqrt{1 -\vt} +\lambda'\frac{1 -|\eta|}{1 -\rho\eta} \), which is equivalent to 
\begin{equation*}
  \frac{(1 -|\rho|)\sqrt{1 +\eta^2 - 2\rho\eta}}{1 - 2|\rho| +\rho\eta} \sqrt{1 -\vt} + \frac{|\rho| -|\eta|}{1 +|\eta|}\frac{\sqrt{1 +\eta^2 - 2\rho\eta}}{1 - 2|\rho| +\rho\eta} \sqrt{1 - 2\vt} \geq \frac{\sqrt{1 +\eta^2 - 2\rho\eta} -(1 -\rho\eta)}{|\eta|(1 -|\rho|)}\sqrt{1 - \vt}
\end{equation*}

We first look at one sufficient condition, \( \frac{\sqrt{1 +\eta^2 - 2\rho\eta} -(1 -\rho\eta)}{|\eta|(1 -|\rho|)} \leq 1 \) 
By simplifying this inequality, we  get \( \sqrt{1 +\eta^2 - 2\rho\eta} \leq 1 +|\eta| - 2\rho\eta \). The RHS is positive, because: (i) If \( |\rho| \leq 0.5 \), \( 1 +|\eta|(1 - 2|\rho|) > 0 \); (ii) If \( |\rho| >0. 5\), recall \( 1 - 2|\rho| +\rho\eta > 0\implies |\eta| \geq \frac{2|\rho| - 1}{|\rho|}\implies  1 +|\eta|(1 - 2|\rho|) \geq 1 -\frac{(2|\rho| -1)^2}{|\rho|} \geq 0 \) for \( |\rho| \geq 0.5 \). Then we can equare both sides and proceed, and finally getting \( \eta\rho \leq \frac{1}{2} \). 

As a result, when \( \rho\eta \leq \frac{1}{2} \), we already have a sufficient condition for what we want to prove. When \( \rho\eta > \frac{1}{2} \), we look at another sufficient condition:

We only need to prove another sufficient condition, by looking at the coefficients of \( \sqrt{1 -\vt} \), 
\begin{equation*}
  \frac{(1 -|\rho|)\sqrt{1 +\eta^2 - 2\rho\eta}}{1 - 2|\rho| +\rho\eta} \geq \frac{\sqrt{1 +\eta^2 - 2\rho\eta} -(1 -\rho\eta)}{\eta(1 -|\rho|)}
\end{equation*}
which is equivalent to verifying
\begin{equation*}
  (1 -|\eta| - 2|\rho| + 3\eta\rho -|\eta|\rho^2) \sqrt{1 +\eta^2 - 2\rho \eta} \leq (1 -\rho\eta)(1 - 2|\rho| +\rho\eta)
\end{equation*}
It is elementary mathematics that \( (RHS - LHS) \) is always positive as a function of \( (|\rho|,|\eta|) \) under \( \rho\eta \geq \frac{1}{2} \), \( 1 - 2|\rho| +\rho\eta > 0 \) and \( 0 <|\eta| <|\rho|  \). 

\textit{Summarising the cases of negative correlation:}
The diagram is 
\begin{align*}
  \sqrt{r} =\max \big\{&~ 1 + \sqrt{1 -\vt},\left[ \frac{1 -|\eta|}{1 -|\rho|} + \frac{\sqrt{1 +\eta^2 - 2\rho\eta}}{1 - |\rho|} \right]\sqrt{1 -\vt},\\
   &~\frac{\sqrt{1 +\eta^2 - 2\rho\eta}}{(1 +|\eta|)(1 -|\rho|)} \sqrt{1 - 2\vt} + \frac{1}{1 -|\rho|},\frac{\sqrt{1 +\eta^2 - 2\rho\eta}}{(1 - 2|\rho| +\rho\eta)_ +} \left[ \sqrt{1 -\vt} +\frac{1 -|\eta|}{1 +|\eta|}\sqrt{1 - 2\vt} \right]\big\}
\end{align*}

\subsection{Proof of Lemma~\ref{suppthm:sol.path.en}}  \label{subsec:proof-EN-solution}

Recall the optimization in \eqref{enproof-optimization}; the solution \( b =(b_1,b_2) \) has to set the sub-gradient of the objective function to zero. As a result, the equation of the sub-gradient for \( b =(b_1,b_2) \)  is:
\begin{equation*}
  \begin{bmatrix} 
  1 & \rho \\
  \rho & 1
  \end{bmatrix} \begin{bmatrix} b_1 \\ b_2 \end{bmatrix} 
  + \sqrt{q} \begin{bmatrix} 
  \sgn(b_1) \\
  \sgn(b_2)
  \end{bmatrix} 
  + \mu \begin{bmatrix} 
  b_1 \\
  b_2
  \end{bmatrix} = \begin{bmatrix} h_1 \\ h_2 \end{bmatrix} 
\end{equation*}

Now we begin to find out the solution path. Fixing \( \mu \), we decrease \( \sqrt{q} \) from a sufficiently large value to see when the variables enter the model. We assume \( h_1 > 0 \) and \( 0 <\abs{h_2} < h_1 \). Other cases can be computed in a similar way. 

\textit{Stage 1}: When \( \sqrt{q} \) is so large that neither of \( (x_j,x_{j + 1}) \) is in the model
\begin{equation*}
  \begin{bmatrix} 
  1 & \rho \\
  \rho & 1
  \end{bmatrix} \begin{bmatrix} 0 \\ 0 \end{bmatrix} 
  + \sqrt{q} \begin{bmatrix} 
  \sgn(0) \\
  \sgn(0)
  \end{bmatrix} 
  + \mu \begin{bmatrix} 
  0 \\
  0
  \end{bmatrix} = \begin{bmatrix} h_1 \\ h_2 \end{bmatrix},
\end{equation*}
it requires \( \sqrt{q} \geq \sqrt{q_1} =\max\{\abs{h_1},\abs{h_2}\} = h_1 \).

\textit{Stage 2}: When \( \sqrt{q} \) crosses \( \sqrt{q_1} = h_1 \), we assert that variable \( x_j \) has to enter the model, while \( x_{j + 1} \) does not. This is because:
\begin{itemize}
  \item \( b_1 \) has to be positive. If it is negative, we have 
  \( b_1 -\sqrt{q} +\mu \cdot 2b_1 = h_1 \), which implies \( b_1 \) has the same sign as \( h_1 +\sqrt{q} > 0 \), which is a contradition.
  \item \( b_2 \) cannot enter the model at this point. Otherwise, we have \( (1 +\mu)b_2 +\sqrt{q} \sgn(b_2) = h_2 \), for \( \abs{h_2} <\sqrt{q} < h_1 \). Considering the sign of \( b_2 \),  we have a contradition.
\end{itemize}
So \( b_1 \) has to enter the model as a positive number. Now the equation is 
\begin{equation*}
  \begin{bmatrix} 
  1 & \rho \\
  \rho & 1
  \end{bmatrix} \begin{bmatrix} b_1 \\ 0 \end{bmatrix} 
  + \sqrt{q} \begin{bmatrix} 
  1 \\
  \sgn(0)
  \end{bmatrix} 
  + \mu \begin{bmatrix} 
  b_1 \\
  0
  \end{bmatrix} = \begin{bmatrix} h_1 \\ h_2 \end{bmatrix} 
\end{equation*}
Thus \( b_1 =\frac{h_1 -\sqrt{q}}{1 +\mu} \) and \( \rho b_1 +\sqrt{q} \sgn(0) = h_2 \). Since \( \sgn(0)\in[ - 1,1] \), we need 
\begin{equation}\label{suppeq:sol.path.en.kkt}
  \abs{h_2 -\frac{\rho}{1 +\mu}h_1 +\frac{\rho}{1 +\mu}\sqrt{q}} \leq \sqrt{q}
\end{equation}
By discussing the sign of the content of the absolute value as \( \sqrt{q} \) decreases, we have the following two cases:

\textit{Stage 3, Case 1}: When \( h_2 >\eta h_1 \) (recall we define \( \eta =\frac{\rho}{1 +\mu}\) as a shorthand), and \( \sqrt{q} \) crosses \( \sqrt{q_2} =\frac{h_2 -\eta h_1}{1 -\eta} \), then \( x_{j + 1} \) enters the model and \( b_2 \) is positive. This is because Equation~\eqref{suppeq:sol.path.en.kkt} is now \( h_2 -\eta h_1 +\eta\sqrt{q} \leq \sqrt{q} \); as \( \sqrt{q} =\sqrt{q_2} = \frac{h_2 -\eta h_1}{1 - \eta}\), the sub-gradient of \( \abs{b_2} \) is taking the value of \( 1\in\sgn(0) \), so \( b_2 \) has to enter the model as a positive number. 

In this case, we solve 
\begin{equation*}
  \begin{bmatrix} 
  1 & \rho \\
  \rho & 1
  \end{bmatrix} \begin{bmatrix} b_1 \\ b_2 \end{bmatrix} 
  + \sqrt{q} \begin{bmatrix} 
  1 \\
  1
  \end{bmatrix} 
  + \mu \begin{bmatrix} 
  b_1 \\
  b_2
  \end{bmatrix} = \begin{bmatrix} h_1 \\ h_2 \end{bmatrix} 
\end{equation*}
and get \begin{align*}
  b_1 =\frac{\frac{h_1 -\sqrt{q}}{1 +\mu} -\eta\frac{h_2 -\sqrt{q}}{1 +\mu}}{1 - \eta^2},\quad 
  b_2 =\frac{\frac{h_2 -\sqrt{q}}{1 +\mu} -\eta\frac{h_1 -\sqrt{q}}{1 +\mu}}{1 - \eta^2}.
\end{align*}

\textit{Stage 3, Case 2}: When \( h_2 <\eta h_1 \), and \( \sqrt{q} \) crosses \( \sqrt{q_2} = \frac{\eta h_1 - h_2}{1 +\eta }\), \( x_{j + 1} \) enters the model and \( b_2 \) is negative. This is because reviewing Equation~\ref{suppeq:sol.path.en.kkt}, we always have \( h_2 -\eta h_1 +\eta \sqrt{q} < \sqrt{q} \), and thus when \( \sqrt{q} \) is small enough to make \( \abs{h_2 -\eta h_1 + \eta \sqrt{q}} = \sqrt{q} \), it has to be \( -h_2 +\eta h_1 -\eta \sqrt{q} = \sqrt{q} \). As a result, when \( \sqrt{q} \) crosses \( \sqrt{q_2} =\frac{\eta h_1 - h_2}{1 +\eta } \), \( b_2 \) enters the model as a negative number. Solving
\begin{equation*}
  \begin{bmatrix} 
  1 & \rho \\
  \rho & 1
  \end{bmatrix} \begin{bmatrix} b_1 \\ b_2 \end{bmatrix} 
  + \sqrt{q} \begin{bmatrix} 
  1 \\
  -1
  \end{bmatrix} 
  + \mu \begin{bmatrix} 
  b_1 \\
  b_2
  \end{bmatrix} = \begin{bmatrix} h_1 \\ h_2 \end{bmatrix} 
\end{equation*}
we have \begin{align*}
  b_1 =\frac{\frac{h_1 -\sqrt{q}}{1 +\mu} -\eta \frac{h_2 +\sqrt{q}}{1 +\mu}}{1 - \eta ^2},\quad 
  b_2 =\frac{\frac{h_2 +\sqrt{q}}{1 +\mu} -\eta \frac{h_1 -\sqrt{q}}{1 +\mu}}{1 - \eta ^2}.
\end{align*}



\section{Proof of Proposition~\ref{prop:bridge} (Marginal Regression)}
\label{sec:marginal.regression}

Proposition~\ref{prop:bridge} is about the connection of Elastic net to Lasso and marginal regression. 

To prove the assertion about Lasso, we only need to quote the results from the Corollary 4.2 of \citet{ke2020power} on the phase curves of Lasso. In fact, the phase curves of Lasso can be exactly obtained by setting \( \mu = 0 \) in Theorem~\ref{thm:elastic-net}. As \( \mu \) decreases from some positive value to zero, the curves in Theorem~\ref{thm:elastic-net} just converges downwards to the phase curves of Lasso.


To prove the assertion about marginal regression, we need to fully study the variable selection method based on marginal regression, which we will be devoted to for the rest of this section.

\begin{definition}\label{suppdef:marginal.regression}
  Marginal regression refers to the variable selection method which ranks all the variabels according to \( \{ \abs{X'y}_j:j\in[p] \}  \) and sets some cutoff point \( t \) for the ranking. Then the variables \( \{ j\in[p]:\abs{X'y}_j > t \}  \) is selected. 
\end{definition}
Soft-thresholded marginal regression behaves the same as Definition~\ref{suppdef:marginal.regression} in terms of variable selection.

\begin{remark}
If we focus on \( (1 +\mu)\hat\beta^{\mathrm{EN}} \), then Lasso and soft-thresholded marginal regression are just two limits of \( \mu = 0 \) and \( \mu =\infty \), in terms of the solution and its path, the shape of the rejection region, and phase curves. 
\end{remark}







As described in Section~\ref{suppsec:sketch}, our proof still consists of three parts: (a) deriving the rejection region, (b) obtaining the rate of convergence of $\mathbb{E}[H(\hat{\beta},\beta)]$, and (c) calculating the phase diagram. 

\paragraph{Part 1: Deriving the rejection region.} 
According to Definition~\ref{suppdef:marginal.regression}, the variable selection based on marginal regression can be decomposed to every correlated pair of variables, \( (x_j,x_{j + 1}) \). It directly thresholds \( (x_j'y,x_{j + 1}'y) \) with \( t \), and if we divide \( (x_j'y,x_{j + 1}'y, t) \) with \( \sqrt{2\log(p)} \), it is equivalent to thresholding $h_1=x_j'y/\sqrt{2\log(p)}$, $h_2=x_{j+1}'y/\sqrt{2\log(p)}$ with \( t' = t/\sqrt{2\log(p)} \).

The solution path of marginal regression is very straight forward, so we present it in Lemma~\ref{supplem:sol.path.marreg} without proof.

\begin{lemma}[the solution path of marginal regression]\label{supplem:sol.path.marreg}
  The definition of \( (h_1,h_2,\hat b_1,\hat b_2) \) follows from that of Lemma~\ref{suppthm:sol.path.en}, and we still assume \( h_1 > \abs{h_2} \geq 0 \). The solution path of marginal regression defined in \ref{suppdef:marginal.regression} can be describes as:
  \begin{enumerate}
  \item When \( t' \geq  h_1 \), we have \( \hat b_1 =\hat b_2 = 0 \).
  \item When \( \abs{h_2} \leq t' < h_1 \), we have \( \hat b_1 \neq 0 \) and \( \hat b_2 = 0 \).
  \item When \( t < \abs{h_2} \), we have \( \hat b_1 \neq 0 \) and \( \hat b_2 \neq 0 \).
  \end{enumerate}
\end{lemma}

We now use Lemma~\ref{supplem:sol.path.marreg} to derive the rejection region \( \cal R \) of marginal regression. Recall that \( \cal R \) is the set of \( h =(h_1,h_2)' \) such that \( \hat b_1\neq 0 \).  In fact, the same procedure of Elastic net can be copied here, except that the specific behavior of the variable selection method is different. Lemma~\ref{supplem:sol.path.marreg}  tells us it is just 
\begin{align} \label{suppeq:RejRegion.marreg}
  {\cal R} &= \{(h_1,h_2): h_1> t'\}\cup \{(h_1, h_2): h_1 <- t'\}
\end{align}  

\paragraph{Part 2. Analyzing the Hamming error.} 
The first steps of analysing Elastic net applies here as well, and we just need to compute  $d_{\Sigma}(\mu_{00}, {\cal R})$, $d_{\Sigma}(\mu_{01}, {\cal R})$, $d_{\Sigma}(\mu_{10}, {\cal R}^c)$, and $d_{\Sigma}(\mu_{11}, {\cal R}^c)$ given the different shape of \( \cal R \). Then we can compute \[
  \FP_p=L_p p^{1- \min\bigl\{ d^2_{\Sigma}(\mu_{00},{\cal R}), \;\; \vt + d^2_{\Sigma}(\mu_{01},{\cal R})\bigr\}}, \qquad \FN_p = L_p p^{1-\min\bigl\{\vt + d^2_{\Sigma}(\mu_{10},{\cal R}^c),\;\; 2\vt + d^2_{\Sigma}(\mu_{11},{\cal R}^c)\bigr\}}.
  \]
   Finally $\mathbb{E}[H(\hat{\beta},\beta)]=\FP_p+\FN_p=L_p p^{1-h(t';\vt,r)}$. 

\begin{theorem}\label{suppthm:hamm.marreg}
  Under the conditions of Theorem~\ref{thm:elastic-net}, let \( t = t' \sqrt{2\log(p)} \) in marginal regression defined in \ref{suppdef:marginal.regression}. As \( p\to\infty \), 
  \begin{align*}
    \FP_p =&~ L_p \cdot p^{1 -\min \left\{ t'^2,\;\;\vt + (t' -|\rho| \sqrt{r})_ +^2 \right\}} \\
    \FN_p =&~ L_p \cdot p^{1 -\min \left\{ \vt +(\sqrt{r} - t')_ +^2,\;\;2\vt + [(1 +\rho)\sqrt{r} - t']_ +^2 \right\}}
  \end{align*}
\end{theorem}

\paragraph{Part 3. Calculating the phase diagram.} 
By Theorem~\ref{suppthm:hamm.marreg}, the Hamming error is $\FP_p+\FN_p=L_p p^{1-h(t';\vartheta,r)}$, where  
\begin{equation}\label{suppeq:h(qvtr).marreg}
  h(t';\vartheta,r) = \min\Bigl\{\min\bigl\{ t'^2, \  \vt + (t' -|\rho| \sqrt{r})_ +^2 \bigr\},\;\;  \min\bigl\{\vt + (\sqrt{r} - t')_ +^2,\  2\vt + [(1 +\rho)\sqrt{r} - t']_ +^2 \bigr\} \Bigr\}.
\end{equation}

The first steps of the proof of Elastic net can also be applied directly. We still  try to find $t'^*$ that maximizes $h(t';\vartheta,r)$ and then investigate the conditions on $(r,\vartheta)$ such that $h(t'^*; \vartheta,r)>1$ or $\vartheta<h(t'^*; \vartheta,r)<1$ or $h(t'^*; \vartheta,r)\leq \vartheta$. 

We can still prove that \( r= \vt \)  is the boundary between the Regions of Almost Full Recovery and No Recovery, i.e., the boundary separating $\vartheta<h(q^*; \vartheta,r)<1$ and  $h(q^*; \vartheta,r)\leq \vartheta$. The proof just needs slight modification to the proof of Elastic net. 

For the rest of Part 3, we try to find the boundary between the Regions of Exact Recovery and Almost Full Recovery, i.e., the boundary separating $h(t'^*; \vartheta,r)>1$ and  $\vartheta<h(t'^*; \vartheta,r)<1$. We can still leverage the important fact that Equation~\ref{suppeq:important.relationship} holds at the boundary: 
\begin{equation*}
  \min\bigl\{ t'^2, \  \vt + (t' -|\rho| \sqrt{r})_ +^2 \bigr\} =  \min\bigl\{\vt + (\sqrt{r} - t')_ +^2,\  2\vt + [(1 +\rho)\sqrt{r} - t']_ +^2 \bigr\} = 1
\end{equation*}
The above equation still gives us \( 2\times 2 = 4 \) cases respectively for \( \rho > 0 \) and \( \rho <0 \). We discuss them one by one and summarise the results when the full phase curves are complete. 

When \( \rho> 0 \), we set out to prove the final phase diagram is 
\begin{equation*}
  \sqrt{r} =\max \left\{  1 + \sqrt{1 -\vt},\,\frac{2}{1 -\rho}\sqrt{1 -\vt}  \right\}
\end{equation*}

\textit{First}, if \( t'^2 = \vt + (\sqrt{r} - t')_ +^2 = 1\) and \( \vt + (t' -|\rho| \sqrt{r})_ +^2 \geq 1 \), \( 2\vt + [(1 +\rho)\sqrt{r} - t']_ +^2 \geq 1 \), then \( t' = 1 \) and \( \sqrt{r} = 1 + \sqrt{1 -\vt} \). 

We also need to meet two requirements: \( \sqrt{r} \leq  \frac{1 - \sqrt{1 -\vt}}{\rho} \) from \( \vt + (t' -|\rho| \sqrt{r})_ +^2 \geq 1 \) and \( \sqrt{r} \geq \frac{1 + \sqrt{1 - 2\vt}}{1 +\rho} \) from \( 2\vt + [(1 +\rho)\sqrt{r} - t']_ +^2 \geq 1 \). The second requirement is not restrictive, and the first one is equivalent to \( \sqrt{1 -\vt} \leq \frac{1 -\rho}{1 +\rho} \).

\textit{Second}, if \( t'^2 = 2\vt + [(1 +\rho)\sqrt{r} - t']_ +^2 = 1\) and \( \vt + (\sqrt{r} - t')_ +^2 \geq 1 \), \( \vt + (t' -|\rho| \sqrt{r})_ +^2 \geq 1 \), then we have no admissible curve. This is because \( \sqrt{r} = \frac{1 + \sqrt{1 - 2\vt}}{1 +\rho} \) and it is required that \( \sqrt{r} \geq 1 + \sqrt{1 -\vt} \) by \( \vt + (\sqrt{r} - t')_ +^2 \geq 1 \). It gives us no admissible \( \sqrt{r} \). 

\textit{Summarising the first two cases}, we have only one curve \( \sqrt{r} = 1 + \sqrt{1 -\vt} \) under the constraint \( \sqrt{1 -\vt} \leq \frac{1 -\rho}{1 +\rho} \). 

\textit{Third}, if \( \vt + (t' -|\rho| \sqrt{r})_ +^2 = \vt + (\sqrt{r} - t')_ +^2 = 1 \), and \( t' \geq 1 \), \( 2\vt + [(1 +\rho)\sqrt{r} - t']_ +^2  \geq 1 \), the equality gives us \( \sqrt{r} = \frac{2}{1 -\rho}\sqrt{1 -\vt} \) and \( t' = \frac{1 +\rho}{1 -\rho}\sqrt{1 -\vt}\). We are also required to have \( t' \geq 1 \) and \( \sqrt{r} \geq \frac{\sqrt{1 - 2\vt}}{1 +\rho} + \frac{\sqrt{1 -\vt}}{1 -\rho} \) (not restrictive). 

\textit{Fourth}, if \( \vt + (t' -|\rho| \sqrt{r})_ +^2 =2\vt + [(1 +\rho)\sqrt{r} - t']_ +^2 = 1\), and \( t' \geq 1 \), \( \vt + (\sqrt{r} - t')_ +^2 \geq 1 \), then we have no admissible curve. In fact, the equality gives us \( \sqrt{r} = \frac{t' - \sqrt{1 -\vt}}{\rho} =\frac{t' + \sqrt{1 - 2\vt}}{1 +\rho}\), \( t' = \rho \sqrt{1 - 2\vt} + (1 +\rho) \sqrt{1 -\vt} \). We are also required to have \( \sqrt{r} \geq  t' + \sqrt{1 -\vt} \Leftrightarrow \sqrt{1 - 2\vt} \geq \frac{1 +\rho}{1 - \rho}\sqrt{1 -\vt} \), which gives a contradition.

\textit{Summarising the last two cases}, we have only one curve \( \sqrt{r} = \frac{2}{1 -\rho}\sqrt{1 -\vt} \) under the constraint \( \sqrt{1 -\vt} \geq \frac{1 -\rho}{1 +\rho} \). 

\textit{Summarising the cases of positive correlation}, we have proven the phase diagram is \( \sqrt{r} =\max \left\{  1 + \sqrt{1 -\vt},\,\frac{2}{1 -\rho}\sqrt{1 -\vt}  \right\} \). 

When \( \rho < 0 \), we then prove that the phase diagram is  
\begin{equation*}
  \sqrt{r} =\max \left\{  1 + \sqrt{1 -\vt},\,\frac{2}{1 -|\rho|}\sqrt{1 -\vt},\, \frac{\sqrt{1 -\vt} + \sqrt{1 - 2\vt}}{1 - 2|\rho|},\,\frac{1 + \sqrt{1 - 2\vt}}{1 -|\rho|}  \right\}
\end{equation*}

\textit{First}, if \( t'^2 = \vt + (\sqrt{r} - t')_ +^2 = 1\) and \( \vt + (t' -|\rho| \sqrt{r})_ +^2 \geq 1 \), \( 2\vt + [(1 -|\rho|)\sqrt{r} - t']_ +^2 \geq 1 \), then \( t' = 1 \) and \( \sqrt{r} = 1 + \sqrt{1 -\vt} \). 

We also have two requirements, \( \sqrt{r} \leq \frac{1 - \sqrt{1 -\vt}}{|\rho|} \) and \( \sqrt{r} \geq \frac{1 + \sqrt{1 - 2\vt}}{1 -|\rho|} \).  We have studied the first one when \( \rho > 0 \), and know it is \( 1+ \sqrt{1 -\vt} \geq \frac{2}{1 -|\rho|}\sqrt{1 -\vt} \). The RHS of the second requirement is actually a new curve we will see later.

\textit{Second}, if \( t'^2 =2\vt + [(1 -|\rho|)\sqrt{r} - t']_ +^2 = 1\) and \( \vt + (t' -|\rho| \sqrt{r})_ +^2 \geq 1 \), \(  \vt + (\sqrt{r} - t')_ +^2  \geq 1 \), we have \( \sqrt{r} = \frac{1 + \sqrt{1 - 2\vt}}{1 -|\rho|}\) and \( \sqrt{r} \leq  \frac{1 - \sqrt{1 -\vt}}{|\rho|} \), \( \sqrt{r} \geq 1 + \sqrt{1 -\vt}. \). 

\textit{Summarising the first two cases}, we have \( \sqrt{r} =\max \left\{ 1 + \sqrt{1 -\vt},\frac{1 + \sqrt{1 - 2\vt}}{1 -|\rho|} \right\} \)  and we need \( \sqrt{r} \leq \frac{1 - \sqrt{1 -\vt}}{|\rho|} \).

\textit{Third}, if \( \vt + (t' -|\rho| \sqrt{r})_ +^2 = \vt + (\sqrt{r} - t')_ +^2 = 1\) and \( t' \geq 1 \), \( 2\vt + [(1 -|\rho|)\sqrt{r} - t']_ +^2 \geq 1 \), then \( \sqrt{r} = \frac{2}{1 - |\rho|}\sqrt{1 -\vt} \), and the two other requirements are \( t' = \frac{1 +|\rho|}{1 -|\rho|}\sqrt{1 -\vt} \geq 1 \) and \( \sqrt{r} \geq  \frac{\sqrt{1 - 2\vt}}{1 -|\rho|} + \frac{1 +|\rho|}{(1 -|\rho|)^2} \sqrt{1 -\vt} \). 

In the next case, we will get another curve \( \sqrt{r} = \frac{\sqrt{1 -\vt} + \sqrt{1 - 2\vt}}{1 - 2|\rho|}\). In the above inequalities, the last one corresponds to \( \frac{2}{1 -|\rho|} \sqrt{1 -\vt} \geq  \frac{\sqrt{1 -\vt} + \sqrt{1 - 2\vt}}{1 - 2|\rho|} \). \( t' = \frac{1 +|\rho|}{1 -|\rho|}\sqrt{1 -\vt} \geq 1 \) in the above inequalities is just \( 1 + \sqrt{1 -\vt} \leq  \frac{2}{1 -|\rho|} \sqrt{1 -\vt} \).

\textit{Fourth}, if \( \vt + (t' -|\rho| \sqrt{r})_ +^2 = 2\vt + [(1 -|\rho|)\sqrt{r} - t']_ +^2  = 1\) and \( t' \geq 1 \), \(  \vt + (\sqrt{r} - t')_ +^2\geq 1 \), then we know from the equality that
\begin{equation*}
  \sqrt{r} = \frac{t' - \sqrt{1 -\vt}}{|\rho|} =\frac{t' + \sqrt{1 - 2\vt}}{1 -|\rho|}
\end{equation*}
and we will get this when we solve for \( t' \):
\begin{equation*}
  \frac{t'}{|\rho|} - \frac{t'}{1 -|\rho|} =\frac{\sqrt{1 - 2\vt}}{1 -|\rho|} + \frac{\sqrt{1 -\vt}}{|\rho|}.
\end{equation*}
If \( |\rho| \geq \frac{1}{2} \), this equation admits no positive solution for \( t' \). Recall that in the \textit{first and second} cases, we also needed \( \sqrt{r} \leq \frac{1 -\sqrt{1 -\vt}}{|\rho|} \) and \( \sqrt{r} \geq \frac{1 + \sqrt{1 - 2\vt}}{1 -|\rho|} \) in the cases of \( FP_1 \) being tight. When \( |\rho| \geq \frac{1}{2} \), \( \frac{1 + \sqrt{1 - 2\vt}}{1 -|\rho|} \leq \frac{1 -\sqrt{1 -\vt}}{|\rho|} \)  has no solution either, so there cannot be any curve in the interval \( 0 <\vt <\frac{1}{2} \). 

if \( |\rho| < \frac{1}{2} \), we can proceed to have the two requirements:
\begin{equation*}
  \begin{cases} 
    t' =&~ \frac{1}{1 - 2|\rho|} \left[ |\rho| \sqrt{1 - 2\vt} +(1 -|\rho|) \sqrt{1 -\vt} \right] \geq 1\\
    \sqrt{r} =&~ \frac{\sqrt{1 -\vt} + \sqrt{1 - 2\vt}}{1 - 2|\rho|} \geq \frac{|\rho|}{1 - 2|\rho|} \sqrt{1 - 2\vt} + \frac{2 - 3|\rho|}{1 - 2|\rho|} \sqrt{1 -\vt}
  \end{cases} 
\end{equation*}
In the above inequalities, \( \sqrt{r} = \frac{\sqrt{1 -\vt} + \sqrt{1 - 2\vt}}{1 - 2|\rho|} \geq \frac{|\rho|}{1 - 2|\rho|} \sqrt{1 - 2\vt} + \frac{2 - 3|\rho|}{1 - 2|\rho|} \sqrt{1 -\vt} \) is equivalent to \( \frac{\sqrt{1 -\vt} + \sqrt{1 - 2\vt}}{1 - 2|\rho|} \geq \frac{2}{1 -|\rho|}\sqrt{1 -\vt} \). The requirement on \( t' \) is equivalent to  \( \sqrt{r} = \frac{\sqrt{1 -\vt} + \sqrt{1 - 2\vt}}{1 - 2|\rho|} \geq \frac{1 + \sqrt{1 - 2\vt}}{1 -|\rho|} \).

\textit{Summarising the cases of negative correlation}: 
The final phase diagram is 
\begin{equation*}
  \sqrt{r} =\max \left\{  1 + \sqrt{1 -\vt},\,\frac{2}{1 -|\rho|}\sqrt{1 -\vt},\, \frac{\sqrt{1 -\vt} + \sqrt{1 - 2\vt}}{1 - 2|\rho|},\,\frac{1 + \sqrt{1 - 2\vt}}{1 -|\rho|}  \right\}
\end{equation*}
When \( \rho \leq - \frac{1}{2} \), \( \sqrt{r} \) has no finite value for \( \vt\in(0,1/2) \), and we do not have Region of Exact Recovery or \( h(t';\vt,r) > 1 \) at all.

\section{Proof of Theorem~\ref{thm:SCAD} (SCAD)}\label{suppsec:scad}

As described in Section~\ref{suppsec:sketch}, our proof for SCAD still consists of three parts: (a) deriving the rejection region, (b) obtaining the rate of convergence of $\mathbb{E}[H(\hat{\beta},\beta)]$, and (c) calculating the phase diagram. 

Before starting, we first recall the definition of SCAD. An alternative way to write the derivative of the penalty function is:
\begin{equation}\label{suppeq:definition.SCAD}
  q'(\theta)=\begin{cases} 
     \lambda \cdot \sgn(\theta) & \text{if } \abs\theta < \lambda\\
     \frac{a\lambda -\theta}{a - 1}\cdot \sgn(\theta) & \text{if } \lambda < \abs\theta < a\lambda \\
     0 & \text{if } \abs\theta > a\lambda
   \end{cases} 
 \end{equation}
for \( \theta\in\R,\, a > 2\, \lambda > 0 \). 

\paragraph{Part 1: Deriving the rejection region.} 
Just like the first steps of Elastic net, we define \( h =(h_1,h_2)'\in\R^2 \) where $h_1=x_j'y/\sqrt{2\log(p)}$, $h_2=x_{j+1}'y/\sqrt{2\log(p)}$; $\lambda=\lambda' \sqrt{2\log(p)}$;  $(\hat{\beta}_j, \hat{\beta}_{j+1})=\sqrt{2\log(p)}(\hat{b}_1, \hat{b}_2)$ are the  entries of the estimator \( \hat\beta^{\mathrm{{S}CAD}} \) corresponding to \( (x_j,x_{j + 1}) \). 

The estimator of SCAD is \( \hat{\beta}^{\text {SCAD }}=\operatorname{argmin}_{\beta}\left\{\|y-X \beta\|^{2} / 2+Q_{\lambda}(\beta)\right\} \). Like Elastic net, it can be decomposed into bivariate sub-problems of each pair of correlated variables. By dividing the bivariate sub-problem of \( (x_j,x_{j + 1}) \) by \( \sqrt{2\log(p)} \), we have
\beq \label{suppeq:optimization.scad}
L(b)\equiv \frac{1}{2}b'\begin{bmatrix}1&\rho\\\rho' & 1\end{bmatrix}b + b'h+\lambda'(q'(b_1)  + q'(b_2))
\eeq
and the minimizer of the optimization~\eqref{suppeq:optimization.scad} is \( (\hat b_1,\hat b_2) \).  The next lemma proves the solution to \eqref{suppeq:optimization.scad} when \( h_1 >\abs{h_2}\), and it is proven in Section~\ref{suppsec:sol.path.scad}.


\begin{lemma}[the solution path of SCAD] 
   \label{suppthm:sol.path.scad}
  Consider the optimization in \eqref{suppeq:optimization.scad}. Suppose \( h_1 >\abs{h_2} \), and suppose \( \rho \geq 0 \). 
  \begin{itemize}
    \item When \( \lambda' \geq \lambda'_1 =\max \left\{ \abs{h_1} ,\abs{h_2}\right\} \), \( \hat b_1=\ \hat b2 = 0 \).
    \item If \( (\rho - \frac{1}{a})h_1 < h_2 < \begin{cases} (\rho + \frac{1}{a})h_1 & \text{if }a > \frac{2}{1-\rho} \\ \frac{1+\rho}{2}h_1 & \text{if } a \leq  \frac{2}{1-\rho} \end{cases}  \),
    \begin{enumerate}
      \item When \( \lambda' < \lambda'_1\) and \( \lambda' \geq \abs{h_2 -\rho h_1} \),  \( \hat b_1 \neq 0 \) and \( \hat b_2 = 0 \). 
      \item When \( \lambda' < \abs{h_2 -\rho h_1} \), \( \hat b_1 \neq 0 \) and \( \hat b_2 \neq 0 \). 
    \end{enumerate}
    \item If \( a > \frac{2}{1 -\rho} \) and \( (\rho + \frac{1}{a})h_1 < h_2 < \frac{1 +\rho}{2}h_1 \): 
    \begin{enumerate}
      \item When \( \lambda' <\lambda'_1 \) and \( \lambda' \geq \frac{(a - 2)h_2 -\rho(a - 1)h_1}{a - 2 - a\rho} \), \( \hat b_1 \neq 0 \) and \( \hat b_2 = 0 \).
      \item When \( \lambda' < \frac{(a - 2)h_2 -\rho(a - 1)h_1}{a - 2 - a\rho}\), \( \hat b_1 \neq 0 \) and \( \hat b_2 \neq 0 \).
    \end{enumerate}
    \item If \( h_2 <(\rho - \frac{1}{a})h_1 \), \(\forall\, a \):
    \begin{enumerate}
      \item When \( \lambda' <\lambda'_1 \) and \( \lambda' \geq \frac{\rho(a - 1)h_1 -(a - 2)h_2}{a + a\rho - 2} \), \( \hat b_1 \neq 0 \) and \( \hat b_1 = 0 \).
      \item When \( \lambda' < \frac{\rho(a - 1)h_1 -(a - 2)h_2}{a + a\rho - 2}\), \( \hat b_1 \neq 0 \) and \( \hat b_2 \neq 0 \).
    \end{enumerate}
    \item If  \( h_2 \leq \frac{ - 1 +\rho}{2}h_1 \) or \( h_2 \geq  \frac{1 +\rho}{2}h_1 \): 
    \begin{enumerate}
      \item When \( \lambda' < \lambda'_1 \) and \( \lambda' \geq \lambda_2'^{(2)} = \frac{\abs{h_2 -\rho h_1}}{1 -\rho} \), \( \hat b_1 \neq 0 \) and \( \hat b_1 = 0 \).
      \item When \( \lambda' <\lambda_2'^{(2)} \), \( \hat b_1 \neq 0 \) and \( \hat b_2 \neq 0 \).
    \end{enumerate}
  \end{itemize}
\end{lemma}

We did not require \( \rho \geq  0 \) in the solution path of Elastic net, but here \( \rho \geq  0 \) is needed to cut down unnecessary discussion. The proof of Elastic net has shown that  the solution path of \( h_1>\abs{h_2}\) and \( \rho \geq  0\) is enough to draw the whole rejection region. 

Still requiring \( \rho > 0 \), the rejection region looks different for \( a > \frac{2}{1 -\rho} \) and \( a \leq \frac{2}{1 -\rho} \). The first steps are the same as those of Elastic net, and we only present the rejection region here:

When \( a \geq \frac{2}{1 -\rho} \): 
\begin{align} \label{proof-scad-rjRegion_1}
  {\cal R} &= \{(h_1,h_2): h_1-\rho h_2 >\lambda'(1-\rho),\, h_1 >\lambda',\,h_1 -\frac{|\rho|(a -1)}{a - 2}h_2 >\lambda'(1 -\frac{a\rho}{a -2})\}\cr 
  &\;\; \cup \{(h_1,h_2):h_1 -\rho h_2 >\lambda',h_2 > a\lambda'\}\cup \{(h_1, h_2): h_1-\rho h_2 >\lambda'(1+\rho)\}\cr
  &\;\; \cup \{(h_1, h_2): h_1 - \frac{\rho(a- 1)}{a - 2} h_2 > \lambda' (1 + \frac{a\rho}{a - 2}),h_1 -\rho h_2 >\lambda'\}\cr 
  &\;\; \cup \{(h_1,h_2): -h_1 +\rho h_2 >\lambda'(1-\rho),\, h_1 <-\lambda',\, -h_1 +\frac{|\rho|(a -1)}{a - 2}h_2 >\lambda'(1 -\frac{a\rho}{a -2})\}\cr 
  &\;\; \cup \{(h_1,h_2): -h_1 +\rho h_2 >\lambda',h_2 <- a\lambda'\}\cup \{(h_1, h_2): -h_1 +\rho h_2 >\lambda'(1+\rho)\}\cr
  &\;\; \cup \{(h_1, h_2): -h_1 + \frac{\rho(a- 1)}{a - 2} h_2 > \lambda' (1 + \frac{a\rho}{a - 2}), -h_1 +\rho h_2 >\lambda'\}
\end{align} 

When \( a \leq \frac{2}{1 -{\rho}} \):
\begin{align} \label{proof-scad-rjRegion_2}
  {\cal R} &= \{(h_1,h_2): h_1-\rho h_2 >\lambda'(1-\rho),\, h_1 >\lambda',\,h_1 >\frac{1 +\rho}{2}h_2 \}\cr 
  &\;\; \cup \{(h_1,h_2):h_1 -\rho h_2 >\lambda',h_2 > \frac{2\lambda'}{1 -\rho} \}\cup \{(h_1, h_2): h_1-\rho h_2 >\lambda'(1+\rho)\}\cr
  &\;\; \cup \{(h_1, h_2): h_1 - \frac{\rho(a- 1)}{a - 2} h_2 > \lambda' (1 + \frac{a\rho}{a - 2}),h_1 -\rho h_2 >\lambda'\}\cr 
  &\;\; \cup \{(h_1,h_2): -h_1 +\rho h_2 >\lambda'(1-\rho),\, h_1 <-\lambda',\,h_1 <\frac{1 +\rho}{2}h_2 \}\cr 
  &\;\; \cup \{(h_1,h_2): -h_1 +\rho h_2 >\lambda',h_2 <- \frac{2\lambda'}{1 -\rho} \}\cup \{(h_1, h_2): -h_1 +\rho h_2 >\lambda'(1+\rho)\}\cr
  &\;\; \cup \{(h_1, h_2): -h_1 + \frac{\rho(a- 1)}{a - 2} h_2 > \lambda' (1 + \frac{a\rho}{a - 2}), -h_1 +\rho h_2 >\lambda'\}\cr 
\end{align} 

When \( \rho <0 \), we apply the same sign-flipping technique in the proof of Elastic net and still use the rejection region of positive correlation. Such technique requires considering one more case for \( \rho < 0 \), which is the elliptical distance from the center \( \mu_{11} =((1 -\abs{\rho})\sqrt{r}, -(1 -\abs{\rho})\sqrt{r}) \) to \( {\cal R}^c \) (plotted with positive correlation \( \abs{\rho} > 0 \)). 

\paragraph{Part 2. Analyzing the Hamming error.} we allow \( \rho\in( -1,1) \) from now on. The analysis of Hamming error follows the same procedure as that of Elastic net. It is only the shape of \( \cal R \) which is different. For \( a \geq \frac{2}{1 -|\rho|} \) and \( a \leq \frac{2}{1 -|\rho|} \), we respectively present a theorem for the Hamming error rate. 
\begin{theorem}\label{suppthm:hamm.scad.1}
  Suppose the conditions of Theorem~\ref{thm:SCAD} hold. Let $\lambda=\lambda'\sqrt{2\log(p)}$ in SCAD. Define a few important points in the rejection region (as noted in Figure~\ref{suppfig:hamm.scad.large.a}): \( A(\lambda',\lambda'),\,B((1 +|\rho|)\lambda',2\lambda'),\,C((1 + a|\rho|)\lambda',a\lambda'),\, D((1 -|\rho|)\lambda', - 2\lambda') \).  As $p\to\infty$, 
  \[
  \FP_p=L_p p^{1- \min\bigl\{ \lambda'^2, \;\; \vt + f_1(\sqrt{r}, \lambda')\bigr\}}, \qquad \FN_p = L_p p^{1-\min\bigl\{\vt + f_2(\sqrt{r}, \lambda'),\;\; 2\vt + f_3(\sqrt{r}, \lambda')\bigr\}}, 
  \]
  where (below, $d^2_{|\rho|}(u,v)$ is as in Definition~\ref{def:EllipsDistance}) 
  \begin{align*}
   f_1(\sqrt{r},\lambda') & = \begin{cases} 
    (\lambda'-|\rho| \sqrt{r})^2 & \text{if }\sqrt{r} \leq \frac{\lambda'}{1 +|\rho|} \\
    \frac{1}{1 -|\rho|^2}d_{|\rho|}^2(A, (|\rho| \sqrt{r},\sqrt{r})) & \text{if } \frac{\lambda'}{1+|\rho|} \leq \sqrt{r} \leq \lambda' \\
    \frac{1 -|\rho|}{1 +|\rho|}\lambda'^2 & \text{if }\lambda' \leq \sqrt{r} \leq 2\lambda' \\
    \min \{ \frac{\lambda'^2}{1 -\rho^2},\frac{1}{1 -\rho^2} d^2(B, (|\rho| \sqrt{r},\sqrt{r}))\} & \text{if }2\lambda' \leq \sqrt{r} \leq \lambda'\left[2 + \frac{|\rho| -\rho^2}{(a - 2)(1 +|\rho|)}\right] \\
    \min\left\{\frac{\lambda'^2}{1 -\rho^2} ,\ \frac{\left[ \lambda'(1 -\frac{a|\rho|}{a - 2}) +\frac{|\rho| \sqrt{r}}{a - 2} \right]^2}{1 +\frac{\rho^2(a - 1)^2}{(a - 2)^2} -\frac{2\rho^2(a - 1)}{a - 2}}\right\} & \text{if } \sqrt{r} \geq \lambda'\left[2 + \frac{|\rho| -\rho^2}{(a - 2)(1 +|\rho|)}\right]
  \end{cases} \cr
  f_2(\sqrt{r}, \lambda') &= 
  \begin{cases} 
    \min \begin{cases} 
      (\sqrt{r} -\lambda')_ +^2 \\
      \frac{1}{1 -\rho^2}\left[(1 -\rho^2)\sqrt{r} -(1 -|\rho|)\lambda'\right] \\
      \frac{1}{1 +\frac{\rho^2(a - 1)^2}{(a - 2)^2} -\frac{2\rho^2(a - 1)}{a - 2}}\left[\lambda'\left(1 - \frac{a|\rho|}{a - 2}\right) - \sqrt{r} \cdot \left(1 -\frac{\rho^2(a - 1)}{a - 2}\right)\right]^2
      \end{cases} \text{if }\sqrt{r} \leq \frac{a(a - 2)(1 -\rho^2) +|\rho|}{|\rho|(a - 1)(1 - \rho^2)}\lambda'\\
      \frac{1}{1 -\rho^2}d_{|\rho|}^2(C, (\sqrt{r},|\rho| \sqrt{r})) \qquad\qquad\qquad \text{if } \frac{a(a - 2)(1 -\rho^2) +|\rho|}{|\rho|(a - 1)(1 - \rho^2)} \lambda'\leq \sqrt{r} \leq \frac{a\lambda'}{|\rho|} \\
      \frac{1}{1 -\rho^2}\left[(1 -\rho^2)\sqrt{r}-\lambda'\right]^2 \qquad\qquad\quad\ \text{if }\sqrt{r} \geq \frac{a\lambda'}{|\rho|}
  \end{cases} 
\end{align*}
The definition of \( f_3(\sqrt{r},\lambda') \) is different for different signs of \( \rho \). When \( \rho > 0\):
\begin{align*}
  f_3(\sqrt{r}, \lambda') &= \frac{1}{1 -\rho^2} \cdot \begin{cases} 
    \min \begin{cases} 
      \left[(1 -\rho^2)\sqrt{r} -(1 -\rho)\lambda'\right]_ +^2 \\
      h(\sqrt{r},\lambda')
      \end{cases} & \text{  if } \sqrt{r} \leq \frac{a\lambda'}{1+\rho} \\
      \left[(1 -\rho^2)\sqrt{r}-\lambda'\right]^2 & \text{  if } \sqrt{r} \geq \frac{a\lambda'}{1 +\rho}
  \end{cases} 
\end{align*}
where \begin{equation*}
  h(\sqrt{r},\lambda') = \begin{cases} 
    \frac{ (1 -\rho^2)}{1 +\frac{\rho^2(a - 1)^2}{(a - 2)^2} -\frac{2\rho^2(a - 1)}{a - 2}}\left[\lambda'\left(1 - \frac{a\rho}{a - 2}\right) - (1 +\rho)\sqrt{r} \cdot \left(1 -\frac{\rho(a - 1)}{a - 2}\right)\right]^2 & \text{ if } \sqrt{r} \leq \frac{\lambda'}{1 +\rho} \cdot \frac{a(a - 2)(1 -\rho^2) +\rho}{(a - 2)(1 -\rho^2) +\rho -\rho^2} \\
    d^2(C, ((1 +\rho)\sqrt{r},(1 +\rho) \sqrt{r})) & \text{ if } \sqrt{r} \geq \frac{\lambda'}{1 +\rho} \cdot \frac{a(a - 2)(1 -\rho^2) +\rho}{(a - 2)(1 -\rho^2) +\rho -\rho^2}
  \end{cases} 
\end{equation*}
When \( \rho < 0 \), 
\begin{align*}
  f_3(\sqrt{r},\lambda') =\frac{1}{1 -\rho^2} \cdot  \begin{cases} 
      \left[(1 -\rho^2)\sqrt{r}-(1+|\rho|)\lambda'\right]^2 & 
      \text{if } \sqrt{r} \leq \frac{2\lambda'}{1 -|\rho|} \\
      \min \begin{cases} 
        \left[(1 -\rho^2)\sqrt{r} -\lambda'\right]^2\\
        k(\lambda',a)
      \end{cases} & \text{if } \sqrt{r} \geq \frac{2\lambda'}{1 -|\rho|}
    \end{cases}
\end{align*}
where 
\begin{equation*}
  k(\lambda',a) = \begin{cases} 
    d^2\left(D,\left((1 -|\rho|)\sqrt{r},-(1-|\rho|)\sqrt{r}\right)\right) \qquad \text{   if } \frac{2\lambda'}{1 -|\rho|} \leq \sqrt{r} \leq \frac{\lambda'}{1 -|\rho|}\left[2 + \frac{|\rho| +\rho^2}{(a - 2)(1 -\rho^2) - (|\rho| +\rho^2)} \right]\\
    \frac{ (1 -\rho^2)}{1 +\frac{\rho^2(a - 1)^2}{(a - 2)^2} -\frac{2\rho^2(a - 1)}{a - 2}}\left[ -\lambda'\left(1 + \frac{a|\rho|}{a - 2}\right) + (1 -|\rho|)\sqrt{r} \cdot \left(1 +\frac{|\rho|(a - 1)}{a - 2}\right)\right]^2 \\
    \qquad\qquad\qquad{if } \sqrt{r} \geq \frac{\lambda'}{1 -|\rho|}\left[2 + \frac{|\rho| +\rho^2}{(a - 2)(1 -\rho^2) - (|\rho| +\rho^2)} \right]
  \end{cases} 
\end{equation*}
  \end{theorem}

\begin{proof}[Proof of Theorem~\ref{suppthm:hamm.scad.1}]
  See the rejection region in Figure~\ref{suppfig:hamm.scad.large.a} for a visualization of the rejection region. 
  \begin{figure}[h!]
    \centering
    \includegraphics[width=1.1\textwidth]{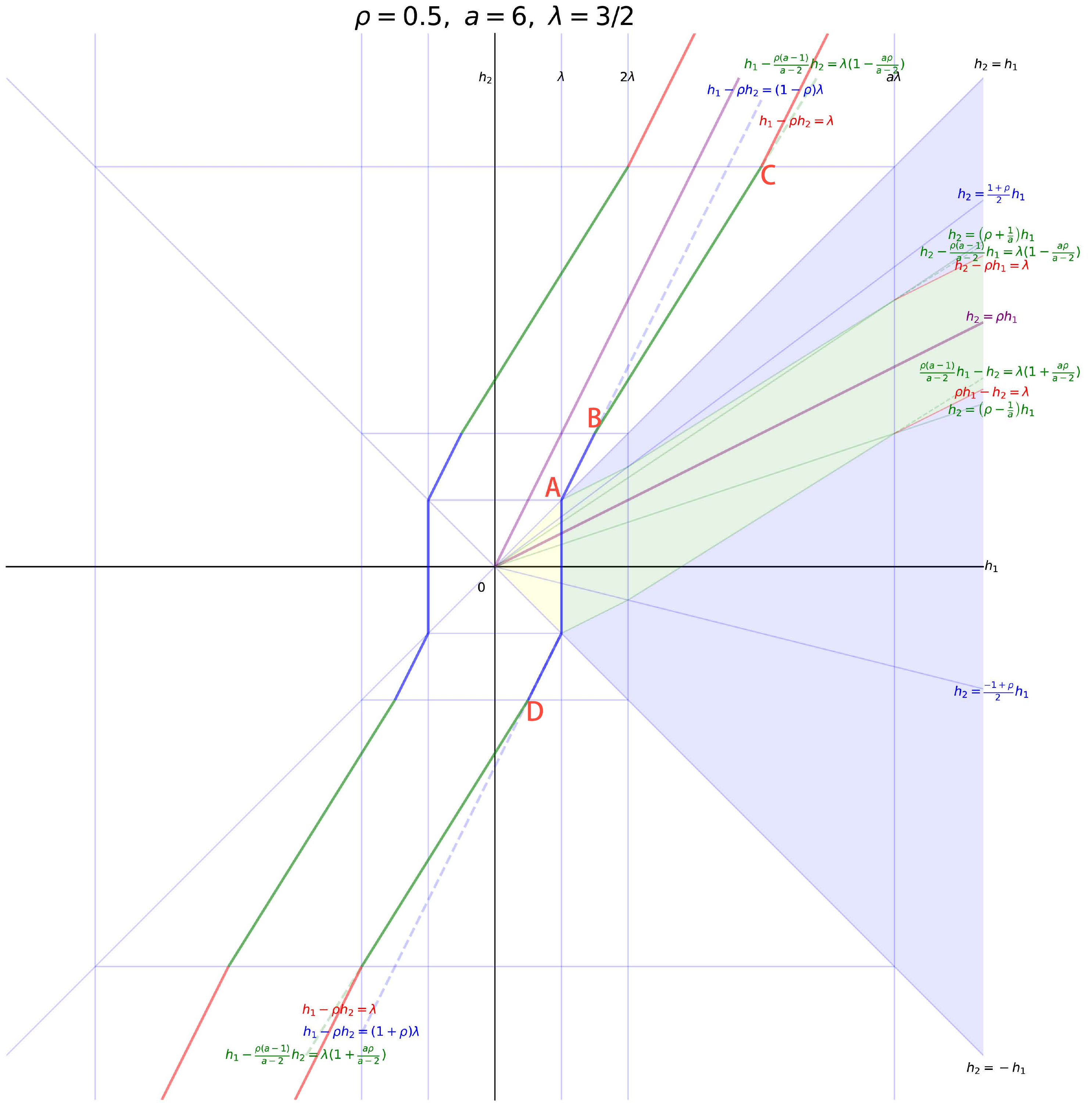}
    \caption{When \( a > \frac{2}{1 -|\rho|} \), the rejection region looks like this}
    \label{suppfig:hamm.scad.large.a}
  \end{figure}

\textit{The ellipsoid centered at the point \( \mu_{00} \)}: Easy to see the rate is \( L_p \cdot p^{1 -\lambda'^2} \). 

\textit{The ellipsoid centered at the point \( \mu_{01} \) }: We set out to find out \( f_1(\sqrt{r},\lambda') \).
Similar to Lasso, we have: (i) When \( \sqrt{r} \leq \frac{\lambda'}{1 +|\rho|} \), \( f_1(\sqrt{r},\lambda') =(\lambda'-|\rho| \sqrt{r})^2 \).
(ii) When  \( \frac{\lambda'}{1 +|\rho|} \leq \sqrt{r} \leq \lambda' \), \( f_1(\sqrt{r},\lambda') = \frac{1}{1 -\rho^2} d^2(A, (|\rho| \sqrt{r},\sqrt{r}))\). The point \( A \) is noted in Figure~\ref{suppfig:hamm.scad.large.a}.
(iii) When \( \lambda' \leq \sqrt{r} \leq 2\lambda' \), \( f_1(\sqrt{r},\lambda') = \frac{1 -|\rho|}{1 +|\rho|}\lambda'^2\).

Then we need to investigate the green segment in Figure~\ref{suppfig:hamm.scad.large.a}. When the ellipsoid is tangent to the green segment on the \textit{right} side (i.e. \( BC \)), and the tangent point is above Point \( B \), then using Lemma~\ref{supplem:distance},
\begin{equation*}
  \sqrt{r} +\frac{ -\frac{|\rho|}{a - 2}\left[\lambda'\left(1 - \frac{a|\rho|}{a - 2}\right) + \sqrt{r} \cdot \frac{|\rho|}{a - 2}\right]}{1 +\frac{\rho^2(a - 1)^2}{(a - 2)^2} -\frac{2\rho^2(a - 1)}{a - 2}} \geq 2\lambda'
\end{equation*}
which implies \( \sqrt{r} \geq \lambda'\left[2 + \frac{|\rho| -\rho^2}{(a - 2)(1 +|\rho|)}\right]  \)

When \( \sqrt{r} \geq \lambda'\left[2 + \frac{|\rho| -\rho^2}{(a - 2)(1 +\rho)}\right]  \), the ellipsoid either intersects with the green line segment \( BC \), or the red segment beyond \( C \). However, we need to eliminate the possibility of the ellipsoid having a smaller radius when tangent to the segments on the left. 

We will see that the line segments on the left can indeed be eliminated. This is because when the ellipsoid is tangent to both the green lines on the left and right, 
\begin{align*}
  \frac{(1 -\rho^2)\left[ \lambda'(1 -\frac{a|\rho|}{a - 2}) +\frac{|\rho| \sqrt{r}}{a - 2} \right]^2}{1 +\frac{\rho^2(a - 1)^2}{(a - 2)^2} -\frac{2\rho^2(a - 1)}{a - 2}} =&~ \frac{(1 -\rho^2)\left[ -\lambda'(1 +\frac{a|\rho|}{a - 2}) +\frac{|\rho| \sqrt{r}}{a - 2} \right]^2}{1 +\frac{\rho^2(a - 1)^2}{(a - 2)^2} -\frac{2\rho^2(a - 1)}{a - 2}}\\
  \implies \lambda'(1 -\frac{a|\rho|}{a - 2}) +\frac{|\rho| \sqrt{r}}{a - 2} =&~ \lambda'(1 +\frac{a|\rho|}{a - 2}) -\frac{|\rho| \sqrt{r}}{a - 2} \implies \sqrt{r} = a\lambda
\end{align*}
As a result, when  \( \sqrt{r} \leq a\lambda \), we can ignore the possibility of the ellipsoid intersecting the green or red segments on the left side. When \( \sqrt{r} \geq a\lambda \), the right side still has the smaller distance. 

when \( 2\lambda' \leq \sqrt{r} \leq  \lambda'\left[2 + \frac{|\rho| -\rho^2}{(a - 2)(1 +|\rho|)}\right]  \), \( f_1(\sqrt{r},\lambda') = \frac{1}{1 -\rho^2} d^2(B, (|\rho| \sqrt{r},\sqrt{r}))\). The point \( B \) is noted in Figure~\ref{suppfig:hamm.scad.large.a}.

when \( \sqrt{r} \geq  \lambda'\left[2 + \frac{|\rho| -\rho^2}{(a - 2)(1 +|\rho|)}\right]  \), \begin{equation*}
  f_1(\sqrt{r},\lambda') = \frac{1}{1 -\rho^2} \min\left\{\lambda'^2,\ \frac{(1 -\rho^2)\left[ \lambda'(1 -\frac{a|\rho|}{a - 2}) +\frac{|\rho| \sqrt{r}}{a - 2} \right]^2}{1 +\frac{\rho^2(a - 1)^2}{(a - 2)^2} -\frac{2\rho^2(a - 1)}{a - 2}}\right\}. 
\end{equation*}

\textit{The ellipsoid centered at the  \( \mu_{10} \)}: We only explain one tricky point: When the tangent point to the segment \( BC \) is precisely Point \( C \), 
\begin{equation*}
  |\rho|\sqrt{r} +\frac{ -\frac{|\rho|}{a - 2}\left[\lambda'\left(1 - \frac{a|\rho|}{a - 2}\right) - \sqrt{r} \cdot \left(1 -\frac{\rho^2(a - 1)}{a - 2}\right)\right]}{1 +\frac{\rho^2(a - 1)^2}{(a - 2)^2} -\frac{2\rho^2(a - 1)}{a - 2}} = a\lambda'
\end{equation*}
then \( \sqrt{r} = \frac{a(a - 2)(1 -\rho^2) +|\rho|}{|\rho|(a - 1)(1 - \rho^2)} \). 






  \textit{The ellipsoid centered at  \( \mu_{11} =\left((1 +\rho)\sqrt{r},(1 +\rho)\sqrt{r}\right) \), when \( \rho > 0 \) }: We still explain only one important point: When the ellipsoid is tangent to the green segment precisely at Point \( C \), 
  \begin{equation*}
    (1 +\rho)\sqrt{r} +\frac{ -\frac{\rho}{a - 2}\left[\lambda'\left(1 - \frac{a\rho}{a - 2}\right) - (1 +\rho)\sqrt{r} \cdot \left(1 -\frac{\rho(a - 1)}{a - 2}\right)\right]}{1 +\frac{\rho^2(a - 1)^2}{(a - 2)^2} -\frac{2\rho^2(a - 1)}{a - 2}} = a\lambda'
  \end{equation*}
  then 
  \( (1 +\rho)\sqrt{r} =\frac{a(a - 2)(1 -\rho^2) +\rho}{(a - 2)(1 -\rho^2) +\rho -\rho^2}\lambda' \).
  





  \textit{The ellipsoid centered at  \( \mu_{11} = \left((1 -|\rho|)\sqrt{r}, -(1 -|\rho|)\sqrt{r}\right) \), when \( \rho < 0 \) }: We explain one important point: when the ellipsoid is tangent to the green segment at the Point \( D \). Now we have
  \begin{equation*}
    -(1 -|\rho|)\sqrt{r} +\frac{ -\frac{|\rho|}{a - 2}\left[\lambda'\left(1 + \frac{a|\rho|}{a - 2}\right) - (1 -|\rho|)\sqrt{r} \cdot \left(1 +\frac{|\rho|(a - 1)}{a - 2}\right)\right]}{1 +\frac{\rho^2(a - 1)^2}{(a - 2)^2} -\frac{2\rho^2(a - 1)}{a - 2}} = - 2\lambda'
  \end{equation*}
  then 
  \( (1 -|\rho|)\sqrt{r} =\lambda'\left[2 + \frac{|\rho| +\rho^2}{(a - 2)(1 -\rho^2) - (|\rho| +\rho^2)} \right]\).
  
  
  Note that even when \( \frac{2\lambda'}{1 -|\rho|} \leq \sqrt{r} \leq \frac{\lambda'}{1 -|\rho|}\left[2 + \frac{|\rho| +\rho^2}{(a - 2)(1 -\rho^2) - (|\rho| +\rho^2)} \right] \), and the ellipsoid intersects with the rejection region at Point \( D \), it may be tangent to the red segment without being tangent to the green segment. This is especially true when \( a < 2 +\frac{|\rho|}{1 -|\rho|} \) (but this only happens when \( a < \frac{2}{1 -|\rho|} \), the next section.)
\end{proof}

\begin{theorem}\label{suppthm:hamm.scad.2}
  Suppose the conditions of Theorem~\ref{thm:SCAD} hold. Let $\lambda=\lambda'\sqrt{2\log(p)}$ in SCAD. Define a few important points (as noted in the rejection region in Figure~\ref{suppfig:hamm.scad.small.a}): \( A(\lambda',\lambda'),\,B((1 +|\rho|)\lambda',2\lambda'),\,C(\frac{1 + |\rho|}{1 -|\rho|} \lambda',\frac{2\lambda'}{1 -|\rho|} ),\, D((1 -|\rho|)\lambda', - 2\lambda') \).  As $p\to\infty$, 
  \[
  \FP_p=L_p p^{1- \min\bigl\{ \lambda'^2, \;\; \vt + f_1(\sqrt{r}, \lambda')\bigr\}}, \qquad \FN_p = L_p p^{1-\min\bigl\{\vt + f_2(\sqrt{r}, \lambda'),\;\; 2\vt + f_3(\sqrt{r}, \lambda')\bigr\}}, 
  \]
  where (below, $d^2_{|\rho|}(u,v)$ is as in Definition~\ref{def:EllipsDistance}) 
  \begin{align*}
   f_1(\sqrt{r},\lambda') & =\begin{cases} 
    (\lambda'-|\rho| \sqrt{r})^2 & \text{if }\sqrt{r} \leq \frac{\lambda'}{1 +|\rho|} \\
    \frac{1}{1 -\rho^2}d^2(A, (|\rho| \sqrt{r},\sqrt{r})) & \text{if } \frac{\lambda'}{1+|\rho|} \leq \sqrt{r} \leq \lambda' \\
    \frac{1 -|\rho|}{1 +|\rho|} \lambda'^2 & \text{if }\lambda' \leq \sqrt{r} \leq 2\lambda' \\
    \min\left\{ \frac{\lambda'^2}{1 -\rho^2},\frac{1}{1 -\rho^2}d^2(B, (|\rho| \sqrt{r},\sqrt{r}))  \right\} & \text{if }2\lambda' \leq \sqrt{r} \leq \frac{5 + 3|\rho|}{2 + 2|\rho|} \lambda' \\
    \min\left\{\frac{\lambda'^2}{1 -\rho^2},\ \frac{(1-|\rho|) r}{(5+3|\rho|)}\right\}  & \text{if } \sqrt{r} \geq \frac{5 + 3|\rho|}{2 + 2|\rho|} \lambda'
  \end{cases}  \cr
  f_2(\sqrt{r}, \lambda') &= 
  \begin{cases} 
    \min \begin{cases} 
      (\sqrt{r} -\lambda')_ +^2 \\
      \frac{1}{1 -\rho^2}\left[(1 -\rho^2)\sqrt{r} -(1 -|\rho|)\lambda'\right]^2 \\
      \frac{(1 -|\rho|)(2 +|\rho|)^2}{5 + 3|\rho|} r
      \end{cases} & 
       \text{if }\sqrt{r} \leq \frac{5 + 3|\rho|}{(1 -|\rho|)(1 +|\rho|)^2}\lambda'\\
      \frac{1}{1 -\rho^2}d^2(C, (\sqrt{r},|\rho| \sqrt{r})) & \text{if } \frac{5 + 3|\rho|}{(1 -|\rho|)(1 +|\rho|)^2}\lambda' \leq \sqrt{r} \leq \frac{2\lambda'}{|\rho|(1-|\rho|)} \\
      \frac{1}{1 -\rho^2}\left[(1 -\rho^2)\sqrt{r}-\lambda'\right]^2& \text{if }\sqrt{r} \geq \frac{2\lambda'}{|\rho|(1-|\rho|)}
  \end{cases} 
\end{align*}
The definition of \( f_3(\sqrt{r},\lambda') \) is different for different signs of \( \rho \). When \( \rho > 0\):
\begin{align*}
  f_3(\sqrt{r}, \lambda') &= \frac{1}{1 -\rho^2} \cdot \begin{cases} 
    \min \begin{cases} 
      \left[(1 -\rho^2)\sqrt{r} -(1 -\rho)\lambda'\right]_ +^2 \\
      h(\sqrt{r},\lambda')
      \end{cases} & \text{  if } \sqrt{r} \leq \frac{2\lambda'}{1-\rho^2} \\
      \left[(1 -\rho^2)\sqrt{r}-\lambda'\right]^2 & \text{  if } \sqrt{r} \geq \frac{2\lambda'}{1 -\rho^2}
  \end{cases} 
\end{align*}
where \begin{equation*}
  h(\sqrt{r},\lambda') = \begin{cases} 
    \frac{(1 -\rho^2)^2(1 +\rho)}{5 + 3\rho}r & \text{ if } \sqrt{r} \leq \frac{\lambda'}{1 +\rho} \cdot \frac{5 + 3\rho}{(1 -\rho)(3 +\rho)} \\
    d^2(C, ((1 +\rho)\sqrt{r},(1 +\rho) \sqrt{r})) & \text{ if } \sqrt{r} \geq \frac{\lambda'}{1 +\rho} \cdot \frac{5 + 3\rho}{(1 -\rho)(3 +\rho)}
  \end{cases} 
\end{equation*}
When \( \rho < 0 \), 
\begin{align*}
  f_3(\sqrt{r},\lambda') =\frac{1}{1 -\rho^2} \cdot  \begin{cases} 
      \left[(1 -\rho^2)\sqrt{r}-(1+|\rho|)\lambda'\right]^2 & 
      \text{if } \sqrt{r} \leq \frac{2\lambda'}{1 -|\rho|} \\
      \min \begin{cases} 
        \left[(1 -\rho^2)\sqrt{r} -\lambda'\right]^2\\
        k(\lambda',a)
      \end{cases} & \text{if } \sqrt{r} \geq \frac{2\lambda'}{1 -|\rho|}
    \end{cases}
\end{align*}
where 
\begin{equation*}
  k(\lambda',a) = \begin{cases} 
    d^2\left(D,\left((1 -|\rho|)\sqrt{r},-(1-|\rho|)\sqrt{r}\right)\right) \qquad \text{   if } \frac{2\lambda'}{1 -|\rho|} \leq \sqrt{r} \leq \frac{\lambda'}{1 -|\rho|}\left[2 + \frac{|\rho| +\rho^2}{(a - 2)(1 -\rho^2) - (|\rho| +\rho^2)} \right]\\
    \frac{ (1 -\rho^2)}{1 +\frac{\rho^2(a - 1)^2}{(a - 2)^2} -\frac{2\rho^2(a - 1)}{a - 2}}\left[ -\lambda'\left(1 + \frac{a|\rho|}{a - 2}\right) + (1 -|\rho|)\sqrt{r} \cdot \left(1 +\frac{|\rho|(a - 1)}{a - 2}\right)\right]^2 \\
    \qquad\qquad\qquad{if } \sqrt{r} \geq \frac{\lambda'}{1 -|\rho|}\left[2 + \frac{|\rho| +\rho^2}{(a - 2)(1 -\rho^2) - (|\rho| +\rho^2)} \right]
  \end{cases} 
\end{equation*}
  \end{theorem}

\begin{proof}[Proof of Theorem~\ref{suppthm:hamm.scad.2}]
See the different rejection region in Figure~\ref{suppfig:hamm.scad.small.a}.
\begin{figure}[h!]
  \centering
  \includegraphics[width=1.1\textwidth]{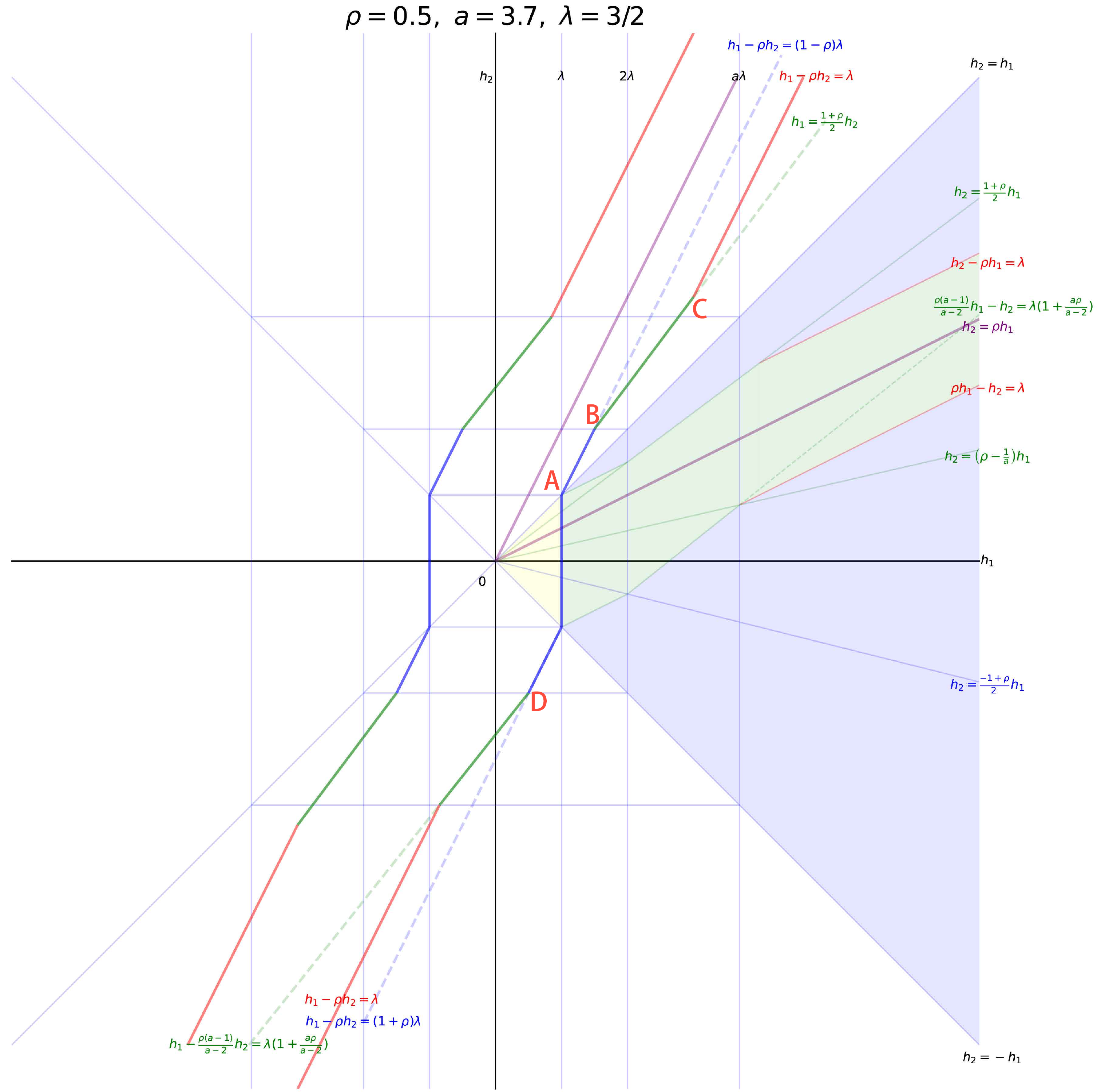}
  \caption{When \( a \leq  \frac{2}{1 -|\rho|} \), the rejection region looks like this}
  \label{suppfig:hamm.scad.small.a}
\end{figure}
  
\textit{The ellipsoid centered at \( \mu_{00} \):} The rate is \( L_p \cdot p^{1 -\lambda'^2} \) 

\textit{The ellipsoid centered at \( \mu_{01} \):} 
Still similar to Lasso, we have
\begin{itemize}
  \item when \( \sqrt{r} \leq \frac{\lambda'}{1 +|\rho|} \), \( f_1(\sqrt{r},\lambda') = (\lambda'-|\rho| \sqrt{r})^2 \).
  \item when  \( \frac{\lambda'}{1 +|\rho|} \leq \sqrt{r} \leq \lambda' \), \( f_1(\sqrt{r},\lambda') = \frac{1}{1 -\rho^2} d^2(A, (|\rho| \sqrt{r},\sqrt{r}))\). The point \( A \) has been defined in Theorem~\ref{suppthm:hamm.scad.2} and noted in Figure~\ref{suppfig:hamm.scad.small.a}. 
  \item when \( \lambda' \leq \sqrt{r} \leq 2\lambda' \), \( f_1(\sqrt{r},\lambda')= \lambda'^2\).
\end{itemize}

Then we need to investigate the green segment in Figure~\ref{suppfig:hamm.scad.small.a}. When the ellipsoid is tangent to the green segment on the \textit{right} side (i.e. \( BC \)), and the tangent point is above Point \( B \), then using Lemma~\ref{supplem:distance},
\begin{equation*}
  \sqrt{r} - \frac{\frac{1}{4}(1 -|\rho|)^2 \sqrt{r}}{1 +(\frac{1 +|\rho|}{2} )^2 -|\rho|(1 +|\rho|)} \geq 2\lambda'
\end{equation*}
which implies \( \sqrt{r} \geq \frac{5 + 3|\rho|}{2 + 2|\rho|} \lambda'. \)

When \( \sqrt{r} \geq \frac{5 + 3|\rho|}{2 + 2|\rho|} \lambda'  \), the ellipsoid either intersects with the green line segment \( BC \), or the red segment beyond \( C \). However, we need to eliminate the possibility of the ellipsoid having a smaller radius when tangent to the segments on the left.

Actually, we will see that the line segments on the left can indeed be eliminated, without doing any computation. The case of \( a \leq \frac{2}{1 -|\rho|} \) is a degenerate case, as we have \( \frac{|\rho|(a - 1)}{a - 2} \geq \frac{1 +|\rho|}{2} \).
when \( a \leq \frac{2}{1 -|\rho|} \). From the computation in the \( a > \frac{2}{1 -|\rho|} \) counterpart, the green and red segments on the left sides can be ignored. 

So we can continue the discussion and present the rest two cases:
\begin{itemize}
  \item when \( 2\lambda' \leq \sqrt{r} \leq  \frac{5 + 3|\rho|}{2 + 2|\rho|} \lambda'  \), \( \text{rate} = d^2(B, (|\rho| \sqrt{r},\sqrt{r}))\). The point \( B \) is noted in Figure~\ref{suppfig:hamm.scad.small.a}.
  \item when \( \sqrt{r} \geq  \frac{5 + 3|\rho|}{2 + 2|\rho|} \lambda' \), \begin{equation*}
    \text{rate} = \min\left\{\lambda'^2,\ \frac{(1 -|\rho|^2)(1-|\rho|) r}{(5+3|\rho|)}\right\} 
  \end{equation*}
\end{itemize}

\textit{The ellipsoid centered at \( \mu_{10} \) }: Only one special point needs to be investigated. When the tangent point to the segment \( BC \) is precisely Point \( C\left( \frac{1 +|\rho|}{1 -|\rho|}\lambda',\ \frac{2\lambda'}{1 -|\rho|} \right) \), 
\begin{equation*}
  |\rho|\sqrt{r} +\frac{ -\frac{1 -|\rho|}{2}\left[\frac{|\rho|(1 +|\rho|)}{2} -1\right]\sqrt{r}}{1 +(\frac{1 +|\rho|}{2} )^2 -|\rho|(1 +|\rho|)} = \frac{2\lambda'}{1 -|\rho|}
\end{equation*}
then \( \sqrt{r} = \frac{5 + 3|\rho|}{(1 -|\rho|)(1 +|\rho|)^2}\lambda' \).


\textit{The ellipsoid centered at \( \mu_{11} =\left((1 +\rho)\sqrt{r},(1 +\rho)\sqrt{r}\right) \),  when \( \rho > 0 \).} We explain one special point: When the ellipsoid is tangent to the green segment precisely at Point \( C \), 
\begin{equation*}
  (1 +\rho)\sqrt{r} +\frac{ \frac{(1 -\rho)(1 -\rho^2)}{4}\sqrt{r}}{1 +(\frac{1 +\rho}{2} )^2 -\rho(1 +\rho)} = \frac{2\lambda'}{1 -\rho}
\end{equation*}
then \( \sqrt{r} = \frac{5 + 3\rho}{(1 -\rho^2)(3 +\rho)}\lambda' \).


\textit{The ellipsoid centered at \( \mu_{11} =\left((1 -|\rho|)\sqrt{r}, -(1 -|\rho|)\sqrt{r}\right) \), only when \( \rho < 0 \)}: This case is identical to the counterpart proof for \( a \geq \frac{2}{1 -|\rho|} \), and nothing needs to be changed.
\end{proof}

\paragraph{Part 3. Calculating the phase diagram, for \( a \leq \frac{2}{1-|\rho|} \).}  The boundary between the Regions of Almost Full Recovery and No Recovery is still \( r =\vt \), and it can be proven in the same manner as that of Elastic net. For the rest of this part, we focus on the boundary between Exact Recovery and Almost Full Recovery.

We focus on the case of \( a < \frac{2}{1 -|\rho|}\), because: First, the phase diagram of SCAD when \( a < \frac{2}{1 -|\rho|} \)  is worse than Lasso's diagram when \( \rho > 0 \), and  becomes the same as Lasso when \( a \) is sufficiently larger than \( \frac{2}{1 -\rho} \). When \( \rho <0 \), the phase diagram is better than that of Lasso when \( a \leq  \frac{2}{1 -|\rho|} \), and numerical results show that when \( a > \frac{2}{1 -|\rho|} \), the diagram is monotonically moving upwards towards Lasso's diagram when \( a \) is increasing. Second, when \( a <\frac{2}{1 -|\rho|} \), Theorem~\ref{suppthm:hamm.scad.2} does not depend on \( a \) in its most part, and is much easier to compute.  To sum up, the case of \( a > \frac{2}{1-|\rho|} \) is much more tedious in computation but less informative.

\paragraph{We start from the case of \( \rho > 0 \).}
Before diving into the proof, we give an overall account for the diagram:
\begin{enumerate}
  \item The diagram is the same as that of Lasso, only except that when \( \rho < 0.179 \), there is a tiny difference in a small neighborhood of \( \vt= 0\), slightly worse than Lasso. See equation~\eqref{eq:extra.curve}
  \item As long as \( a < \frac{2}{1 -\rho} \), the phase diagram does not depend on the specific value of \( a \).    
\end{enumerate}

Then we move on to the proof, which has four cases just like the proof of Elastic net.

\textit{First, \( \lambda'^2 =\vt + f_2(\sqrt{r},\lambda') = 1 \) and \( \vt +f_1(\sqrt{r},\lambda') \geq 1\),\( 2\vt + f_3(\sqrt{r},\lambda') \geq 1\):} We know \( \lambda' = 1 \). From the definitin of \( f_2(\sqrt{r},\lambda') \), we know \( \sqrt{r} \geq  1 + \sqrt{1 -\vt} >1\). 

Then we start our discussion on the conditional expression of \( f_2(\sqrt{r},\lambda') \). (Note that numerically, \( \min_\rho \frac{5 + 3\rho}{(1 -\rho)(1 +\rho)^2} = 4.848\).) 

If \( 1 -\vt = \frac{1}{1 -\rho^2}d^2(C,(\rho \sqrt{r},\sqrt{r})) \): As we know \( d^2(C,(\rho \sqrt{r},\sqrt{r})) \geq \frac{(1 -\rho^2)(1 -\rho)(2 +\rho)^2}{5 + 3\rho} r \), so \( \sqrt{r} \leq \sqrt{1 -\vt}\sqrt{\frac{5 + 3\rho}{(1 -\rho)(2 +\rho)^2}} \), which contradicts the pre-condition that \( \sqrt{r} \geq \frac{5 + 3\rho}{(1 -\rho)(1 +\rho)^2} \). We have no curve in this case. If \( 1 -\vt = \frac{1}{1 -\rho^2}[(1 -\rho^2)\sqrt{r} - 1]^2 \), then \( \sqrt{r} = \sqrt{\frac{1 -\vt}{1 -\rho^2}} + \frac{1}{1 -\rho^2} \). This again contradicts the pre-condition that \( \sqrt{r} \geq \frac{2}{\rho(1 -\rho)} \). We have no curve in this case.

  As a result, we can only have \( 1 <\sqrt{r} <\frac{5 + 3\rho}{(1 -\rho)(1 +\rho)^2} \) and one of the following three:
  \begin{align*}
    & \sqrt{r} = 1 + \sqrt{1 -\vt}\\
    & \sqrt{r} = \sqrt{ \frac{1 -\vt}{1 -\rho^2}} + \frac{1}{1 +\rho} \\
    & \sqrt{r} = \sqrt{1 -\vt}\sqrt{\frac{5 + 3\rho}{(1 -\rho)(2 +\rho)^2}}
  \end{align*}
We then discuss the three curves one by one, starting from the last one. 

\begin{enumerate}
  \item \( \sqrt{r} = \sqrt{1 -\vt}\sqrt{\frac{5 + 3\rho}{(1 -\rho)(2 +\rho)^2}} \). We need to look at \( FP_2 \) to eliminate this curve. 

    When \( 1 \leq \sqrt{r} \leq 2 \), we have \( \sqrt{1 -\vt} \leq \sqrt{\frac{1 -\rho}{1 +\rho}} \). In \( FN_1 \), for the last term to be the minimum among the three 
     \[ \min \begin{cases} 
      (1 -\rho^2)(\sqrt{r} -\lambda')_ +^2 \\
      \left[(1 -\rho^2)\sqrt{r} -(1 -\rho)\lambda'\right]^2 \\
      \frac{(1 -\rho^2)(1 -\rho)(2 +\rho)^2}{5 + 3\rho} r
      \end{cases}, \]
      we need \( \sqrt{1 -\vt}\sqrt{\frac{5 + 3\rho}{(1 -\rho)(2 +\rho)^2}} \geq  \sqrt{ \frac{1 -\vt}{1 -\rho^2}} + \frac{1}{1 +\rho}  \), which implies
      \begin{equation*}
        \sqrt{\frac{1 -\rho}{1 +\rho}} \left[\sqrt{\frac{5 + 3\rho}{(1 -\rho)(2 +\rho)^2}} -\frac{1}{\sqrt{1 -\rho^2}}\right] \geq \frac{1}{1 +\rho}. 
      \end{equation*} 
      Simplify this for a few steps and we can see the contradiction.

      When \( \sqrt{r} \geq \frac{5 + 3\rho}{2 + 2\rho} \), we can see the contradiction by simplifying this inequality itself. 

      When \( 2 < \sqrt{r} < \frac{5 + 3\rho}{2 + 2\rho} \), by looking at 
      \begin{align*}
        & \sqrt{1 -\vt}\sqrt{\frac{5 + 3\rho}{(1 -\rho)(2 +\rho)^2}} \leq \frac{5 + 3\rho}{2 + 2\rho} \\
        & \sqrt{1 -\vt}\sqrt{\frac{5 + 3\rho}{(1 -\rho)(2 +\rho)^2}} \geq \frac{1}{1 +\rho} + \sqrt{\frac{1 -\vt}{1 -\rho^2}}
      \end{align*}
      we can see that no \( \rho\in(0,1) \) can admit a possible \( \sqrt{1 -\vt} \).

  \item \( \sqrt{r} = \frac{1}{1 +\rho} +\sqrt{\frac{1 -\vt}{1 -\rho^2}} \). We need to look at \( FP_2 \) to eliminate this curve. 

    We already have \( \sqrt{r} \geq 1 \); when \( 1 < \sqrt{r} \leq 2 \), from the rate of \( FP_2 \), we have \( 1 -\vt \leq \frac{1 -\rho}{1 + \rho}. \) 
    However, In \( FN_1 \), for the middle term to be the minimum among the three 
    \[ \min \begin{cases} 
     (1 -\rho^2)(\sqrt{r} -\lambda')_ +^2 \\
     \left[(1 -\rho^2)\sqrt{r} -(1 -\rho)\lambda'\right]^2 \\
     \frac{(1 -\rho^2)(1 -\rho)(2 +\rho)^2}{5 + 3\rho} r
     \end{cases}, \] we need
     \begin{equation*}
      \frac{1}{1 +\rho} +\sqrt{\frac{1 -\vt}{1 -\rho^2}} > 1 + \sqrt{1 -\vt}.
     \end{equation*} The upper and lower bound of \( \sqrt{1 -\vt} \) would render this case impossible. 

    When \( 2 < \sqrt{r} \leq \frac{5 + 3\rho}{2 + 2\rho} \), using the expression of \eqref{eq:dist.B}, we need 
     \begin{equation*}
       \rho^2 r - 2(1 +\rho)\sqrt{r} + 4 \geq 0
     \end{equation*} in which we use \( \sqrt{r} \) to express \( \sqrt{1 -\vt} \). By letting \( \sqrt{r} = 2 \) or \( \frac{5 + 3\rho}{2 + 2\rho} \), we can see they are both negative, so we have a contradiction. 

     When \( \sqrt{r} > \frac{5 + 3\rho}{2 + 2\rho}\), we have 
     \begin{align*}
     & \frac{1}{1 +\rho} + \sqrt{\frac{1 -\vt}{1 -\rho^2}} >  \frac{5 + 3\rho}{2 + 2\rho} \\
      & \sqrt{\frac{5 + 3\rho}{1 -\rho}}\sqrt{1 -\vt} \leq \frac{1}{1 +\rho} + \sqrt{\frac{1 -\vt}{1 -\rho^2}}
     \end{align*} 
     and the upper bound on \( \sqrt{1 -\vt} \) is even smaller than the lower bound. Contradiction.

\end{enumerate}

Now we are only left with \( \sqrt{r} = 1 + \sqrt{1 -\vt} \). We need it to meet the following requirements:
\begin{equation*}
  \begin{cases} 
  1 + \sqrt{1 -\vt} \geq \frac{1}{1 + p} + \sqrt{\frac{1 -\vt}{1 -\rho^2}} & \text{for it to be the smallest among the three}\\
  1 + \sqrt{1 -\vt} \geq \sqrt{1 -\vt}\sqrt{\frac{5 + 3\rho}{(1 -\rho)(2 +\rho)^2}} & \text{for it to be the smallest among the three}\\
  1 + \sqrt{1 -\vt} \leq \frac{5 + 3\rho}{(1 -\rho)(1 +\rho)^2} & \text{pre-condition; not restrictive} \\
  \vt >\frac{2\rho}{1 +\rho}  & \text{for \( \vt + f_1(\sqrt{r},\lambda') \geq 1 \) } \\
  2\vt + f_3(\sqrt{r},\lambda') \geq 1
  \end{cases} 
\end{equation*}
Among the first 4 requirements, the fourth one is can imply all of the rest. Then we look at \( FN_{2} \), and show it is always \( o(1) \) when \(\vt >\frac{2\rho}{1 +\rho}   \). When \( \rho > \frac{1}{3} \), \( \vt > \frac{1}{2} \); when \( \rho < \frac{1}{3} \), we discuss as follows: 
\begin{itemize}
  \item When \( \sqrt{r} \leq \frac{\lambda'}{1 +\rho} \cdot \frac{5 + 3\rho}{(1 -\rho)(3 +\rho)} \), 
  
  For the first term, we need \( 1 + \sqrt{1 -\vt} > \sqrt{\frac{1 - 2\vt}{1 -\rho^2}} + \frac{1}{1 +\rho} \) so that the exponent is negative.  \( (LHS - RHS) \) is increasing in \( \vt \), and verifying \( \vt = \frac{2\rho}{1 +\rho} \) is enough.
  
  For the next term, we need \( \sqrt{\frac{5 + 3\rho}{(1 -\rho^2)(1 +\rho)}}\sqrt{1 -2\vt} \leq 1 + \sqrt{1 -\vt} \). Now we need to verify \( \vt =\max\{\frac{2\rho}{1 +\rho},1-\left(\frac{5 +3 \rho }{ (1 -\rho^2) (\rho+3)}-1\right)^2\} \), and it still holds.  
    
    \textit{As long as \( k \leq 2 \), \( k \sqrt{1 -\vt} - \sqrt{1 - 2\vt} \) is increasing in \( \vt \).  \( \sqrt{\frac{(1 -\rho^2)(1 +\rho)}{5 + 3\rho}} \leq 0.454167 \) numerically.  } 
  \item When \(   \sqrt{r} \geq \frac{2\lambda'}{1 -\rho^2}\): impossible, because \( \sqrt{r} = 1 + \sqrt{1 -\vt} \).  
  \item When \(  \frac{\lambda'}{1 +\rho} \cdot \frac{5 + 3\rho}{(1 -\rho)(3 +\rho)} \leq \sqrt{r} \leq \frac{2\lambda'}{1 -\rho^2} \): We know \begin{equation*}
    d^2(C, ((1 +\rho)\sqrt{r},(1 +\rho) \sqrt{r})) \geq \max\left\{\frac{(1 -\rho^2)^2(1 +\rho)}{5 + 3\rho}r,\ \left[(1 -\rho^2)\sqrt{r}-\lambda'\right]^2\right\}
  \end{equation*}
  So a sufficient condition is 
  \begin{equation*}
    \sqrt{\frac{1 - 2\vt}{1 -\rho^2}} + \frac{1}{1 -\rho^2} \leq 1 + \sqrt{1 -\vt}
  \end{equation*}
  with \( \vt =\frac{2\rho}{1 +\rho} \). When \( \rho \leq \frac{1}{3} \), this holds.
\end{itemize}

To conclude, we have verified that \( FP_1 = FN_1 \) can only admit \( \sqrt{r} = 1 + \sqrt{1 -\vt} \), for \( \vt > \frac{2\rho}{1 +\rho} \); and this curve indeed meets all the requirements. 

\textit{Second}, if \( \vt + f_1(\sqrt{r},\lambda') = \vt + f_2(\sqrt{r},\lambda') = 1\), and \( \lambda' \geq  1\), \( 2\vt+ f_3(\sqrt{r},\lambda') \geq 1 \), we also need to discuss along the conditional expression of \( f_1(\sqrt{r},\lambda') \).

When \( \lambda' \leq \sqrt{r} \leq 2\lambda' \) in \( FP_2 \), now we have \( \lambda' = \sqrt{\frac{1 +\rho}{1 -\rho}}\sqrt{1 -\vt}. \) Because we want \( \lambda' \geq 1 \), all we need is to require \( \lambda' > 1 \), which implies \(\vt <\frac{2\rho}{1 +\rho}.  \) 

Since \( \min_{\rho}\frac{5 + 3\rho}{(1 -\rho)(1 +\rho)^2} = 4.848 \), we only need to consider the first case in conditional expression of \( f_2(\sqrt{r},\lambda') \). This gives us three possible curves:
\begin{align*}
  &\sqrt{r} = (1 + \sqrt{\frac{1 +\rho}{1 -\rho}})\sqrt{1 -\vt}  \\
  & \sqrt{r} = 2 \sqrt{\frac{1 -\vt}{1 -\rho^{2}}}\\
  & \sqrt{r} = \sqrt{\frac{5 + 3\rho}{(1 -\rho)(2 +\rho)^2}}\sqrt{1 -\vt}
\end{align*}
Actually, \( \forall \ \rho\in(0,1) \), \( 1 + \sqrt{\frac{1 +\rho}{1 -\rho}} > \max \left\{ \frac{2}{\sqrt{1 -\rho^2}},\ \sqrt{\frac{5 + 3\rho}{(1 -\rho)(2 +\rho)^2}} \right\} \), which means that in the expression of \( f_2 \) when 
\begin{equation*}
  \min \begin{cases} 
    (\sqrt{r} -\lambda')_ +^2 \\
    \frac{1}{1 -\rho^2}\left[(1 -\rho^2)\sqrt{r} -(1 -\rho)\lambda'\right]^2 \\
    \frac{(1 -\rho)(2 +\rho)^2}{5 + 3\rho} r
    \end{cases},
\end{equation*}
 neither of the last two lines cannot produce a curve and be the minimum at the same time. 

So we are only left with \( \sqrt{r} =(1 + \sqrt{\frac{1 +\rho}{1 -\rho}})\sqrt{1 -\vt} \). When \( \vt < \frac{2\rho}{1 +\rho} \), we already have \( \lambda' \geq 1 \), \( \FPtwo \geq 1 \) and \( \FNone \geq 1 \), and we are left to verify \( \FNtwo \geq 1 \). Actually, this does not always hold when \( \rho \) is very small, in which case we need one more curve.

When \( \rho \geq 0.197 \), \( \sqrt{r} \leq \lambda' \cdot \frac{5 + 3\rho}{(1 -\rho^2)(3 +\rho)} \) always holds, and we only need to consider the first case. We need 
\begin{align*}
  \begin{cases} 
    & (1 + \sqrt{\frac{1 +\rho}{1 -\rho}})\sqrt{1 -\vt} \geq \frac{1}{1 +\rho}\sqrt{\frac{1 +\rho}{1 -\rho}}\sqrt{1 -\vt} + \sqrt{\frac{1 - 2\vt}{1 -\rho^2}} \\
    & (1 + \sqrt{\frac{1 +\rho}{1 -\rho}})\sqrt{1 -\vt} \geq \sqrt{1 -2\vt}\sqrt{\frac{5 + 3\rho}{(1 +\rho)(1 -\rho^2)}}
  \end{cases} 
\end{align*}
The first one always holds for \( \rho\in(0,1) \). We can separate \( \vt \) and \( \rho \) into two sides of the inequality and see this.
The second one always holds for \( \rho \geq  0.183 \). We can separate \( \vt \) and \( \rho \) into two sides of the inequality and see this.

As a result, when \(  \rho > 0.197 \), \( FN_2 = o(1) \).

When \( 0.183 \leq \rho < 0.197 \),  \( \sqrt{r} \) is always in the second case, and \( FN_2 = o(1) \) still holds, because \( d^2(C, ((1 +\rho)\sqrt{r},(1 +\rho) \sqrt{r})) \geq \frac{(1 -\rho^2)^2(1 +\rho)}{5 + 3\rho}r. \) 

When \( \rho < 0.183 \), we need to look at the expression of \(  d^2(C, ((1 +\rho)\sqrt{r},(1 +\rho) \sqrt{r})) \). 
\begin{equation}\label{eq:dist.C}
  d^2(C, ((1 +\rho)\sqrt{r},(1 +\rho) \sqrt{r})) = 2(1 - \rho)(1 +\rho)^2 r - 2(1 +\rho)(3 + \rho)\lambda'\sqrt{r} + \frac{5 + 3\rho}{1 -\rho} \lambda'^2 
\end{equation}
In our case, we want \( d^2(C, ((1 +\rho)\sqrt{r},(1 +\rho) \sqrt{r})) \geq (1 -\rho^2)(1 - 2\vt) \) which is \begin{equation*}
  2 (1 +\rho) \left(1 + \sqrt{\frac{1 +\rho}{1 -\rho}}\right)^2 -\frac{2 (3 +\rho)}{1 -\rho} \sqrt{\frac{1 +\rho}{1 -\rho}} \left(1 +\sqrt{\frac{1 +\rho}{1 -\rho}}\right)+\frac{3 \rho+5}{(1 -\rho)^3} \geq \frac{1 - 2\vt}{1 -\vt}.
\end{equation*}
There is no simpler form even if we further break this down. 

When \( \rho > 0.179 \), the LHS is always greater than 1, and not restrictive to \( \vt \). When \( \rho < 0.179 \), this imposes a lower bound on \( \vt \); when \( \vt \) is very small, there will be another curve above \( \sqrt{r} =(1 +\sqrt{\frac{1 +\rho}{1 -\rho}})\sqrt{1 -\vt} \). 

To sum up, now we have \( \sqrt{r} =\max \left\{ 1 + \sqrt{1-\vt}, (1 +\sqrt{\frac{1 +\rho}{1 -\rho}})\sqrt{1 -\vt}\right\} \), but when \( \rho < 0.179 \), we seem to need one more curve which is unknown for now.

When \( \sqrt{r} \geq \frac{5 + 3\rho}{2 + 2\rho}\lambda' \) in \( f_1(\sqrt{r},\lambda') \), if \( \lambda'^2 \leq \frac{(1 -\rho^2)(1 -\rho)r}{5 + 3\rho} \), then we need \( \lambda^2 =(1 -\vt)(1 -\rho^2) \), which contradicts \( \lambda > 1 \) (for \( FP_1 \)). So we must have \( \lambda'^2 > \frac{(1 -\rho^2)(1 -\rho)r}{5 + 3\rho} \), and \( \sqrt{r} = \sqrt{\frac{5 + 3\rho}{1 -\rho}}\sqrt{1 -\vt} \). 

However, this curve \( \sqrt{r} = \sqrt{\frac{5 + 3\rho}{1 -\rho}}\sqrt{1 -\vt} \) is always greater than the \( \sqrt{r} =(1 + \sqrt{\frac{1 +\rho}{1 -\rho}})\sqrt{1 -\vt} \) we have computed. By requiring \( \sqrt{r}\geq \frac{5 + 3\rho}{2 + 2\rho}\lambda' \), we have \( 1 <\lambda' \leq \frac{2(1 +\rho)\sqrt{1 -\vt}}{\sqrt{5 + 3\rho}\sqrt{ 1 -\rho}} \). Notice that \( \sqrt{1 -\vt} \geq \frac{\sqrt{(5 + 3\rho)(1 -\rho)}}{2(1 +\rho)} \) is a strictly tighter requirement than \( \vt <\frac{2\rho}{1 +\rho} \). As a result, even if this curve exists, it is not part of the boundary in the phase diagram.

When \( 2\lambda' < \sqrt{r} < \frac{5 + 3\rho}{2 + 2\rho}\lambda'\), \( \lambda'^2 \) cannot be smaller than \( d^2(B, (\rho \sqrt{r},\sqrt{r}))  \); otherwise \(\lambda'^2 =(1 -\rho^2)(1 -\vt)  \) which contradicts \( \lambda' > 1 \) in \( FP_1 \). 

Then we need the following things:
\begin{equation*}
  \begin{cases} 
    2\lambda' < \sqrt{r} < \frac{5 + 3\rho}{2 + 2\rho}\lambda' \\
    d^2(B, (\rho \sqrt{r},\sqrt{r})) =(1 -\rho^2)(1 -\vt)\\
    \lambda' \geq 1\\
    \FNtwo \geq 1\\
    \sqrt{r} \geq \lambda' + \sqrt{1 -\vt}\\
    \sqrt{r} \geq \sqrt{\frac{1 -\vt}{1 -\rho^2} } +\frac{\lambda'}{1 +\rho} \\
    \sqrt{r} \geq \sqrt{\frac{5 + 3\rho}{(1 -\rho)(2 +\rho)^2} }\sqrt{1 -\vt}
  \end{cases} 
\end{equation*}
and one of the last three inequalities must attain equality. 

If \( \sqrt{r} =\lambda' + \sqrt{1 -\vt} \), then \( \sqrt{r} > 2\lambda' \) would imply \( \lambda' < \sqrt{1 -\vt} \), contradicting \( \lambda' > 1 \).

If \( \sqrt{r} = \sqrt{\frac{1 -\vt}{1 -\rho^2} } +\frac{\lambda'}{1 +\rho} \), we need to look at the expression of \( d^2(B, (\rho \sqrt{r},\sqrt{r})) \) computed in \eqref{eq:dist.B}. We let \( \lambda^* = \frac{1}{\sqrt{1 -\vt}}\lambda' \), and \( \sqrt{r} = \sqrt{\frac{1 -\vt}{1 -\rho^2} } +\frac{\lambda^* \sqrt{1 -\vt}}{1 +\rho} \). Now we have \(\frac{2 + 4\rho + 3\rho^2}{(1 +\rho)^2}{\lambda^*}^2 -\frac{2(1 + 2\rho)}{(1 +\rho)\sqrt{1 -\rho^2}}{\lambda^*} +\frac{\rho^2}{1 -\rho^2} = 0  \). However, 
\begin{align*}
  \sqrt{r} > 2\lambda'\implies&~ \lambda^* < \frac{1 +\rho}{1 + 2\rho}\frac{1}{\sqrt{1 -\rho^2}}\\
  \sqrt{r} < \frac{5 + 3\rho}{2 + 2\rho}\lambda' \implies&~ \lambda^* > \frac{2}{3}\frac{1}{\sqrt{1 -\rho^2}}
\end{align*}
Plug such lower bound and upper bound into the quadratic equation, and the values are negative at both the upper and lower bounds. Thus we know that it has no solution for \( \lambda' \) at all. 

If \( \sqrt{r} = \sqrt{\frac{5 + 3\rho}{(1 -\rho)(2 +\rho)^2} }\sqrt{1 -\vt} \), then we have the following two requirements:
\begin{align*}
\begin{cases} 
  \sqrt{r} \geq&~  \sqrt{\frac{1 -\vt}{1 -\rho^2} } +\frac{\lambda'}{1 +\rho} \implies \lambda' \leq (1 +\rho) \left[ \sqrt{\frac{5 + 3\rho}{(1 -\rho)(2 +\rho)^2}} - \frac{1}{\sqrt{1 -\rho^2}}\right]\sqrt{1 -\vt} \\
\sqrt{r} \leq&~ \frac{5 + 3\rho}{2 + 2\rho}\lambda' \implies \lambda' \geq \frac{2(1 +\rho)}{5 + 3\rho}\sqrt{\frac{5 + 3\rho}{(1 -\rho)(2 +\rho)^2}}\sqrt{1 -\vt}
\end{cases} 
\end{align*}
and the upper bound is smaller than the lower bound, contradiction. 

Tp sum up the \textit{second} case, we have \( \sqrt{r} =\max \left\{ 1 + \sqrt{1-\vt}, (1 +\sqrt{\frac{1 +\rho}{1 -\rho}})\sqrt{1 -\vt}\right\} \), but when \( \rho < 0.179 \), we seem to need one more curve which is unknown for now.

\textit{Third,} if \( \lambda'^2 =\FNtwo = 1 \) and \( \FPtwo \geq 1,\,\FNone \geq  1\), we will eventually have now curve in this case.
We start from some basic requirements:

Now we have \( \lambda' = 1 \). When \( \lambda' = 1 \) is fixed, the exponents of \( FP \) and \( FN \) are all decreasing in \( \sqrt{r} \). For \( FN_1 \), we thus need
\begin{equation*}
  \sqrt{r} \geq \max\left\{ 1+ \sqrt{1 -\vt},\ \sqrt{\frac{1 -\vt}{1 -\rho^2}} + \frac{1}{1 +\rho},\ \sqrt{\frac{5 + 3\rho}{(1 -\rho)(2 +\rho)^2}}\sqrt{1 -\vt} \right\}
\end{equation*}
Even if we finally find an admissible curve with \( \lambda'^2 =\FNtwo = 1 \), it cannot be lower than \( \sqrt{r} = 1 + \sqrt{1 -\vt} \), so we can require \( \vt < \frac{2\rho}{1 +\rho} \).

Now we look at \( \FPtwo \geq 1 \). If \( \sqrt{r} \leq 2 \), then \( \vt < \frac{2\rho}{1 +\rho} \) gives us a contradiction. Thus \( \sqrt{r} \geq 2 \). We still need \( FPtwo \geq  \) additionally to \( \sqrt{r} \geq 2 \).

Then we discuss the conditional expression of \( f_3(\sqrt{r},\lambda') \) one by one.

When \( \sqrt{r} \geq \frac{2\lambda'}{1 -\rho^2} \),  \(
  \sqrt{r} = \sqrt{\frac{1 -2\vt}{1 -\rho^2}} + \frac{1}{1 -\rho^2}
\) which contradicts \( \sqrt{r} \geq \frac{2\lambda'}{1 -\rho^2} \).

When \( \frac{\lambda'}{1 +\rho} \cdot \frac{5 + 3\rho}{(1 -\rho)(3 +\rho)} \leq \sqrt{r} < \frac{2\lambda'}{1 -\rho^2} \), we have \begin{align*}
  & d^2(C, ((1 +\rho)\sqrt{r},(1 +\rho) \sqrt{r})) = (1 -\rho^2)(1 - 2\vt) \\
  & d^2(C, ((1 +\rho)\sqrt{r},(1 +\rho) \sqrt{r})) \geq  \frac{(1 -\rho^2)^2(1 +\rho)}{5 + 3\rho}r 
\end{align*}
and thus
\( \sqrt{r} \leq \sqrt{\frac{5 + 3\rho}{(1 +\rho)^2(1 -\rho)}}\sqrt{1 - 2\vt} \).

For \( \FPtwo \geq 1 \), we look at the conditional expression of \( f_1 \), which now can only take one of the last two cases. 
\begin{itemize}
  \item if it is the last case, then we need \( \sqrt{r} \geq \sqrt{\frac{5 + 3\rho}{1 -\rho}}\sqrt{1 -\vt} \), which does not hold. 
  \item if it is the fourth case, then we need \begin{equation*}
    d^2(B, (\rho \sqrt{r},\sqrt{r})) \geq (1 -\rho^2)(1 -\vt).
  \end{equation*}
  We already have the expression of \( d^2(B, (\rho \sqrt{r},\sqrt{r})) \) in \eqref{eq:dist.B}, and the expression of \( d^2(C, ((1 +\rho)\sqrt{r},(1 +\rho) \sqrt{r})) \) in \eqref{eq:dist.C}:
  \begin{align*}
    & r - 4 \sqrt{r} +\frac{5 + 3\rho}{1 +\rho} \geq 1 -\vt \\
    & r - \frac{3 +\rho}{1 -\rho^2}\sqrt{r} + \frac{5 + 3\rho}{2(1 -\rho^2)^2} = \frac{1 - 2\vt}{2(1 +\rho)} \\
    \implies & \frac{1 - 2\vt}{2(1 +\rho)} +  \frac{3 +\rho}{1 -\rho^2}\sqrt{r} - \frac{5 + 3\rho}{2(1 -\rho^2)^2} - 4 \sqrt{r} +\frac{5 + 3\rho}{1 +\rho} \geq 1 -\vt \\
    \implies & \frac{1 }{2(1 +\rho)} +  \frac{3 +\rho}{1 -\rho^2}\sqrt{r} - \frac{5 + 3\rho}{2(1 -\rho^2)^2} - 4 \sqrt{r} +\frac{5 + 3\rho}{1 +\rho} \geq 1 -\frac{\rho}{1 +\rho} \vt 
  \end{align*}
  It turns out that when we regard \( \sqrt{r} \) as an independent variable, and let it vary in the interval \( (2,2.5) \) as in \( FP_2 \), we always have, \( \forall \ \rho\in(0,1) \), 
\begin{equation*}
  \frac{1 }{2(1 +\rho)} +  \frac{3 +\rho}{1 -\rho^2}\sqrt{r} - \frac{5 + 3\rho}{2(1 -\rho^2)^2} - 4 \sqrt{r} +\frac{5 + 3\rho}{1 +\rho} < 1 -\frac{\rho}{1 +\rho} 
\end{equation*}
so we have a contradiction.
\end{itemize}

To sum up, no curve in this case.

When \( \sqrt{r} \leq \frac{\lambda'}{1 +\rho} \cdot \frac{5 + 3\rho}{(1 -\rho)(3 +\rho)} \), if \( \left[(1 -\rho^2)\sqrt{r} -(1 -\rho)\lambda'\right]_ +^2 \leq \frac{(1 -\rho^2)^2(1 +\rho)}{5 + 3\rho}r \), then we have \( \sqrt{r} = \sqrt{\frac{1 - 2\vt}{1 -\rho^2}} + \frac{1}{1 +\rho}\), which contradicts \( \sqrt{r} \geq  \sqrt{\frac{1 - \vt}{1 -\rho^2}} + \frac{1}{1 +\rho}\) required by \( FN_1 \). If \( \left[(1 -\rho^2)\sqrt{r} -(1 -\rho)\lambda'\right]_ +^2 > \frac{(1 -\rho^2)^2(1 +\rho)}{5 + 3\rho}r \), then we have \( \sqrt{r} = \sqrt{\frac{5 + 3\rho}{(1 +\rho)^2(1 -\rho)}}\sqrt{1 - 2\vt} \). This is a tedious case.
\begin{itemize}
  \item Upper bound on \( \vt \): \( \sqrt{r} \geq 2 \), which implies \( \vt \leq \min\left\{ \frac{2\rho}{1+\rho},\ \frac{1}{2} -\frac{2(1 +\rho)^2(1 -\rho)}{5 + 3\rho} \right\} \).
  \item Lower bound on \( \vt \): Even if the curve is admissible, it only makes a difference if it is smaller than \( (1 + \sqrt\frac{1 +\rho}{1 -\rho})\sqrt{1 -\vt} \). This gives us \( \sqrt{\frac{1 - 2\vt}{1 -\vt}} < \left(1 + \sqrt\frac{1 +\rho}{1 -\rho}\right) \sqrt{\frac{(1 +\rho)^2(1 -\rho)}{5 + 3\rho}} = : \phi(\rho). \)
  Also, \( \sqrt{r} \leq \frac{5 + 3\rho}{(2 + 2\rho)} \) implies \( \vt \geq \frac{1}{2} - \frac{1}{8}(5 + 3\rho)(1 -\rho) \).
\end{itemize}
We then examine \( \FPtwo \). 
\begin{itemize}
  \item When \( \sqrt{r} \geq \frac{5 + 3\rho}{2 + 2\rho}\lambda' \geq \frac{5 + 3\rho}{2 + 2\rho} \), we easily have a contradiction because \( \sqrt{r} = \sqrt{\frac{5 + 3\rho}{(1 +\rho)^2(1 -\rho)}}\sqrt{1 - 2\vt} \). 
  \item When  \( \sqrt{r} < \frac{5 + 3\rho}{2 + 2\rho}\lambda'  \), we need to look at \( d^2(B, (\rho \sqrt{r},\sqrt{r})) \). It now becomes \(\sqrt{r} - 2 \geq \sqrt{\frac{2\rho}{1 +\rho} -\vt}  \).  
The derivative of \( (LHS - RHS) \) w.r.t. \( \vt \) is
\begin{equation*}
  \frac{1}{2 \sqrt{\frac{2\rho}{1 +\rho} -\vt}}\left[ 1 - \sqrt{\frac{2(5 + 3\rho)}{(1 +\rho)^2(1 - \rho)} } \sqrt{\frac{\frac{2\rho}{1 +\rho} -\vt}{\frac{1}{2} -\vt}}\right]
\end{equation*} 
from which we can see \( (LHS - RHS) \) is either increasing, decreasing, or first-decreasing-then-increasing. If we evaluate \( (LHS - RHS) \) at the smallest and largest \( \vt \) and they are both negative (\( \forall \rho \)), then we have a contradiction.
\begin{itemize}
  \item When \( \rho \geq 0.183 \), \( \phi(\rho) > 1 \). The \( (LHS - RHS) \) is below 0 at both \( \max\{0,\ \frac{1}{2}-\frac{1}{8} (1 -\rho) (5 +3\rho)\} \) and \( \min\left\{ \frac{2\rho}{1+\rho},\ \frac{1}{2} -\frac{2(1 +\rho)^2(1 -\rho)}{5 + 3\rho} \right\} \), for all \( 0.183 \leq \rho \leq 1 \).
  \item When \( \rho < 0.183 \), \( 0 <\phi(\rho) < 1 \), and it poses a lower bound on \( \vt \). 
  \begin{itemize}
    \item When \( \rho \leq  0.091 \), the lower bound is greater than the upper bound, so this case does not exist for any \( \vt \).
    \item When \( \rho\in(0.091, 0.183) \), we can verify that the \( (LHS - RHS) \) is below 0 at both \( \max\{0,\ \frac{1}{2}-\frac{1}{8} (1 -\rho) (5 +3\rho)\} \) and \( \min\left\{ \frac{2\rho}{1+\rho},\ \frac{1}{2} -\frac{2(1 +\rho)^2(1 -\rho)}{5 + 3\rho},\ \frac{1 -\phi^2(\rho)}{2 -\phi^2(\rho)} \right\} \).
  \end{itemize}
\end{itemize}
\end{itemize}
We finally know that \( FP_1 = FN_2 \) gives nothing. 

\textit{Fourth}, if \( \FPtwo = \FNtwo = 1 \) and \( \lambda' \geq 1\), \( \FNone \geq 1 \), we will get the last curve. 

From \( \FNone \geq 1 \), we know \( \sqrt{r} \geq \lambda' \), so we start from the case \( \lambda' \leq  \sqrt{r} \leq 2\lambda' \) in \( f_1(\sqrt{r},\lambda') \).

When \( \lambda' \leq  \sqrt{r} \leq 2\lambda' \) in \( f_1(\sqrt{r},\lambda') \), we have \( \lambda' = \sqrt{\frac{1 +\rho}{1 -\rho}}\sqrt{1 -\vt} \). If \( \sqrt{r} \leq \frac{\lambda'}{1 +\rho} \cdot \frac{5 + 3\rho}{(1 -\rho)(3 +\rho)} \) in \( f_3(\sqrt{r},\lambda') \), we could have either of the following two,
\begin{align*}
  & \sqrt{r} = \frac{1}{\sqrt{1 -\rho^2}}\left( \sqrt{1 -\vt} + \sqrt{1 - 2\vt} \right)\\
  & \sqrt{r} = \sqrt{\frac{5 + 3\rho}{(1 +\rho)^2(1 - \rho)}}\sqrt{1 -2\vt}
\end{align*} 
When it is the former, for \( \FNone \geq 1 \), we need \( \sqrt{r} \geq (1 + \sqrt{\frac{1 +\rho}{1 -\rho}})\sqrt{1 -\vt} \), which implies \( \sqrt{\frac{1 - 2\vt}{1 -\vt}} \geq \rho+ \sqrt{1 -\rho^2} > 1 \), which is impossible.

When it is the latter,  for \( \FNone \geq 1\), we need \( \sqrt{r} \geq (1 + \sqrt{\frac{1 +\rho}{1 -\rho}})\sqrt{1 -\vt} \), which implies \( \sqrt{\frac{1 - 2\vt}{1 -\vt}} \geq \sqrt{\frac{(1 +\rho)^2(1 -\rho)}{5 + 3\rho}}\left( 1 + \sqrt{\frac{1 +\rho}{1 -\rho}} \right). \) 
Also, \( \sqrt{r} \leq \frac{\lambda'}{1 +\rho} \cdot \frac{5 + 3\rho}{(1 -\rho)(3 +\rho)} \) implies \( \sqrt{\frac{1 - 2\vt}{1 -\vt}} \leq \frac{\sqrt{(5 + 3\rho) (1 +\rho)}}{(1 -\rho) (3 +\rho)} \). 
However, now either the lower bound \( \sqrt{\frac{(1 +\rho)^2(1 -\rho)}{5 + 3\rho}}\left( 1 + \sqrt{\frac{1 +\rho}{1 -\rho}} \right) > 1 \geq \sqrt{\frac{1 - 2\vt}{1 -\vt}} \), or the upper bound is smaller than the lower bound. 

If \( \frac{\lambda'}{1 +\rho} \cdot \frac{5 + 3\rho}{(1 -\rho)(3 +\rho)} \leq \sqrt{r} \leq \frac{2\lambda'}{1 -\rho^2} \) in \( f_3(\sqrt{r},\lambda') \), we need to solve a quadratic function of \( \sqrt{r} \), i.e. \( d^2(C, ((1 +\rho)\sqrt{r},(1 +\rho) \sqrt{r})) =(1 -\rho^2)(1 - 2\vt) \). The expression of the LHS is already in \eqref{eq:dist.C}. Then we have 
\begin{equation}\label{eq:extra.curve}
  \sqrt{r} =\frac{3 +\rho}{2(1 -\rho^2)}\sqrt{\frac{1 +\rho}{1 -\rho}}\sqrt{1 -\vt} + \frac{1}{2}\sqrt{\frac{2(1 - 2\vt)}{1 +\rho} - \frac{(1 -\vt)}{(1 -\rho)^2}}.
\end{equation}
We take the larger root, because when \( \sqrt{r} \) takes the smaller one, the ellipsoid is actually still tangent to the green line segment in Figure~\ref{suppfig:hamm.scad.small.a}. Thus the smaller root should be discarded. 

Of course, we also list all the requirements it must meet. They are loose only except the first one.
\begin{equation*}
  \begin{cases} 
    \sqrt{r} \geq (1 + \sqrt{\frac{1 +\rho}{1 -\rho}})\sqrt{1 -\vt}\\
  \vt \leq \frac{2\rho}{1 +\rho}\\
  \frac{5 + 3\rho}{(1 -\rho^2)(3 +\rho)} \sqrt{\frac{1 +\rho}{1 -\rho}} \sqrt{1 -\vt}\leq \sqrt{r} \leq 2\sqrt{\frac{1 +\rho}{1 -\rho}}\sqrt{1 -\vt}
  \end{cases} 
\end{equation*}

When \( \sqrt{r} \geq\frac{5 + 3\rho}{2 + 2\rho}\lambda' \) in \( f_1(\sqrt{r},\lambda') \),  We know from \( FP_2 \) that \( \sqrt{r} =\sqrt{\frac{5 + 3\rho}{1 -\rho}}\sqrt{1 -\vt} \). (
  Because we need \( \lambda' \geq  1 \) for \( FP_1 \), the other term is not possible.) If \( \sqrt{r} \leq \frac{\lambda'}{1 +\rho} \cdot \frac{5 + 3\rho}{(1 -\rho)(3 +\rho)} \) in \( f_3(\sqrt{r},\lambda') \), it is only possible that \( \sqrt{r} = \sqrt{\frac{1 - 2\vt}{1 -\rho^2}} +\frac{\lambda'}{1 +\rho} \). However, \( FN_1 = o(1)\) requires that 
\( \sqrt{r} \geq  \sqrt{\frac{1 - \vt}{1 -\rho^2}} +\frac{\lambda'}{1 +\rho} \), so this case is not possible. If \( \sqrt{r} \geq \frac{2\lambda'}{1 -\rho^2} \) in \( f_3(\sqrt{r},\lambda') \), we have \( \sqrt{r} =  \sqrt{\frac{1 - 2\vt}{1 -\rho^2}} +\frac{\lambda'}{1 -\rho^2}\). Because \( \sqrt{r} \geq 2\left( \sqrt{r} - \sqrt{\frac{1 - 2\vt}{1 -\rho^2}} \right) \), we have \( \sqrt{r} \leq 2\sqrt{\frac{1 - 2\vt}{1 -\rho^2}} \implies \lambda' \leq \sqrt{(1 - 2\vt)(1 -\rho^2)}\), which contradicts \( \lambda' >1 \). If \( \frac{\lambda'}{1 +\rho} \cdot \frac{5 + 3\rho}{(1 -\rho)(3 +\rho)} < \sqrt{r} < \frac{2\lambda'}{1 -\rho^2}\) in \( f_3(\sqrt{r},\lambda') \), we know that 
\begin{equation*}
  (1 -\rho^2)(1 - 2\vt) =d^2(C, ((1 +\rho)\sqrt{r},(1 +\rho) \sqrt{r})) \geq \frac{(1 -\rho^2)^2(1 +\rho)}{5 + 3\rho}r
\end{equation*}
and thus \( \sqrt{r} \leq \sqrt{\frac{5 + 3\rho}{(1 -\rho)(1 +\rho)^2}}\sqrt{1 - 2\vt} \) which contradicts \( \sqrt{r} =\sqrt{\frac{5 + 3\rho}{1 -\rho}}\sqrt{1 -\vt} \).

When \( \sqrt{r} \in (2\lambda', \frac{5 + 3\rho}{2 + 2\rho}\lambda') \) in \( f_1(\sqrt{r},\lambda') \): 

If \( \sqrt{r} \leq \frac{\lambda'}{1 +\rho} \cdot \frac{5 + 3\rho}{(1 -\rho)(3 +\rho)} \) in \( f_3(\sqrt{r},\lambda') \), and  \( \sqrt{r} = \sqrt{\frac{1 - 2\vt}{1 -\rho^2}} +\frac{\lambda'}{1 +\rho} \), we have a contradiction. This is because \( FN_1 = o(1)\) requires that 
\( \sqrt{r} \geq  \sqrt{\frac{1 - \vt}{1 -\rho^2}} +\frac{\lambda'}{1 +\rho} \), so this case is not possible.

If \( \sqrt{r} \leq \frac{\lambda'}{1 +\rho} \cdot \frac{5 + 3\rho}{(1 -\rho)(3 +\rho)} \) in \( f_3(\sqrt{r},\lambda') \), and \( \sqrt{r} =\sqrt{\frac{5 + 3\rho}{(1 +\rho)^2(1 -\rho)}}\sqrt{1 - 2\vt} \), this turns out a very tedious case because we generally need to work with \( (\vt,\ \lambda',\ \rho) \) at the same time. For completeness, we include a rigorous proof anyway.


Briefly speaking, we let \( \lambda^* = \frac{1}{\sqrt{1 -\vt}} \lambda'\). Because all we need is 
\begin{equation*}
  \begin{cases} 
    d^2(B, (\rho \sqrt{r},\sqrt{r})) =(1 -\rho^2)(1 -\vt)\quad (\FPtwo \geq 1)\\
    \sqrt{r} =\sqrt{\frac{5 + 3\rho}{(1 +\rho)^2(1 -\rho)}}\sqrt{1 - 2\vt}\quad (\FNtwo \geq 1)\\
    \lambda' \geq 1\\
    \sqrt{r} \geq \lambda' + \sqrt{1 -\vt}\quad (\FNone \geq 1)\\
    \sqrt{r} \geq \sqrt{\frac{1 -\vt}{1 -\rho^2}} +\frac{\lambda'}{1 +\rho}\quad (\FNone \geq 1) \\
    \sqrt{r} \geq \sqrt{\frac{5 + 3\rho}{(1 -\rho)(2 +\rho)^2}}\sqrt{1 -\vt}\quad (\FNone \geq 1) \\
    2\lambda' \leq \sqrt{r} \leq \frac{5 + 3\rho}{2 + 2\rho}\lambda' \quad (\FPtwo \geq 1 )\\
    \sqrt{r} \leq \frac{5 + 3\rho}{(1 -\rho^2)(3 +\rho)}\lambda'\quad (\FNtwo \geq 1)\\
  \end{cases} 
\end{equation*}
by cleaning this up, we have, letting \( x = \sqrt{\frac{1 - 2\vt}{1 -\vt}} \), 
\begin{equation}\label{eq:complex}
  \begin{cases} 
  q(x)\defeq\frac{5 + 3\rho}{(1 +\rho)^2(1 -\rho)}x^2 - 4 \sqrt{\frac{5 + 3\rho}{(1 +\rho)^2(1 -\rho)}} \lambda^* x + \left( \frac{5 + 3\rho}{1 +\rho}  \right) (\lambda^*)^2 - 1 = 0 \\
  \lambda^* \geq 1\\
  x \geq \max \left\{ \frac{(1 +\rho)^2(1 -\rho)}{5 + 3\rho}\left( \frac{\lambda^*}{1 +\rho} + \frac{1}{\sqrt{1 -\rho^2}} \right),\ \frac{1 +\rho}{2 +\rho},\ 2\lambda^*\sqrt{\frac{(1 +\rho)^2(1 -\rho)}{5 + 3\rho}} \right\} \\
  x \leq  \min\left\{1,\ \frac{\lambda^*}{2}\sqrt{(5 + 3\rho)(1 -\rho)},\ \frac{\lambda^*}{3 +\rho}\sqrt{\frac{5 + 3\rho}{1 -\rho}} \right\}
  \end{cases} 
\end{equation} 
If we want some admissible \( x \) to exist, we need the upper bound to be greater than the lower bound in the last two inequalities. This will give us \( \underline{B}(\rho) \leq \lambda^* \leq \overline{B}(\rho) \). (The expressions of \( \underline{B}(\rho) \) and \( \overline{B}(\rho) \) can be explicitly written, but are omitted for brevity.)  
\begin{itemize}
  \item For some \( \rho\in(0,1) \) and \( \lambda^*\in[\underline{B}(\rho),\,   \overline{B}(\rho)] \), we try to plot \( q(x) \) at a suitable \( x \). If \( q(x) < 0 \) always holds, then \( x \) has no solution, and this case is eliminated.
  \item \( q(x) \) is a quadratic function of \( x \). Its axis of symmetry is \( 2\lambda^*\sqrt{\frac{(1 +\rho)^2(1 -\rho)}{5 + 3\rho}}, \) which we already know is smaller than \( x \). Thus \( q(x) \) is increasing in admissible \( x \), and we only need to evaluate \( q(x) \) at the maximum \( x \): \( x = \min\left\{1,\ \frac{\lambda^*}{2}\sqrt{(5 + 3\rho)(1 -\rho)},\ \frac{\lambda^*}{3 +\rho}\sqrt{\frac{5 + 3\rho}{1 -\rho}} \right\}\).
  \item Now we are left to prove a bivariate function is below zero, which can be easily verified given all the requirements on \( \lambda' \) and \( \forall\, \rho >0 \). 
\end{itemize}

When \( \sqrt{r} \geq \frac{2\lambda'}{1 -\rho^2} \) in \( f_3(\sqrt{r},\lambda') \) , we have \( \sqrt{r} =  \sqrt{\frac{1 - 2\vt}{1 -\rho^2}} +\frac{\lambda'}{1 -\rho^2}\). Because \( \sqrt{r} \geq 2\left( \sqrt{r} - \sqrt{\frac{1 - 2\vt}{1 -\rho^2}} \right) \), we have \( \sqrt{r} \leq 2\sqrt{\frac{1 - 2\vt}{1 -\rho^2}} \implies \lambda' \leq \sqrt{(1 - 2\vt)(1 -\rho^2)}\), which contradicts \( \lambda' >1 \).

When \( \frac{\lambda'}{1 +\rho} \cdot \frac{5 + 3\rho}{(1 -\rho)(3 +\rho)} < \sqrt{r} < \frac{2\lambda'}{1 -\rho^2}\) in \( f_3(\sqrt{r},\lambda') \), we have 
\begin{align*}
  &d^2(C, ((1 +\rho)\sqrt{r},(1 +\rho) \sqrt{r})) =(1 -\rho^2)(1 - 2\vt)\\
  & d^2(B, (\rho\sqrt{r}, \sqrt{r})) = (1 -\rho^2)(1 - \vt).
\end{align*}
This is an even more tedious case, and we eliminate this case as follows:
We first list all of the requirements we have:
  \begin{equation*}
    \begin{cases} 
      d^2(B, (\rho\sqrt{r}, \sqrt{r})) = (1 -\rho^2)(1 - \vt)\quad (\FPtwo \geq 1)\\
      d^2(C, ((1 +\rho)\sqrt{r},(1 +\rho) \sqrt{r})) =(1 -\rho^2)(1 - 2\vt)\quad (\FNtwo \geq 1)\\
      \lambda' \geq 1\\
      2\lambda' \leq \sqrt{r} \leq \frac{5 + 3\rho}{2 + 2\rho}\lambda'\quad (\FPtwo \geq 1)\\
      \sqrt{r} \geq \max \left\{ \sqrt{\frac{1 -\vt}{1 -\rho^2}} +\frac{\lambda'}{1 +\rho},\quad  \sqrt{\frac{5 + 3\rho}{(1 -\rho)(2 +\rho)^2}}\sqrt{1 -\vt}\right\}\quad (\FNone \geq 1) \\
      \sqrt{r} \geq \frac{5 + 3\rho}{(1 -\rho^2)(3 +\rho)}\lambda'\quad (\FPtwo \geq 1)
    \end{cases} 
  \end{equation*}
  Define \( x = \sqrt{\frac{r}{1 -\vt}} \) and \( \lambda^* = \frac{1}{\sqrt{1 -\vt}}\lambda' \). We know the upper and lower bounds of \( x \), from the last three inequalities:
  \begin{equation*}
    \max \left\{\frac{5 + 3\rho}{(1 -\rho^2)(3 +\rho)}\lambda^*,
    2\lambda^*, 
    \frac{1}{\sqrt{1 -\rho^2}} +\frac{\lambda^*}{1 +\rho},
    \sqrt{\frac{5 + 3\rho}{(1 -\rho)(2 +\rho)^2}}
    \right\}
    \leq x \leq \frac{5 + 3\rho}{2 + 2\rho}\lambda^*.
  \end{equation*}
  For admissible \( x \) to exist, we need \( \rho \leq 0.415 \) and \( \lambda^* \geq \max\left\{ \frac{2}{3 \sqrt{1 -\rho^2}}, \frac{2(1 +\rho)}{(2 +\rho) \sqrt{(5 + 3\rho)(1 -\rho)}} \right\} \). 

  Then we want to know the upper and lower bounds of \( \lambda^* \) given \( \rho \leq 0.415 \). In terms of the upper bound, we know from \( d^2(B, (\rho\sqrt{r}, \sqrt{r})) = (1 -\rho^2)(1 - \vt) \geq (1 -\rho)^2 \lambda'^2 \) 
  that \( \lambda^* \leq \sqrt{\frac{1 +\rho}{1 -\rho}} \). To sum up, 
  \begin{equation*}
    \max\left\{1,\ \frac{2}{3 \sqrt{1 -\rho^2}}, \frac{2(1 +\rho)}{(2 +\rho) \sqrt{(5 + 3\rho)(1 -\rho)}} \right\} \leq \lambda^* \leq \sqrt{\frac{1 +\rho}{1 -\rho}}.
  \end{equation*}

  Having the upper and lower bound on \( \lambda' \), we look at the first two quadratic functions, and we know that 
  \begin{equation*}
    \begin{cases} 
    x^2 - 4x\lambda^* +\frac{5 + 3\rho}{1 +\rho}{\lambda^*}^2 = 1\\
    2(1 +\rho)x^2 - \frac{2(3 +\rho)}{1 -\rho}x\lambda^* + \frac{5 + 3\rho}{(1 -\rho)^2(1 +\rho)}{\lambda^*}^2 = \frac{1 - 2\vt}{1 -\vt} \leq 1
    \end{cases} 
  \end{equation*}
  The LHS of the second quadratic equation, as a function of \( x \), has the axis of symmetry at \( \frac{3 +\rho}{2(1 -\rho^2)}\lambda^* > \frac{5 + 3\rho}{2 + 2\rho}\lambda^*\), so it is decreasing in \( x \) (fixing \( \lambda^* \)). 

  We look closer at \( x^2 - 4x\lambda^* +\frac{5 + 3\rho}{1 +\rho}{\lambda^*}^2 = 1 \),
  which gives \( x = 2\lambda^* + \sqrt{1 -\frac{1 -\rho}{1 +\rho} {\lambda^*}^2} \leq 2\lambda^* + \sqrt{1 -\frac{1 -\rho}{1 +\rho} {\lambda^*}} \). 
  
  Let \( y = \min \left\{ \frac{5 + 3\rho}{2 + 2\rho}\lambda^*, 2\lambda^* + \sqrt{1 -\frac{1 -\rho}{1 +\rho} {\lambda^*}} \right\}\) be an upper bound on \( x \). Both terms in \( y \) are   increasing in \( \lambda^* \leq  \sqrt{\frac{1 +\rho}{1 -\rho}} \), while \( x \) has no such monoticity. 
  \begin{align*}
     2(1 +\rho)x^2 - \frac{2(3 +\rho)}{1 -\rho}x\lambda^* + \frac{5 + 3\rho}{(1 -\rho)^2(1 +\rho)}{\lambda^*}^2 \geq 2(1 +\rho)y^2 - \frac{2(3 +\rho)}{1 -\rho}y\lambda^* + \frac{5 + 3\rho}{(1 -\rho)^2(1 +\rho)}{\lambda^*}^2
  \end{align*}

  The RHS of the above line, viewed as a function of \( \lambda^* \) and fixing \( y \), is also decreasing in \( \lambda^* \), because the axis of symmetry
  \begin{equation*}
    \frac{(1 -\rho^2)(3 +\rho)}{5 + 3\rho}y =\frac{(1 -\rho^2)(3 +\rho)}{5 + 3\rho}\min \left\{ \frac{5 + 3\rho}{2 + 2\rho}\lambda^*, 2\lambda^* + \sqrt{1 -\frac{1 -\rho}{1 +\rho} {\lambda^*}} \right\} \geq \lambda^*.
  \end{equation*}
  (The above line can be proven for \( \lambda^* < \sqrt{\frac{1 +\rho}{1 -\rho}} \) and \( \rho < 0.415 \) as we have required.)

  Thus, as a whole, \( \frac{\dif{RHS(y(\lambda^*),\lambda^*)}}{\dif{\lambda^*}} = \frac{\partial RHS(y,\lambda^*)}{\partial y} \cdot \frac{\partial y}{\partial \lambda^*} +\frac{\partial RHS(y,\lambda^*)}{\partial \lambda^*} \leq 0 \). So the \( RHS \) is decreasing in \( \lambda^* \). 
  
  When we let \( \lambda^* = \sqrt{\frac{1 +\rho}{1 -\rho}} \) which is the maximum, the \( RHS \) is a univariate function of \( \rho\in(0,0.415) \), which is always greater than \( 1 \). It cannot be equal to \( \frac{1 - 2\vt}{1 -\vt} \), and now we have a contradiction. 

  \paragraph{We then look at the case of \( \rho < 0 \).}

Before diving into the proof, we first re-iterate the phase curves in Theorem~\ref{thm:SCAD} in an equivalent way. As the proof is tedious, the simplified form of the diagram in Theorem~\ref{thm:SCAD} may not be recognizable, so we describes the diagram in an equivalent way in Theorem~\ref{suppthm:scad.sayagain} again, making it more consistent with what we will see in the proof. 

To ease the notation, we recall in Theorem~\ref{thm:SCAD} we defined 
\begin{equation*}
  \sqrt{h_6(\vt)} = \sqrt{\frac{1 - 2\vt}{1 -\rho^2}} + \frac{\frac{1 - 2|\rho|}{1 -|\rho|}\sqrt{\frac{1 -2\vartheta}{1 -\rho^2}} + \sqrt{\left[ \left( \frac{1 - 2|\rho|}{1 -|\rho|} \right)^2 + \frac{1 -|\rho|}{1 +|\rho|} \right](1 -\vt) - \frac{1 - 2\vt}{(1 +|\rho|)^2}}}{(1 -|\rho|)\left[ \left( \frac{1 - 2|\rho|}{1 -|\rho|} \right)^2 + \frac{1 -|\rho|}{1 +|\rho|} \right]}
\end{equation*}

\begin{theorem}[Re-iterating Theorem~\ref{thm:SCAD} in an equivalent way for \( \rho < 0 \)]\label{suppthm:scad.sayagain} For \( \rho < 0 \) and \( a <\frac{2}{1-|\rho|} \), the half of the phase diagram of SCAD when \( \vt\in[\frac{1}{2},1) \) is the same as that of Lasso and SCAD for positive \( \rho \). When \( \vt < \frac{1}{2} \), from left to right:

When \( |\rho| \geq 0.535 \) (approximately), \( \sqrt{r} =\max \left\{ \sqrt{\frac{5 + 3|\rho|}{1 -|\rho|}} \sqrt{1 -\vt}, \  \sqrt{\frac{1 - 2\vt}{1 -\rho^2}} + \frac{1}{1 -|\rho|}\right\}. \) 

    When \( \frac{1}{2} \leq |\rho| < 0.535 \), 
    \begin{equation*}
      \sqrt{r} = \begin{cases} 
        \sqrt{\frac{5 + 3|\rho|}{1 -|\rho|}} \sqrt{1 -\vt} & \mbox{if} \sqrt{\frac{1 - 2\vt}{1 -\vt}} \geq  \left( 1 - \frac{2(1 +|\rho|)}{(1 -|\rho|)(5 + 3|\rho|)} \right) \sqrt{(5 + 3|\rho|)(1 +|\rho|)} \\
        \sqrt{h_6(\vt)} & \mbox{if} \sqrt{\frac{1 - 2\vt}{1 -\vt}} <  \left( 1 - \frac{2(1 +|\rho|)}{(1 -|\rho|)(5 + 3|\rho|)} \right) \sqrt{(5 + 3|\rho|)(1 +|\rho|)}
      \end{cases} 
    \end{equation*}

When \( 0.3965 \leq |\rho| < \frac{1}{2} \),
    \begin{equation*}
      \sqrt{r} = \begin{cases} 
        \sqrt{\frac{5 + 3|\rho|}{1 -|\rho|}} \sqrt{1 -\vt} & \mbox{if} \sqrt{\frac{1 - 2\vt}{1 -\vt}} \geq  \left( 1 - \frac{2(1 +|\rho|)}{(1 -|\rho|)(5 + 3|\rho|)} \right) \sqrt{(5 + 3|\rho|)(1 +|\rho|)} \\
        \sqrt{h_6(\vt)} & \mbox{if} \sqrt{\frac{1 - 2\vt}{1 -\vt}} <  \left( 1 - \frac{2(1 +|\rho|)}{(1 -|\rho|)(5 + 3|\rho|)} \right) \sqrt{(5 + 3|\rho|)(1 +|\rho|)} \mbox{ and} \sqrt{\frac{1 - 2\vt}{1 -\vt}} > \frac{(1 +|\rho|)(1 - 2|\rho|)}{1 -|\rho|}\\
        \sqrt{\frac{1 - 2\vt}{1 -|\rho|^2}} +  \sqrt{\frac{1 +|\rho|}{1 -|\rho|}}\frac{\sqrt{1 -\vt}}{1-|\rho|} & \mbox{if} \sqrt{\frac{1 - 2\vt}{1 -\vt}} \leq  \frac{(1 +|\rho|)(1 - 2|\rho|)}{1 -|\rho|}
      \end{cases} 
    \end{equation*}

When \( \frac{1}{3} \leq |\rho| < 0.3965 \),
    \begin{equation*}
      \sqrt{r} = \begin{cases} 
        \sqrt{\frac{5 + 3|\rho|}{1 -|\rho|}} \sqrt{1 -\vt} & \mbox{if} \sqrt{\frac{1 - 2\vt}{1 -\vt}} \geq  \left( 1 - \frac{2(1 +|\rho|)}{(1 -|\rho|)(5 + 3|\rho|)} \right) \sqrt{(5 + 3|\rho|)(1 +|\rho|)} \\
        \sqrt{h_6(\vt)} & \mbox{if} \sqrt{\frac{1 - 2\vt}{1 -\vt}} <  \left( 1 - \frac{2(1 +|\rho|)}{(1 -|\rho|)(5 + 3|\rho|)} \right) \sqrt{(5 + 3|\rho|)(1 +|\rho|)} \mbox{ and} \sqrt{\frac{1 - 2\vt}{1 -\vt}} > \frac{(1 +|\rho|)(1 - 2|\rho|)}{1 -|\rho|}\\
        \sqrt{\frac{1 - 2\vt}{1 -|\rho|^2}} +  \sqrt{\frac{1 +|\rho|}{1 -|\rho|}}\frac{\sqrt{1 -\vt}}{1-|\rho|} & \mbox{if} \sqrt{\frac{1 - 2\vt}{1 -\vt}} \leq  \frac{(1 +|\rho|)(1 - 2|\rho|)}{1 -|\rho|} \mbox{ and}\sqrt{\frac{1 - 2\vt}{1 -\vt}} \geq \sqrt{1 -|\rho|^2} - \frac{|\rho|(1 +|\rho|)}{(1 -|\rho|)} \\
        \left( 1 + \sqrt{\frac{1 +|\rho|}{1 -|\rho|}} \right) \sqrt{1 -\vt} & \mbox{if} \sqrt{\frac{1 - 2\vt}{1 -\vt}} < \sqrt{1 -|\rho|^2} - \frac{|\rho|(1 +|\rho|)}{(1 -|\rho|)}
      \end{cases} 
    \end{equation*}

When \( 0.311 \leq |\rho| < \frac{1}{3} \),
    \begin{equation*}
      \sqrt{r} = \begin{cases} 
        \sqrt{\frac{5 + 3|\rho|}{1 -|\rho|}} \sqrt{1 -\vt} & \mbox{if} \sqrt{\frac{1 - 2\vt}{1 -\vt}} \geq  \left( 1 - \frac{2(1 +|\rho|)}{(1 -|\rho|)(5 + 3|\rho|)} \right) \sqrt{(5 + 3|\rho|)(1 +|\rho|)} \\
        \sqrt{h_6(\vt)} & \mbox{if} \sqrt{\frac{1 - 2\vt}{1 -\vt}} <  \left( 1 - \frac{2(1 +|\rho|)}{(1 -|\rho|)(5 + 3|\rho|)} \right) \sqrt{(5 + 3|\rho|)(1 +|\rho|)} \mbox{ and} \sqrt{\frac{1 - 2\vt}{1 -\vt}} > \frac{(1 +|\rho|)(1 - 2|\rho|)}{1 -|\rho|}\\
        \sqrt{\frac{1 - 2\vt}{1 -|\rho|^2}} +  \sqrt{\frac{1 +|\rho|}{1 -|\rho|}}\frac{\sqrt{1 -\vt}}{1-|\rho|} & \mbox{if} \sqrt{\frac{1 - 2\vt}{1 -\vt}} \leq  \frac{(1 +|\rho|)(1 - 2|\rho|)}{1 -|\rho|} \mbox{ and}\sqrt{\frac{1 - 2\vt}{1 -\vt}} \geq \sqrt{1 -|\rho|^2} - \frac{|\rho|(1 +|\rho|)}{(1 -|\rho|)} \\
        \max \{\left( 1 + \sqrt{\frac{1 +|\rho|}{1 -|\rho|}} \right) \sqrt{1 -\vt},& 1 + \sqrt{1 -\vt}\}  \mbox{ if} \sqrt{\frac{1 - 2\vt}{1 -\vt}} < \sqrt{1 -|\rho|^2} - \frac{|\rho|(1 +|\rho|)}{(1 -|\rho|)}
      \end{cases} 
    \end{equation*}

 When \( 0.28832 \leq |\rho| < 0.311 \),
    \begin{equation*}
      \sqrt{r} = \begin{cases} 
        \sqrt{h_6(\vt)} & \mbox{if} \sqrt{\frac{1 - 2\vt}{1 -\vt}} > \frac{(1 +|\rho|)(1 - 2|\rho|)}{1 -|\rho|}\\
        \sqrt{\frac{1 - 2\vt}{1 -|\rho|^2}} +  \sqrt{\frac{1 +|\rho|}{1 -|\rho|}}\frac{\sqrt{1 -\vt}}{1-|\rho|} & \mbox{if} \sqrt{\frac{1 - 2\vt}{1 -\vt}} \leq  \frac{(1 +|\rho|)(1 - 2|\rho|)}{1 -|\rho|} \mbox{ and}\sqrt{\frac{1 - 2\vt}{1 -\vt}} \geq \sqrt{1 -|\rho|^2} - \frac{|\rho|(1 +|\rho|)}{(1 -|\rho|)} \\
        \max \{\left( 1 + \sqrt{\frac{1 +|\rho|}{1 -|\rho|}} \right) \sqrt{1 -\vt},& 1 + \sqrt{1 -\vt}\}  \mbox{ if} \sqrt{\frac{1 - 2\vt}{1 -\vt}} < \sqrt{1 -|\rho|^2} - \frac{|\rho|(1 +|\rho|)}{(1 -|\rho|)}
      \end{cases} 
    \end{equation*}

 When \( 0 \leq |\rho| < 0.28832 \),
    \begin{equation*}
      \sqrt{r} = \begin{cases} 
        \sqrt{h_6(\vt)} & \mbox{if} \sqrt{\frac{1 - 2\vt}{1 -\vt}} > \frac{(1 +|\rho|)(1 - 2|\rho|)}{1 -|\rho|}\\
        \sqrt{\frac{1 - 2\vt}{1 -|\rho|^2}} +  \sqrt{\frac{1 +|\rho|}{1 -|\rho|}}\frac{\sqrt{1 -\vt}}{1-|\rho|} & \mbox{if} \sqrt{\frac{1 - 2\vt}{1 -\vt}} \leq  \frac{(1 +|\rho|)(1 - 2|\rho|)}{1 -|\rho|} \mbox{ and}\sqrt{\frac{1 - 2\vt}{1 -\vt}} \geq \sqrt{1 -|\rho|^2} - \frac{|\rho|(1 +|\rho|)}{(1 -|\rho|)} \\
        \max \{ \sqrt{\frac{1-2 \vartheta}{1-|\rho|^{2}}}+\frac{1}{1-|\rho|},& 1 + \sqrt{1 -\vt} \} \mbox{ if} \sqrt{\frac{1 - 2\vt}{1 -\vt}} < \sqrt{1 -|\rho|^2} - \frac{|\rho|(1 +|\rho|)}{(1 -|\rho|)}
      \end{cases} 
    \end{equation*}
\end{theorem}

Wo note two things additionally: First, the exact computation of the numerical results (e.g. ``0.28832'') are covered in the rest of the proof. Second, 
\begin{equation*}
  \left( 1 - \frac{2(1 +|\rho|)}{(1 -|\rho|)(5 + 3|\rho|)} \right) \sqrt{(5 + 3|\rho|)(1 +|\rho|)} \equiv \frac{3-4|\rho|-3 \rho^{2}}{(1-|\rho|)} \sqrt{\frac{1+|\rho|}{5+3|\rho|}}
\end{equation*}
always holds, and the shorter RHS is shown in Theorem~\ref{thm:SCAD}.

Before proving Theorem~\ref{suppthm:scad.sayagain} (and Theorem~\ref{thm:SCAD} at the same time). we put two important results here.
\begin{itemize}
  \item In \( f_1(\sqrt{r},\lambda') \) (recall Theorem~\ref{suppthm:hamm.scad.2}), when \( 2\lambda' < \sqrt{r} < \frac{5 + 3|\rho|}{2 + 2|\rho|}\lambda' \), 
      \begin{align}
        d^2(B,(|\rho| \sqrt{r},\sqrt{r})) =&~ (1 -|\rho|^2)\left[r - 4\sqrt{r}\lambda' +\frac{5 + 3|\rho|}{1 +|\rho|}\lambda'^2\right] \label{eq:dist.B}.
      \end{align}
  \item In \( f_3(\sqrt{r},\lambda') \), when \( \sqrt{r} \geq \frac{2\lambda'}{1 -|\rho|} \), \begin{align}
    d^2\left(D,\left((1 -|\rho|)\sqrt{r},-(1-|\rho|)\sqrt{r}\right)\right) =(1 -|\rho|^2)\left[ 2(1 -|\rho|)r - 2(3 -|\rho|)\lambda' \sqrt{r} + \frac{5 - 3|\rho|}{1 -|\rho|}\lambda'^2\right]
  \end{align}
\end{itemize} 
Then we move on to the proof, which has four sections:
\( FP_1 = FN_1 \), \( FP_2 = FN_1 \), \( FP_1 = FN_2 \), \( FP_2 = FN_2 \). 

\textit{First}, if \( \lambda'^2 =\FNone =1 \) and \( \FPtwo \geq 1 \), \( \FNtwo \geq 1 \), we will prove that \( \sqrt{r} = 1 + \sqrt{1 -\vt} \) is part of the diagram with the condition
\begin{equation*}
  \begin{cases} 
    \vt \geq \frac{2|\rho|}{1 +|\rho|}\\
    1 + \sqrt{1 -\vt} \geq \sqrt{\frac{1 - 2\vt}{1 -\rho^2}} + \frac{1}{1 -|\rho|}
  \end{cases} 
\end{equation*}
When \( \vt > \frac{1}{2} \), the second condition can be ignored. Even when \( \vt \leq  \frac{1}{2} \), 
the second condition is restrictive (stronger than the first one) only when \( |\rho| < 0. 28832 \) approximately.  (\( LHS - RHS \) is increasing in \( \vt \).)

We know \( \lambda' = 1 \) and \( \sqrt{r} > 1 \). From the previous discussion in the case of positive correlation, we already know that only one curve is possible, which is \begin{equation*}
  \sqrt{r} = 1 + \sqrt{1 -\vt}
\end{equation*}

We still need it to meet the following requirements:
\begin{equation*}
  \begin{cases} 
  1 + \sqrt{1 -\vt} \geq \frac{1}{1 + |\rho|} + \sqrt{\frac{1 -\vt}{1 -\rho^2}} & \text{for it to be the smallest among the three in }f_2\\
  1 + \sqrt{1 -\vt} \geq \sqrt{1 -\vt}\sqrt{\frac{5 + 3|\rho|}{(1 -|\rho|)(2 +|\rho|)^2}} & \text{for it to be the smallest among the three in }f_2\\
  1 + \sqrt{1 -\vt} \leq \frac{5 + 3|\rho|}{(1 -|\rho|)(1 +|\rho|)^2} & \text{pre-condition in \( f_2 \); not restrictive} \\
  \vt >\frac{2|\rho|}{1 +|\rho|}  & \text{for \( \FPtwo \geq 1 \) } \\
  \FNtwo \geq 1
  \end{cases} 
\end{equation*}
Among the first 4 requiremenrs, the fourth one can imply the rest. Finally, we look at \( \FNtwo \geq 1 \).

When \( \vt \geq  \frac{1}{2} \), we naturally have \( \FNtwo \geq 1 \). When \( |\rho| \geq \frac{1}{3} \), \( \vt >\frac{2|\rho|}{1 +|\rho|} \geq \frac{1}{2} \) always holds, and no more discussion is needed. When \( |\rho| < \frac{1}{3} \), we proceed to the following discussion. 

For \( |\rho| < \frac{1}{3} \), since  \(\sqrt{r} = 1 + \sqrt{1 -\vt} \leq 2 < \frac{2\lambda'}{1 -|\rho|} \), we need \(1 + \sqrt{1 -\vt} \geq \sqrt{\frac{1 - 2\vt}{1 -\rho^2}} + \frac{1}{ 1-|\rho|}  \). 


\textit{Second}, if \( \FPtwo =\FNone = 1\) and \( \lambda' \geq  1 \), \( \FNtwo \geq 1 \), we will prove that we can have one curve at most, \( \sqrt{r} = \left( 1 + \sqrt{\frac{1 +|\rho|}{1 -|\rho|}} \right) \sqrt{1 -\vt} \), and it exists in the interval:
  (Define \( \psi(|\rho|)\defeq  \sqrt{1 -\rho^2}\left( 1 - \frac{|\rho|}{1 -|\rho|}\sqrt{\frac{1 +|\rho|}{1 -|\rho|}} \right)\).) 
\begin{itemize}
  \item When \( 0.28832 \leq |\rho| \leq 0.3965 \), the curve exists in the interval \( \left[ \frac{1 - \psi(|\rho|)^2}{2 - \psi(|\rho|)^2} ,\ \frac{2|\rho|}{1 +|\rho|} \right)\).
  \item When \( |\rho| < 0.28832\), the curve does not exist. 
  \item When \( |\rho| > 0.3965 \), the curve exists in the interval \( [\frac{1}{2},\frac{2|\rho|}{1 +|\rho|}) \). 
\end{itemize}

We discuss the conditional expression of \( f_1(\sqrt{r},\lambda') \) to prove such result:

When \( \lambda' \leq \sqrt{r} \leq 2\lambda' \) in \( f_1(\sqrt{r},\lambda') \), we have \( \lambda' = \sqrt{\frac{1 +|\rho|}{1 -|\rho|}}\sqrt{1 -\vt}. \) Because we want \( \lambda' > 1 \), it implies \(  \vt <\frac{2|\rho|}{1 +|\rho|}. \) 

We have discussed and eliminated several curves in the case of positive correlation in the case of positive correlation, so now we are only left with one curve: \( \sqrt{r} = \left( 1 + \sqrt{\frac{1 +|\rho|}{1 -|\rho|}} \right) \sqrt{1 -\vt} \) and we only need to additionally verify \( \FNtwo \geq 1 \). 

Since \(  \vt <\frac{2|\rho|}{1 +|\rho|} \), when \( |\rho| > \frac{1}{3} \), \( FN_2 = o(1) \) naturally holds for \( \vartheta\in[\frac{1}{2},\frac{2|\rho|}{1 +|\rho|}) \). We only need to discuss \( \vt\in(0,\frac{1}{2}) \); when  \( |\rho| \leq \frac{1}{3} \), we need to verify  \( FN_2 = o(1) \) for all \( \vt \leq [0, \frac{2|\rho|}{1 +|\rho|})\). 

Now it can be verified that \( \sqrt{r} \leq \frac{2}{1 -|\rho|}\lambda' \) always holds, so we need 
\begin{equation*}
  \left( 1 + \sqrt{\frac{1 +|\rho|}{1 -|\rho|}} \right) \sqrt{1 -\vt} \geq \sqrt{\frac{1 - 2\vt}{1-\rho^2}} + \frac{1}{1 -|\rho|} \sqrt{\frac{1 +|\rho|}{1 -|\rho|}}\sqrt{1 -\vt}
\end{equation*}
which implies 
\begin{equation}\label{eq:FP_2.FN_1}
  \psi(|\rho|) =\sqrt{1 -\rho^2}\left( 1 - \frac{|\rho|}{1 -|\rho|}\sqrt{\frac{1 +|\rho|}{1 -|\rho|}} \right) \geq \sqrt{\frac{1 - 2\vt}{1 -\vt}}. 
\end{equation}

By taking a close look at the function of \( |\rho| \) on the LHS:
\begin{itemize}
  \item When \( \frac{1}{3} <|\rho| < 0.3965 \) (approximately), \( \psi(|\rho|) \)  is positive and smaller than 1. If we define \( \psi(|\rho|)\defeq  \sqrt{1 -\rho^2}\left( 1 - \frac{|\rho|}{1 -|\rho|}\sqrt{\frac{1 +|\rho|}{1 -|\rho|}} \right)\), then the curve \( \sqrt{r} = \left( 1 + \sqrt{\frac{1 +|\rho|}{1 -|\rho|}} \right) \sqrt{1 -\vt} \) 
  exists in the interval \( \left[ \frac{1 - \psi(|\rho|)^2}{2 - \psi(|\rho|)^2} ,\ \frac{2|\rho|}{1 +|\rho|} \right)\). 
  \item When \( |\rho| \geq 0.3965 \), \( \psi(|\rho|) \) is negative. The curve \( \sqrt{r} = \left( 1 + \sqrt{\frac{1 +|\rho|}{1 -|\rho|}} \right) \sqrt{1 -\vt} \) 
  exists only in the interval \( [\frac{1}{2},\frac{2|\rho|}{1 +|\rho|}) \).  
  \item When \( 0.28832 \leq |\rho| \leq \frac{1}{3} \), the curve \( \sqrt{r} = \left( 1 + \sqrt{\frac{1 +|\rho|}{1 -|\rho|}} \right) \sqrt{1 -\vt} \) 
  exists in the interval \( \left[ \frac{1 - \psi(|\rho|)^2}{2 - \psi(|\rho|)^2} ,\ \frac{2|\rho|}{1 +|\rho|} \right)\). 
  \item When \( |\rho| < 0.28832\), this curve does not exist at all, because on the RHS of the inequality~\eqref{eq:FP_2.FN_1}, the smallest value it can take is \( \sqrt{\frac{1 - 2(2|\rho|)/(1 +|\rho|)}{1 -(2|\rho|)/(1 +|\rho|)}} \). When \( |\rho| < 0.28832\), even this smallest value is greater than the LHS, so the inequality cannot hold. 
\end{itemize}

When \( \sqrt{r} \geq \frac{5 + 3|\rho|}{2 + 2|\rho|}\lambda' \) in \( FP_2 \), or when \( 2\lambda' < \sqrt{r} < \frac{5 + 3|\rho|}{2 + 2|\rho|}\lambda'\), the same proof for positive correlation can be used to eliminate these cases. 

To sum up the whole case of \( \FPtwo = \FNone = 1 \): We can have one curve at most, \( \sqrt{r} = \left( 1 + \sqrt{\frac{1 +|\rho|}{1 -|\rho|}} \right) \sqrt{1 -\vt} \), which exists in the interval mentioned above.

Now we are left with the \textit{third} and \textit{fourth} cases, both requiring \( \FNtwo = 1 \). We assume \( \vt \leq \frac{1}{2} \) from now on, because the different definition of \( f_3(\sqrt{r},\lambda') \) makes no difference when \( \vt \geq  \frac{1}{2}\), and the proof for \( \rho >0 \) can be copied. 

\textit{Third,} if \( \lambda'^2 =\FNtwo = 1 \) and \( \FPtwo \geq 1 \), \( \FNone \geq 1 \), since \( \lambda' =1 \) is fixed, all of \( (f_1(\sqrt{r},\lambda'),f_2(\sqrt{r},\lambda'),f_3(\sqrt{r},\lambda')) \) are increasing in \( \sqrt{r} \). As a result, the requirement from \( \FNone \geq 1 \) is just 
\begin{equation}\label{eq:req.FN_1}
  \sqrt{r} \geq \max\left\{ 1+ \sqrt{1 -\vt},\ \sqrt{\frac{1 -\vt}{1 -\rho^2}} + \frac{1}{1 +|\rho|},\ \sqrt{\frac{5 + 3|\rho|}{(1 -|\rho|)(2 +|\rho|)^2}}\sqrt{1 -\vt} \right\}
\end{equation}

When \( \sqrt{r} \leq \frac{2\lambda'}{1 -|\rho|}  \) in \( f_3(\sqrt{r},\lambda') \), we have \( \sqrt{r} = \sqrt{\frac{1 -2\vt}{1 -|\rho|^2}} + \frac{1}{1 -|\rho|} \), and we will see this curve does not exist in the diagram. We need to look at \( \FPtwo \geq 1 \):

If \( \sqrt{r} \leq 2 \) in \( f_1(\sqrt{r},\lambda') \), then we can limit \( |\rho| <\frac{1}{2} \) because we have assumed \( \vt < \frac{1}{2} \), and now we need \( \frac{2|\rho|}{1 +|\rho|} \leq \vt < \frac{1}{2} \),  
    \begin{itemize}
      \item When \( 0.28832 \leq |\rho| \leq \frac{1}{2} \), this case does not exist. This is because we also need \( \sqrt{r} =\sqrt{\frac{1 -2\vt}{1 -|\rho|^2}} + \frac{1}{1 -|\rho|} \geq 1 + \sqrt{1 -\vt}. \) 
      But we already know from the start of the \textit{first} case, that \( \vt \geq \frac{2|\rho|}{1 +|\rho|}  \) implies \(1 + \sqrt{1 -\vt} \geq  \sqrt{\frac{1 -2\vt}{1 -|\rho|^2}} + \frac{1}{1 -|\rho|} \). 
      \item When \( |\rho| < 0.28832\), this case {indeed exists}, though it is visually only a tiny segment. The curve exists in the interval defined by 
      \begin{equation*}
        \begin{cases} 
        \vt \geq \frac{2|\rho|}{1 +|\rho|}\\
        \sqrt{r} =\sqrt{\frac{1 -2\vt}{1 -|\rho|^2}} + \frac{1}{1 -|\rho|} \geq 1 + \sqrt{1 -\vt}
        \end{cases} 
      \end{equation*}
      in which the second line is an upper bound. (Actually, $\vt \geq \frac{2|\rho|}{1 +|\rho|} \implies \sqrt{\frac{1 -2\vt}{1 -|\rho|^2}} + \frac{1}{1 -|\rho|} \leq 2$)

      In terms of \eqref{eq:req.FN_1} (the requirements of \( \FNone \)), we still need to verify \( \sqrt{r} \geq \max\left\{\sqrt{\frac{1 -\vt}{1 -|\rho|^2}} + \frac{1}{1 +|\rho|},\ \sqrt{\frac{5 + 3|\rho|}{(1 -|\rho|)(2 +|\rho|)^2}}\sqrt{1 -\vt}  \right\} \). Using \( \sqrt{1 -\vt} \leq \sqrt{r} - 1 \) to replace \( \sqrt{1 -\vt} \), we will find both of these are much weaker than \( \sqrt{r} \leq 2 \) and not restrictive.
    \end{itemize}

 If \( \sqrt{r} \geq \frac{5 + 3|\rho|}{2 + 2|\rho|} \) in \( f_1(\sqrt{r},\lambda') \), then, in terms of \( \FPtwo \geq 1 \), we are required to have \begin{equation*}
  \begin{cases} 
  \sqrt{r} \geq  \frac{5 + 3|\rho|}{2 + 2|\rho|}\\
  \sqrt{r} \geq \sqrt{\frac{5 + 3|\rho|}{1 -|\rho|}}\sqrt{1 -\vt}
  \end{cases} 
\end{equation*}
Now, by using \( \frac{1}{2}\sqrt{r} \geq \frac{5 + 3|\rho|}{2(2 + 2|\rho|)} \) and \( \frac{1}{2}\sqrt{r}  \geq \frac{1}{2}\sqrt{\frac{5 + 3|\rho|}{1 -|\rho|}}\sqrt{1 -\vt}\), we can easily verify that the requirements of \( \FNone \geq 1 \) in \eqref{eq:req.FN_1} all holds. So we only need to focus on \( \FPtwo \geq 1 \). 

 We further argue that \( \sqrt{r} \geq \sqrt{\frac{5 + 3|\rho|}{1 -|\rho|}}\sqrt{1 -\vt} \) implies \( \sqrt{r} \geq  \frac{5 + 3|\rho|}{2 + 2|\rho|} \).
    \begin{itemize}
      \item If \( \sqrt{\frac{5 + 3|\rho|}{1 -|\rho|}}\sqrt{1 -\vt} <  \frac{5 + 3|\rho|}{2 + 2|\rho|} \), this condition itself imposes a lower bound on \( \vt \). Since \( \vt < \frac{1}{2} \), we actually need \( |\rho| \leq 0.3798 \) for such lower bound to be smaller than \( \frac{1}{2} \). However, when \( |\rho| \) is too small, \( \sqrt{r} \geq \sqrt{\frac{5 + 3|\rho|}{1 -|\rho|}}\sqrt{1 -\vt} \) admits no solution at all. Namely, 
      \begin{equation*}
        \sqrt{(5 + 3|\rho|)(1 +|\rho|)}\sqrt{1 -\vt} - \sqrt{1 - 2\vt} \leq \frac{\sqrt{1 -|\rho|^2}}{1 -|\rho|}
      \end{equation*}
      admits no solution. The LHS is not monotone, but its minimum in \( \vt\in(0,\frac{1}{2}) \) is taken at 
      \begin{equation*}
        \vt = \frac{(5 + 3|\rho|)(1 +|\rho|) - 4}{2(5 + 3|\rho|)(1 +|\rho|) - 4}.
      \end{equation*}
      At this point, the LHS is greater than the RHS, so it admits no solution. 
      \item Since \( \sqrt{\frac{5 + 3|\rho|}{1 -|\rho|}}\sqrt{1 -\vt} \geq  \frac{5 + 3|\rho|}{2 + 2|\rho|} \), we have \( \sqrt{r} \geq \sqrt{\frac{5 + 3|\rho|}{1 -|\rho|}}\sqrt{1 -\vt} \implies \sqrt{r} \geq  \frac{5 + 3|\rho|}{2 + 2|\rho|} \).
    \end{itemize}
 As a result, \( \sqrt{r} = \sqrt{\frac{1 -2\vt}{1 -|\rho|^2}} + \frac{1}{1 -|\rho|} \) makes part of the boundary when it is greater than \( \sqrt{\frac{5 + 3|\rho|}{1 -|\rho|}}\sqrt{1 -\vt} \). We will see later, that the latter is also part of the boundary, when it is greater than  \(  \sqrt{\frac{1 -2\vt}{1 -|\rho|^2}} + \frac{1}{1 -|\rho|}  \).

   If \( 2 < \sqrt{r} < \frac{5 + 3|\rho|}{2 + 2|\rho|} \) in \( f_1(\sqrt{r},\lambda') \) , then in terms of \( \FPtwo \geq 1 \), we need \( d^2(B,(|\rho| \sqrt{r},\sqrt{r})) \geq (1 -|\rho|^2)(1 -\vt) \). Because \( 2 <  \sqrt{\frac{1 -2\vt}{1 -|\rho|^2}} + \frac{1}{1 -|\rho|} < \frac{5 + 3|\rho|}{2 + 2|\rho|} \), we need \( |\rho| \leq 0.535 \); otherwise no solution for \( \vt \).

   We now need to use \eqref{eq:dist.B}. When \( \vt \geq \frac{2|\rho|}{1 +|\rho|} \), \( d^2(B,(|\rho| \sqrt{r},\sqrt{r})) \geq (1 -\rho^2)(1 -\vt) \) always holds. However, \( \vt \geq \frac{2|\rho|}{1 +|\rho|} \) also makes \( \sqrt{r} < 2 \), which contradicts \( \sqrt{r} \geq 2 \). When \( \vt < \frac{2|\rho|}{1 +|\rho|} \), \( d^2(B,(|\rho| \sqrt{r},\sqrt{r})) \geq (1 -\rho^2)(1 -\vt) \) implies \( \sqrt{r} \geq 2 + \sqrt{\frac{2|\rho|}{1 +|\rho|} -\vt}. \) 
    which is  
    \begin{equation*}
      \sqrt{1 -\rho^2}\sqrt{\frac{2|\rho|}{1 +|\rho|} -\vt} - \sqrt{1 - 2\vt} \leq (\frac{1}{1 -|\rho|} - 2)\sqrt{1 -\rho^2}.
    \end{equation*}
    We take a close look at the LHS, as a function of \( \vt \). 
    \begin{itemize}
      \item Using the lower bound of \( \vt \) implied by \( \sqrt{\frac{1 -2\vt}{1 -\rho^2}} + \frac{1}{1 -|\rho|} < \frac{5 + 3|\rho|}{2 + 2|\rho|} \), the LHS is greater than the RHS. 
      \item Using the upper bound of \( \vt \) implied by \( \sqrt{\frac{1 -2\vt}{1 -\rho^2}} + \frac{1}{1 -|\rho|} > 2 \), the LHS is greater than the RHS. 
      \item The LHS is either increasing in \( \vt \), or decreasing in \( \vt \), or first-increasing-then-decreasing. Since \( LHS > RHS \) holds at both ends of the interval, it holds for all \( \vt \). Now we have a contradition. 
  \end{itemize}

  When \( \sqrt{r} > \frac{2\lambda'}{1 -|\rho|}  \) in \( f_3(\sqrt{r},\lambda') \), this case produces no curve in the diagram, but
there is no easy way to  eliminate this case. 
We still need to use some tedious calculation.

If the smallest term is \( \left[(1 -\rho^2)\sqrt{r} -\lambda'\right]^2 \) in \( f_3(\sqrt{r},\lambda') \), then we have \( \sqrt{r} = \sqrt{\frac{1 - 2\vt}{1 -\rho^2}} + \frac{1}{1 -\rho^2} \)
which contradicts \( \sqrt{r} \geq \frac{2}{1 -|\rho|} \).

If \( \frac{2 -|\rho|}{1 -|\rho|} < a < \frac{2}{1 -|\rho|} \), we recall the form of \( k(\lambda',a) \):
\begin{equation*}
  k(\lambda',a) = \begin{cases} 
    d^2\left(D,\left((1 -|\rho|)\sqrt{r},-(1-|\rho|)\sqrt{r}\right)\right) \qquad \text{   if } \frac{2\lambda'}{1 -|\rho|} \leq \sqrt{r} \leq \frac{\lambda'}{1 -|\rho|}\left[2 + \frac{|\rho| +\rho^2}{(a - 2)(1 -\rho^2) - (|\rho| +\rho^2)} \right]\\
    \frac{ (1 -\rho^2)}{1 +\frac{\rho^2(a - 1)^2}{(a - 2)^2} -\frac{2\rho^2(a - 1)}{a - 2}}
    \left[ -\lambda'\left(1 + \frac{a|\rho|}{a - 2}\right) + (1 -|\rho|)\sqrt{r} \cdot \left(1 +\frac{|\rho|(a - 1)}{a - 2}\right)\right]^2 \\
    \qquad\qquad\qquad{if } \sqrt{r} \geq \frac{\lambda'}{1 -|\rho|}\left[2 + \frac{|\rho| +\rho^2}{(a - 2)(1 -\rho^2) - (|\rho| +\rho^2)} \right]
  \end{cases} 
\end{equation*}

\begin{itemize}
  \item When \(  \sqrt{r} \geq \frac{\lambda'}{1 -|\rho|}\left[2 + \frac{|\rho| +\rho^2}{(a - 2)(1 -\rho^2) - (|\rho| +\rho^2)} \right]  \), the expression of \( \sqrt{r} \) is 
  \begin{equation*}
    \sqrt{r} =\frac{1 + \frac{a|\rho|}{a - 2} + \sqrt{1 - 2\vt}\sqrt{1 +\frac{\rho^2(a - 1)^2}{(a - 2)^2} -\frac{2\rho^2(a - 1)}{a - 2}}}{(1 -|\rho|)(1 +\frac{|\rho|(a - 1)}{a - 2})}
  \end{equation*}
  and actually \( \sqrt{r} <  \frac{\lambda'}{1 -|\rho|}\left[2 + \frac{|\rho| +\rho^2}{(a - 2)(1 -\rho^2) - (|\rho| +\rho^2)} \right] \), which gives a contradiction. To prove this, we take \( \vt = 0 \), and re-arrange the terms:
  \begin{align*}
    \sqrt{1 +\frac{\rho^2(a - 1)^2}{(a - 2)^2} -\frac{2\rho^2(a - 1)}{a - 2}} <&~ 1 +|\rho| +\frac{|\rho|(1 +\frac{|\rho|(a - 1)}{a - 2})}{(a - 2)(1 -|\rho|) -|\rho|}\\
    =&~ \frac{\rho^2 + (a - 2)^2(1 -\rho^2)}{(a - 2)^2(1 -|\rho|) -|\rho|(a - 2)} \\
    \Leftrightarrow (a - 2)^2 + \rho^2 (a - 1)^2 - 2\rho^2 (a - 1)(a -2)  <&~ \frac{\left[ \rho^2 + (a - 2)^2(1 -\rho^2) \right]^2}{\left[ (a - 2)(1 -|\rho|) -|\rho| \right]^2}
  \end{align*}
  Multiply each side with \( \left[ (a - 2)(1 -|\rho|) -|\rho| \right]^2 \), and we have a polynomial. We can then factorize \( LHS - RHS \)  and get \( -2(a - 2)(a - 1)|\rho|(1 -|\rho|)\left[\rho^2 + (1 -\rho^2)(a - 2)^2\right] < 0 \).
\item When \(  \sqrt{r} < \frac{\lambda'}{1 -|\rho|}\left[2 + \frac{|\rho| +\rho^2}{(a - 2)(1 -\rho^2) - (|\rho| +\rho^2)} \right]  \), we calculate the distance related to Point \( D \) and eventually get \( 2(1 -|\rho|)\sqrt{r} - 2(3 -|\rho|)\sqrt{r}+ \frac{5 -3|\rho|}{1 -|\rho|} = 1 - 2\vt \) 
which implies \( \sqrt{r} =\frac{3 -|\rho|}{2(1 -|\rho|)} + \frac{1}{2}\sqrt{1 -\frac{4\vt}{1 -|\rho|}} \).
This expression is not too complicated, and we can easily verify \( \sqrt{r} < \frac{2}{1 -|\rho|} \) which gives us a contradition. 
\end{itemize}

If \( a \leq \frac{2 -|\rho|}{1 -|\rho|}  \), we still have \( \sqrt{r} =\frac{3 -|\rho|}{2(1 -|\rho|)} + \frac{1}{2}\sqrt{1 -\frac{4\vt}{1 -|\rho|}} < \frac{2}{1 -|\rho|}\) like the case above, which is still a contradiction. 

\textit{Fourth}, if \( \FPtwo =\FNtwo = 1 \) and \( \lambda' \geq 1 \), \( \FNone \geq  1 \), then: 

For \( \FNone \geq  1 \): When \( \lambda' \) is fixed, the exponents of \( FP \) and \( FN \) are all decreasing in \( \sqrt{r} \). We thus need
\begin{equation}\label{eq:FN_1.new}
  \sqrt{r} \geq \max\left\{ \lambda' + \sqrt{1 -\vt},\ \sqrt{\frac{1 -\vt}{1 -\rho^2}} + \frac{\lambda'}{1 +|\rho|},\ \sqrt{\frac{5 + 3|\rho|}{(1 -|\rho|)(2 +|\rho|)^2}}\sqrt{1 -\vt} \right\}
\end{equation}

When \( \sqrt{r} \leq \frac{2\lambda'}{1 -|\rho|} \) in \( \FNtwo \geq 1 \), for \( f_1(\sqrt{r},\lambda') \):

If \( \lambda' \leq \sqrt{r} \leq 2\lambda' \) in \( f_1(\sqrt{r},\lambda') \), we have \(  \lambda' = \sqrt{\frac{1 +|\rho|}{1 -|\rho|}}\sqrt{1 -\vt} \), and thus \( \vt \leq \frac{2|\rho|}{1 +|\rho|} \) and \( \sqrt{r} = \sqrt{\frac{1 - 2\vt}{1 -|\rho|^2}} + \frac{1}{1 -|\rho|} \sqrt{\frac{1 +|\rho|}{1 -|\rho|}}\sqrt{1 -\vt}. \) Since \( \sqrt{r} \leq 2\lambda' \), we need both \( |\rho| <\frac{1}{2} \) and \( \sqrt{\frac{1 - 2\vt}{1 -\vt}} \leq \frac{(1 +|\rho|)(1 - 2|\rho|)}{1 -|\rho|} \).
When we come to verify the conditions in \eqref{eq:FN_1.new}, the first one still dominates the others. As a result, the curve \begin{equation*}
  \sqrt{r} = \sqrt{\frac{1 - 2\vt}{1 -|\rho|^2}} + \frac{1}{1 -|\rho|} \sqrt{\frac{1 +|\rho|}{1 -|\rho|}}\sqrt{1 -\vt}.
\end{equation*} exists under the following conditions:
\begin{equation*}
  \begin{cases} 
    |\rho| < \frac{1}{2}\\
    \sqrt{\frac{1 - 2\vt}{1 -\vt}} \leq \frac{(1 +|\rho|)(1 - 2|\rho|)}{1 -|\rho|} \\
    \sqrt{\frac{1 - 2\vt}{1 -\vt}} \geq \sqrt{1 -|\rho|^2} - \frac{|\rho|(1 +|\rho|)}{(1 -|\rho|)}
  \end{cases} 
\end{equation*}
(The last one implies \( \vt \leq \frac{2|\rho|}{1 +|\rho|} \).)

If \( \sqrt{r} \geq \frac{5 + 3|\rho|}{2 + 2|\rho|} \) in \( f_1(\sqrt{r},\lambda') \) , we have \( \sqrt{r} = \sqrt{\frac{5 + 3|\rho|}{1 -|\rho|}} \sqrt{1 -\vt} \)
and \( \lambda' \) is computed with \( \sqrt{r} = \sqrt{\frac{1 - 2\vt}{1 -|\rho|^2}} + \frac{\lambda'}{1 -|\rho|}. \)

Now, \( \sqrt{r} = \sqrt{\frac{5 + 3|\rho|}{1 -|\rho|}} \sqrt{1 -\vt} \geq \frac{5 + 3|\rho|}{2 + 2|\rho|} \) requires \begin{equation}\label{suppeq:one.more.requirement}
  \sqrt{\frac{1 - 2\vt}{1 -\vt}} \geq \left( 1 - \frac{2(1 +|\rho|)}{(1 -|\rho|)(5 + 3|\rho|)} \right) \sqrt{(5 + 3|\rho|)(1 +|\rho|)}.
\end{equation}
and  \( \FNone \geq 1 \) requires \( \sqrt{\frac{1 - 2\vt}{1 -\vt}} \geq \frac{1 +|\rho|}{1 -|\rho|} \left( 1 - 2|\rho| \sqrt{\frac{5 + 3|\rho|}{1 +|\rho|}} \right) \) and \( \sqrt{\frac{1 - 2\vt}{1 -\vt}} \geq \frac{\sqrt{1 -\rho^2}}{1 -|\rho|} - \frac{|\rho|}{1 -|\rho|} \sqrt{(5 + 3|\rho|)(1 +|\rho|)}  \). The conditions required by \( \FNone \geq 1 \) are actually even weaker than \( \sqrt{r} \geq \frac{5 + 3|\rho|}{2 + 2|\rho|} \). \( \lambda' \geq 1 \) requires \( \sqrt{\frac{5 + 3|\rho|}{1 -|\rho|}} \sqrt{1 -\vt} \geq \sqrt{\frac{1 - 2\vt}{1 -\rho^2}} + \frac{1}{1 -|\rho|}. \)  
    \begin{itemize}
      \item When \( |\rho| \geq  0.535 \), the RHS of \begin{equation*}
        \sqrt{\frac{1 - 2\vt}{1 -\vt}} \geq \left( 1 - \frac{2(1 +|\rho|)}{(1 -|\rho|)(5 + 3|\rho|)} \right) \sqrt{(5 + 3|\rho|)(1 +|\rho|)}.
      \end{equation*} is negative, thus not restrictive. 
      
      Taking the requirement \( \sqrt{\frac{5 + 3|\rho|}{1 -|\rho|}} \sqrt{1 -\vt} \geq \sqrt{\frac{1 - 2\vt}{1 -\rho^2}} + \frac{1}{1 -|\rho|} \) into account, now the boundary consists of 
      \begin{equation*}
        \max \left\{ \sqrt{\frac{5 + 3|\rho|}{1 -|\rho|}} \sqrt{1 -\vt}, \  \sqrt{\frac{1 - 2\vt}{1 -\rho^2}} + \frac{1}{1 -|\rho|}\right\}.
      \end{equation*}
      \item When \( |\rho| < 0.535 \): \( \sqrt{\frac{5 + 3|\rho|}{1 -|\rho|}} \sqrt{1 -\vt} \geq \sqrt{\frac{1 - 2\vt}{1 -\rho^2}} + \frac{1}{1 -|\rho|} \) is not restrictive, because it can be re-written as  \begin{equation*}
        \sqrt{(5 + 3|\rho|)(1 +|\rho|)}\sqrt{1 -\vt} - \sqrt{1 - 2\vt} \geq \frac{\sqrt{1 -\rho^2}}{1 -|\rho|}
      \end{equation*}
      which always holds, because we have shown the minimum of the LHS is taken at \begin{equation*}
        \vt = \frac{(5 + 3|\rho|)(1 +|\rho|) - 4}{2(5 + 3|\rho|)(1 +|\rho|) - 4},
      \end{equation*} and the minimum can be verified to be greater than the RHS. Thus \begin{equation*}
        \sqrt{\frac{5 + 3|\rho|}{1 -|\rho|}} \sqrt{1 -\vt} \geq \sqrt{\frac{1 - 2\vt}{1 -\rho^2}} + \frac{1}{1 -|\rho|}.\end{equation*} is not restrictive. We only need requirement~\eqref{suppeq:one.more.requirement}.
    \end{itemize}

If \( 2\lambda' < \sqrt{r} < \frac{5 + 3|\rho|}{2 + 2|\rho|}\lambda' \) in \( f_1(\sqrt{r},\lambda') \), this case is very tedious. We present the closed form of \( \lambda' \) and \( \sqrt{r} \):
  \begin{align*}
    \lambda' =&~ \left[ \left( \frac{1 - 2|\rho|}{1 -|\rho|} \right)^2 + \frac{1 -|\rho|}{1 +|\rho|} \right]^{ - 1} \cdot \left[ \frac{1 - 2|\rho|}{1 -|\rho|}\sqrt{\frac{1 -2\vartheta}{1 -\rho^2}} + \sqrt{\left[ \left( \frac{1 - 2|\rho|}{1 -|\rho|} \right)^2 + \frac{1 -|\rho|}{1 +|\rho|} \right](1 -\vt) - \frac{1 - 2\vt}{(1 +|\rho|)^2}}\right] \\
    \sqrt{r} = &~ \sqrt{\frac{1 - 2\vt}{1 -\rho^2}} + \frac{\lambda'}{1 -|\rho|} = 2\lambda' + \sqrt{1 -\vt -\frac{1 -|\rho|}{1 +|\rho|}\lambda'^2}
  \end{align*}

  In terms of the requirements from \( 2\lambda' < \sqrt{r} < \frac{5 + 3|\rho|}{2 + 2|\rho|}\lambda' \): First, \( \sqrt{r} > 2\lambda' \) will give us \begin{align*}
    \text{either } & |\rho| \geq \frac{1}{2} 
    \text{ or } \left( |\rho| < \frac{1}{2} \text{ and } \sqrt{\frac{1 - 2\vt}{1 -\vt}} > \frac{(1 +|\rho|)(1 - 2|\rho|)}{1 -|\rho|}\right)
  \end{align*}
  Second, \( \sqrt{r} < \frac{5 + 3|\rho|}{2 + 2|\rho|}\lambda' \) will give us     \begin{align*}
    |\rho| < 0.535& \text{ (the same numerical value which appreared before)} \\
    \text{and }\sqrt{\frac{1 - 2\vt}{1 -\vt}} < &~ \left( 1 - \frac{2(1 +|\rho|)}{(1 -|\rho|)(5 + 3|\rho|)} \right) \sqrt{(5 + 3|\rho|)(1 +|\rho|)}.
  \end{align*}

  In terms of other requirements, we show that they are can be implied by the two conditions we have just arrived at. 

  We first look at the requirement \( \lambda' \geq 1 \). We need \( \sqrt{\frac{1 - 2\vt}{1 -\vt}} \leq (1 +|\rho|) \sqrt{\left( \frac{1 - 2|\rho|}{1 -|\rho|} \right)^2 + \frac{1 -|\rho|}{1 +|\rho|}} \) to make the content of the square root positive, but this is implied by \( \sqrt{r} \leq \frac{5 + 3|\rho|}{2 + 2|\rho|}\lambda'\). Then \( \lambda' \geq  1 \) is equivalent to: (let \( x = \sqrt{\frac{1 - 2\vt}{1 -\vt}} \))
  \begin{align*}
    & \frac{1 - 2|\rho|}{1 -|\rho|}\sqrt{\frac{1 -2\vartheta}{1 -\rho^2}} + \sqrt{\left[ \left( \frac{1 - 2|\rho|}{1 -|\rho|} \right)^2 + \frac{1 -|\rho|}{1 +|\rho|} \right](1 -\vt) - \frac{1 - 2\vt}{(1 +|\rho|)^2}} \geq  \left( \frac{1 - 2|\rho|}{1 -|\rho|} \right)^2 + \frac{1 -|\rho|}{1 +|\rho|} \\
    \Leftrightarrow& \frac{1}{\sqrt{2 - x^2}} \left[ \frac{1 - 2|\rho|}{1 -|\rho|}\frac{x}{\sqrt{1 -\rho^2}} + \sqrt{\left[ \left( \frac{1 - 2|\rho|}{1 -|\rho|} \right)^2 + \frac{1 -|\rho|}{1 +|\rho|} \right] - \frac{x^2}{(1 +|\rho|)^2}}  \right] \geq  \left( \frac{1 - 2|\rho|}{1 -|\rho|} \right)^2 + \frac{1 -|\rho|}{1 +|\rho|} \\
    &\text{for } \max\left\{ 0,  \frac{(1 +|\rho|)(1 - 2|\rho|)}{1 -|\rho|}\right\} \leq x \leq \min\left\{ 1, \left( 1 - \frac{2(1 +|\rho|)}{(1 -|\rho|)(5 + 3|\rho|)} \right) \sqrt{(5 + 3|\rho|)(1 +|\rho|)} \right\} \\ &\text{and } |\rho|<0.535
  \end{align*}
  It always holds, because we can verify the graph of \( (LHS - RHS) \) as a bi-variate function of \( (|\rho|,x) \) is always above zero.

We then look at the requirement \( \sqrt{r} \leq \frac{2\lambda'}{1 -|\rho|} \): This is equivalent to 
\begin{equation*}
  \frac{1 - 2\vt}{1 -\vt} \leq \frac{\left(\frac{1-2 |\rho|}{1 -|\rho|}\right)^2+\frac{1 -|\rho|}{1 +|\rho|}}{\frac{\left(\frac{1 -|\rho|}{1 +|\rho|}-\frac{2 |\rho| (1-2 |\rho|)}{(1 -|\rho|)^2}\right)^2}{1 -\rho^2}+\frac{1}{(1 +|\rho|)^2}}
\end{equation*}
which is always weaker than \( \sqrt{\frac{1 - 2\vt}{1 -\vt}} \leq \min\left\{ 1, \left( 1 - \frac{2(1 +|\rho|)}{(1 -|\rho|)(5 + 3|\rho|)} \right) \sqrt{(5 + 3|\rho|)(1 +|\rho|)} \right\}. \) 

We finally look at the requirement from \( \FNone \geq 1 \), or equivalently \eqref{eq:FN_1.new}. First, we need to verify \( \sqrt{r} \geq \lambda' + \sqrt{1 -\vt} \). Since \( \sqrt{r} = \sqrt{\frac{1 - 2\vt}{1 -\rho^2}} + \frac{\lambda'}{1 -|\rho|}\), this is equivalent to 
\( \sqrt{\frac{1 - 2\vt}{1 -\rho^2}} + \frac{|\rho|}{1 -|\rho|}\lambda' \geq \sqrt{1 -\vt} \). It naturally holds if \( |\rho| \leq 0.5 \), because \( \lambda' \geq 1 \). When \( 0.5 <|\rho| < 0.535 \), still letting \( x = \sqrt{\frac{1 - 2\vt}{1 -\vt}} \), we have 
\begin{align*}
  &\sqrt{\frac{1 - 2\vt}{1 -\rho^2}} + \frac{|\rho|}{1 -|\rho|}\lambda' \geq \sqrt{1 -\vt}\\
  \Leftrightarrow & \frac{x}{\sqrt{1 -\rho^2}} + \frac{|\rho|\lambda'}{(1 -|\rho|)}\sqrt{2 - x^2} \geq 1 \\
  \Leftarrow & \frac{x}{\sqrt{1 -\rho^2}} + \frac{|\rho|}{(1 -|\rho|)}\sqrt{2 - x^2} \geq 1
\end{align*}
The LHS is either increasing in \( x\in(0,1) \), or first-increasing-then-decreasing. When \( x = 0 \) or \( 1 \), the inequality holds for \( 0.5 < |\rho| < 0.535 \), so it always holds.

We still need to verify \( \sqrt{r} \geq \max \left\{ \sqrt{\frac{1 -\vt}{1 -\rho^2}} + \frac{\lambda'}{1 +|\rho|},\ \sqrt{\frac{5 + 3|\rho|}{(1 -|\rho|)(2 +|\rho|)^2}}\sqrt{1 -\vt} \right\} \). With \( \sqrt{1 -\vt} \leq \sqrt{r} -\lambda' \), we can get rid of \( \sqrt{1 -\vt} \) and arrange either of the requirements as an inequality between \( \sqrt{r} \) and \( \lambda' \). Such an inequality will be weaker than \( \sqrt{r} \leq \frac{5 + 3|\rho|}{2 + 2|\rho|}\lambda' \).

When \( \sqrt{r} > \frac{2\lambda'}{1 -|\rho|} \) in \( f_3(\sqrt{r},\lambda') \), we will see this case does not produce any curve in the diagram. We still need to discuss two cases in terms of \( f_1(\sqrt{r},\lambda') \):

If \( \sqrt{r} \geq  \frac{5 + 3|\rho|}{2 + 2|\rho|}\lambda' \) in \( f_1(\sqrt{r},\lambda') \): We have \( \sqrt{r} = \sqrt{\frac{5 + 3|\rho|}{1 -|\rho|}}\sqrt{1 -\vt} \). Because \( \sqrt{r} > \frac{2\lambda'}{1 -|\rho|} \), we have \begin{equation*}
  \lambda' < \frac{1}{2}\sqrt{(5 + 3|\rho|)(1 -|\rho|)}\sqrt{1 -\vt}
\end{equation*}
To admit a solution for \( \lambda' \geq 1 \), we need \( |\rho| < \frac{1}{3} \) so that the RHS is large enough.  
\begin{itemize}
  \item When \( \sqrt{r} = \sqrt{\frac{1 - 2\vt}{1 -\rho^2}} + \frac{\lambda'}{1 -\rho^2}\), \( \lambda' =(1 -\rho^2)\left( \sqrt{\frac{5 + 3|\rho|}{1 -|\rho|}}\sqrt{1 -\vt} - \sqrt{\frac{1 - 2\vt}{1 -\rho^2}}\right) \). Combining this with \( \lambda' < \frac{1}{2}\sqrt{(5 + 3|\rho|)(1 -|\rho|)}\sqrt{1 -\vt} \), we can get a requirement for \( x = \sqrt{\frac{1 - 2\vt}{1 -\vt}} \):
  \begin{equation*}
    x > \sqrt{(5 + 3|\rho|)(1 +|\rho|)} - \frac{1}{2}\sqrt{\frac{5 + 3|\rho|}{1 +|\rho|}} 
  \end{equation*}
  but the RHS is always greater than 1. Thus we have a contradition.
  \item When \( \sqrt{r} =\frac{3 -|\rho|}{2(1 -|\rho|)}\lambda' + \frac{1}{2} \sqrt{\frac{2(1 - 2\vt)}{1 -|\rho|} -\frac{1 +|\rho|}{1 -|\rho|}\lambda'^2} \): We temporarily ignore the relationship between \( \lambda' \) and \( \vt \), and the \( \lambda' \) which maximizes \( \sqrt{r} =\frac{3 -|\rho|}{2(1 -|\rho|)}\lambda' + \frac{1}{2} \sqrt{\frac{2(1 - 2\vt)}{1 -|\rho|} -\frac{1 +|\rho|}{1 -|\rho|}\lambda'^2} \) is \( \lambda' = \sqrt{\frac{2(3 -|\rho|)^2(1 - 2\vt)}{(3 -|\rho|)^2(1 +|\rho|) +(1 -|\rho|)(1 +|\rho|)^2}} \). Even with the maximizer \( \lambda' \), we still have \[ \frac{3 -|\rho|}{2(1 -|\rho|)}\lambda' + \frac{1}{2} \sqrt{\frac{2(1 - 2\vt)}{1 -|\rho|} -\frac{1 +|\rho|}{1 -|\rho|}\lambda'^2} < \sqrt{\frac{5 + 3|\rho|}{1 -|\rho|}}\sqrt{1 -\vt} \] for \( |\rho| < \frac{1}{3} \). Thus we have no solution for \( \lambda' \).
  \item When \( \frac{2 -|\rho|}{1-|\rho|} < a < \frac{2}{1 -|\rho|} \) and \( \sqrt{r} \geq \frac{\lambda'}{1 -|\rho|}\left[2 + \frac{|\rho| +\rho^2}{(a - 2)(1 -\rho^2) - (|\rho| +\rho^2)} \right]  \), the contradiction comes from the fact that 
  \begin{equation*}
    \sqrt{r} = \sqrt{\frac{5 + 3|\rho|}{1 -|\rho|}}\sqrt{1 -\vt} < \frac{\lambda'}{1 -|\rho|}\left[2 + \frac{|\rho| +\rho^2}{(a - 2)(1 -\rho^2) - (|\rho| +\rho^2)} \right]
  \end{equation*}
  for any \( \frac{2 -|\rho|}{1-|\rho|} < a < \frac{2}{1 -|\rho|} \), \( |\rho| < \frac{1}{3} \) and \( \lambda' \geq 1 \). 
\end{itemize}

If \( \frac{2\lambda'}{1 -|\rho|} < \sqrt{r} < \frac{5 + 3|\rho|}{2 + 2|\rho|}\lambda' \) in \( f_1(\sqrt{r},\lambda') \):
  For this case to exist, we need \( \frac{2}{1 -|\rho|} <\frac{5 + 3|\rho|}{2 + 2|\rho|} \), which requires \( |\rho| < 0.1547 \). Solving \( d^2(B,(|\rho| \sqrt{r},\sqrt{r})) =(1 -\rho^2)(1 -\vt) \) in \( FP_2 \), we have \( \sqrt{r} = 2\lambda' + \sqrt{1 -\vt -\frac{1 -|\rho|}{1 +|\rho|}\lambda'^2} \) 
  \begin{itemize}
    \item When \( \sqrt{r} = \sqrt{\frac{1 - 2\vt}{1 -\rho^2}} + \frac{\lambda'}{1 -\rho^2}\), we need \( \sqrt{\frac{1 - 2\vt}{1 -\rho^2}} + \frac{\lambda'}{1 -\rho^2} > \frac{2}{1 -|\rho|}\lambda' \) which contradicts \( \lambda' \geq 1 \) when \( |\rho| < 0. 1547\).
    \item When \( d^2\left(D,\left((1 -|\rho|)\sqrt{r},-(1-|\rho|)\sqrt{r}\right)\right) \), the \( (\sqrt{r},\lambda') \) pair are given by 
    \begin{align*}
      \begin{cases} 
        \sqrt{r} =&~ 2\lambda' + \sqrt{1 -\vt -\frac{1 -|\rho|}{1 +|\rho|}\lambda'^2}\\
        \sqrt{r} =&~ \frac{3 -|\rho|}{2(1 -|\rho|)}\lambda' + \frac{1}{2} \sqrt{\frac{2(1 - 2\vt)}{1 -|\rho|} -\frac{1 +|\rho|}{1 -|\rho|}\lambda'^2}
      \end{cases} 
    \end{align*}
    which implies \( \frac{1 - 3|\rho|}{2(1 -|\rho|)}\lambda' + \sqrt{1 -\vt -\frac{1 -|\rho|}{1 +|\rho|}\lambda'^2} = \frac{1}{2} \sqrt{\frac{2(1 - 2\vt)}{1 -|\rho|} -\frac{1 +|\rho|}{1 -|\rho|}\lambda'^2}. \) 

    However, the above equation has no solution for \(\lambda' \), because the LHS is always greater than the RHS. We look at \( (LHS - RHS) \) from now on, and prove it is positive:
    
    Let \( \lambda^* =\frac{\lambda'}{\sqrt{1 -\vt}} \), and \( \lambda^* \geq 1 \) as well. Then
    \begin{align*}
      \frac{(LHS - RHS)}{\sqrt{1 -\vt}} = &~ 
      \frac{1 - 3|\rho|}{2(1 -|\rho|)}\lambda^* + \sqrt{1 -\frac{1 -|\rho|}{1 +|\rho|}{\lambda^*}^2} - \frac{1}{2} \sqrt{\frac{2}{1 -|\rho|}\frac{1 - 2\vt}{1-\vt} -\frac{1 +|\rho|}{1 -|\rho|}{\lambda^*}^2}\\
      \geq &~ \frac{1 - 3|\rho|}{2(1 -|\rho|)}\lambda^* + \sqrt{1 -\frac{1 -|\rho|}{1 +|\rho|}{\lambda^*}^2} - \frac{1}{2} \sqrt{\frac{2}{1 -|\rho|} -\frac{1 +|\rho|}{1 -|\rho|}{\lambda^*}^2}
    \end{align*}
    Thus we only need to prove \begin{equation*}
      \frac{1 - 3|\rho|}{2(1 -|\rho|)}\lambda^* + \sqrt{1 -\frac{1 -|\rho|}{1 +|\rho|}{\lambda^*}^2} > \frac{1}{2} \sqrt{\frac{2}{1 -|\rho|} -\frac{1 +|\rho|}{1 -|\rho|}{\lambda^*}^2}
    \end{equation*}
    From the content of the square roots, we can see that \( \lambda^* \leq \sqrt{\frac{1 +|\rho|}{1 - |\rho|}} \).

    Square both sides, and ignore the cross term on the LHS, and we can actually prove a stronger result, 
    \begin{equation*}
      \left( \frac{1 - 3|\rho|}{2(1 -|\rho|)} \right)^2 {\lambda^*}^2 + \left( 1 -\frac{1 -|\rho|}{1 +|\rho|}{\lambda^*}^2  \right) \geq \frac{1}{2(1 -|\rho|)} -\frac{1 +|\rho|}{4(1 -|\rho|)}{\lambda^*}^2.
    \end{equation*}
    We only need to verify the two ends, \( \lambda^* = 1 \) and \( \lambda^* = \sqrt{\frac{1 +|\rho|}{1 -|\rho|}} \), to see that this inequality holds for \( 0 <|\rho| < 0.1547 \).  
    \item When \( \frac{2 -|\rho|}{1-|\rho|} < a < \frac{2}{1 -|\rho|} \) and \( \sqrt{r} \geq \frac{\lambda'}{1 -|\rho|}\left[2 + \frac{|\rho| +\rho^2}{(a - 2)(1 -\rho^2) - (|\rho| +\rho^2)} \right]  \):
    
    \( \frac{\lambda'}{1 -|\rho|}\left[2 + \frac{|\rho| +\rho^2}{(a - 2)(1 -\rho^2) - (|\rho| +\rho^2)} \right]  \) is at least \( \frac{3\lambda'}{1 -|\rho|} \), which is still greater than  \( \sqrt{r} = 2\lambda' + \sqrt{1 -\vt -\frac{1 -|\rho|}{1 +|\rho|}\lambda'^2} \). 
  \end{itemize}
We have finished discussing the last case. To sum up the whole phase diagram, it is exactly Theorem~\ref{suppthm:scad.sayagain}. 

\subsection{Proof of Lemma~\ref{suppthm:sol.path.scad}}  \label{suppsec:sol.path.scad}

(We have assumed \( \rho > 0 \) in Lemma~\ref{suppthm:sol.path.scad}.)

Recall optimization \ref{suppeq:optimization.scad} and our assumption \( b_1 >\abs{h_2} \). The equation of the sub-gradient for \( b =(b_1,b_2)' \) is:
  \begin{equation}\label{suppeq:subgrad.equations.scad}
    \begin{bmatrix} 1 & \rho \\ \rho & 1 \end{bmatrix} \begin{bmatrix}b_1 \\ b_2 \end{bmatrix} + \begin{bmatrix} q'(b_1)\\q'(b_2) \end{bmatrix} = \begin{bmatrix} h_1 \\ h_2 \end{bmatrix}. 
  \end{equation}

When \( \lambda' \) is sufficiently large, neither of \( (b_1,b_2) \) is nonzero. We investigate the process of decreasing \( \lambda' \) from \( \infty \), and discuss the major stages along the way.

  \textit{Stage 1:} When \( \lambda' \) is large, both \( (b_j,b_{j + 1}) \) are zero, and SCAD behaves like Lasso.  When \( \lambda' \) is large, Equation~\eqref{suppeq:subgrad.equations.scad} becomes \(  \lambda' \cdot \sgn(0) = h_1,\ \lambda' \cdot \sgn(0) = h_2 \) and so we need 
  \begin{equation} \label{eq:lambda-1}
    \lambda' \geq \lambda'_1 = \max\{\abs{h_1},\abs{h_2}\}.  
  \end{equation}
  
  \textit{Stage 2:}
  When \( \lambda' \) crosses \( \lambda'_1 =\max \left\{ \abs{h_1},\abs{h_2} \right\} \), since we have assumed \( h_1 >\abs{h_2} \geq 0 \), \( b_1\) becomes positive. To see this, consider \( \lambda' \) in a very small interval \( (\lambda_1' -\delta,\lambda'_1) =(\abs{h_1} -\delta,\abs{h_1}) \):
  \begin{equation*}
    \begin{bmatrix} 1 & \rho \\ \rho & 1 \end{bmatrix} \begin{bmatrix}b_1  \\ 0 \end{bmatrix} + \begin{bmatrix} \lambda'\cdot \sgn(b_1) \\ \lambda' \cdot \sgn0 \end{bmatrix} = \begin{bmatrix} h_1 \\ h_2 \end{bmatrix}. 
  \end{equation*}
  Now we have \( b_1 = h_1 -\lambda'\cdot \sgn(b_1)  \). 
  From this equation, we know \( b_1 \) has the same sign as \( h_1 \), so \( b_1 \) enters the model as a positive number. On the other hand, \( b_2 \) cannot enter the model before \(b_1\), or we have a contradiction. This is because we would have 
  \( b_2 = h_2 -\lambda' \cdot \sgn(b_2) \), but the signs of the LHS and RHS can never agree.
  
  Now that \( b_1 \) is positive, for the above system of equations to admit a solution, we also need 
  \begin{equation}\label{suppeq:sol.path.constraint-1.scad}
    \abs{h_2 +\rho \lambda' -\rho h_1} <\lambda'
  \end{equation}
  
  \textit{Stage 3:} When \( \lambda' \) continues to decrease, we have two possible cases when the solution path enters the next stage. First, \( b_1 \) continues to increase and becomes larger than \( \lambda' \); then its gradient will change according to the definition of SCAD (see \eqref{suppeq:definition.SCAD}), while \( b_2 = 0 \) all along. Second, \( b_2 \) enters the model before \( b_1 \) gets larger than \( \lambda' \). 

  We start from the first case mentioned above. In this case, at the next critical point \( \lambda' =\lambda'_2 \), we would have \( b_1 =\lambda' \) while \( b_2 = 0 \) still. Then
  \begin{equation*}
    \begin{bmatrix} 1 & \rho \\ \rho & 1 \end{bmatrix} \begin{bmatrix}\lambda'  \\ 0 \end{bmatrix} + \begin{bmatrix} \lambda' \\ \lambda' \cdot \sgn0 \end{bmatrix} = \begin{bmatrix} h_1 \\ h_2 \end{bmatrix}. 
  \end{equation*} 
  Then we have \( \lambda_2'^{(1)} = \frac{1}{2}h_1 \). 
  
  For this equation to admit a solution, we need \( \abs{h_2 -\rho \lambda'_2} =\abs{h_2 - \frac{1}{2}\rho h_1} <\lambda'_2 = \frac{1}{2}h_1 \)   which is \begin{equation}\label{suppeq:sol.path.constraint-2.scad}
    \frac{ - 1 +\rho}{2}h_1 < h_2 < \frac{1 +\rho}{2}h_1.
  \end{equation}
  With this constraint~\eqref{suppeq:sol.path.constraint-2.scad} and \( \lambda'\in[\lambda_2'^{(1)},\lambda'_1] \), we can also go back to check the Condition~\eqref{suppeq:sol.path.constraint-1.scad} in the previous stage, and we can see it holds.

  We then consider the next case in which \( b_2 \) enters the model first. 
  This case is essentially Lasso. We solve
  \begin{equation*}
    \begin{bmatrix} 1 & \rho \\ \rho & 1 \end{bmatrix} \begin{bmatrix} b_1  \\ b_2 \end{bmatrix} + \begin{bmatrix} \lambda' \\ \lambda' \cdot \sgn(b_2) \end{bmatrix} = \begin{bmatrix} h_1 \\ h_2 \end{bmatrix}. 
  \end{equation*}
  when \( \lambda'\in(\lambda_2'^{(2)} -\delta,\lambda_2'^{(2)}) \). We then have two cases, depending on the sign of \( b_2 \) when it enters the model. If \( b_2 \) enters the model as a positive number, then
 eventually we have 
  \begin{equation*}
    0 < \rho h_1 < h_2 < h_1,\text{ and } \lambda' < \lambda'_2 =\frac{h_2-\rho h_1}{1-\rho}
  \end{equation*}
If \( b_2 \) enters the model as a negative number, then 
   eventually we have 
  \begin{equation*}
    - h_1 < h_2 <\rho h_1, \text{ and } \lambda' < \lambda'_2 = \frac{\rho h_1 - h_2}{1+\rho}.
  \end{equation*}

  With these constraints and \( \lambda'\in[\lambda'^{(2)},\lambda'_1] \), we can also go back to check the Condition~\eqref{suppeq:sol.path.constraint-1.scad} in the previous stage. It holds, and we omit the details for brevity.

We have discussed the two cases in the third stage, and we need to decide which one actually happens. When \( h_2 \geq \frac{1 +\rho}{2}h_1 \) or \( h_2 \leq \frac{ - 1 +\rho}{2}h_1 \), requirement~\ref{suppeq:sol.path.constraint-2.scad} is not met, and thus the second case is the case that happens. The subset of rejection region when \( h_2 \geq \frac{1 +\rho}{2}h_1 \) or \( h_2 \leq \frac{ - 1 +\rho}{2}h_1 \) is the same as that of Lasso. Now both \( b_1 \) and \( b_2 \) are nonzero, and we need not discuss any further.

When \( \frac{ - 1 +\rho}{2}h_1 < h_2 < \frac{1 +\rho}{2}h_1 \), we can verify that 
\begin{align*}
  \lambda_2'^{(1)} =\frac{1}{2}h_1 > \lambda_2'^{(2)} =\frac{h_2 -\rho h_1}{1 -\rho} \text{ when } h_2 \geq  \rho h_1 \\
  \lambda_2'^{(1)} =\frac{1}{2}h_1 > \lambda_2'^{(2)} = \frac{\rho h_1 - h_2}{1 +\rho} \text{ when } h_2 <  \rho h_1
\end{align*}
So \( b_1 \) becomes larger than \( \lambda' \) before \( b_2 \) enters the model.

\textit{Stage 3:} When \( (h_1,h_2) \) satisfies \( \frac{ - 1 +\rho}{2}h_1 < h_2 < \frac{1 +\rho}{2}h_1 \) in the last stage, we stil need to find out when \( b_2 \) enters the model after \( b_1 \) becomes greater than \( \lambda' \).
When  \( \lambda' < \lambda_2'^{(1)} = \frac{1}{2}h_1 \) we still have two possible cases to discuss: First, \( b_1 \) continues to grow larger than \( (a\lambda') \), making the expression of \( q'(b_1) \) different again,  while \( b_2 \) is still at zero. Second, \( b_2 \) enters the model before \( b_1 \)  hits \( (a\lambda') \). 

Before the discussion of the two cases, we look at the system of equations when \( \lambda'\in(\lambda'_2 -\delta,\lambda'_2) \) for a very small \( \delta \).
\begin{equation*}
  \begin{bmatrix} 1 & \rho \\ \rho & 1 \end{bmatrix} \begin{bmatrix} b_1  \\ 0 \end{bmatrix} + \begin{bmatrix} \frac{a\lambda' - b_1}{a - 1} \\ \lambda' \cdot \sgn(0) \end{bmatrix} = \begin{bmatrix} h_1 \\ h_2 \end{bmatrix}. 
\end{equation*}

The solution of \( b_1 \) is \( b_1 = \frac{(a - 1)h_1 - a\lambda'}{a - 2} \) and  the sub-gradient for \( b_2 \) requires
\begin{equation}\label{suppeq:sol.path.constraint-3.scad}
  \abs{h_2 -\rho b_1} <\lambda'
\end{equation}

We start from the first case in which \( b_1 \) reaches \( (a\lambda) \)  first. At the point \( \lambda' =\lambda_3'^{(1)} \), we have \( b_1 = a\lambda' \) and 
\begin{equation*}
  \begin{bmatrix} 1 & \rho \\ \rho & 1 \end{bmatrix} \begin{bmatrix} a\lambda'  \\ 0 \end{bmatrix} + \begin{bmatrix} 0 \\ \lambda' \cdot \sgn(0) \end{bmatrix} = \begin{bmatrix} h_1 \\ h_2 \end{bmatrix}. 
\end{equation*}
Thus \( \lambda_3'^{(1)} = \frac{h_1}{a} \). In terms of the sub-gradient \( \sgn(0) \) , we need \(  \abs{h_2 -\rho h_1} < \frac{h_1}{a}, \) and thus \(  (\rho - \frac{1}{a})h_1< h_2 <(\rho + \frac{1}{a})h_1 \).
Compare the above equation with Equation~\ref{suppeq:sol.path.constraint-2.scad}, and we get
\begin{equation}\label{suppeq:sol.path.constraint-4.scad}
  (\rho - \frac{1}{a})h_1 < h_2 < \min \left\{(\rho + \frac{1}{a})h_1,\frac{1+\rho}{2}h_1  \right\} = \begin{cases} (\rho + \frac{1}{a})h_1 & \text{if }a > \frac{2}{1-\rho} \\ \frac{1+\rho}{2}h_1 & \text{if } a \leq  \frac{2}{1-\rho} \end{cases} 
\end{equation}

Then we consider the second case in which \( b_2 \) enters the model first. 
We look at the condition~\ref{suppeq:sol.path.constraint-3.scad} to find \( \lambda_3'^{(2)} \), because  condition~\ref{suppeq:sol.path.constraint-3.scad} would become tight at \( \lambda =\lambda_3'^{(2)} \). It requires
\begin{equation*}
  -\lambda' + \frac{\rho(a - 1)h_1 - a\rho \lambda'}{a - 2} < h_2 < \lambda' + \frac{\rho(a - 1)h_1 - a\rho \lambda'}{a - 2}
\end{equation*}
The left half of the inequality is equivalent to  (We have implicitly used \( \rho > 0 \).)
\begin{equation*}
  \lambda' > \frac{\rho(a - 1)h_1 -(a - 2)h_2}{a + a\rho - 2}  
\end{equation*}
The right half of the inequality is \begin{equation}\label{suppeq:sol.path.constraint-5.scad}
  (a - 2)h_2 -\rho(a - 1)h_1 <(a - 2 - a\rho)\lambda'.
\end{equation}
It turns out that we still need to discuss whether \( a > \frac{2}{1 -\rho} \):

When \( a > \frac{2}{1 -\rho} \), \( a - 2 - a\rho > 0 \) and requirement~\eqref{suppeq:sol.path.constraint-5.scad} is restrictive. We have
\begin{equation*}
  \lambda' > \lambda_3'^{(2)} =\max\left\{\frac{(a - 2)h_2 -\rho(a - 1)h_1}{a - 2 - a\rho},\:\frac{\rho(a - 1)h_1 -(a - 2)h_2}{a + a\rho - 2}\,  \right\}.
\end{equation*}
Of course only one of the two terms will be positive, and it involves the discussion of whether \((a - 2)h_2 > \rho(a - 1)h_1 \) or not.

When \( a \leq  \frac{2}{1 -\rho} \), we can prove requirement~\eqref{suppeq:sol.path.constraint-5.scad}
  always holds without any requirement. To see this, just plug  \( \lambda' =\lambda'_2 = \frac{h_1}{2} \) into requirement~\eqref{suppeq:sol.path.constraint-5.scad}, and we will see \( (a - 2)h_2 -\rho(a - 1)h_1 \leq (a - 2 - a\rho)\lambda'_2 < 0 \) because it is equivalent to \( h_2 <\frac{1 +\rho}{2}h_1  \). In other words, requirement~\eqref{suppeq:sol.path.constraint-5.scad} is not restrictive, and we only need the left half:
  \begin{equation*}
    \lambda' > \lambda_3'^{(2)} =\frac{\rho(a - 1)h_1 -(a - 2)h_2}{a + a\rho - 2}.
  \end{equation*}

Finally, we need to decide how to choose between the two cases:
\begin{itemize}
  \item When \( a > \frac{2}{1 -\rho} \) and \( (\rho - \frac{1}{a})h_1 < h_2 < (\rho + \frac{1}{a})h_1 \), we should choose \( \lambda_3'^{(1)} = \frac{h_1}{a} \). This is because now \( \frac{h_1}{a} > \max\left\{\frac{(a - 2)h_2 -\rho(a - 1)h_1}{a - 2 - a\rho},\:\frac{\rho(a - 1)h_1 -(a - 2)h_2}{a + a\rho - 2}\,  \right\}\) always holds. 
  \item When \( a \leq \frac{2}{1 -\rho} \) and \( (\rho - \frac{1}{a})h_1 < h_2 < \frac{1 +\rho}{2}h_1 \), we should choose \( \lambda_3'^{(1)} = \frac{h_1}{a} \). This is because now \( \frac{h_1}{a} > \frac{\rho(a - 1)h_1 -(a - 2)h_2}{a + a\rho - 2}\) always holds.
  \item When \( a > \frac{2}{1 -\rho} \) and \( h_2 >(\rho + \frac{1}{a})h_1 \), we should choose \( \lambda_3'^{(2)} = \frac{(a - 2)h_2 -\rho(a - 1)h_1}{a - 2 - a\rho} > 0\). 
  \item When \( h_2 <(\rho - \frac{1}{a})h_1 \), for arbitrary \( a \), we should choose \( \lambda_3'^{(2)} = \frac{\rho(a - 1)h_1 -(a - 2)h_2}{a + a\rho - 2} > 0\).
\end{itemize}

In the discussion above, if the Condition~\eqref{suppeq:sol.path.constraint-4.scad} is met and \( \lambda' <\lambda_3'^{(1)} = \frac{h_1}{a} \), we still have not seen \( b_2 \) in the model when \( b_1 \) hits \( a\lambda' \), and we need to discuss further.

\textit{Stage 4:} 
When \( \lambda' \) fall below \( \lambda_3'^{(1)} \), \( b_1 \) is greater than \( a\lambda' \). Now the system of equations:
\begin{equation*}
  \begin{bmatrix} 1 & \rho \\ \rho & 1 \end{bmatrix} \begin{bmatrix} b_1  \\ 0 \end{bmatrix} + \begin{bmatrix} 0 \\ \lambda' \cdot \sgn(0) \end{bmatrix} = \begin{bmatrix} h_1 \\ h_2 \end{bmatrix}. 
\end{equation*}
Now we have \( b_1 = h_1 \), and \( \lambda' > \abs{h_2 -\rho h_1}. \)  
Thus we know \( \lambda'_4 =\abs{h_2 -\rho h_1} \). (The sign of \( (h_2 -\rho h_1) \) is not determined yet.) When \( \lambda' \) fall below \( \lambda'_4 \), \( b_2 \) inevitably enters the model. 

\section{Proof of Proposition~\ref{prop:SCAD-larger-a} (Comparing SCAD and Lasso)}
Proposition~\ref{prop:SCAD-larger-a} makes two assertions, the first about positive \( \rho \) and the second about negative \( \rho\). We prove them respectively.

When \( \rho < 0 \), it would be obvious that the diagram of SCAD is lower than the diagram of Lasso. Reviewing the diagram of Lasso in the main text, it is the maximum of four curves:
\beq 
 U_{\mathrm{Lasso}}(\vt)= \begin{cases} 
    \max \left\{ h_1(\vt),h^*_2(\vt) \right\}, & \text{ when }\rho \geq 0, \\
    \max \left\{ h_1(\vt),h^*_2(\vt),h^*_3(\vt),h^*_4(\vt) \right\}, &\text{ when }\rho < 0,
    \end{cases} 
\eeq
where \( h_1(\vt) =(1 + \sqrt{1 -\vt})^2 \),  $h^*_2(\vt) = \bigl( 1 + \sqrt{\frac{1 +\abs{\rho}}{1 -\abs{\rho}}} \bigr)^2 (1 -\vt)$, $h^*_3(\vt) = \frac{1}{(1 -\abs*{\rho})^2} \bigl( 1 + \sqrt{\frac{1 -\abs{\rho}}{1 +\abs{\rho}}}\sqrt{1 - 2\vt} \bigr)^2$, and $h^*_4(\vt) = \frac{1}{(1 -\abs*{\rho})^2} \bigl(\sqrt{\frac{1 +\abs{\rho}}{1 -\abs{\rho}}} \sqrt{1 -\vt} + \sqrt{\frac{1 -\abs{\rho}}{1 +\abs{\rho}}}\sqrt{1 - 2\vt} \bigr)^2$. 

In terms of SCAD, \( h_1(\vt) =(1 + \sqrt{1 -\vt})^2 \),  $h^*_2(\vt) = \bigl( 1 + \sqrt{\frac{1 +\abs{\rho}}{1 -\abs{\rho}}} \bigr)^2 (1 -\vt)$ and $h^*_3(\vt) = \frac{1}{(1 -\abs*{\rho})^2} \bigl( 1 + \sqrt{\frac{1 -\abs{\rho}}{1 +\abs{\rho}}}\sqrt{1 - 2\vt} \bigr)^2$ are also present in Theorem~\ref{thm:SCAD}. (The notation of the corresponding curve of \( h^*_3(\vt) \) is different, \( h^*_3(\vt) = \left( \frac{1}{1 -\rho} +\sqrt{\frac{1- 2\vt}{1 -\rho^2}} \right) \).) The only different curve is the last one; we compare only the last curve below:
\begin{align*}
  h_{\text{Lasso}}(\vt) =&~ \frac{1}{(1 -\abs*{\rho})^2} \bigl(\sqrt{\frac{1 +\abs{\rho}}{1 -\abs{\rho}}} \sqrt{1 -\vt} + \sqrt{\frac{1 -\abs{\rho}}{1 +\abs{\rho}}}\sqrt{1 - 2\vt} \bigr)^2 \\
  h_{\text{SCAD}}(\vartheta)=&~\left\{\begin{array}{ll}\left(\frac{5+3|\rho|}{1-|\rho|}\right)(1-\vartheta), & \text { if } \sqrt{\frac{1-2 \vartheta}{1-\vartheta}} \geq \frac{3-4|\rho|-3 \rho^{2}}{(1-|\rho|)} \sqrt{\frac{1+|\rho|}{5+3|\rho|}} \\ \frac{1}{(1-|\rho|)^{2}}\left(\sqrt{\frac{1+|\rho|}{1-|\rho|}} \sqrt{1-\vartheta}+\sqrt{\frac{1-|\rho|}{1+|\rho|}} \sqrt{1-2 \vartheta}\right)^{2}, & \text { if } \sqrt{\frac{1-2 \vartheta}{1-\vartheta}} \leq \frac{(1+|\rho|)(1-2|\rho|)}{1-|\rho|} \\ h_{6}(\vartheta) & \text { other wise }\end{array}\right.\\
  \leq &~ \min \left\{ \left(\frac{5+3|\rho|}{1-|\rho|}\right)(1-\vartheta),\, \frac{1}{(1 -\abs*{\rho})^2} \bigl(\sqrt{\frac{1 +\abs{\rho}}{1 -\abs{\rho}}} \sqrt{1 -\vt} + \sqrt{\frac{1 -\abs{\rho}}{1 +\abs{\rho}}}\sqrt{1 - 2\vt} \bigr)^2 \right\} \\
  \leq&~ h_{\text{Lasso}}(\vt)
\end{align*}
Thus the assertion when \( \rho < 0 \) is proven. 




When \( \rho> 0 \),  from the details proof in Section~\ref{suppsec:scad}, we know that when \( a \leq  \frac{2}{1-\rho} \), the diagram of SCAD is the same as that of Lasso except when \( \rho < 0.179 \) in a tiny neighborhood of \( \vt = 0 \). See Figure~\ref{suppfig:scad.tinydifference} for an example.
\begin{figure}[h!]
  \centering
  \includegraphics[width=0.4\textwidth]{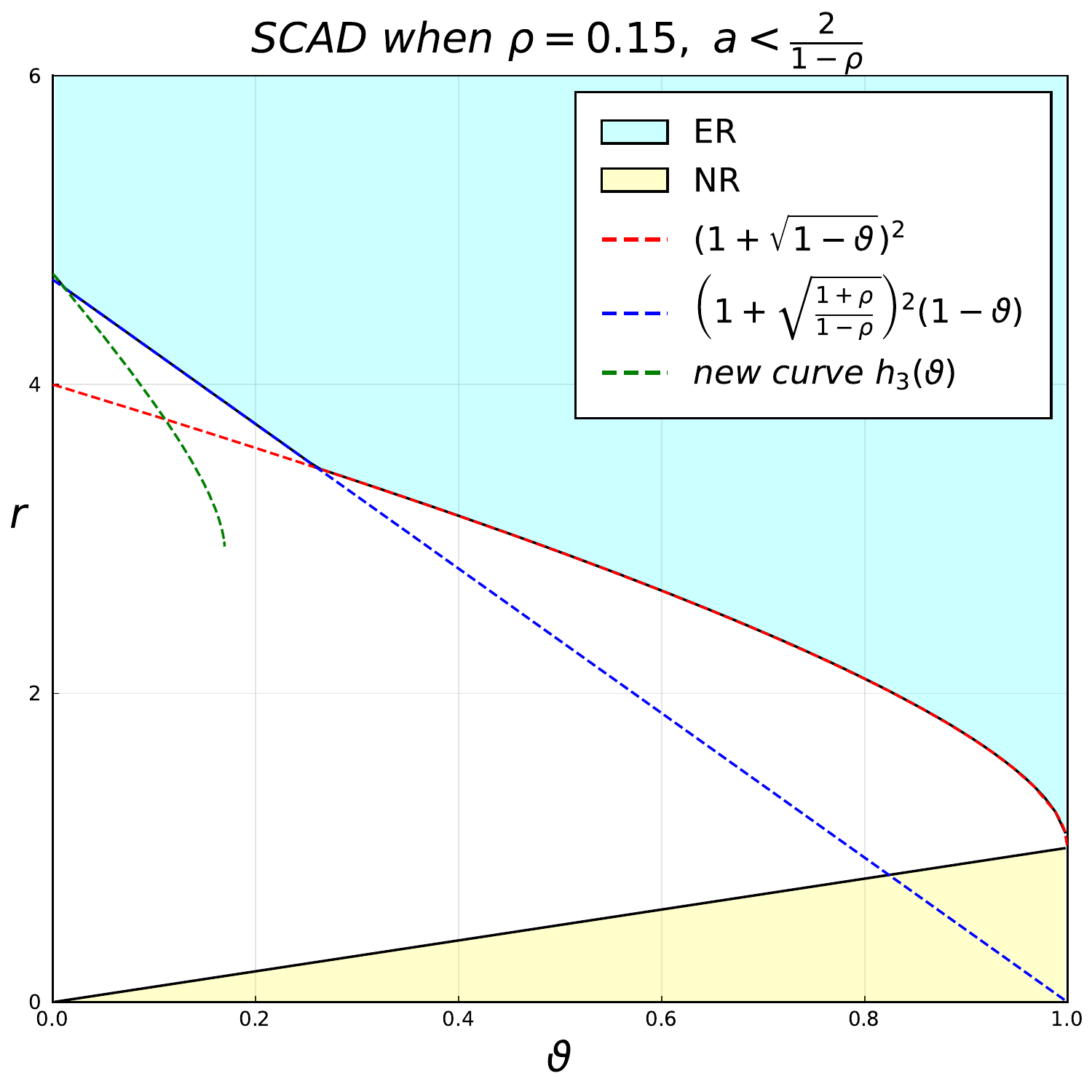}
  \caption{The phase diagram when \( 0 \leq \rho < 0.179\) with the newly added curve.}
  \label{suppfig:scad.tinydifference}
\end{figure}

When \( a > \frac{2}{1 -\rho} \) and increases, the penalty function of SCAD converges to that of Lasso and so does the rejection region. Eventually the tiny corner will vanish, and the diagram of SCAD with optimal \( (a^*,\lambda^*) \) and \( \rho > 0 \) will be the same as Lasso.

\section{Proof of Theorem~\ref{thm:thresh-lasso} (Thresholded Lasso)}\label{suppsec:thresLasso}

As described in Section~\ref{suppsec:sketch}, our proof has three parts: (a) deriving the rejection region, (b) obtaining the rate of convergence of $\mathbb{E}[H(\hat{\beta},\beta)]$, and (c) calculating the phase diagram. 

\paragraph{Part 1: Deriving the rejection region.} 
Recall that the rejection region ${\cal R}$ is as defined in \eqref{def:RejRegion}.
Still use the scaled version of \( (\lambda,t,x_j'y,x_{j + 1}'y) \): Define $h_1=x_j'y/\sqrt{2\log(p)}$, $h_2=x_{j+1}'y/\sqrt{2\log(p)}$, $\lambda'=\lambda/\sqrt{2\log(p)}$ and \( t' = t/\sqrt{2\log(p)} \). Consider Lasso decomposed into bivariate sub-problems, and for \( (x_j,x_{j + 1}) \), $(\hat{b}_1, \hat{b}_2)$ minimizes
\beq \label{suppeq:Lassoproof-optimization}
L(b)\equiv \frac{1}{2}b'\begin{bmatrix}1&\rho\\\rho' & 1\end{bmatrix}b + b'h+\lambda'\|b\|_1
\eeq
It is seen that $(\hat{\beta}_j, \hat{\beta}_{j+1})=\sqrt{2\log(p)}(\hat{b}_1, \hat{b}_2)$. Thresholded Lasso applies threshold \( t \) to \( (\hat{\beta}_j, \hat{\beta}_{j+1}) \), which is equivalent to thresholding \( (\hat{b}_1, \hat{b}_2) \) with \( t' \). 

Fix $\rho\geq 0$. The next lemma gives the explicit solution to \eqref{enproof-optimization} in the case of $h_1>|h_2|$. It is proved in Section~\ref{subsec:proof-thresLasso-solution}. 
\begin{lemma}\label{supplem:sol.path.thresLasso}
  Consider the variable selection method by solving the optimization~\ref{suppeq:Lassoproof-optimization} and then thresholding the solution with \( t' \); if \( (\hat b_1) \) (or \( (\hat b_2) \)) survives the thresholding, then variable \( x_j \) (or \( x_{j + 1} \)) is selected. Suppose \( h_1 >\abs{h_2} \) and \( \rho \geq  0 \), then
  \begin{itemize}
    \item When \( \lambda' > h_1 \), neither of \( (x_j,x_{j + 1}) \) is selected.
    \item If \( h_1 \geq \lambda' \) and \( \rho h_1 -\lambda'( +\rho) \leq h_2 \leq \rho h_1 +\lambda'(1 -\rho) \), then:  When \( h_1 \leq \lambda' + t' \), neither of \( (x_j,x_{j + 1}) \) is selected. When \( h_1 >\lambda'+ t' \), only \( x_j \)  is selected.
    \item If \( h_1 \geq \lambda' \) and \( h_2 > \rho h_1 +\lambda'(1 -\rho) \), then 
    \begin{enumerate}
      \item When \( h_1 <\rho h_2 +\lambda'(1 -\rho) + t'(1 -\rho^2) \), neither of \( (x_j,x_{j + 1}) \) is selected.  
      \item When \( h_1 \geq \rho h_2 +\lambda'(1 -\rho) + t'(1 -\rho^2) \) and \( h_2 \leq \rho h_1 + \lambda'(1 -\rho) + t'(1 -\rho^2) \),  only \( x_j \)  is selected.
      \item When \( h_2 > \rho h_1 + \lambda'(1 -\rho) + t'(1 -\rho^2) \),  both \( (x_j,x_{j + 1}) \)  are selected.
    \end{enumerate}
    \item If \( h_1 \geq \lambda' \) and \( h_2 < \rho h_1 -\lambda'(1 +\rho) \), then 
    \begin{enumerate}
      \item When \( h_1 <\rho h_2 +\lambda'(1 +\rho) + t'(1 -\rho^2) \), neither of \( (x_j,x_{j + 1}) \) is selected.  
      \item When \( h_1 \geq \rho h_2 +\lambda'(1 +\rho) + t'(1 -\rho^2) \) and \( h_2 \leq \rho h_1 + \lambda'(1 +\rho) + t'(1 -\rho^2) \),  only \( x_j \)  is selected.
      \item When \( h_2 > \rho h_1 + \lambda'(1 +\rho) + t'(1 -\rho^2) \),  both \( (x_j,x_{j + 1}) \)  are selected.
    \end{enumerate}
  \end{itemize}
\end{lemma}

Following  similar reasoning to that of Elastic net, we can use Lemma~\ref{supplem:sol.path.thresLasso} to write explicitely the rejection region \( \cal R \), which is the region in \( \R^2 \) where the value of \( (h_1,h_2) \) implies \( x_j \) will get selected eventually. The rejection region of Thresholded Lasso  and \( \rho > 0 \) is
\begin{align}\label{suppeq:rejRegion.thresLasso}
  {\cal R} &= \{(h_1,h_2): h_1>\rho h_2 + \lambda'(1-\rho) + t(1 -\rho^2),\, h_1 >\lambda' + 't'\}\cr
&\;\; \cup \{(h_1, h_2): h_1>\rho h_2 + \lambda'(1+\rho) + t'(1 -\rho^2)\} \cup \{(h_1, h_2): h_1 <\rho h_2 - \lambda'(1+\rho) - t'(1 -\rho^2)\}\cr
&\;\; \cup \{(h_1,h_2): h_1 <\rho h_2 - \lambda'(1-\rho) - t(1 -\rho^2),\, h_1 <-\lambda' - 't'\}. 
\end{align}

See Figure~\ref{suppfig:rejection.region.en} for a visualization of the rejection region.

\begin{figure}[H]
  \centering
  \includegraphics[width=0.7\textwidth]{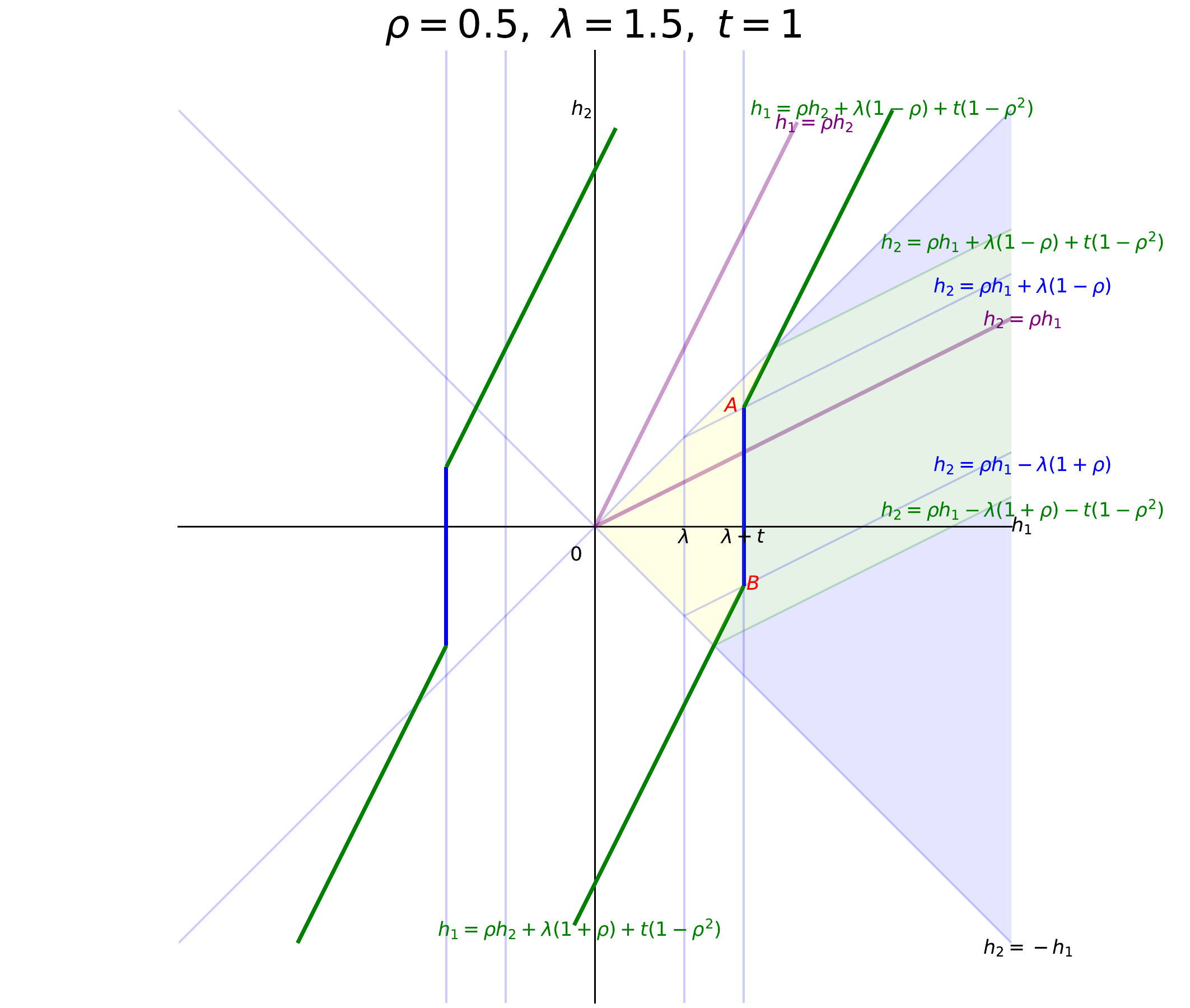}
  \caption{The rejection region of Thresholded Lasso for $\rho\geq 0$.}
  \label{suppfig:rejection.region.thresLasso}
\end{figure}

\paragraph{Part 2. Analyzing the Hamming error.} 
The discussion of Elastic net can be applied here as well, and we present Theorem~\ref{suppthm:hamm.thresLasso} directly.  
\begin{theorem}\label{suppthm:hamm.thresLasso}
  Suppose the conditions of Theorem~\ref{thm:thresh-lasso} hold. Let $\lambda'=\lambda/\sqrt{2\log(p)}$ and \( t' = t/\sqrt{2\log(p)} \)  in Thresholded Lasso. The correlation \( \rho\in( - 1,1) \). 
  As $p\to\infty$, 
  \[
  \FP_p=L_p p^{1- \min\bigl\{ f_1(\sqrt{r}, \lambda',t'), \;\; \vt + f_2(\sqrt{r}, \lambda',t')\bigr\}}, \qquad \FN_p = L_p p^{1-\min\bigl\{\vt + f_3(\sqrt{r}, \lambda',t'),\;\; 2\vt + f_4(\sqrt{r}, \lambda',t')\bigr\}}, 
  \]
  where (below, $d^2_{|\rho|}(u,v)$ is as in Definition~\ref{def:EllipsDistance}) 
  \begin{align*}
    f_1(\sqrt{r},\lambda',t') & = \min \Bigl\{ \frac{1}{1 -\rho^2}\left[(1 +|\rho|) \lambda' + (1 -\rho^2)t'\right]^2,\;\; (\lambda' + t')^2  \Bigr\} \cr
    f_2(\sqrt{r},\lambda',t') & = \begin{cases} 
     (\lambda' + t' -|\rho| \sqrt{r})^2 & \text{if } \sqrt{r} \leq \frac{\lambda'(1-|\rho|)}{1-\rho^2} \\
     d^2_{|\rho|}((\lambda' + t',\lambda' +|\rho|t'),(|\rho| \sqrt{r},\sqrt{r})) & \text{if } \sqrt{r} \in (\frac{\lambda'(1-|\rho|)}{1-\rho^2}, \lambda' + |\rho| t')\\
     \frac{1}{1 -\rho^2}\left[\lambda'(1 -|\rho|) + t'(1 -\rho^2)\right]^2 & \text{if } \sqrt{r} \geq \lambda' + |\rho| t' 
   \end{cases} 
  \cr
  f_3(\sqrt{r}, \lambda',t') &= \min\Bigl\{ (\sqrt{r} -\lambda' - t')_ +^2, \;\;\frac{1}{1 -\rho^2}\left[(1 -\rho^2)\sqrt{r} - t'(1 -\rho^2) -\lambda'(1 -|\rho|)\right]^2
          \Bigr\},\cr
  f_4(\sqrt{r}, \lambda',t') &= \frac{1}{1 -\rho^2}\left[(1 -\rho^2)\sqrt{r} - t'(1 -\rho^2) -\lambda'(1 -\rho)\right]^2
  \end{align*}
  \end{theorem}

\begin{remark}\label{suppremark:threslasso}
  When \( \rho > 0 \) in Theorem~\ref{suppthm:hamm.thresLasso}, we notice that \( f_3(\sqrt{r},\lambda',t') \leq f_4(\sqrt{r},\lambda',t') \), and thus \( \FN_p \)  can be simplified:
  \begin{equation*}
    \FN_p = L_p p^{1 -\vt - f_3(\sqrt{r}, \lambda',t')}.
  \end{equation*}
  When \( \rho < 0 \), such simplification is not available.
\end{remark}

The proof of Theorem~\ref{suppthm:hamm.thresLasso} is easy given the simple rejection region shown in Figure~\ref{suppfig:rejection.region.thresLasso}, and we omit it for brevity.

\paragraph{Part 3. Calculating the phase diagram.} 
The boundary line between Almost Full Recovery and No Recovery is still \( r =\vt \), and the proof is similar to that of Elastic net. The rest of this part calculates the curve between Almost Full Recovery and Exact Recovery.

In such calculation, thresholded Lasso has two tuning parameters, \( (\lambda',t') \), and thus we need one more equality additional to the important fact noted in the Part 3 of Elastic net. In other words, we not only need 
\begin{equation*}
  \min\bigl\{ f_1(\sqrt{r}, \lambda',t'), \;\; \vt + f_2(\sqrt{r}, \lambda',t')\bigr\} =\min\bigl\{\vt + f_3(\sqrt{r}, \lambda',t'),\;\; 2\vt + f_4(\sqrt{r}, \lambda',t')\bigr\} = 1
\end{equation*}
but also need one more equation. This gives us more than \( 4 \) cases for other methods. (For brevity, we use \( f_i\, (i = 1,2,3,4)\) as shorthand of \( f_i(\sqrt{r},\lambda',t') \) for the rest of this part.)

For the rest of this section, we use a clearer way to discuss all the cases; that is, we discuss each possible curve and find out whether they can be present in some interval of \( \vt \). 

\textit{We first talk about \( \rho > 0 \)}: In this case, since we always need \( \vt +f_3 \geq 1 \), we much have \( \sqrt{r} \geq \lambda' + t' \). As a result, for \( \vt +f_2 \geq 1 \), it can only be \( \tfpthree \geq 1 \) since \( \sqrt{r} \geq \lambda' + t' \). Also, according to Remark~\ref{suppremark:threslasso}, we can ignore the requirement \( 2\vt +f_4 \geq 1 \) for \( \rho \geq 0 \).

\textit{First}, we study the curve \( \sqrt{r} = 2 \sqrt{\frac{1 -\vt}{1 -\rho^2}} \), which is the curve given by letting \[ \tfpthree =\tfntwo = 1 \] in \( f_2 \) and \( f_3 \). We also need \( \vt +f_1,\vt +f_3 \geq 1 \) and one more equality. One possible case is
\begin{equation*}
  \begin{cases} 
  \tfpone \geq 1 \\
  \tfptwo \geq 1\\
  \tfpthree = 1\\
  \tfnone = 1\\
  \tfntwo = 1 
  \end{cases} 
\end{equation*}
which gives \( \lambda' = \frac{1 +\rho}{\rho} \left(\frac{1}{\sqrt{1 -\rho^2}} - 1\right)\sqrt{1 -\vt} \) and \( t' = \sqrt{\frac{1 +\vt}{1 -\rho^2} } - \frac{1}{\rho}\left(\frac{1}{\sqrt{1 -\rho^2}} - 1\right)\sqrt{1 -\vt}\). In this case, \( t' \geq 0 \) always holds; \( \tfptwo \geq 1 \implies \tfpone \geq 1 \), and \( \tfptwo \geq 1 \) is equivalent to \( \sqrt{r} = 2 \sqrt{\frac{1 -\vt}{1 -\rho^2}} \geq  1 + \sqrt{1 -\vt}\). This is a sufficient condition for this curve to show up in the diagram. 

\textit{Second}, we study the curve \( \sqrt{r} = 1 + \sqrt{1 -\vt} \), which is given by \( \tfptwo =\tfnone = 1 \). We also need \( f_1,\vt +f_2,\vt +f_3 \geq 1 \) and one more equality. Depending on which requirement to take equality, we discuss two possible cases:

If \( \tfpone =1 \), and \( \tfpthree \geq 1 \), \( \tfntwo \geq 1 \), then we will get \( \lambda' = \frac{\sqrt{1 -\rho^2} -(1 -\rho^2)}{\rho(1 +\rho)} \) and \( t' = \frac{(1 +\rho) - \sqrt{1 -\rho^2}}{\rho(1 +\rho)}\). From the two inequality requirements, we get 
\begin{equation*}
  \sqrt{1 -\vt} \leq \min \left\{ \frac{ - \sqrt{1 -\rho^2} + 2\rho + 2}{(1 +\rho)^2}\sqrt{1 -\rho^2},\,\frac{1 -\rho}{1 +\rho} \right\} =\frac{1 -\rho}{1 +\rho}
\end{equation*}

If \( \tfntwo = 1 \) and \( \tfpone \geq 1 \), \( \tfpthree \geq 1 \), then we will get \( \lambda' = \frac{1 +\rho}{\rho} \left(\frac{1}{\sqrt{1 -\rho^2}} - 1\right)\sqrt{1 -\vt}\), \( t' = 1 -\lambda' \). From \( t' \geq 0 \) and the two inequality requirements, we get 
\begin{equation*}
  \frac{1 -\rho}{1 +\rho} \leq \sqrt{1 -\vt} \leq \min \left\{ \left(\frac{2}{\sqrt{1 -\rho^2} } - 1\right)^{ - 1},\,\frac{\rho}{1 +\rho}\left(\frac{1}{\sqrt{1 -\rho^2}} - 1\right)^{ - 1} \right\} = \left(\frac{2}{\sqrt{1 -\rho^2} } - 1\right)^{ - 1}
\end{equation*}

Taking the intersection of the first two cases, we already know that \( \sqrt{r} = 1 + \sqrt{1 -\vt} \) exists as long as \( \sqrt{r} = 1+ \sqrt{1 -\vt} \geq 2 \sqrt{\frac{1 -\vt}{1 -\rho^2}} \). There is one more case left, but the interval of \( \vt \) for \( \sqrt{r} = 1 + \sqrt{1 -\vt} \) to exist will be a subset of what we already have, so we omit it.

Now we already seem to have the whole phase curve, but the tricky part of thresholded Lasso having two tunable parameters is that we might have multiple curves for the same \( \vt \), and we need to take the minimum across all the curves. Thus we need to continue discussing all the other curves. For \( \rho > 0 \), we have three more to go.

\textit{Third}, we study the curve \( \sqrt{r} = \frac{2\sqrt{1 -\rho^2} -(1 +\rho)}{(1 -\rho)\sqrt{1 -\rho^2}}\sqrt{1 -\vt} + \frac{1}{\sqrt{1 -\rho^2}}\), given by \[ \tfpone =\tfnone =\tfntwo = 1. \] 
Now we have \( \lambda' = \frac{1 +\rho}{\rho} \left(\frac{1}{\sqrt{1 -\rho^2}} - 1\right)\sqrt{1 -\vt}\) and \( t' =\frac{1}{\sqrt{1 -\rho^2}} - \frac{1 +\rho}{\rho(1 -\rho)}\left(\frac{1}{\sqrt{1 -\rho^2}} - 1\right)\sqrt{1 -\vt} \). From \( t' \geq 0 \) and \( f_1,\vt + \vt+f_2 \geq 1 \), we get 
\begin{equation*}
  \sqrt{1 -\vt} \leq \min \left\{ \left[ \frac{(1 +\rho)^2}{\rho \sqrt{1 -\rho^2}}\left(\frac{1}{\sqrt{1 -\rho^2}} - 1\right) \right]^{ - 1},\,\frac{1 -\rho}{1 +\rho},\,\left(2 -\frac{2 \sqrt{1 -\rho^2}}{1 -\rho} +\frac{1 +\rho}{1 -\rho}\right)^{ - 1} \right\} =\frac{1 -\rho}{1 +\rho}
\end{equation*}
As we can see, when \( \sqrt{1 -\vt} \leq \frac{1 -\rho}{1 +\rho} \), now we have 2 curves, both of which seem to be the boundary. We must take the \textit{lower one} then.

\textit{Fourth}, we study the curve \( \sqrt{r} = \left(1+\frac{1+\rho}{2 \sqrt{1-\rho^{2}}}\right) \sqrt{1-\vartheta}+\frac{1-\rho}{2 \sqrt{1-\rho^{2}}}\) given by \[ \tfpone =\tfpthree =\tfnone = 1. \]
Now we have \( \lambda' = \frac{1}{2\rho}\left(1 - \sqrt{1 -\vt}\right)\sqrt{1 -\rho^2}\) and \( t' = \frac{1}{2\rho \sqrt{1 -\rho^2}}\left[(1 + \rho)\sqrt{1 -\vartheta} -(1 -\rho)\right]\). 

From \( t' \geq 0 \) and \( f_1, \vt+f_3 \geq 1 \), we need \begin{equation*}
  \max \left\{ \frac{1 -\rho}{1 +\rho},\,\frac{2 \sqrt{1 -\rho^2} -(1 -\rho)}{1 +\rho} \right\} \leq \sqrt{1 -\vt} \leq \frac{1 -\rho}{2} \cdot \left(\frac{3 -\rho}{2} - \sqrt{1 -\rho^2}\right)^{ - 1}
\end{equation*}
This gives us an empty set, because actually
\begin{equation*}
  \frac{1 -\rho}{2} \cdot \left(\frac{3 -\rho}{2} - \sqrt{1 -\rho^2}\right)^{ - 1} < \frac{2 \sqrt{1 -\rho^2} -(1 -\rho)}{1 +\rho}.
\end{equation*}

\textit{Fifth}, we study the curve \( \sqrt{r} = \sqrt{\frac{1-\vartheta}{1-\rho^{2}}}+\frac{2 \rho+2-\sqrt{1-\rho^{2}}}{(1+\rho)^{2}}\) given by \begin{equation*}
  \tfpone =\tfptwo =\tfntwo = 1.
\end{equation*}
We get \( \lambda' =\frac{\sqrt{1 -\rho^2} -(1 -\rho^2)}{\rho(1 +\rho)} \) and \( t' = \frac{(1 +\rho) - \sqrt{1 -\rho^2}}{\rho(1 +\rho)}\). 
From \(  \vt+f_2, \vt+f_3 \geq 1\), we get 
\begin{equation*}
  \frac{1 -\rho}{1 +\rho} \leq \sqrt{1 -\vt} \leq \frac{2 \sqrt{1 -\rho^2} -(1 -\rho)}{1 +\rho}
\end{equation*}

\textit{Summarising all the five curves when \( \rho > 0 \)}: We will elimiate the \textit{third} and \textit{fifth} curve. For the \textit{third} curve, when \( \sqrt{1 -\vt} \leq \frac{1 -\rho}{1 +\rho} \), it is always larger than the other curve \( \sqrt{r} = 1 + \sqrt{1 -\vt} \). In other words, when \( \sqrt{1 -\vt} \leq \frac{1 -\rho}{1 +\rho} \), we always have \begin{equation*}
  \sqrt{r} = \frac{2\sqrt{1 -\rho^2} -(1 +\rho)}{(1 -\rho)\sqrt{1 -\rho^2}}\sqrt{1 -\vt} + \frac{1}{\sqrt{1 -\rho^2}} \geq \sqrt{1 -\vt } + 1
\end{equation*}
In fact, \( (LHS - RHS) \) takes its minimum at \( \sqrt{1 -\vt} =\frac{1 -\rho}{1 +\rho} \), which is exactly zero.

For the \textit{fifth} curve, when \( \frac{1 -\rho}{1 +\rho} \leq \sqrt{1 -\vt} \leq \frac{2 \sqrt{1 -\rho^2} -(1 -\rho)}{1 +\rho} \), we have
\begin{equation*}
  \sqrt{r} = \sqrt{\frac{1-\vartheta}{1-\rho^{2}}}+\frac{2 \rho+2-\sqrt{1-\rho^{2}}}{(1+\rho)^{2}} \geq \max \left\{ 1 + \sqrt{1 -\vt}, 2 \sqrt{\frac{1 -\vt}{1 -\rho^2}} \right\}
\end{equation*}
which can be verified using  \( \frac{1 -\rho}{1 +\rho} \leq \sqrt{1 -\vt} \leq \frac{2 \sqrt{1 -\rho^2} -(1 -\rho)}{1 +\rho} \) in a similar manner. To sum up, for \( \rho \geq 0 \), the phase curve of thresholded Lasso is 
\begin{equation*}
  \sqrt{r} =\max \left\{ 1 + \sqrt{1 -\vt}, 2 \sqrt{\frac{1 -\vt}{1 -\rho^2}}  \right\}.
\end{equation*}

\textit{We then talk about \( \rho < 0 \) }, which now requires additionally \( 2\vt + f_4 \geq 1 \), or 
 \[ 2\vt + \left[(1 -\rho^2)\sqrt{r}-\lambda'(1 +|\rho|)-(1 -\rho^2)t'\right]_+ ^2 \geq 1. \]
When \( \vartheta \geq \frac{1}{2} \), this newly added requirement has no effects, and the right half (\( \vt \geq \frac{1}{2} \))  of the phase diagram should be the same as that of \( \rho \geq 0 \). As a result, we can limit ourselve to consider \( \vt \leq \frac{1}{2} \).

Also note that we used to ignore \( 2\vt + f_4 \geq 1 \) for \( \rho \geq 0 \) because it is not restrictive; now we have added it, and it is the only difference between the cases of \( \rho \geq 0 \) and \( \rho < 0 \) (because \( f_1, \vt+f_2,g_3 \) only rely on \( |\rho| \)). As a result, we only need to discuss this additional requirement; since we have eliminated three curves when \( \rho > 0 \), there is no need to discuss them again.

\textit{First}, we study the curve \( \sqrt{r} = 2 \sqrt{\frac{1 -\vt}{1 -\rho^2}} \), which is the curve given by letting \[ \tfpthree =\tfntwo = 1 \] in \(  \vt+f_2 \) and \(  \vt+f_3 \). We also need \( f_1, \vt+f_3, 2\vt+f_4 \geq 1 \) and one more quality. When \( \rho > 0 \), we used to consider only one case, but now we need to discuss all four cases and take the union to get the interval of \( \vt \), because each case has different \( (\lambda',t') \) and lead to different intervals of \( \vt \).

If \( \tfpthree =\tfnone =\tfntwo =1 \), and \(\vt + f_1, 2\vt + f_4 \geq 1 \), this is the case we covered when \( \rho \geq  0 \). Then \( \tfnthree \geq 1 \) implies \(  \sqrt{\frac{1 - 2\vt}{1 -\vt} } \leq \frac{2 \sqrt{1 -\rho^2} -(1 +|\rho|)}{1 -|\rho|}\) and other requirements imply \( \sqrt{1-\vt}\left(\frac{2}{\sqrt{1-\rho^2}}-1\right) > 1 \). (The first constraint implies \( \vt \geq  \) some value, and the second one implies \( \vt \leq  \) some value.) The overall requirements are:
\begin{equation*}
  \sqrt{1-\vt}\left(\frac{2}{\sqrt{1-\rho^2}}-1\right) > 1,\text{and }\sqrt{\frac{1 - 2\vt}{1 -\vt} } \leq \frac{2 \sqrt{1 -\rho^2} -(1 +|\rho|)}{1 -|\rho|}
\end{equation*}

If \( \tfpone =\tfpthree =\tfntwo = 1 \), and other terms are greater than one, then we have \( \lambda' = \frac{1 -\rho^2}{2|\rho|} \cdot \frac{1 - \sqrt{1 -\vt}}{\sqrt{1 -\rho^2}}\), \( t' = \frac{1}{\sqrt{1 -\rho^2}} -\frac{\lambda'}{1 -|\rho|} \). In this case, \( t' \geq 0 \) and \( f_1,\vt + f_3,2\vt + f_4 \geq 1 \)  implies 
\begin{equation*}
  \sqrt{1 -\vt} \geq \max \left\{ \frac{1 -|\rho|}{1 +|\rho|},\,\frac{2 \sqrt{1 -\rho^2} -(1 -|\rho|)}{1 +|\rho|} ,\,\frac{1 -|\rho|}{(3 - |\rho|) - 2 \sqrt{1 -\rho^2}} \right\} =\frac{2 \sqrt{1 -\rho^2} -(1 -|\rho|)}{1 +|\rho|} 
\end{equation*}

If \( \tfptwo =\tfpthree =\tfntwo = 1 \) and other terms are greater than one, then we have \( \lambda' = \frac{1 +|\rho|}{|\rho|}\left(1 - \sqrt{\frac{1 -\vt}{1 -\rho^2}} \right)\) and \( t' = 1 -\lambda' \). In this case, \( \lambda' \geq 0 \) and \( t' \geq 0 \) implies \( \frac{\sqrt{1 -\rho^2}}{1 +|\rho|} \leq \sqrt{1 -\vt} \leq \sqrt{1 -\rho^2 } \), which is weaker than the requirements from \( f_1,\vt + f_3 \) and \( 2\vt + f_4 \):
\begin{equation*}
  \left(\frac{2}{\sqrt{1 -\rho^2}} - 1\right)^{ - 1} \leq \sqrt{1 -\vt} \leq  \frac{2 \sqrt{1 -\rho^2} -(1 -|\rho|)}{1 +|\rho|},\text{and}\;\; \frac{2 \sqrt{1 -\rho^2}}{1 -|\rho|} + \sqrt{1 - 2\vt} \leq \frac{3 -|\rho|}{1 -|\rho|} \sqrt{1 -\vt}
\end{equation*}

If \( \tfpthree =\tfntwo =\tfnthree = 1 \) and other terms are greater than one, then we have \( \lambda' = \frac{1 -\rho^2}{2|\rho|}\frac{1}{\sqrt{1 -\rho^2}}\left(\sqrt{1 -\vt} - \sqrt{1 - 2\vt}\right) \) and \( t' = \sqrt{\frac{1 -\vt}{1 -\rho^2}} -\frac{\lambda'}{1 +|\rho|} \). In this case, \( \lambda' \geq 0 \) and \( t' \geq 0 \) implies \( \sqrt{\frac{1 - 2\vt}{1 -\vt}} \geq 1 - \frac{2|\rho|}{1 -|\rho|} \), which is still weaker than the requirements from \( f_1,\vt + f_3 \geq 1 \):
\begin{equation*}
  \sqrt{\frac{1 - 2\vt}{1 -\vt}} \geq \frac{2 \sqrt{1 -\rho^2} -(1 +|\rho|)}{1 -|\rho|},\text{and}\;\;\frac{2 \sqrt{1 -\rho^2}}{1 -|\rho|} + \sqrt{1 - 2\vt} \leq \frac{3 -|\rho|}{1 -|\rho|} \sqrt{1 -\vt} 
\end{equation*}

Taking the union over all the four cases of curve \( \sqrt{r} = 2 \sqrt{\frac{1 -\vt}{1 -\rho^2}} \), it exists for \( \vt \) satisfying \(  \sqrt{1 -\vt} \geq \left(\frac{2}{\sqrt{1 -\rho^2}} - 1\right)^{ - 1} \) and \( \frac{2 \sqrt{1 -\rho^2}}{1 -|\rho|} + \sqrt{1 - 2\vt} \leq \frac{3 -|\rho|}{1 -|\rho|} \sqrt{1 -\vt} \), 
which is equivalent to 
\begin{equation*}
  \sqrt{r} = 2 \sqrt{\frac{1 -\vt}{1 -\rho^2}} \geq \max \left\{ 1 + \sqrt{1 -\vt},\;1+\frac{1+|\rho|}{2} \sqrt{\frac{1-\vartheta}{1-\rho^{2}}}+\frac{1-|\rho|}{2} \sqrt{\frac{1-2 \vartheta}{1-\rho^{2}}}   \right\}
\end{equation*}

\textit{Second}, we study the curve \( \sqrt{r} = 1 + \sqrt{1 -\vt} \), which is given by \( \tfptwo =\tfnone = 1 \). We also need \( f_1,\vt +f_2,\vt +f_3 \geq 1 \) and one more equality. When \( \rho \geq 0 \), we have discussed two cases; now we discuss the two old cases with the additional requirement \( 2\vt + f_4 \geq 1 \), and two additional cases.

If \( \tfpone = \tfptwo =\tfnone = 1 \) and other terms are greater than one, we have already considered this in a previous section. Now we only add the requirement \( \tfnthree \geq 1 \). The final requirement on \( \vt \) is  \( \sqrt{1 -\vt} \leq \frac{1 -|\rho|}{1 +|\rho|} \) and \( (1 - \sqrt{1 -\rho^2}) + \sqrt{1 - 2\vt} \leq \sqrt{1 - \rho^2}\sqrt{1 - \vt} \).

If \( \tfptwo =\tfnone = \tfntwo = 1 \) and other terms are greater than one, we have already considered this in a previous section. Now we only add the requirement \( \tfnthree \geq 1 \). The final requirement on \( \vt \) is \( \frac{1 -|\rho|}{1 + |\rho|} \leq \sqrt{1 -\vt} \leq \left(\frac{2}{\sqrt{1 -\rho^2}} - 1\right)^{ - 1} \) and \( \sqrt{\frac{1 - 2\vt}{1 -\vt}} \leq \frac{2 \sqrt{1 -\rho^2} -(1 +|\rho|)}{1 -|\rho|} \)

If \( \tfpthree =\tfptwo =\tfnone = 1 \) and other terms are greater than one, this is a new case, and we have \( \lambda' = \frac{1 +|\rho|}{|\rho|}\left(1 - \sqrt{\frac{1 -\vt}{1 -\rho^2} }\right)\), \( t' = 1 -\lambda' \). \( \lambda',t' \geq 0 \) requires \( \frac{\sqrt{1 -\rho^2}}{1 +|\rho|} \leq \sqrt{1 -\vt} \leq \sqrt{1 -\rho^2} \), and the requirements from \( f_1,\vt + f_3,2\vt + f_4 \geq 1 \) are \( \sqrt{1 -\vt} \leq \min \left\{ \frac{2 \sqrt{1 -\rho^2} -(1 -|\rho|)}{1 +|\rho|},\,\left(\frac{2}{\sqrt{1 -\rho^2}} - 1\right)^{ - 1} \right\} \) and \( \frac{1 +|\rho|}{1 -|\rho|}\sqrt{1 -\rho^2} + \sqrt{1 - 2\vt} \leq \left(\frac{1 +|\rho|}{1 -|\rho|} +\sqrt{1 -\rho^2}\right)\sqrt{1 -\vt} \). Taking the intersection, the overall requirement on \( \vt \) is 
\begin{equation*}
  \frac{\sqrt{1 -\rho^2}}{1 +|\rho|} \leq \sqrt{1 -\vt} \leq \left(\frac{2}{\sqrt{1 -\rho^2}} - 1\right)^{ - 1}\text{and}\;\frac{1 +|\rho|}{1 -|\rho|}\sqrt{1 -\rho^2} + \sqrt{1 - 2\vt} \leq \left(\frac{1 +|\rho|}{1 -|\rho|} +\sqrt{1 -\rho^2}\right)\sqrt{1 -\vt}
\end{equation*}

If \( \tfptwo =\tfnone =\tfnthree = 1 \) and other terms are greater than one, then we have \( \lambda' = \frac{1 -|\rho|}{|\rho|} \left(\sqrt{1 -\vt} - \sqrt{\frac{1 - 2\vt}{1 -\rho^2}}\right)\), \( t' = 1 -\lambda' \). The requirements from \( \lambda',t' \geq 0 \) and \( f_1,\vt + f_2,\vt + f_3 \geq 1 \) are:
\begin{align*}
  \sqrt{\frac{1 - 2\vt}{1 -\vt}} \leq&~ \frac{2 \sqrt{1 -\rho^2} -(1 +|\rho|)}{1 -|\rho|} \\
  \frac{|\rho| \sqrt{1 -\rho^2}}{1 -|\rho|} + \sqrt{1 - 2\vt} \geq&~ \sqrt{1 -\rho^2} \cdot \sqrt{1 -\vt} \\
  (1 - \sqrt{1 -\rho^2}) + \sqrt{1 - 2\vt} \leq&~  \sqrt{1 -\rho^2} \cdot \sqrt{1 -\vt} \\
  \frac{1 +|\rho|}{1 -|\rho|}\sqrt{1 -\rho^2} + \sqrt{1 - 2\vt} \geq&~ \left(\frac{1 +|\rho|}{1 -|\rho|} +\sqrt{1 -\rho^2}\right)\sqrt{1 -\vt}
\end{align*}
Taking intersection, the first and the last can imply the other two, so the final requirements are just the first and the last one. 

Taking the union over all the four cases of \( \sqrt{r} = 1 + \sqrt{1 -\vt} \), we have \( \sqrt{1 -\vt} \leq \left(\frac{2}{\sqrt{1 -\rho^2}} - 1\right)^{ - 1} \) and \( \sqrt{\frac{1 - 2\vt}{1 -\vt}} \leq \frac{2 \sqrt{1 -\rho^2} -(1 +|\rho|)}{1 -|\rho|} \), which is equivalent to 
\begin{equation*}
  \sqrt{r} = 1 + \sqrt{1 -\vt} \geq \max \left\{ 2 \sqrt{\frac{1 -\vt}{1 -\rho^2}},\,1+\frac{1+|\rho|}{2} \sqrt{\frac{1-\vartheta}{1-\rho^{2}}}+\frac{1-|\rho|}{2} \sqrt{\frac{1-2 \vartheta}{1-\rho^{2}}} \right\}
\end{equation*}  

\textit{Third}, we study the curve \( \sqrt{r} =\sqrt{\frac{1 - 2\vt}{1 -\rho^2}} +\frac{1}{\sqrt{1 -\rho^2}} \) given by \( \tfpone =\tfnthree = 1 \). We need one more equality constraint and other terms greater than one. We will see this curve does not exist for ant \( \vt \) at all. 

If the additional equality is \( \tfptwo = 1 \), then we have \( \lambda' = \frac{\sqrt{1 -\rho^2} -(1 -\rho^2)}{|\rho|(1 +|\rho|)}\) and \( t' = \frac{(1 +|\rho|) - \sqrt{1 -\rho^2}}{|\rho|(1 +|\rho|)}\). However, this case admits no \( \vt \), because \( \tfpthree \geq 1 \) requires \( \sqrt{1 -\vt} \leq \frac{2 \sqrt{1 -\rho^2} -(1 -|\rho|)}{1 +|\rho|} \) and \( \tfnone \geq 1 \) requires \( (1 - \sqrt{1 -\rho^2}) + \sqrt{1 - 2\vt} \geq~  \sqrt{1 -\rho^2} \cdot \sqrt{1 -\vt} \); these two requirements have no intersection.

If the additional equality is \( \tfpthree = 1 \), then we have \( \lambda' = \frac{1 -\rho^2}{2 |\rho|}\frac{1}{\sqrt{1 -\rho^2}}\left(1 -\sqrt{1 -\vt}\right)\) and \(  t' = \frac{1}{\sqrt{1 -\rho^2}} -\frac{\lambda'}{1 -|\rho|}\).  However, this case admits no \( \vt \), because \( t' \geq 0 \) requires \( \sqrt{1 -\vt} > \frac{1 -|\rho|}{1 +|\rho|} \) and \( \tfntwo \geq 1 \) requires \( 1 + \sqrt{1 - 2\vt} \geq 2 \sqrt{1 -\vt} \);  these two requirements have no intersection.

If the additional equality is \( \tfnone = 1 \), then \( \lambda' = \frac{1 -|\rho|}{|\rho|} \left(\sqrt{1 -\vt} - \sqrt{\frac{1 - 2\vt}{1 -\rho^2}}\right)\) and \(  t' = \frac{1}{\sqrt{1 -\rho^2}} -\frac{\lambda'}{1 -|\rho|}\). However, this case admits no \( \vt \), because \( \tfptwo \geq 2 \) requires \( (1 -\sqrt{1 -\rho^2}) + \sqrt{1 - 2\vt} \geq \sqrt{1 -\rho^2}\sqrt{1 -\vt} \) and \( \tfntwo \geq 1 \) requires \( \sqrt{\frac{1 - 2\vt}{1 -\vt}} \leq \frac{2 \sqrt{1 -\rho^2} -(1 +|\rho|)}{1 -|\rho|} \); these two requirements have no intersection.

If the additional equality is \( \tfntwo = 1 \), then \( \lambda' = \frac{1 -\rho^2}{2|\rho|}\frac{1}{\sqrt{1 -\rho^2}}\left(\sqrt{1 -\vt} - \sqrt{1 - 2\vt}\right)\) and \( t' = \frac{1}{\sqrt{1 -\rho^2}} -\frac{\lambda'}{1 -|\rho|} \). However, this case admits no \( \vt \), because \( t' \geq 0 \), \( \tfptwo \geq 1 \) and \( \tfpthree \geq 1 \) respectively requires 
\begin{align*}
&\frac{2|\rho|}{1 +|\rho|} + \sqrt{1 - 2\vt} \geq \sqrt{1 -\vt} \\
& \frac{2(1 - \sqrt{1 -\rho^2})}{1 +|\rho|} + \sqrt{1 - 2\vt} \geq \sqrt{1 -\vt}\\
&1 + \sqrt{1 - 2\vt} \geq 2 \sqrt{1 -\vt}
\end{align*}
These three requirements admit no \( \vt\in(0,1) \). 

To sum up, the \textit{third} surve  does not show up in the phase diagram.

\textit{Fourth}, we study the curve \(\sqrt{r} =\sqrt{\frac{1-2 \vartheta}{1-\rho^{2}}}+\frac{2(1-|\rho|)-\sqrt{(1-\vartheta)\left(1-\rho^{2}\right)}}{(1-|\rho|)^{2}}\) given by  \[ \tfptwo = \tfpthree = \tfnthree = 1.\] We have \( \lambda' = \frac{1 +|\rho|}{|\rho|}\left(1 - \sqrt{\frac{1 -\vt}{1 -\rho^2} }\right) \) and \( t' = 1 -\lambda' \). \( \lambda',t' \geq 0 \) requires \( \frac{\sqrt{1 -\rho^2}}{1 +|\rho|} \leq \sqrt{1 -\vt} \leq \sqrt{1 -\rho^2} \), and we also need the requirements from \( f_1,\vt +f_3 \geq 1 \). The overall requirement on \( \vt \) is:
\begin{align*}
  \frac{2 \sqrt{1 -\rho^2} -(1 -|\rho|)}{1 +|\rho|} \leq&~ \sqrt{1 -\vt} \leq\sqrt{1 -\rho^2} \\
  \frac{1 +|\rho|}{1 -|\rho|}\sqrt{1 -\rho^2} + \sqrt{1 - 2\vt} \geq&~ \left(\frac{1 +|\rho|}{1 -|\rho|} +\sqrt{1 -\rho^2}\right)\sqrt{1 -\vt}\\
  \frac{2 \sqrt{1 -\rho^2}}{1 -|\rho|} + \sqrt{1 - 2\vt} \geq&~ \frac{3 -|\rho|}{1 -|\rho|} \sqrt{1 -\vt}
\end{align*}
This is also an empty set, and this curve can never be present in the phase diagram.

To prove this, we note that \( \phi(\vt) = \frac{2 \sqrt{1 -\rho^2}}{1 -|\rho|} + \sqrt{1 - 2\vt} - \frac{3 -|\rho|}{1 -|\rho|} \sqrt{1 -\vt} \) is a ``first increasing, then decreasing'' function of \( \vt\in[0,\frac{1}{2}] \). When \( |\rho| \geq \frac{1 + 2 \sqrt{2}}{7} \), the maximum of \( \frac{2 \sqrt{1 -\rho^2}}{1 -|\rho|} + \sqrt{1 - 2\vt} - \frac{3 -|\rho|}{1 -|\rho|} \sqrt{1 -\vt} \) is not positive, and does not admit a curve. When \( |\rho| < \frac{1 + 2 \sqrt{2}}{7} \), we add the requirement of \( \sqrt{1 -\vt} \geq \frac{2 \sqrt{1 -\rho^2} -(1 -|\rho|)}{1 +|\rho|} \), even the largest \( \vt \) is still on the left side of the peak of the maximum point of \( \phi(\vt) \)  and still makes it negative.

\textit{Fifth}, we study the curve \(\sqrt{r} =\left(\frac{2}{1+|\rho|}+\frac{1}{\sqrt{1-\rho^{2}}}\right) \sqrt{1-\vartheta}-\frac{1-|\rho|}{1+|\rho|} \sqrt{\frac{1-2 \vartheta}{1-\rho^{2}}}\), which is given by \[ \tfpthree =\tfnone =\tfnthree = 1. \] Now we have \( \lambda' = \frac{1 -|\rho|}{|\rho|} \left(\sqrt{1 -\vt} - \sqrt{\frac{1 - 2\vt}{1 -\rho^2}}\right)\) and \( t' = \frac{1}{\sqrt{1 -\rho^2}} -\frac{\lambda'}{1 -|\rho|} \). The requirements of \( \lambda',t' \geq 0 \) and \( f_1,\vt +f_3 \geq 0\) are
\begin{align*}
  \sqrt{1 -\rho^2} - \frac{|\rho|(1 +|\rho|)}{1 -|\rho|} \leq \sqrt{\frac{1 - 2\vt}{1 -\vt}} \leq&~ \frac{2 \sqrt{1 -\rho^2} -(1 +|\rho|)}{1 -|\rho|} \\
  \frac{1 +|\rho|}{2} + \sqrt{1 - 2\vt} \leq&~ \left(\frac{1 +|\rho|}{2} + \sqrt{1 -\rho^2}\right)\sqrt{1 -\vt} \\
  \frac{1 +|\rho|}{1 -|\rho|}\sqrt{1 -\rho^2} + \sqrt{1 - 2\vt} \leq&~ \left(\frac{1 +|\rho|}{1 -|\rho|} +\sqrt{1 -\rho^2}\right)\sqrt{1 -\vt}
\end{align*}

This may not be an empty set. However, even when it is not an empty set, the curve is actually either greater than \( \sqrt{r} = 1 + \sqrt{1 -\vt} \) or \( \sqrt{r} = 2 \sqrt{\frac{1 -\vt}{1 -\rho^2}} \). To see this:
\begin{itemize}
  \item In terms of the existence of \( \sqrt{r} = 1 + \sqrt{1 -\vt} \) and \( \sqrt{r} = 2 \sqrt{\frac{1 -\vt}{1 -\rho^2}} \): When \( \frac{1 -|\rho|}{1 + |\rho|} \leq \sqrt{1 -\vt} \leq \left(\frac{2}{\sqrt{1 -\rho^2}} - 1\right)^{ - 1} \) and \( \sqrt{\frac{1 - 2\vt}{1 -\vt}} \leq \frac{2 \sqrt{1 -\rho^2} -(1 +|\rho|)}{1 -|\rho|} \), we have proven that \( \sqrt{r} = 1 + \sqrt{1 -\vt} \) is one segment of the phase diagram.  When \( \sqrt{1 -\vt} \geq \left(\frac{2}{\sqrt{1 -\rho^2}} - 1\right)^{ - 1} \) and \( \sqrt{\frac{1 - 2\vt}{1 -\vt}} \leq \frac{2 \sqrt{1 -\rho^2} -(1 +|\rho|)}{1 -|\rho|} \), we have proven that \( \sqrt{r} = 2 \sqrt{\frac{1 -\vt}{1 -\rho^2}} \)  is one segment of the phase diagram. We can prove that 
  \begin{align*}
    &\sqrt{1 -\rho^2} - \frac{|\rho|(1 +|\rho|)}{1 -|\rho|} \leq \sqrt{\frac{1 - 2\vt}{1 -\vt}} \leq~ \frac{2 \sqrt{1 -\rho^2} -(1 +|\rho|)}{1 -|\rho|}\\ 
    \implies & \begin{cases} 
    \text{either }\frac{1 -|\rho|}{1 + |\rho|} \leq \sqrt{1 -\vt} \leq \left(\frac{2}{\sqrt{1 -\rho^2}} - 1\right)^{ - 1} \text{ and }\sqrt{\frac{1 - 2\vt}{1 -\vt}} \leq \frac{2 \sqrt{1 -\rho^2} -(1 +|\rho|)}{1 -|\rho|} \\
    \text{or }\sqrt{1 -\vt} \geq \left(\frac{2}{\sqrt{1 -\rho^2}} - 1\right)^{ - 1} \text{ and }\sqrt{\frac{1 - 2\vt}{1 -\vt}} \leq \frac{2 \sqrt{1 -\rho^2} -(1 +|\rho|)}{1 -|\rho|} 
    \end{cases}
  \end{align*}
  so one of \( \sqrt{r} = 1 + \sqrt{1 -\vt} \) and \( \sqrt{r} = 2 \sqrt{\frac{1 -\vt}{1 -\rho^2}} \) exists as long as the \textit{fifth} curve exists.
  \item In the latter case, it is greater than \( \sqrt{r} = 2 \sqrt{\frac{1 -\vt}{1 -\rho^2}} \) which exists in the same region. (can be easily verified)
  \item In the former case, we can assume \( |\rho| \leq 3 - 2 \sqrt{2} \) because  we need \( \sqrt{1 -\vt} \geq \frac{1 -|\rho|}{1 + |\rho|} \) to hold for some \( \vt\in(0,\frac{1}{2}) \). Using \( |\rho| \leq 3 - 2 \sqrt{2} \), we can prove  \( \sqrt{\frac{1 - 2\vt}{1 -\vt}} \geq \sqrt{1 -\rho^2} - \frac{|\rho|(1 +|\rho|)}{1 -|\rho|} \implies \sqrt{1 -\vt} \geq \frac{1 -|\rho|}{1 + |\rho|} \). Now we can easily verify the curve is greater than \( \sqrt{r} = 1 + \sqrt{1 -\vt} \) which exists in the same region
\end{itemize}
As a result, this curve does not play a part in the final phase diagram either. 

\textit{Sixth}, we study the curve \(\sqrt{r} =1+\frac{1+|\rho|}{2} \sqrt{\frac{1-\vartheta}{1-\rho^{2}}}+\frac{1-|\rho|}{2} \sqrt{\frac{1-2 \vartheta}{1-\rho^{2}}}\), given by 
\begin{equation*}
  \tfptwo =\tfntwo =\tfnthree = 1
\end{equation*}
We get \( \lambda' =\frac{1 -\rho^2}{2|\rho|}\frac{1}{\sqrt{1 -\rho^2}}\left(\sqrt{1 -\vt} - \sqrt{1 - 2\vt}\right) \) and \( t' = 1 -\lambda' \). The requirements from \( t' \geq 0 \), \( \tfpone \geq 1 \), \( \tfpthree \geq 1 \) and \( \tfnone \geq 1 \) are respectively
\begin{align*}
 &  \frac{2|\rho|}{\sqrt{1 -\rho^2}} +\sqrt{1 - 2\vt}\geq \sqrt{1 -\vt} \\
 & \frac{2(1 - \sqrt{1 -\rho^2})}{1 +|\rho|} +\sqrt{1 - 2\vt} \leq \sqrt{1 -\vt} \\
 & \frac{2 \sqrt{1 -\rho^2}}{1 -|\rho|} + \sqrt{1 - 2\vt} \geq \frac{3 -|\rho|}{1 -|\rho|} \sqrt{1 -\vt} \\
 & \sqrt{\frac{1 - 2\vt}{1 -\vt}} \geq \frac{2 \sqrt{1 -\rho^2} -(1 +|\rho|)}{1 -|\rho|}
\end{align*}

Taking the intersection, the last two inequalities can imply the rest, and it is equivalent to 
\begin{equation*}
  \sqrt{r} =1+\frac{1+|\rho|}{2} \sqrt{\frac{1-\vartheta}{1-\rho^{2}}}+\frac{1-|\rho|}{2} \sqrt{\frac{1-2 \vartheta}{1-\rho^{2}}}  \geq \max \left\{1 + \sqrt{1 -\vt},\, 2 \sqrt{\frac{1 -\vt}{1 -\rho^2}} \right\}.
\end{equation*}  

Now we have studied all the curves for \( \rho < 0 \). To sum up, the phase curve is 
\begin{equation*}
  \sqrt{r} =  \max \left\{1 + \sqrt{1 -\vt},\, 2 \sqrt{\frac{1 -\vt}{1 -\rho^2}},\,1+\frac{1+|\rho|}{2} \sqrt{\frac{1-\vartheta}{1-\rho^{2}}}+\frac{1-|\rho|}{2} \sqrt{\frac{1-2 \vartheta}{1-\rho^{2}}} \right\}
\end{equation*}

\subsection{Proof of Lemma~\ref{supplem:sol.path.thresLasso}}  \label{subsec:proof-thresLasso-solution}

Recall the optimization in \eqref{enproof-optimization}; the solution \( b =(b_1,b_2) \) has to set the sub-gradient of the objective function to zero. As a result, the equation of the sub-gradient for \( b =(b_1,b_2) \)  is:
\begin{equation*}
  \begin{bmatrix} 
  1 & \rho \\
  \rho & 1
  \end{bmatrix} \begin{bmatrix} b_1 \\ b_2 \end{bmatrix} 
  + \lambda' \begin{bmatrix} 
  \sgn(b_1) \\
  \sgn(b_2)
  \end{bmatrix} 
   = \begin{bmatrix} h_1 \\ h_2 \end{bmatrix} 
\end{equation*}

Now we begin to find out the solution path. Thresholded Lasso has two steps: First, we run Lasso to select variables from \( (x_j,x_{j + 1})  \); second, a thresholding step with \( t = t'\sqrt{2\log(p)} \) is further performed, and the surviving variables of the two steps are the finally selected ones. Also note that we have required \( \rho \geq 0 \).    

First, we study the behavior of Lasso, and decrease \( \lambda' \) from a sufficiently large value to see when the variables enter the model. We assume \( h_1 > 0 \) and \( 0 <\abs{h_2} < h_1 \). 

The procedure is just setting \( \mu = 0 \) in the proof of Lemma~\ref{suppthm:sol.path.en}, and we summarise the results below:

\begin{itemize}
  \item When \( \lambda' \geq h_1 \), we have \( \hat{b}_1 = \hat{b}_2 = 0 \).
  \item If $h_2\geq \rho h_1$, when \(  \frac{{h_2 -\rho h_1}}{1 -\rho}\leq \lambda' <h_1 \), we have $\hat{b}_1 ={h_1 -\lambda'}$, and $\hat{b}_2 = 0$; 
  
  When $\lambda'<\frac{{h_2 -\rho h_1}}{1 -\rho}$, we have   \[
  \hat{b}_1 =\frac{(h_1 -\rho h_2) -(1 -\rho)\lambda'}{1 -\rho^2},\qquad 
    \hat{b}_2 =\frac{(h_2 -\rho h_1) -(1 -\rho)\lambda'}{1 - \rho^2}; 
   \]
   \item if $h_2<\rho h_1$, when \(  \frac{{-h_2 +\rho h_1}}{1 +\rho}\leq \lambda' <h_1 \), we have $\hat{b}_1 ={h_1 -\lambda'}$, and $\hat{b}_2 = 0$; 
   
   When $\lambda'< \frac{{-h_2 +\rho h_1}}{1 +\rho}$, we have 
\[
    \hat{b}_1 =\frac{(h_1 -\rho h_2) -(1 +\rho)\lambda'}{1 - \rho^2},\quad 
  \hat{b}_2 =\frac{(h_2 -\rho h_1) +(1 +\rho)\lambda'}{1 - \rho^2}.
\]
\end{itemize}

When \( \lambda' \geq h_1 \), for any \( t' \), we have \( \hat b_1 =\hat b_2 = 0 \).

For \( h_2 \geq \rho h_1 \), when \(  \frac{{h_2 -\rho h_1}}{1 -\rho }\leq \lambda' <h_1 \), if \( h_1 \leq \lambda' + t' \), then we still selected neither of \( (x_j,x_{j + 1}) \) in the end. If \( h_1 >\lambda' + t' \), then we will select only \( x_j \). When $\lambda'<\frac{{h_2 -\rho h_1}}{1 -\rho}$, we have  \( \hat b_1 >\hat b_2 > 0 \) and it depends on whether \( t \geq \hat b_1 \), \( \hat b_2 \leq  t < b_1 \) or \( \hat b_2 > t \) how \( (x_j,x_{j + 1}) \) are selected in the end.

For \( h_2 < \rho h_1 \), when \(  \frac{{ -h_2 +\rho h_1}}{1 +\rho }\leq \lambda' <h_1 \), if \( h_1 \leq \lambda' + t' \), then we still selected neither of \( (x_j,x_{j + 1}) \) in the end. If \( h_1 >\lambda' + t' \), then we will select only \( x_j \). When $\lambda'<\frac{{-h_2 +\rho h_1}}{1 -\rho}$, we have  \( \hat b_1 >-\hat b_2 > 0 \) and it depends on whether \( t \geq \hat b_1 \), \( -\hat b_2 \leq  t < b_1 \) or \( -\hat b_2 > t \) how \( (x_j,x_{j + 1}) \) are selected in the end.

\section{Proof of Theorem~\ref{thm:forward} (Forward Selection)}\label{suppsec:forward}

The proof for Forward Selection still consists of three parts: (a) deriving the rejection region, (b) obtaining the rate of convergence of \( \mathbb{E}[H(\hat\beta,\beta)] \), and (c) calculating the phase diagram.

For Forward Selection, we first write $X=[x_1,x_2,\ldots,x_p]$, where $x_i\in \mathbb{R}^n$ for $1\leq i\leq p$. For any subset $A\subset\{1,2,\ldots,p\}$, let $P^{\bot}_{A}$ be the projection  onto the orthogonal complement of the linear space spanned by \( \{ x_i:i\in  A \}  \). 
Before the first part, we formally define Forward Selection in Algorithm~\ref{alg:forward.selection}.
\begin{algorithm}[H]
 \caption{forward selection}
 \label{alg:forward.selection}
 \begin{algorithmic}[1]
     \State Input \( X \) and \( y \) (generated with our own setting.)
     \State Fix \( t > 0 \).
     \State Initialize \( S^{(0)} =\emptyset, \hat\beta^{(0)} = 0 \), \( \hat r^{(0)} = y \).
     \State Initialize \( k = 0 \). 
     \While{true}~\Comment{forward step}
     \State \( k \gets k + 1 \) 
     \State \( \hat r^{(k - 1)}\gets P_{k - 1}^{\bot}y \) \Comment{\( \hat r^{(k - 1)} \) is the residual of the OLS fit of \( Y \) onto \( X_{S^{(k - 1)}} \).}
       \State \( i^* \gets \argmax_{i\notin S^{(k - 1)}}\abs{x_i'\hat r^{(k - 1)}} \)\label{step:entry.rule}
       \State \( \delta^ + \gets  
       \frac{ \abs{x_{i^* }\hat r^{(k - 1)} }}{\norm{ P^{\bot}_{k - 1} x_{i^*}}} \)
       \Comment{the forward gain, equivalent to the decrease in the loss function.}
       \If{\( \delta^ + \leq t \) }
       \State Break.
       \EndIf
       \State \( S^{(k)}\gets S^{(k - 1)}\cup \{ i^* \}  \) \Comment{Note  this step is after checking \(\delta^ +  \leq t \)}
     \EndWhile
     \State \( k\gets k - 1 \) 
     
     \Comment{Because the \( k \)th variable hasn't been added when the ``while'' loop is broken.}
     \State \( \hat\beta =\hat\beta^{\text{ols}}(S^{(k)}) \) 
     \Comment{\( S^{(k)} \) is the set of selected variables.}
 \end{algorithmic}
\end{algorithm}

\begin{remark}
  The stopping rule is equivalent to measuring the decrease in the residual sum of squares. To see this, suppose \( i\in \{ 1,2,\dots ,p \} \) is enrolled at step \( k \), ans \( S^{(k)} = S^{(k - 1)}\cup \{ i \}  \). Then \( ||y - X_{S^{(k)}}\hat\beta^{\text{ols}}(S^{(k)})||^2 = ||P_k^{\bot} y||^2  \) and  \( ||y - X_{S^{(k - 1)}}\hat\beta^{\text{ols}}(S^{(k - 1)})||^2 = ||P_{k - 1}^{\bot} y||^2   \). By adding variable \( i \) into \( S^{(k - 1)} \), the decrease \( ||P_{k - 1}^{\bot} y||^2 -||P_{k}^{\bot} y||^2 \) is equal to the squared norm of the projection of \( P_{k - 1}^{\bot} y \) onto the direction of \( P_{k - 1}^{\bot} x_i \), which is  \( \left(\frac{ \abs{x_{i }\hat r^{(k - 1)} }}{\norm{ P^{\bot}_{k - 1} x_{i}} } \right)^2\) where \( \hat r^{(k - 1)} = P_{k - 1}^{\bot}y \).  
\end{remark}

\paragraph{Part 1: Deriving the rejection region.}
Forward selection is a sequential method, and we first need to show it can be decomposed into bivariate sub-problems. The main reason is that whether some variable \( x_j \) is selected in the end only depends on \( (x_j,x_{j + 1}) \), and has nothing to do with other variables, or the number of steps \( k \). We still use \( (x_j,x_{j + 1}) \) to denote an arbitrary pair of correlated variables:

In terms of forward gain, whenever \( x_{j + 1} \) is not in \( \hat S^{(k - 1)} \) for arbitrary \( k \), \( P^{\bot}_{k - 1}x_j = x_j \), and \( \frac{\abs{x_j' \hat r^{(k - 1)}} }{\norm{ P^{\bot}_{k - 1} x_j}} = \abs{x'_j y} = |h_1|\sqrt{2\log(p)} \). When \( x_{j + 1} \) is already in \( \hat S^{(k - 1)} \) for arbitrary \( k \), \( P^{\bot}_{k - 1}x_j = x_j -\rho x_{j + 1} \), and  \( \frac{\abs{x_j' \hat r^{(k - 1)}} }{\norm{ P^{\bot}_{k - 1} x_j}} = \frac{\abs{(x'_j -\rho x'_{j + 1}) y}} {\sqrt{1 -\rho^2}} = \frac{\abs{h_1 -\rho h_2}}{\sqrt{1 -\rho^2}}\sqrt{2\log(p)} \). 

In terms of the entry rule (``\( i^* \gets \argmax_{i\notin S^{(k - 1)}}\abs{x_i'\hat r^{(k - 1)}} \)'' in Algorithm~\ref{alg:forward.selection}), since  \( x_j'\hat r^{(k - 1)} = x_j'P^\bot_{k - 1}y \), it is still \( x_j'y \) or \( (x_j'y -\rho x_{j + 1}y) \) depending on whether \( x_{j + 1}\in S^{(k - 1)} \). It has nothing to do with specific \( k \) or other variables than \( (x_j,x_{j + 1}) \).


As a result, Algorithm~\ref{alg:forward.selection} under the block-wise diagonal design can be viewed as many bivariate sub-problems going on simultaneously for each pair of correlated variables. In each ``while'' loop, variables from different pairs may be enrolled, but (i) the order of \( (x_j,x_{j + 1}) \) does not depend on \( k \) or other variables, and (ii) the bivariate problem must have terminated when the whole algorithm terminates, and the result of the bivariate problem does not depend on \( k \) or other variables.

Of course, we have assumed that Algorithm~\ref{alg:forward.selection} will always terminate in finite steps, which is true, because each bivariate sub-problem always terminates as we will see in the proof of Lemma~\ref{suppthm:forward.path}. 


Working on a bivariate problem, we can scale everything down by \( \sqrt{2\log(p)} \) and define \( t' = t/\sqrt{2\log(p)} \). Then the solution path can be described in Lemma~\ref{suppthm:forward.path}.

\begin{lem}[Solution path of Forward Selection]\label{suppthm:forward.path}
  Consider the bivariate problem of running Algorithm~\ref{alg:forward.selection} with \( y \) and \( (x_j,x_{j + 1}) \). Suppose \( h_1 > \abs{h_2} \geq 0 \), and \( \rho > 0 \).
\begin{itemize}
  \item When \( t' \geq  h_1\), none of \( (x_j,x_{j + 1}) \) will get selected when the algorithm ends. 
  \item If \( - h_1 < h_2 \leq (\rho - \sqrt{1 -\rho^2}) h_1 \), when \( t' <  h_1\),  both \( (x_j,x_{j + 1}) \) will get selected when the algorithm ends. 
  \item If \( (\rho - \sqrt{1 -\rho^2}) h_1 < h_2 < h_1  \), when \( t\in\left[\frac{\abs{h_2 -\rho h_1}}{\sqrt{1 -\rho^2}},h_1\right) \), only \( x_j \) is in the model before the algorithm ends. 
  \item If \( (\rho - \sqrt{1 -\rho^2}) h_1 < h_2 < h_1  \), when \( t <\frac{\abs{h_2 -\rho h_1}}{\sqrt{1 -\rho^2}} \), both \( (x_j,x_{j + 1}) \) will get selected when the algorithm ends. 
\end{itemize}
  \end{lem}
\begin{proof}[Proof of Lemma~\ref{suppthm:forward.path}]
  Note that we have required \( \rho > 0 \) to avoid unnecessary discussion.
Since \( h_1 >\abs{h_2} \), at any step \( k \) when neither of \( (x_j,x_{j + 1}) \) is in the model, we have \( |x_j' \hat r_{k - 1}| = |h_1| > |x_{j + 1}' \hat r_{k - 1}|=\abs{h_2}\). As a result, if \( t \geq h_1 \), then the algorithm will terminate without selecting either of \( (x_j,x_{j + 1}) \). If \( t > h_1 \), it will select \( x_j \) at some step and proceed to the next ``while'' loop.

After \( x_j \) has been selected, if \( h_2 \geq (\rho + \sqrt{1 -\rho^2})h_1 \) or \( h_2 \leq (\rho - \sqrt{1 -\rho^2})h_1 \), we have \( \frac{\abs{x_{j + 1}' \hat r^{(k - 1)}} }{\norm{ P^{\bot}_{k - 1} x_{j + 1}}} = \frac{\abs{h_2 -\rho h_1}}{\sqrt{1 -\rho^2}} \geq h_1 > t\), and \( x_{j + 1} \) will be selected at some later step. However, since \( \rho > 0 \), we have \( \rho + \sqrt{1 -\rho^2} > 1 \), so we can have \( \frac{\abs{x_{j + 1}' \hat r^{(k - 1)}} }{\norm{ P^{\bot}_{k - 1} x_{j + 1}}} = \frac{\abs{h_2 -\rho h_1}}{\sqrt{1 -\rho^2}} \geq h_1 \) only when \( h_2 \geq (\rho - \sqrt{1 -\rho^2})h_1 \). 

If \( (\rho - \sqrt{1 -\rho^2})h_1 < h_2 < h_1 \), we have \( \frac{\abs{h_2 -\rho h_1}}{\sqrt{1 -\rho^2}} <h_1 \). When \( t <\frac{\abs{h_2 -\rho h_1}}{\sqrt{1 -\rho^2}} \), both of the variables  will be selected; when \( t\in\left[\frac{\abs{h_2 -\rho h_1}}{\sqrt{1 -\rho^2}},h_1\right) \), only \( x_j \) will be selected. 
\end{proof}



We use Lemma~\ref{suppthm:forward.path} to write down the rejection region, as in Figure~\ref{fig:rej}, still for \( \rho > 0 \).
\begin{align} \label{suppeq:forward.rjRegion}
  {\cal R} &= \{(h_1,h_2): h_1-\rho h_2> t' \sqrt{1 -\rho^2},\, h_1> \frac{t' \sqrt{1 -\rho^2}}{1 -\rho} \}\cr
  &\;\; \cup \{(h_1, h_2): h_1 > t',\, h_1> h_2\} \cup \{(h_1, h_2):  h_2<- t,\,h_1-\rho h_2> t' \sqrt{1 -\rho^2} \}\cr
  &\;\; \cup \{(h_1,h_2): -h_1 +\rho h_2> t' \sqrt{1 -\rho^2},\, h_1 <- \frac{t' \sqrt{1 -\rho^2}}{1 -\rho} \}\cr
  &\;\; \cup \{(h_1, h_2): h_1 <- t',\, h_1 < h_2\} \cup \{(h_1, h_2):  h_2 > t,\, -h_1 +\rho h_2> t' \sqrt{1 -\rho^2} \}
  \end{align}  

  \paragraph{Part 2. Analyzing the Hamming error.} 
\begin{figure}[h!]
  \centering
  \includegraphics[width=0.85\textwidth]{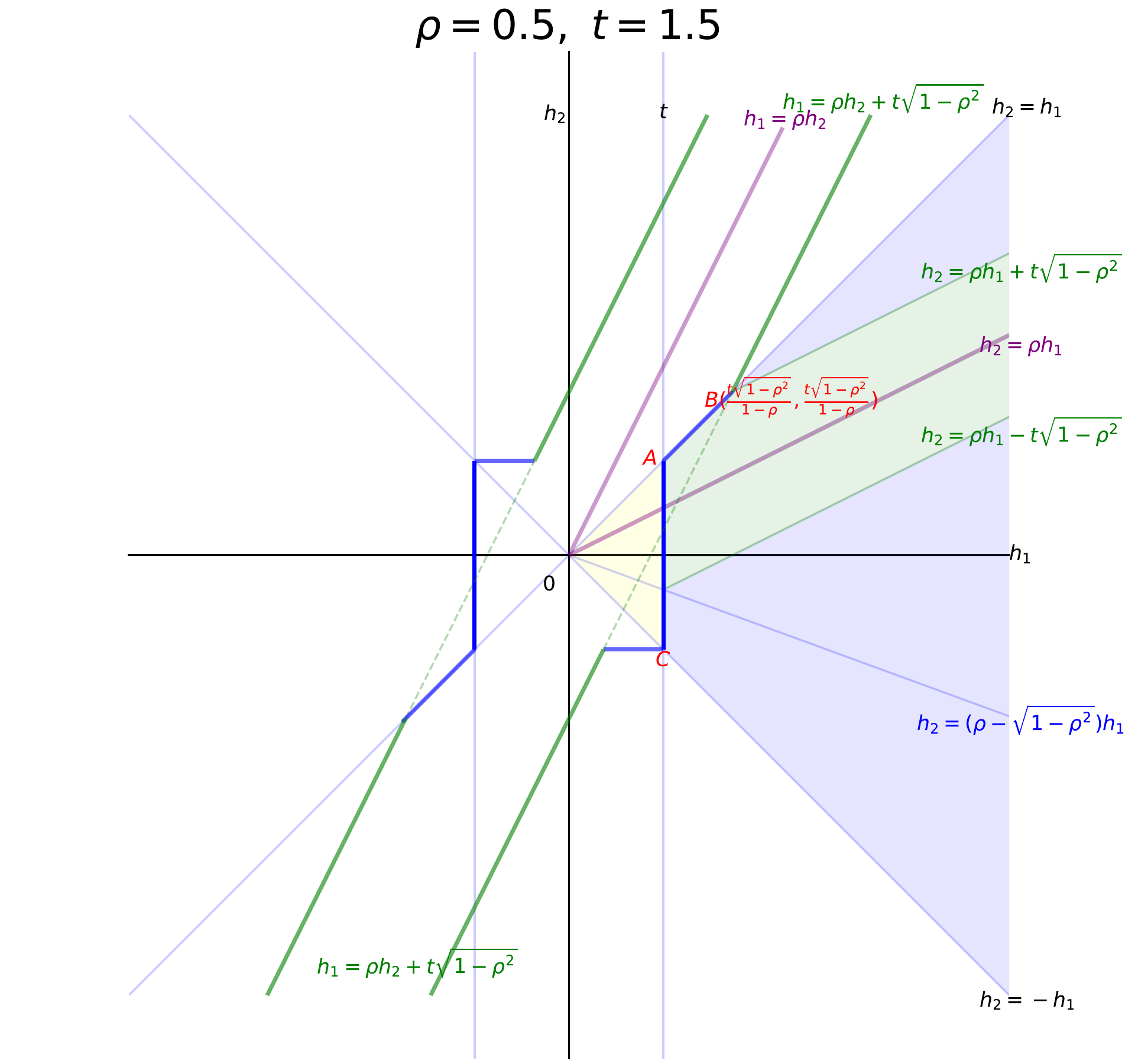}
  \caption{the rejection region of forward selection (\( \rho \geq 0 \)) }
  \label{fig:rej}
\end{figure}

\begin{theorem}\label{suppthm:hamm.forward}
  Suppose the conditions of Theorem~\ref{thm:forward} holds. Let \( t' = t /\sqrt{2\log(p)}\) and \( h_1 = x_j' y /\sqrt{2\log(p)} \), \( h_2 = x_{j + 1}' y /\sqrt{2\log(p)} \). As \( p\to\infty \), 
  \[
  \FP_p=L_p p^{1- \min\bigl\{ t'^2, \;\; \vt + f_1(\sqrt{r}, t')\bigr\}}, \qquad \FN_p = L_p p^{1-\min\bigl\{\vt + f_2(\sqrt{r}, t'),\;\; 2\vt + f_3(\sqrt{r}, t')\bigr\}}, 
  \]
  where (below, $d^2_{|\rho|}(u,v)$ is as in Definition~\ref{def:EllipsDistance}), 
  \begin{align*}
    f_1(\sqrt{r},t') & = \begin{cases} 
      (t' -|\rho| \sqrt{r})^2 & \text{ if } \sqrt{r} \leq \frac{t'}{1 +|\rho|}\\
      \frac{1}{1 -\rho^2}d_{\abs{\rho}}^2((t',t'), (|\rho| \sqrt{r},\sqrt{r})) & \text{ if } \frac{t'}{1 +|\rho|} <\sqrt{r} \leq \frac{2t'}{1+|\rho|} \\
      \min\left\{ \frac{1}{2}(1 -|\rho|) r,\ t'^2 \right\} & \text{ if } \sqrt{r} > \frac{2t'}{1+|\rho|}
    \end{cases}  \cr
   f_2(\sqrt{r}, t') &= \begin{cases} 
    \min \left\{ (\sqrt{r} - t')_ + ^2,\  \frac{1}{2}(1 -|\rho|) r  \right\} & \text{ if } \sqrt{r} \leq \frac{2t'}{\sqrt{1-\rho^2}} \\
    \min \left\{ (\sqrt{r} - t')_ + ^2,\ \frac{1}{1 -\rho^2}d^2_{|\rho|}(B,(\sqrt{r},|\rho| \sqrt{r}))  \right\}& \text{ if }\frac{2t'}{\sqrt{1-\rho^2}} < \sqrt{r} \leq \frac{t'\sqrt{1-\rho^2}}{|\rho|(1-|\rho|)} \\
    \left[ \sqrt{1 -\rho^2}\sqrt{r} - t' \right]^2 & \text{ if } \sqrt{r}>\frac{t' \sqrt{1-\rho^2}}{|\rho|(1-|\rho|)}
  \end{cases}\end{align*}
  The definition of \( f_3(\sqrt{r},t') \) depends on the sign of \( \rho \). When \( \rho > 0 \), 
  \begin{equation*}
    f_3(\sqrt{r}, t') = \left[ \sqrt{1 -\rho^2}\sqrt{r} - t  \right]^2
  \end{equation*}
    When \( \rho < 0 \), 
    \begin{align*}
      f_3(\sqrt{r},t') =&~ \min \left\{  \left[ \sqrt{1 -\rho^2}\sqrt{r} - t  \right]^2,\ d^2_{|\rho|}(C,((1 -|\rho|)\sqrt{r}, -(1 -|\rho|)\sqrt{r})) \right\} \\
      =&~ \min \left\{  \left[ \sqrt{1 -\rho^2}\sqrt{r} - t  \right]^2,\ \frac{2}{1 -|\rho|}\left[ (1 -|\rho|)\sqrt{r} - t
       \right]^2 \right\}
    \end{align*}
  \end{theorem}

The calculation of the elliptical distances are easy given Lemma~\ref{supplem:distance}. 

\paragraph{Part 3. Calculating the phase diagram.} The computation of the boundary between Alomst Full Recovery and No Recovery is almost the same for every method, so we omit the details and conclude that such boundary is \( r =\vt \). 

Then we set out to calculate the boundary between Alomst Full Recovery and Exact Recovery. As usual, we have four cases respectively for \( \rho > 0 \) and \( \rho < 0 \). However, unlike previous methods, forward selection is very easy, so we combine \( \rho > 0 \) and \( \rho < 0 \) in our discussion and use \( \abs{rho} \) all along.

When \( \rho > 0 \), the phase curve is \begin{equation*}
  \sqrt{r} = \max \left\{ 1 + \sqrt{1 -\vt}, \sqrt{\frac{2(1 -\vt)}{1 -\rho}}, \sqrt{\frac{1 - 2\vt}{1 -\rho^2}} + \frac{1}{\sqrt{1 -\rho^2}}\right\}.
\end{equation*}
When \( \rho < 0 \), the phase curve is \begin{equation*}
  \sqrt{r} =\max \left\{ 1 + \sqrt{1 -\vt}, \sqrt{\frac{2(1 -\vt)}{1 -|\rho|}},\sqrt{\frac{1 - 2\vt}{1 -\rho^2}} + \frac{1}{\sqrt{1 -\rho^2}}, \sqrt{\frac{1 - 2\vt}{2(1 -|\rho|)}} + \frac{1}{1 -|\rho|}\right\}
\end{equation*}

\textit{First}, if \( t'^2 = \vt + f_2(\sqrt{r},t') = 1 \) and \( \vt + f_1(\sqrt{r},t') \geq 1 \), \( 2\vt + f_3(\sqrt{r},t') \geq 1 \), we discuss the conditional expression of \( f_2(\sqrt{r},t') \):

When \( \sqrt{r} \leq \frac{2t'}{\sqrt{1 -\rho^2}} \) in \( f_3(\sqrt{r},t') \): Now \( \sqrt{r} =\max \left\{ 1 + \sqrt{1 -\vt},\, \sqrt{\frac{2(1 -\vt)}{1 -|\rho|}} \right\}. \)   We need the following requirements: First, \( \sqrt{r} \leq \frac{2t'}{\sqrt{1 -\rho^2}} = \frac{2}{\sqrt{1 -\rho^2}} \) itself. It is not restrictive, because \( 1 + \sqrt{1 -\vt} \leq 2 \leq \frac{2}{\sqrt{1 -\rho^2}} \) and \( \sqrt{\frac{2(1 -\vt)}{1 -|\rho|}} \leq \sqrt{\frac{2}{1 -|\rho|}} \leq \frac{2}{\sqrt{1 -\rho^2}} \). Second, \( \vt + f_2(\sqrt{r},t') \geq 1 \), which is still not restrictive. When \( \sqrt{r} \geq \frac{2t'}{1 +|\rho|} \), we know \( FP_2 = o(1) \) from \( \sqrt{r} \geq \sqrt{\frac{2(1 -\vt)}{1 -|\rho|}} \) and \( t' = 1 \).When \( 1 < \sqrt{r} <  \frac{2t'}{1 +|\rho|} \), we know    \( FP_2 = o(1) \) from \( d^2(B,(\sqrt{r},|\rho| \sqrt{r})) \geq  \frac{1}{2}(1 +|\rho|)(1 -|\rho|)^2 r \). Third, \( 2\vt + f_3(\sqrt{r},t') \geq 1 \), which requires \( \sqrt{r} \geq \sqrt{\frac{1 - 2\vt}{1 -\rho^2}} + \frac{1}{\sqrt{1 -\rho^2}} \) when the correlation is positive, and \( \sqrt{r} \geq \max \left\{\sqrt{\frac{1 - 2\vt}{1 -\rho^2}} + \frac{1}{\sqrt{1 -\rho^2}},\,\sqrt{\frac{1 - 2\vt}{2(1 -|\rho|)}} + \frac{1}{1 -|\rho|}  \right\}. \) when the correlation is negative.

When \( \frac{2t'}{\sqrt{1-\rho^2}} < \sqrt{r} \leq \frac{t' \sqrt{1 -\rho^2}}{|\rho|(1-|\rho|)} \) in \( f_3(\sqrt{r},t') \): If \( (1 -\vt)(1 -\rho^2) =  d^2(B,(\sqrt{r},|\rho| \sqrt{r}))\), then since \(d^2(B,(\sqrt{r},|\rho| \sqrt{r})) \geq  \frac{1}{2}(1 +|\rho|)(1 -|\rho|)^2 r \), we have \( \sqrt{r} \leq \sqrt{\frac{2(1 -\vt)}{1 -|\rho|}}.  \) 
However, we also need \( \sqrt{r} >\frac{2}{\sqrt{1-\rho^2}} \), which gives a contradiction. This case does not exist. If \( \sqrt{r} = 1 + \sqrt{1 -\vt} \), then it also contradicts \( \sqrt{r} >\frac{2}{\sqrt{1-\rho^2}} \). 

When \( \sqrt{r}>\frac{t' \sqrt{1 -\rho^2}}{|\rho|(1-|\rho|)} \) in \( f_3(\sqrt{r},t') \): \( \sqrt{r} = \sqrt{\frac{1 -\vt}{1 -\rho^2}} + \frac{1}{\sqrt{1 -\rho^2}} \). It cannot meet the requirement \( \sqrt{r} > \frac{t' \sqrt{1 -\rho^2}}{|\rho|(1-|\rho|)}\), so this case does not exist.

To sum up, the \textit{first} case gives
\begin{equation*}
  \sqrt{r} =\max \left\{ 1 + \sqrt{1 -\vt},\, \sqrt{\frac{2(1 -\vt)}{1 -|\rho|}} \right\}.
\end{equation*}
which exists in the region:
\begin{itemize}
  \item \( \left(\frac{1}{2},1\right]\cup \{ \vt \leq \frac{1}{2}: \sqrt{r} \geq \sqrt{\frac{1 - 2\vt}{1 -\rho^2}} + \frac{1}{\sqrt{1 -\rho^2}}\}  \) for \( \rho > 0 \) .
  \item  \( \left(\frac{1}{2},1\right]\cup \left\{ \vt \leq \frac{1}{2}: \sqrt{r} \geq \max \left\{\sqrt{\frac{1 - 2\vt}{1 -\rho^2}} + \frac{1}{\sqrt{1 -\rho^2}},\,\sqrt{\frac{1 - 2\vt}{2(1 -|\rho|)}} + \frac{1}{1 -|\rho|}  \right\}\right\} \) for \( \rho < 0 \) .
\end{itemize}

\textit{Second}, if \( \vt + f_1(\sqrt{r},t') = \vt + f_2(\sqrt{r},t') = 1\) and \( t' \geq 1 \), \( 2\vt + f_3(\sqrt{r},t') \geq 1 \), then we will get nothing in this case. To see this, we first list a few requirements: 
\begin{itemize}
  \item We know that \( t' \geq 1 \);
  \item We know from \( \vt + f_2(\sqrt{r},t') = 1 \) that \( \sqrt{r} \geq \max \left\{ t' + \sqrt{1 -\vt},\, \sqrt{\frac{2(1 -\vt)}{1 -|\rho|}} \right\} \);
  \item We know from \( 2\vt + f_3(\sqrt{r},t') \geq 1 \) that, when \( \vt \leq \frac{1}{2} \), we need \begin{itemize}
    \item \( \sqrt{r} \geq \sqrt{\frac{1 - 2\vt}{1 -\rho^2}} + \frac{t'}{\sqrt{1 -\rho^2}} \) for positive correlation;
    \item \( \sqrt{r} \geq \max \left\{\sqrt{\frac{1 - 2\vt}{1 -\rho^2}} + \frac{t'}{\sqrt{1 -\rho^2}},\,\sqrt{\frac{1 - 2\vt}{2(1 -|\rho|)}} + \frac{t'}{1 -|\rho|}  \right\} \) for negative correlation.
  \end{itemize}
\end{itemize}

From these requirements, even if this case does admit some curve, it can only be higher than the one in the previous \textit{first} case, and exist in a smaller region. 

\textit{Third}, if \( t'^2 = 2\vt + f_3(\sqrt{r},t') = 1  \) and \( \vt + f_1(\sqrt{r},t') \geq  1 \), \( \vt + f_2(\sqrt{r},t') \geq  1 \), then:

When the correlation is positive, we have only one possible curve \( \sqrt{r} =\sqrt{\frac{1 - 2\vt}{1 -\rho^2}} + \frac{1}{\sqrt{1 -\rho^2}}. \) From \( \vt + f_2(\sqrt{r},t') \geq  1 \), we get the requirement \( \sqrt{r} \geq \max \left\{ 1 + \sqrt{1 -\vt},\, \sqrt{\frac{2(1 -\vt)}{1 -\rho}} \right\} \).  For \( \vt + f_1(\sqrt{r},t') \geq  1 \), since we already have \( \sqrt{r} \geq \sqrt{\frac{2(1 -\vt)}{1 -\rho}}  \), we know \( \vt + f_1(\sqrt{r},t') \geq  1 \) always holds. 

When the correlation is negative, we have \( \sqrt{r} =\max \left\{\sqrt{\frac{1 - 2\vt}{1 -\rho^2}} + \frac{1}{\sqrt{1 -\rho^2}},\,\sqrt{\frac{1 - 2\vt}{2(1 -|\rho|)}} + \frac{1}{1 -|\rho|}  \right\}. \) From \( \vt + f_2(\sqrt{r},t') \geq  1 \), we get the requirement \( \sqrt{r} \geq \max \left\{ 1 + \sqrt{1 -\vt},\, \sqrt{\frac{2(1 -\vt)}{1 -|\rho|}} \right\} \). For \( \vt + f_1(\sqrt{r},t') \geq  1 \), since we already have \( \sqrt{r} \geq \sqrt{\frac{2(1 -\vt)}{1 -|\rho|}}  \), we know \( FP_2 = o(1) \) always holds. 

\textit{Fourth}, if \( \vt + f_1(\sqrt{r},t') = 2\vt + f_3(\sqrt{r},t') = 1  \) and \( t' \geq  1 \), \( \vt + f_2(\sqrt{r},t') \geq  1 \), then we will get nothing from this case. To see this, we still list a few requirement:
\begin{itemize}
  \item We know that \( t' \geq 1 \). 
  \item From \( \vt + f_2(\sqrt{r},t') \geq  1 \), we know that \(  \sqrt{r} \geq \max \left\{ t' + \sqrt{1 -\vt},\, \sqrt{\frac{2(1 -\vt)}{1 -|\rho|}} \right\} \).
  \item From \( 2\vt + f_3(\sqrt{r},t') = 1 \), we know that \( \sqrt{r} = \sqrt{\frac{1 - 2\vt}{1 -\rho^2}} + \frac{t'}{\sqrt{1 -\rho^2}}\) for \( \rho > 0 \) and   \( \sqrt{r} =\max \left\{\sqrt{\frac{1 - 2\vt}{1 -\rho^2}} + \frac{t'}{\sqrt{1 -\rho^2}},\,\sqrt{\frac{1 - 2\vt}{2(1 -|\rho|)}} + \frac{t'}{1 -|\rho|}  \right\}\)  for \( \rho < 0 \).
\end{itemize}

Even if this case admits any curve, that curve would be above the curve in the previous \textit{third} case, and exist within a smaller region of \( \vt \).

To sum up, we have the following results:
\begin{itemize}
    \item Phae diagram when the correlation is positive:
    \begin{equation*}
      \sqrt{r} = \max \left\{ 1 + \sqrt{1 -\vt}, \sqrt{\frac{2(1 -\vt)}{1 -\rho}}, \sqrt{\frac{1 - 2\vt}{1 -\rho^2}} + \frac{1}{\sqrt{1 -\rho^2}}\right\}.
    \end{equation*}
    \item Phae diagram when the correlation is negative:
    \begin{equation*}
      \sqrt{r} =\max \left\{ 1 + \sqrt{1 -\vt}, \sqrt{\frac{2(1 -\vt)}{1 -|\rho|}},\sqrt{\frac{1 - 2\vt}{1 -\rho^2}} + \frac{1}{\sqrt{1 -\rho^2}}, \sqrt{\frac{1 - 2\vt}{2(1 -|\rho|)}} + \frac{1}{1 -|\rho|}\right\}
    \end{equation*}
\end{itemize}

\section{Proof of Theorem~\ref{thm:forward-backward} (Forward Backward Selection)}\label{suppsec:forward.backward}

The proof for Forward Selection still has three tasks: (a) deriving the rejection region, (b) obtaining the rate of convergence of \( \mathbb{E}[H(\hat\beta,\beta)] \), and (c) calculating the phase diagram. However, as we will see later, forward backward selection has six cases, each of which has a different shape of the rejection region. After deriving the rejection region, we consider the  \( \mathbb{E}[H(\hat\beta,\beta)] \) and phase curves of the six cases one by one, and summarise the results at the end. 

Defore deriving the rejection region in the first part, we need some clarification about the definition of the forward backward selection we have investigated. 
More precisely, we have simplified the backward step into one thresholding step after the forward selection algorithm, so it is more precisely ``thresholded forward selection''. 

The reason why we have not used a sequential algorithm with alternating forward and backward steps, like FoBa defined in \citet{zhang2011adaptive}, is not compuational simplicity, but to avoid degeneration. We explain briefly why any sequential algorithm with alternating forward and backward steps will either have nonfunctional backward steps, or be unable to terminate at a finite step.

To see this, we review the setting of Lemma~\ref{suppthm:forward.path} about the solution path of forward selection, in which \( h_1 >\abs{h_2} \) and we only consider a bivariate problem. Using the same argument, some version of FoBa can also be decomposed into bivariate subproblems, and it is equivalent to running the algorithm only on \( y \) and \( (x_j,x_{j + 1}) \). 

In brief, in such a bivariate problem with \( h_1 >\abs{h_2} \), if either or both of \( (x_j,x_{j + 1}) \) ever get selected and then deleted at some backward step, then they will be selected back again because they still meet the requirements for a variable to get enrolled. When \( h_1 >\abs{h_2} \), the case of deleting \( x_j \) while leaving \( x_{j + 1} \) still in the model cannot happen, because no deletion rule based on \( (x_j'y,x_{j + 1}'y) \) can delete \( x_j \) without touching \( x_{j +1} \). As a result, the algorithm cannot terminate at a finite step.

If the algorithm termininates at a finite step, then the backward step much have not deleted any of \( (x_j,x_{j + 1}) \), and such algorithm performs the same as forward selection. 

We have explained the degeneration of Foba \citep{zhang2011adaptive}, but we still want to implement some kind of backward step additional to forward selection, because the problem with forward selection is inability to correct the mistakes made in the early steps. Thus it is natural to use one thresholding step at the end.

\paragraph{Part 1: Deriving the rejection region.}
We first work on the solution path, and then compute the rejection region. The forward selection part has been discussed before, and we recall the results in Lemma~\ref{suppthm:forward.path} (re-iterated in an equivalent way):
\begin{enumerate}
  \item When \( t \geq  h_1 \), neither is selected.
  \item When \( t < h_1 \), and \( \rho h_1 - t \sqrt{1 -\rho^2} \leq h_2 \leq \rho h_1 + t  \sqrt{1 -\rho^2}\), only \( x_j \) is selected.
  \item When \( t < h_1 \), and \( \begin{cases} 
  \text{either }h_2 >  \rho h_1 + t  \sqrt{1 -\rho^2} \\
  \text{or }- h_1 < h_2 < \rho h_1 - t  \sqrt{1 -\rho^2}
  \end{cases}  \), both \( x_j \) and \( x_{j + 1} \) are selected.   
\end{enumerate}

Now this is followed by a thresholding step. Before using \( v \) to threshold the results, we note that:
\begin{enumerate}
  \item When \( x_{j + 1} \) is not selected, \( \hat\beta_j = h_1 \);
  \item When both \( (x_j,x_{j + 1}) \) are selected, \( \hat\beta_j = \frac{h_1 -\rho h_2}{1 -\rho^2} \) and \( \hat\beta_{j + 1} =\frac{h_2 -\rho h_1}{1 -\rho^2} \).
\end{enumerate}

When the thresholding is performed, we can naturally describe the solution path of thresholded forward selection as:
\begin{lemma}\label{supplem:sol.path.thres.forward}
  Consider the bivariate problem of running Algorithm~\ref{alg:forward.selection} with \( y \) and \( (x_j,x_{j + 1}) \), followed by thresholding \( (\hat\beta_j,\hat\beta_{j + 1}) \) with \( v \). Define \( h_1 = x_j' y /\sqrt{2\log(p)} \), \( h_2 = x_{j + 1}' y /\sqrt{2\log(p)} \) and \( t' = t/\sqrt{2\log(p)} \), \( v' = v/\sqrt{2\log(p)} \). Suppose \( h_1 > \abs{h_2} \geq 0 \), and \( \rho > 0 \).  Then 
  \begin{itemize}
    \item When \( t' \geq h_1 \), neither is selected. 
    \item When \( t' < h_1 \) and \( \rho h_1 - t'  \sqrt{1 -\rho^2} \leq h_2 \leq \rho h_1 + t'  \sqrt{1 -\rho^2}\), 
    \begin{itemize}
      \item If \(  v' \geq h_1 \), neither is selected. 
      \item If \(  v' < h_1 \), only \( x_j \) is selected.  
    \end{itemize}
    \item When \( t' < h_1 \), and \( \begin{cases} 
      \text{either }h_2 >  \rho h_1 + t'  \sqrt{1 -\rho^2} \\
      \text{or }- h_1 < h_2 < \rho h_1 - t'  \sqrt{1 -\rho^2}
      \end{cases}  \),
      \begin{itemize}
        \item If \( h_2 >  \rho h_1 + t'   \sqrt{1 -\rho^2} \), and \( h_2 -\rho h_1 >  v'(1 -\rho^2) \), both \( (x_j,x_{j + 1}) \) are selected.
        \item If \( h_2 >  \rho h_1 + t'   \sqrt{1 -\rho^2} \), and \(  h_2 -\rho h_1 \leq   v'(1 -\rho^2) < h_1 -\rho h_2\), only \( x_j \) is selected.
        \item  If \( h_2 >  \rho h_1 + t'   \sqrt{1 -\rho^2} \), and \( h_1 -\rho h_2 \leq  v'(1 -\rho^2) \), neither is selected. 
        \item If \( h_2 <  \rho h_1 - t'   \sqrt{1 -\rho^2} \), and  \(  \rho h_1 - h_2 >  v'(1 -\rho^2) \), both \( (x_j,x_{j + 1}) \) are selected.
        \item If \( h_2 <  \rho h_1 - t'   \sqrt{1 -\rho^2} \), and \( \rho h_1 - h_2 \leq   v'(1 -\rho^2) < h_1 -\rho h_2\), only \( x_j \) is selected.
        \item  If \( h_2 <  \rho h_1 - t' \sqrt{1 -\rho^2} \), and \( h_1 -\rho h_2 \leq  v'(1 -\rho^2) \), neither is selected.
      \end{itemize}
  \end{itemize}
\end{lemma}

The rejection region can be complicated, and it has many cases visually. See Figure~\ref{fig:rjRegion.thres.forward}.
\begin{align} \label{suppeq:foba.rjRegion}
  {\cal R} &= \{(h_1,h_2): h_1-\rho h_2> \max\{t' \sqrt{1 -\rho^2},v'(1 -\rho^2)\},\, h_2 >\rho h_1 + t\sqrt{1 -\rho^2} \}\cr
  &\;\; \cup \{(h_1,h_2):  h_1> \max\{t',v'\},\, h_2 \leq \rho h_1 + t\sqrt{1 -\rho^2},\,h_2 \geq \rho h_1 - t\sqrt{1 -\rho^2} \}\cr
  &\;\; \cup \{(h_1, h_2): h_1> h_2,\,h_1 >\max\{t',v'\},\, h_2 >\rho h_1 - t\sqrt{1 -\rho^2}\} \cr
  &\;\; \cup \{(h_1, h_2): h_1 > t',\, h_2 \leq \rho h_1 - t\sqrt{1 -\rho^2},\, h_1 -\rho h_2 > v'(1 -\rho^2) \} \cr
  &\;\; \cup \{(h_1, h_2):  h_2<- t,\,h_1-\rho h_2> \max\{t' \sqrt{1 -\rho^2},v'(1 -\rho^2)\} \}\cr
  &\;\; \cup \{(h_1,h_2): -h_1 +\rho h_2> \max\{t' \sqrt{1 -\rho^2},v'(1 -\rho^2)\},\, h_2 <\rho h_1 - t\sqrt{1 -\rho^2} \}\cr
  &\;\; \cup \{(h_1,h_2):  h_1 <- \max\{t',v'\},\, h_2 \geq \rho h_1 - t\sqrt{1 -\rho^2},\,h_2 \leq \rho h_1 + t\sqrt{1 -\rho^2} \}\cr
  &\;\; \cup \{(h_1, h_2): h_1 < h_2,\,h_1 >\max\{t',v'\},\, h_2 <\rho h_1 + t\sqrt{1 -\rho^2}\} \cr
  &\;\; \cup \{(h_1, h_2): h_1 < -t',\, h_2 \geq \rho h_1 + t\sqrt{1 -\rho^2},\, -h_1 +\rho h_2 > v'(1 -\rho^2) \} \cr
  &\;\; \cup \{(h_1, h_2):  h_2 > t,\, -h_1 +\rho h_2> \max\{t' \sqrt{1 -\rho^2},v'(1 -\rho^2)\} \}
  \end{align}  

\begin{figure}
  \centering
  \begin{subfigure}[b]{0.495\textwidth}
    \centering
    \includegraphics[width=\textwidth]{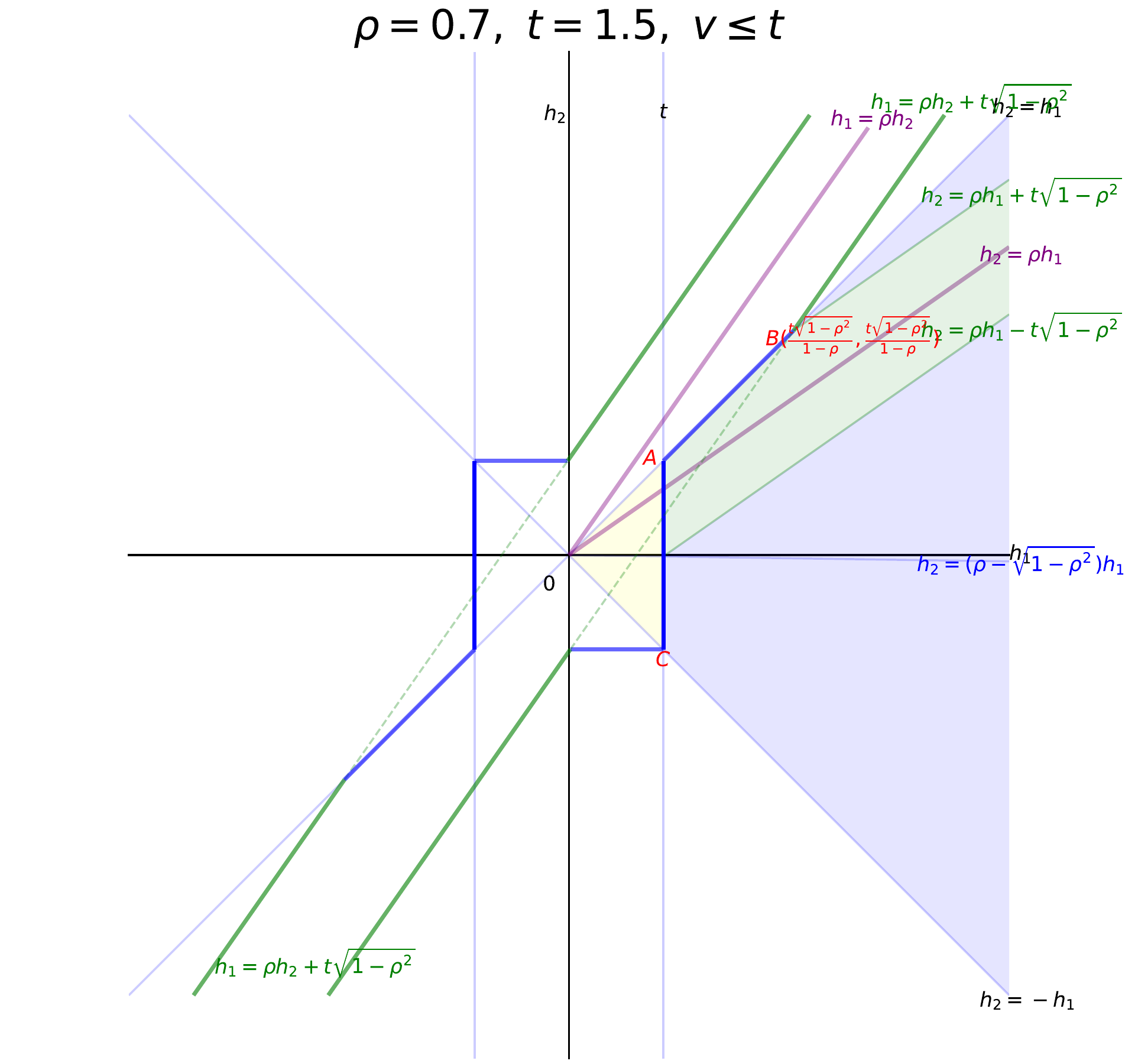}
    \caption[]%
    {{\small When \( v' \leq t' \)}}    
    \label{subfig:small.v}
\end{subfigure}
\hfill
  \begin{subfigure}[b]{0.495\textwidth}  
      \centering 
      \includegraphics[width=\textwidth]{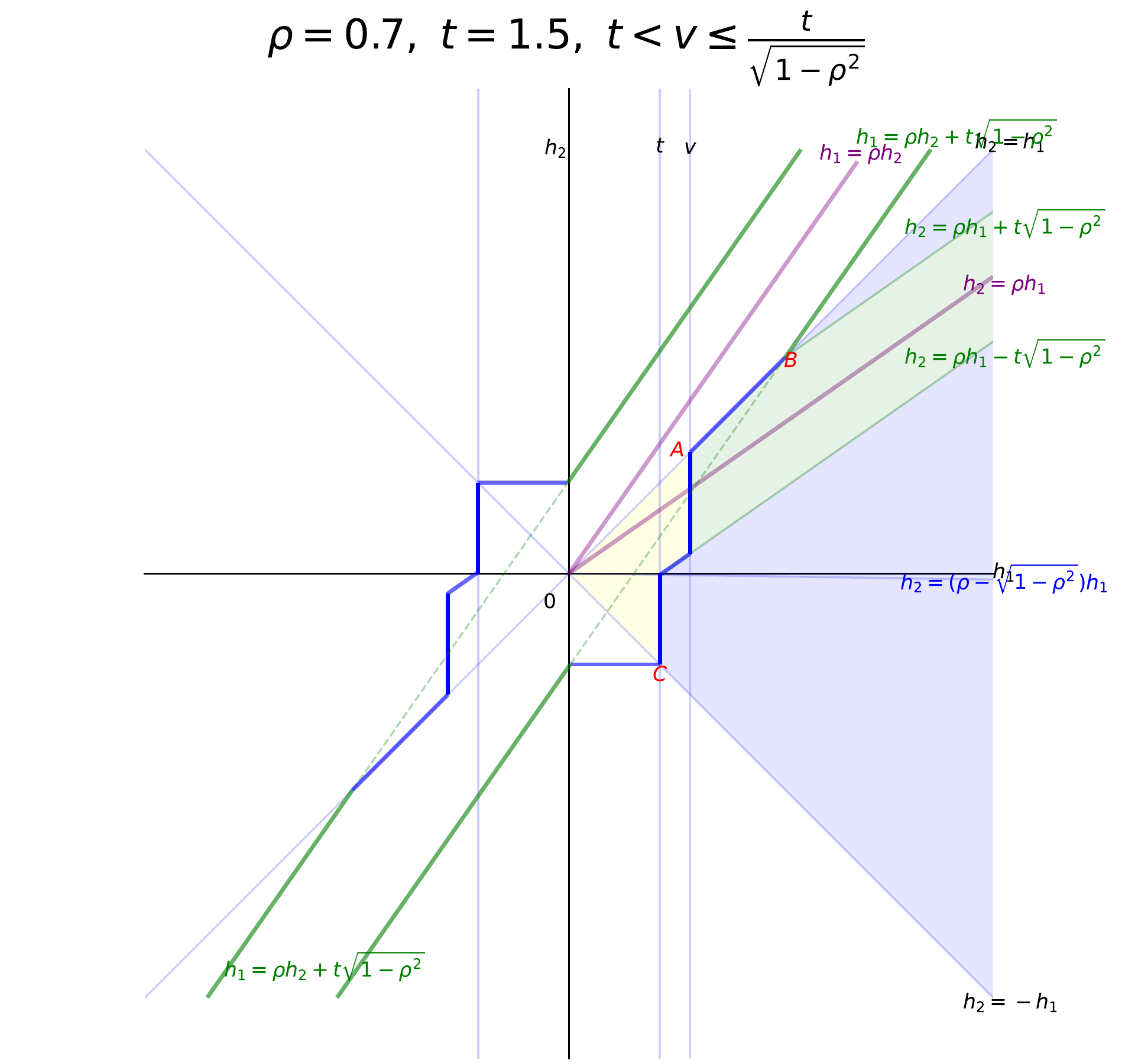}
      \caption[]%
      {{\small When \( t' <v' \leq  \frac{t}{\sqrt{1 -\rho^2}} \)}}    
      \label{subfig:middle.v}
  \end{subfigure}
  \vskip\baselineskip
  \begin{subfigure}[b]{0.495\textwidth}  
    \centering 
    \includegraphics[width=\textwidth]{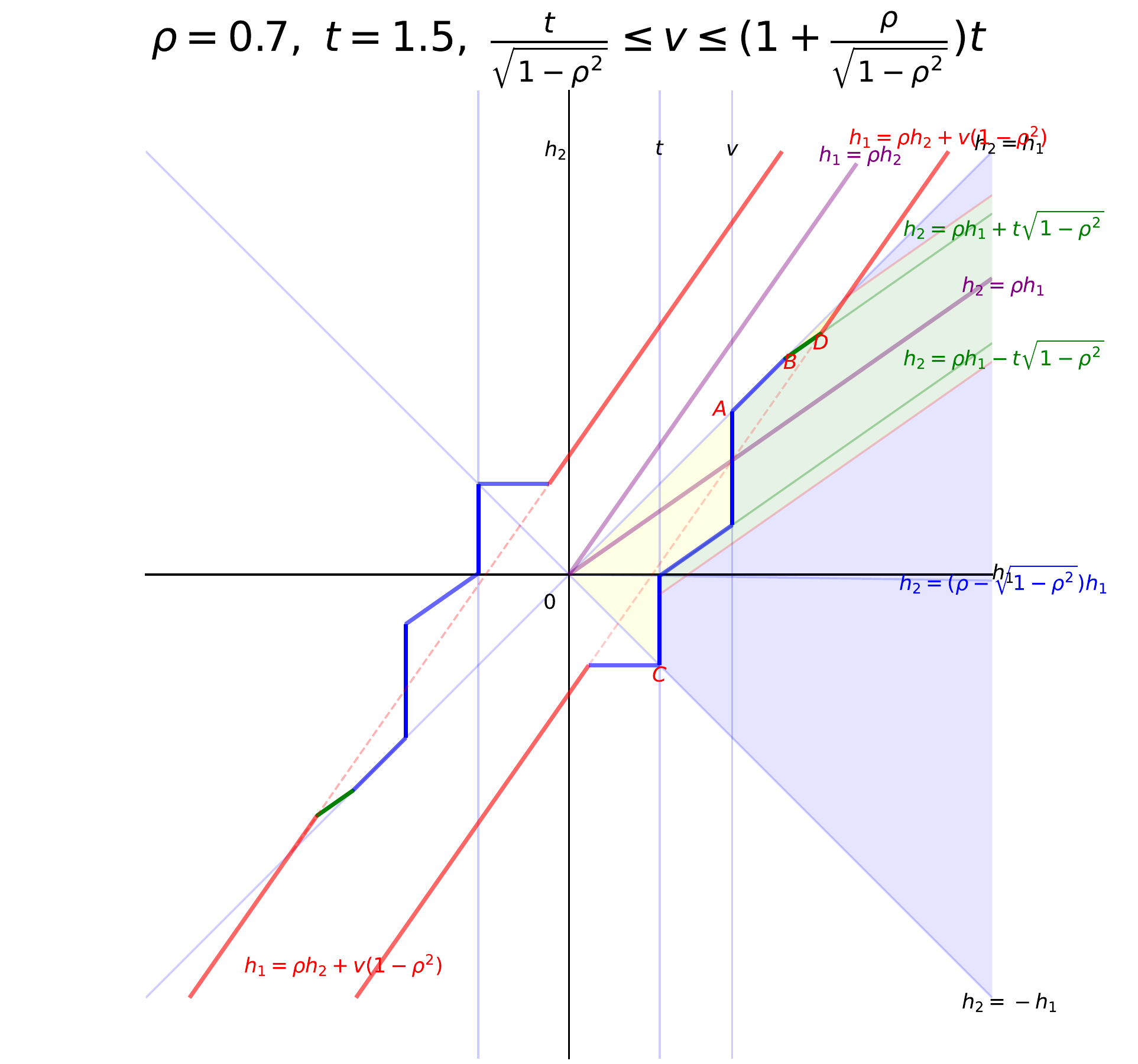}
    \caption[]%
    {{\small When \( \frac{t'}{\sqrt{1 -\rho^2}} < v' \leq t'\left( 1 +\frac{\rho}{\sqrt{1 -\rho^2}} \right)  \)}}    
    \label{subfig:large.v}
\end{subfigure}
\hfill
\begin{subfigure}[b]{0.495\textwidth}  
  \centering 
  \includegraphics[width=\textwidth]{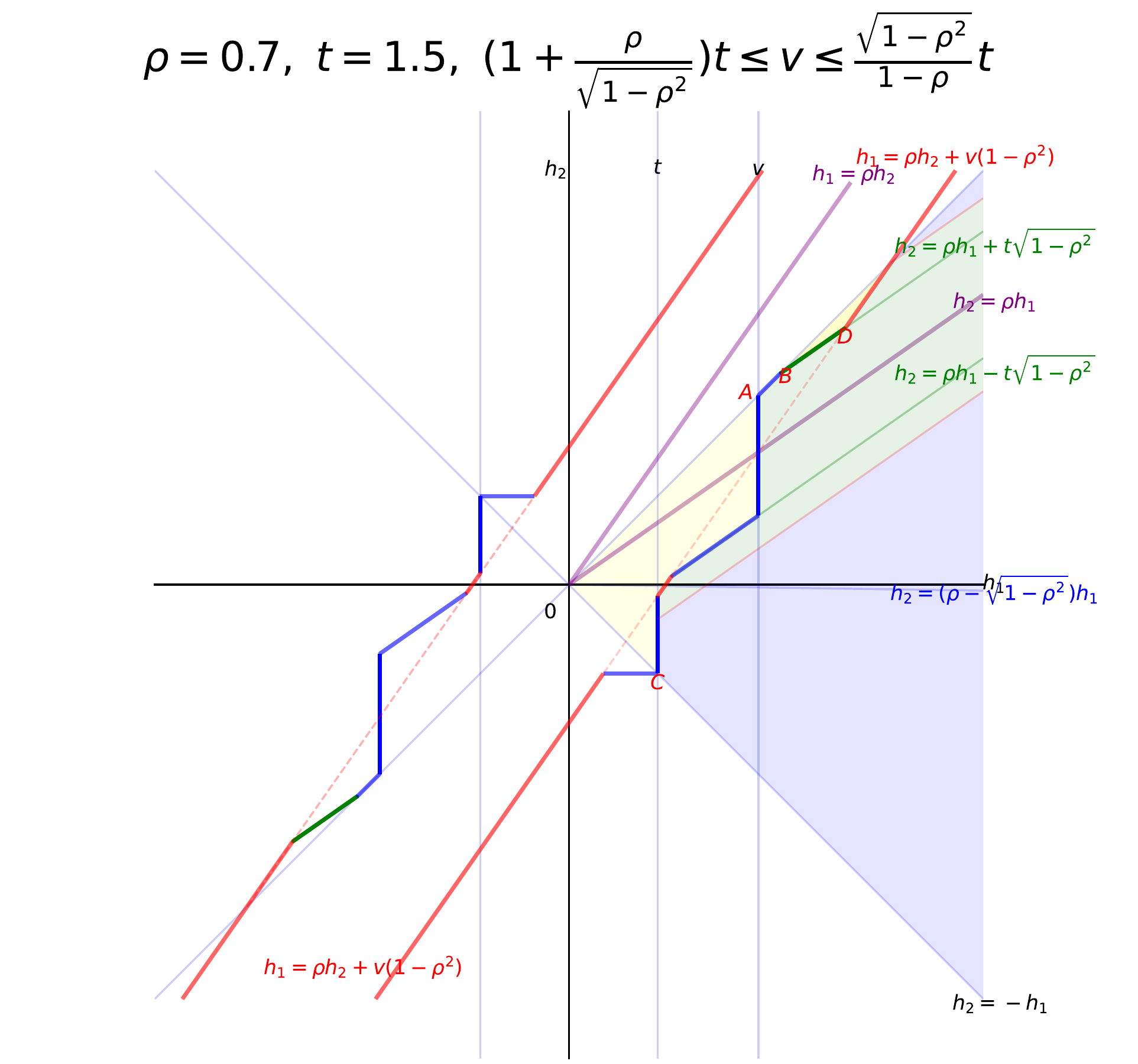}
  \caption[]%
  {{\small When \(t'\left( 1 +\frac{\rho}{\sqrt{1 -\rho^2}} \right)\leq v' \leq\frac{t' \sqrt{1 -\rho^2}}{1 -\rho} \)}}    
  \label{subfig:large.plus.v}
\end{subfigure}
\vskip\baselineskip
\begin{subfigure}[b]{0.495\textwidth}  
  \centering 
  \includegraphics[width=\textwidth]{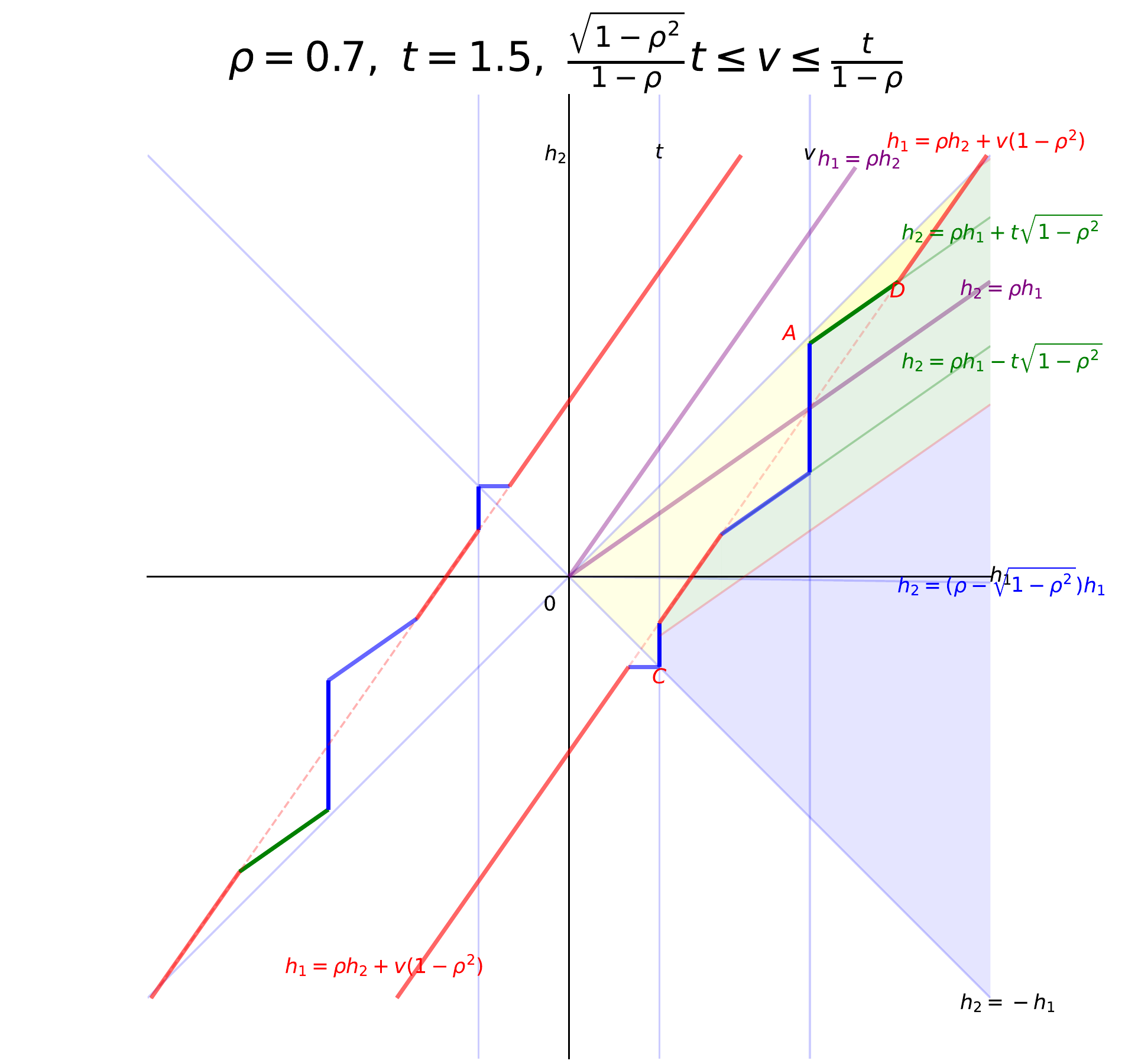}
  \caption[]%
  {{\small When \(\frac{t' \sqrt{1 -\rho^2}}{1 -\rho}\leq v' \leq \frac{t'}{1 -\rho} \)}}    
  \label{subfig:large.plus.plus.v}
\end{subfigure}
\hfill
\begin{subfigure}[b]{0.495\textwidth}  
  \centering 
  \includegraphics[width=\textwidth]{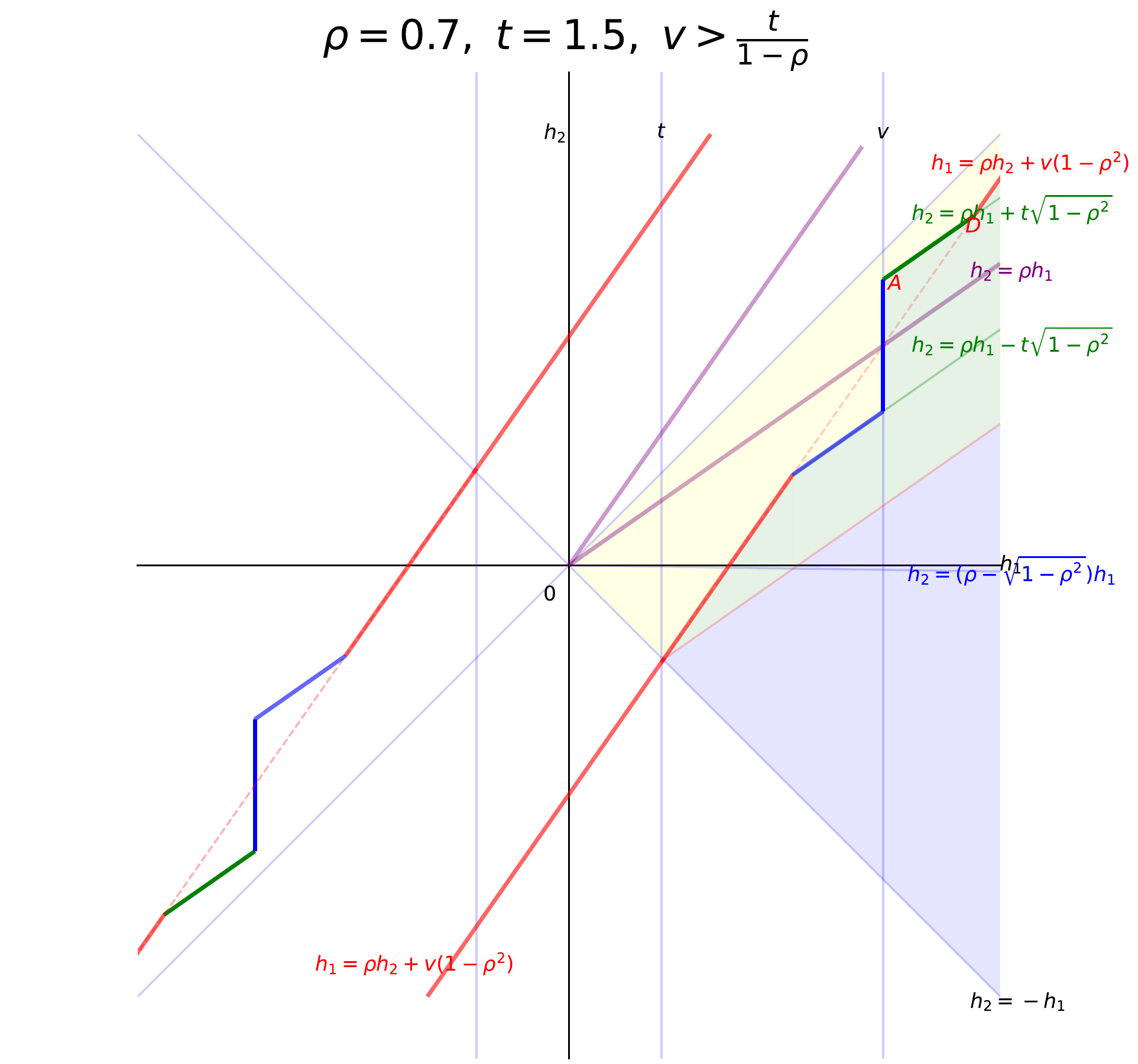}
  \caption[]%
  {{\small When \( v' > \frac{t'}{1 -\rho} \) }}    
  \label{subfig:large.plus.plus.plus.v}
\end{subfigure}
\caption{The rejection region of thresholded forward selection has many cases \( (\rho \geq  0) \).}
\label{fig:rjRegion.thres.forward}
\end{figure}


Due to the many cases of thresholded forward selection, we structure the rest of the proof in a different way: We discuss the six cases shown in Figure~\ref{fig:rjRegion.thres.forward} in the next six parts, and summarise the results for \( \rho \geq 0 \) and \( \rho < 0 \) respectively, at the end of the proof. In other words, each of the six cases has its phase curves, and we take the minimum of all the curves to be the final phase curve. 

\paragraph{Case 1: When \( v' \leq t'\).} 
From the rejection region defined in Equation~\eqref{suppeq:foba.rjRegion}, and Figure~\ref{subfig:small.v}, we know that thresholding does not have any effects in this case. Everything can be copied from forward selection:

The curve between Almost Full Recovery and No Recovery is \( r =\vt \).  The curve between Almost Full Recovery and Exact Recovery is: When \( \rho \geq  0 \),
\begin{equation}\label{eq:boundary.small.v.positive}
  \sqrt{r} =\max \left\{ 1 + \sqrt{1 -\vt},\sqrt{\frac{2(1 -\vt)}{1 -\rho}},\sqrt{\frac{1 - 2\vt}{1 -\rho^2}} + \frac{1}{\sqrt{1 -\rho^2}} \right\}.
\end{equation}
When \( \rho < 0 \),
\begin{equation}\label{eq:boundary.small.v.negative}
  \sqrt{r} =\max \left\{ 1 + \sqrt{1 -\vt},\sqrt{\frac{2(1 -\vt)}{1 - |\rho|}},\sqrt{\frac{1 - 2\vt}{1 - \rho^2}} + \frac{1}{\sqrt{1 - \rho^2}},\sqrt{\frac{1 - 2\vt}{2(1 - |\rho|)}} + \frac{1}{1 - |\rho|} \right\}.
\end{equation}

\paragraph{Case 2: When \( t' \leq v' \leq \frac{t'}{\sqrt{1 -\rho^2}} \).}

\begin{theorem}[The Hamming error rate When \( t' \leq v' \leq \frac{t'}{\sqrt{1 -\rho^2}} \) ]\label{suppthm:hamm.middle.v}
  Suppose the conditions of Theorem~\ref{thm:forward-backward} holds. Let  \( h_1 = x_j' y /\sqrt{2\log(p)} \), \( h_2 = x_{j + 1}' y /\sqrt{2\log(p)} \) and \( v' = v/\sqrt{2\log(p)} \), \( t' = t /\sqrt{2\log(p)}\). We require \( t' \leq v' \leq \frac{t'}{\sqrt{1 -\rho^2}}\). As \( p\to\infty \), 
  \[
  \FP_p=L_p p^{1- \min\bigl\{ \min\{v'^2,2t'^2\}, \;\; \vt + f_1(\sqrt{r}, t',v')\bigr\}}, \qquad \FN_p = L_p p^{1-\min\bigl\{\vt + f_2(\sqrt{r}, t',v'),\;\; 2\vt + f_3(\sqrt{r}, t',v')\bigr\}}, 
  \]
  where (below, $d^2_{|\rho|}(u,v)$ is as in Definition~\ref{def:EllipsDistance}), 
  \begin{align*}
    f_1(\sqrt{r},t',v') & = \begin{cases} 
      (v' - |\rho| \sqrt{r})^2 & \text{ if } \sqrt{r} \leq \frac{v'}{1 +|\rho|}\\
      \frac{1}{1 -\rho^2}d^2_{|\rho|}((v',v'), (|\rho| \sqrt{r},\sqrt{r})) & \text{ if } \frac{v'}{1 +|\rho|} <\sqrt{r} \leq \frac{2v'}{1+|\rho|} \\
      \min\left\{ \frac{1}{2}(1 -|\rho|) r,\ t'^2 \right\} & \text{ if } \sqrt{r} > \frac{2v'}{1+|\rho|}
    \end{cases}  \cr
   f_2(\sqrt{r}, t',v') &= \begin{cases} 
    \min \left\{ (\sqrt{r} - t')_ + ^2,\  \frac{1}{2}(1 -|\rho|) r  \right\} & \text{ if } \sqrt{r} \leq \frac{2t'}{\sqrt{1-\rho^2}} \\
    \min \left\{ (\sqrt{r} - t')_ + ^2,\ \frac{1}{1 -\rho^2}d^2_{|\rho|}(B,(\sqrt{r},|\rho| \sqrt{r}))  \right\}& \text{ if }\frac{2t'}{\sqrt{1-\rho^2}} < \sqrt{r} \leq \frac{t'\sqrt{1-\rho^2}}{|\rho|(1-|\rho|)} \\
    \left[ \sqrt{1 -\rho^2}\sqrt{r} - t' \right]^2 & \text{ if } \sqrt{r}>\frac{t' \sqrt{1-\rho^2}}{|\rho|(1-|\rho|)}
  \end{cases}\end{align*}
  The definition of \( f_3(\sqrt{r},t') \) depends on the sign of \( \rho \). When \( \rho > 0 \), 
  \begin{equation*}
    f_3(\sqrt{r}, t') = \left[ \sqrt{1 -\rho^2}\sqrt{r} - t  \right]^2
  \end{equation*}
    When \( \rho < 0 \), 
    \begin{align*}
      f_3(\sqrt{r},t') =&~ \min \left\{  \left[ \sqrt{1 -\rho^2}\sqrt{r} - t  \right]^2,\ d^2_{|\rho|}(C,((1 -|\rho|)\sqrt{r}, -(1 -|\rho|)\sqrt{r})) \right\} \\
      =&~ \min \left\{  \left[ \sqrt{1 -\rho^2}\sqrt{r} - t  \right]^2,\ \frac{2}{1 -|\rho|}\left[ (1 -|\rho|)\sqrt{r} - t
       \right]^2 \right\}
    \end{align*}
\end{theorem}

\begin{remark}
  The proof of Theorem~\ref{suppthm:hamm.middle.v} is easy, but we need to emphasize one thing: For \( FP_1 \), whose ellipsoid is centered at \( (0,0) \), it may be tangent to \( h_1 = v' \) at \( (v',|\rho| v') \); or, it may intersect the rejection region at the corner \( (t,t(\rho - \sqrt{1 -\rho^2})) \). 
\end{remark}

\begin{theorem}[The phase diagram when \( t \leq v \leq \frac{t}{\sqrt{1 -\rho^2}} \) ]\label{suppthm:phase.middle.v} Suppose the conditions of Theorem~\ref{thm:forward-backward} holds. The boundary between Exact Recovery and and Almost Full Recovery is Equation~\ref{suppeq:boundary.middle.v.positive} when the correlation is positive, and Equation~\ref{suppeq:boundary.middle.v.negative} when the correlation is negative. When \( \rho \geq  0 \), 
\begin{equation}\label{suppeq:boundary.middle.v.positive}
  \sqrt{r} =\max \left\{ 1 + \sqrt{1 -\vt},\sqrt{\frac{2(1 -\vt)}{1 -\rho}},\sqrt{\frac{1 - 2\vt}{1 -\rho^2}} + \sqrt{\frac{1 -\vt}{1 -\rho^2}} \right\}.
\end{equation}
When \( \rho < 0 \),
\begin{align}\label{suppeq:boundary.middle.v.negative}
  \sqrt{r} =  \max \Bigg\{&~ 1 + \sqrt{1 -\vt},\sqrt{\frac{2(1 -\vt)}{1 -|\rho|}},\sqrt{\frac{1 - 2\vt}{1 -\rho^2}} + \sqrt{\frac{1 -\vt}{1 -\rho^2}},\nonumber\\ 
  & \sqrt{\frac{1 - 2\vt}{2(1 -|\rho|)}} + \frac{\sqrt{1 -\vt}}{1 -|\rho|},\sqrt{\frac{1 - 2\vt}{2(1 -|\rho|)}} + \frac{\sqrt{1 -\rho^2}}{1 -|\rho|} \Bigg\}.
\end{align}
\end{theorem}

\begin{proof}[Proof of Theorem~\ref{suppthm:phase.middle.v}]
  Like the proof of forward selection and other methods, We still discuss the \( 2\times 2 = 4 \) cases. For brevity, we use \( f_1,f_2,f_3 \) as shorthand of \( f_1(\sqrt{r},t',v'),f_2(\sqrt{r},t',v'),f_3(\sqrt{r},t',v') \).

  \textit{First}, if \( \min(v'^2,2t'^2) =\vt + f_2 = 1\), we have \( \sqrt{r} =\max \left\{ v' + \sqrt{1 -\vt}, \sqrt{\frac{2(1 -\vt)}{1 -|\rho|}} \right\} \). To ensure \( \vt + f_1 \geq 1 \), we need \( t' \geq \sqrt{1 -\vt} \). For \( f_2 \), it need to meet the requirement \( \sqrt{r} \leq \frac{2t'}{\sqrt{1 -\rho^2}} \). This is not restrictive, because:
  \begin{equation*}
    \begin{cases} 
    t' \geq \sqrt{1 -\vt} \\
    1 \leq  v' \leq \frac{t'}{\sqrt{1 -\rho^2}} \implies t' \geq v' \sqrt{1 -\rho^2}
    \end{cases} \implies 2t' \geq \sqrt{1 -\vt} + v'\sqrt{1 -\rho^2}.
  \end{equation*}
  (For the conditional expression of \( f_2 \), \( \sqrt{r} \leq \frac{2t'}{\sqrt{1 -\rho^2}}  \)  is the only possibility here; to see this, we can just refer to the same part of proof for forward selection.) Finally, \( 2\vt + f_4 \geq 1 \) requires \( \sqrt{r} \geq \sqrt{\frac{1 - 2\vt}{1 -\rho^2}} + \frac{t'}{\sqrt{1 -\rho^2}} \) when \( \rho \geq  0 \), and \( \sqrt{r} \geq \max \left\{ \sqrt{\frac{1 - 2\vt}{1 -\rho^2}} + \frac{t'}{\sqrt{1 -\rho^2}}, \sqrt{\frac{1 - 2\vt}{2(1 -|\rho|)}} + \frac{t'}{1 -|\rho|} \right\} \) when \( \rho \leq 0 \). 

  In the above discussion, the \( (v',t') \) refers to any admissible \( t' \) in this case, so we choose \( v'_{\min} = 1 \) and \( t'_{\min} =\min\left\{ \sqrt{1 -\vt},\sqrt{1 -\rho^2}, \frac{\sqrt{2}}{2}  \right\} \). To sum up,  \( \sqrt{r} = \max \left\{ v'_{\min} + \sqrt{1 -\vt}, \sqrt{\frac{2(1 -\vt)}{1 -|\rho|}} \right\} \), and we require \( \sqrt{r} \geq \sqrt{\frac{1 - 2\vt}{1 -\rho^2}} + \frac{t'_{\min}}{\sqrt{1 -\rho^2}} \) when \( \rho \geq  0 \), and \( \sqrt{r} \geq \max \left\{ \sqrt{\frac{1 - 2\vt}{1 -\rho^2}} + \frac{t'_{\min}}{\sqrt{1 -\rho^2}}, \sqrt{\frac{1 - 2\vt}{2(1 -|\rho|)}} + \frac{t'_{\min}}{1 -|\rho|} \right\} \) when \( \rho \leq 0 \).
  
\textit{Second}, if \( \vt + f_1 =\vt + f_2 = 1 \), this case will not give us any curve. We first need \( v' \geq 1 \) and \( t' \geq \frac{\sqrt{2}}{2}  \), and \( \sqrt{r} = \max \left\{ v' + \sqrt{1 -\vt},\sqrt{\frac{2(1 -\vt)}{1 -\rho}} \right\} \). For \( \vt + f_1 \geq 1 \), we already have \( \sqrt{r} \geq \sqrt{\frac{2(1 -\vt)}{1 -\rho}} \), so we only need \( t' \geq \sqrt{1 -\vt} \). The requirement from \( 2\vt + f_3 \geq 1 \) is still the same as that of the \textit{first} case.

We notice that even if this case admits any curve, it is strictly above the curve yielded by \( FP_1 = FN_1 \), and it exists in a smaller interval of \( \vt \). As a result, we need not discuss this case any further. 

\textit{Third}, if \( \min\left\{ v'^2,2t'^2 \right\} =\vt + f_2 = 1 \), then we immdiately have \( v' \geq 1  \), \( t' \geq \frac{\sqrt{2}}{2} \), and we can limit ourselves to consider \( \vt \leq \frac{1}{2} \). From \( \vt + f_1 \geq 1\) and \( \vt + f_2 \geq 1 \), we have the requirement 
\( \sqrt{r} \geq  \max \left\{ v' + \sqrt{1 -\vt},\sqrt{\frac{2(1 -\vt)}{1 -\rho}} \right\}\)
and 
\( t' \geq \sqrt{1 -\vt} \). Also, since \( v' \leq t'/\sqrt{1 -\rho^2}\), we need \( t' \geq \sqrt{1 -\rho^2} \). 

When \( \rho \geq 0 \), we have \( \sqrt{r} = \sqrt{\frac{1 - 2\vt}{1 -\rho^2}} + \frac{t'}{\sqrt{1 -\rho^2}}\); when \( \rho < 0 \), we have \(  \sqrt{r} = \max \left\{ \sqrt{\frac{1 - 2\vt}{1 -\rho^2}} + \frac{t'}{\sqrt{1 -\rho^2}}, \sqrt{\frac{1 - 2\vt}{2(1 -\rho)}} + \frac{t'}{1 -\rho} \right\}  \). Just like the \textit{first} case, we can take \( v'_{\min} = 1 \) and \( t'_{\min} =\min\left\{ \sqrt{1 -\vt},\sqrt{1 -\rho^2}, \frac{\sqrt{2}}{2}  \right\} \) in the expression of \( \sqrt{r} \). 

\textit{Fourth}, if \( \vt + f_1 = 2\vt + f_3 = 1 \), this case does not give any curve. The discussion is exactly the same as the \textit{second} case: even if this case gives us any curve, it would be strictly above the curve in the \textit{third} case.

To sum up, define \( v'_{\min} = 1 \) and \( t'_{\min} =\min\left\{ \sqrt{1 -\vt},\sqrt{1 -\rho^2}, \frac{\sqrt{2}}{2}  \right\} \): When \( \rho \geq  0 \):
\begin{equation*}
  \sqrt{r} =\max \left\{v'_{\min} + \sqrt{1 -\vt}, \sqrt{\frac{2(1 -\vt)}{1 -|\rho|}},\sqrt{\frac{1 - 2\vt}{1 -\rho^2}} + \frac{t'_{\min}}{\sqrt{1 -\rho^2}} \right\}
\end{equation*}
and add \( \sqrt{\frac{1 - 2\vt}{2(1 -\rho)}} + \frac{t'_{\min}}{1 -\rho} \) into the maximum when \( \rho < 0 \).

We can simplify the expression of the curve above, by deleting a few curves in the maximum: 
  \begin{itemize}
    \item the curve \( \sqrt{r} =\sqrt{\frac{1 - 2\vt}{1 -\rho^2}} + \frac{\sqrt{2}/2}{\sqrt{1 -\rho^2}} \) is always below other curves, and can be omitted. This is because when \( \vt \leq \frac{1}{2} \), \( \sqrt{1 -\vt} \geq \frac{\sqrt{2}}{2}  \), which implies \( \sqrt{r} =\sqrt{\frac{1 - 2\vt}{1 -\rho^2}} + \frac{\sqrt{2}/2}{\sqrt{1 -\rho^2}} \leq \sqrt{\frac{1 - 2\vt}{1 -\rho^2}} + \frac{\sqrt{1 -\vt}}{\sqrt{1 -\rho^2}} \).
    \item for the same reason, the curve \( \sqrt{r} =\sqrt{\frac{1 - 2\vt}{2(1 -\rho)}} + \frac{\sqrt{2}/2}{1 -\rho} \) is also always below other curves, and can be omitted. 
    \item   
    the curve \( \sqrt{r} =\sqrt{\frac{1 - 2\vt}{1 -\rho^2}} + \frac{\sqrt{1 -\rho^2}}{\sqrt{1 -\rho^2}} \) is always below other curves, and can be omitted. 
  \begin{itemize}
    \item When \( \rho \geq \frac{\sqrt{2}}{2}  \), \( \sqrt{1 -\rho^2} \leq \sqrt{1 -\vt} \) for all \( \vt \leq \frac{1}{2} \). Thus \( \sqrt{\frac{1 - 2\vt}{1 -\rho^2}} + \frac{\sqrt{1 -\rho^2}}{\sqrt{1 -\rho^2}} \leq \sqrt{\frac{1 - 2\vt}{1 -\rho^2}} + \frac{\sqrt{1 -\vt}}{\sqrt{1 -\rho^2}} \) 
    \item When \( \rho \geq \frac{\sqrt{2}}{2}  \): If \( \vt \leq \rho^2 \), we still have \( \sqrt{\frac{1 - 2\vt}{1 -\rho^2}} + \frac{\sqrt{1 -\rho^2}}{\sqrt{1 -\rho^2}} \leq \sqrt{\frac{1 - 2\vt}{1 -\rho^2}} + \frac{\sqrt{1 -\vt}}{\sqrt{1 -\rho^2}} \). If \( \rho^2 < \vt \leq \frac{1}{2} \), it can be verified that \( \sqrt{\frac{1 - 2\vt}{1 -\rho^2}} + \frac{\sqrt{1 -\rho^2}}{\sqrt{1 -\rho^2}} \leq 1 + \sqrt{1 -\vt} \).
  \end{itemize}
  \end{itemize}
Now we have arrived at the conclusion of Theorem~\ref{suppthm:phase.middle.v}.
\end{proof}

\paragraph{Case 3: When \(  \frac{t'}{\sqrt{1 -\rho^2}}\leq v' \leq t'( 1 +\frac{|\rho|}{\sqrt{1 -\rho^2}} )\).}
\begin{theorem}[The Hamming error rate When \(  \frac{t'}{\sqrt{1 -\rho^2}}\leq v' \leq t'( 1 +\frac{|\rho|}{\sqrt{1 -\rho^2}} )\) ]\label{suppthm:hamm.large.v}
  Suppose the conditions of Theorem~\ref{thm:forward-backward} holds. Let  \( h_1 = x_j' y /\sqrt{2\log(p)} \), \( h_2 = x_{j + 1}' y /\sqrt{2\log(p)} \) and \( v' = v/\sqrt{2\log(p)} \), \( t' = t /\sqrt{2\log(p)}\). As shorthand notation, define the points \( A(v',v') \),  \( B(\frac{t'\sqrt{1 -\rho^2}}{1 -|\rho|},\frac{t'\sqrt{1 -\rho^2}}{1 -|\rho|}) \), and \( D(v' + \frac{\rho t'}{\sqrt{1 -\rho^2}},\rho v' + \frac{ t'}{\sqrt{1 -\rho^2}}) \) as marked in Figure~\ref{subfig:large.v}. 
  We require \( t'/\sqrt{1 -\rho^2}\leq v' \leq t'( 1 + |\rho|/ \sqrt{1 -\rho^2} )\).  As \( p\to\infty \), 
  \[
  \FP_p=L_p p^{1- \min\bigl\{ \min\{v'^2,2t'^2\}, \;\; \vt + f_1(\sqrt{r}, t',v')\bigr\}}, \qquad \FN_p = L_p p^{1-\min\bigl\{\vt + f_2(\sqrt{r}, t',v'),\;\; 2\vt + f_3(\sqrt{r}, t',v')\bigr\}}, 
  \]
  where (below, $d^2_{|\rho|}(u,v)$ is as in Definition~\ref{def:EllipsDistance}), 
  \begin{align*}
    f_1(\sqrt{r},t',v') & = \begin{cases} 
      (v' -|\rho| \sqrt{r})^2 & \text{ if } \sqrt{r} \leq \frac{v'}{1 +\rho}\\
      \frac{1}{1 -\rho^2}d_{|\rho|}^2(A, (|\rho| \sqrt{r},\sqrt{r})) & \text{ if } \frac{v'}{1 +\rho} <\sqrt{r} \leq \frac{2v'}{1+\rho} \\
      \min\left\{ k(v',t'),\ v'^2(1 -\rho^2) \right\} & \text{ if } \sqrt{r} > \frac{2v'}{1+\rho}
    \end{cases}
  \end{align*}
  where \( k(v',t') \) is defined like: 
  \begin{equation*}
    k(v',t')\defeq \begin{cases} 
      \frac{1}{2}(1 -|\rho|) r & \text{ if } \frac{2v'}{1+|\rho|} \leq \sqrt{r} \leq \frac{2t'}{\sqrt{1-\rho^2}}\\
      \frac{1}{1 -\rho^2}d^2_{|\rho|}\left(B,(|\rho| \sqrt{r},\sqrt{r})\right) & \text{ if }  \frac{2t'}{\sqrt{1-\rho^2}} \leq \sqrt{r} \leq \frac{t \sqrt{1-\rho^2}}{|\rho|(1-|\rho|)} \\
      \left[ \sqrt{1 -\rho^2}\sqrt{r} - t' \right]^2& \text{ if } \sqrt{r} \geq \frac{t \sqrt{1-\rho^2}}{|\rho|(1-|\rho|)}
    \end{cases} 
  \end{equation*}
  \begin{align*} 
   f_2(\sqrt{r}, t',v') &= \begin{cases} 
    \min \left\{(\sqrt{r} - v')_ + ^2,\, \frac{1}{2}(1 -|\rho|) r,\, t'^2  \right\} \qquad\qquad\text{ if } \sqrt{r} \leq v' +|\rho| \frac{t'}{\sqrt{1-\rho^2}} \\
    \min \{ (\sqrt{r} - v')_ + ^2,\,\frac{1}{2}(1 -|\rho|) r,\,\frac{1}{1 -\rho^2}d^2_{|\rho|}(D,(\sqrt{r},|\rho| \sqrt{r})) \}\\
 \qquad\qquad\qquad\qquad\qquad\qquad\qquad\qquad \text{ if } v' +|\rho| \frac{t'}{\sqrt{1-\rho^2}} \leq \sqrt{r} \leq \min\{\sqrt{r_2(v',t')},v' + \frac{t'}{|\rho|\sqrt{1-\rho^2}}\} \\
    \min \left\{ (\sqrt{r} - v')_ + ^2,\ \frac{1}{1 -\rho^2} d^2_{|\rho|}(D,(\sqrt{r},|\rho| \sqrt{r}))\right\}   \text{ if } \min\{\sqrt{r_2},v' + \frac{t'}{|\rho|\sqrt{1-\rho^2}}\} \leq \sqrt{r} \leq v' + \frac{t'}{|\rho|\sqrt{1-\rho^2}} \\
    (1 -\rho^2)\left[ \sqrt{r} - v' \right]^2  \qquad\qquad\qquad\qquad\qquad\quad  \text{if }  \sqrt{r} \geq  v' + \frac{t'}{|\rho|\sqrt{1-\rho^2}}
  \end{cases} 
\end{align*}
where \( r_2 =r_2(v',t') \) is the larger root of the quadratic equation 
    \begin{align}
      \frac{1}{1 - \rho^2}d^2_{|\rho|}(D,(\sqrt{r},|\rho| \sqrt{r})) =&~ \frac{1}{2}(1 -|\rho|) r\cr 
      \Leftrightarrow \frac{1 +|\rho|}{2}r - 2 \left( v' +\frac{|\rho| t'}{\sqrt{1 -\rho^2}}\right)  \sqrt{r} +&~ \left( v'^2 + \frac{t'^2}{1 -\rho^2} +\frac{2|\rho| v' t'}{\sqrt{1 -\rho^2}} \right) = 0 \label{eq:r_2}
    \end{align}
and the explicitely form of  \( r_2(v',t') \) is 
    \begin{align*}
      \sqrt{r_2(v',t')} =&~ \frac{1}{1 +|\rho|} \left[ 2 \left( v' +\frac{|\rho| t'}{\sqrt{1 -\rho^2}}\right) + \sqrt{2(1 -|\rho|) \left( v' - \frac{t'}{\sqrt{1 -\rho^2}}\right) \left( v' + (1 + 2|\rho|) \frac{t'}{\sqrt{1 -\rho^2}} \right)} \right].
    \end{align*}

  The definition of \( f_3(\sqrt{r},t') \) depends on the sign of \( \rho \). When \( \rho > 0 \), 
  \begin{equation*}
    f_3(\sqrt{r}, t') = \left[ \sqrt{1 -\rho^2}\sqrt{r} - t  \right]^2
  \end{equation*}
    When \( \rho < 0 \), 
    \begin{align*}
      f_3(\sqrt{r},t') =&~ \min \left\{  \left[ \sqrt{1 -\rho^2}\sqrt{r} - t  \right]^2,\ d^2_{|\rho|}(C,((1 -|\rho|)\sqrt{r}, -(1 -|\rho|)\sqrt{r})) \right\} \\
      =&~ \min \left\{  \left[ \sqrt{1 -\rho^2}\sqrt{r} - t  \right]^2,\ \frac{2}{1 -|\rho|}\left[ (1 -|\rho|)\sqrt{r} - t
       \right]^2 \right\}
    \end{align*}
\end{theorem}

\begin{proof}[Proof of Theorem~\ref{suppthm:hamm.large.v}]
  We explain one detail, about why \( \sqrt{r_2(v',t')} \) is introduced in \( f_2(\sqrt{r},t',v') \), which corresponds to the ellipsoid centered at \( \mu_{10} =(\sqrt{r},|\rho| \sqrt{r}) \).  Recall that the point \( D \) as noted in in Figure~\ref{subfig:large.v} has cooredinate \( x_D =v' + \frac{|\rho| t'}{\sqrt{1 -\rho^2}},y_D =|\rho| v' + \frac{ t'}{\sqrt{1 -\rho^2}} \). Suppose \( \rho \geq  0 \), as the case of \( \rho < 0 \) can be obtained with symmetry.

  When \( \sqrt{r} \leq x_D = v' + \frac{\rho t'}{\sqrt{1 -\rho^2}} \), the ellipsoid can  be tangent to any one among the three line segments: (i) \( h_1 = v' \), (ii) \( h_2 = h_1 \) or (iii) \( h_2 = \rho h_1 + t' \sqrt{1 -\rho^2} \).
  
  When \( \rho \sqrt{r} \geq y_D =\rho v' + \frac{ t'}{\sqrt{1 -\rho^2}} \), the ellipsoid can either be tangent to  the red line \( h_1 =\rho h_2 + v'(1 -\rho^2) \), or the blue line \( h_1 = v' \) in Figure~\ref{subfig:large.v}.

  When \(v' + \frac{\rho t'}{\sqrt{1 -\rho^2}} \leq  \sqrt{r} \leq  v' + \frac{ t'}{\rho\sqrt{1 -\rho^2}}\), thing are more tricky:
  \begin{itemize}
    \item The ellipsoid can possibly be tangent to \( h_1 = v'\), or it may  intersect point \( D \) and rotate around it. 
    \item However, it is uncertain whether we should include the segment \( h_2 = h_1 \) into the form of the Hamming error. This is because  when \( \sqrt{r} \)  is large, the ellipsoid is at the upper right side of point \( D \), where \( h_2 = h_1 \) does not form the boundary of the rejection region. If we still include it, the final phase diagram will be worse than it actually is. 
    \item To exclude \( h_2 = h_1 \) when it is unwanted, we require \( \sqrt{r} > \sqrt{r_2(v',t')} \).
    \item The place of \( \sqrt{r_2(v',t')} \) is exchangable to the symmetric axis of the quadratic equation~\eqref{eq:r_2}, which is \( \frac{2}{1 +\rho}v' +\frac{2\rho t'}{(1 +\rho)\sqrt{1 -\rho^2}} \). 
  \end{itemize}
\( \sqrt{r_2(v',t')} \) can be greater than \( v' +\frac{t'}{\rho \sqrt{1 -\rho^2}} \), so it is taken minimum with \( v' +\frac{t'}{\rho \sqrt{1 -\rho^2}} \) in the definition of \( f_2 \). 
\end{proof}

Before moving on to the phase diagram, we first introduce two terms to simplify notation:
\begin{definition}\label{def:min.min}
  Define \( v'_{\min}  \) and \( t'_{\min}  \), both as functions of \( \vt \) and \( \rho \). \( v'_{\min} =\max \left\{ 1,\sqrt{\frac{1 -\vt}{1 -\rho^2}},\frac{\sqrt{2}/2}{\sqrt{1 -\rho^2}} \right\} \) and \( t'_{\min} = \max \left\{ \frac{\sqrt{2}}{2}, \frac{\max\{1,\sqrt{\frac{1 -\vt}{1 -\rho^2}}\}}{1 +|\rho|/\sqrt{1 -\rho^2}}\right\} \). 
\end{definition}
Figure~\ref{suppfig:def.minmin} gives an explanation of how \( v'_{\min} \) and \( t'_{\min} \) are defined. 
\begin{figure}[h!]
  \centering
  \includegraphics[width=0.6\textwidth]{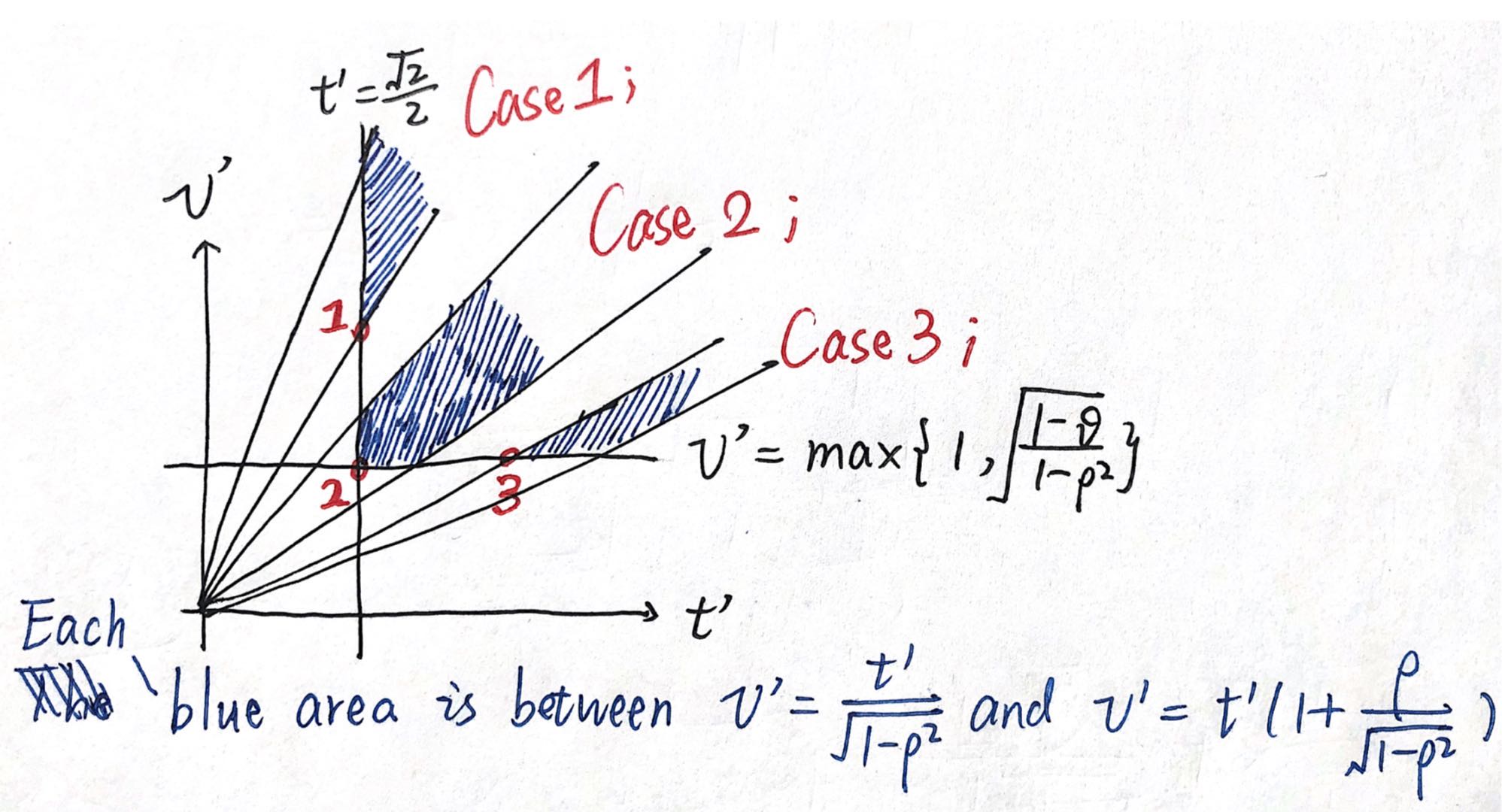}
  \caption{the definition of \( v'_{\min} \) and \( t'_{\min} \)}
  \label{suppfig:def.minmin}
\end{figure}

\begin{theorem}[The phase diagram when \(  \frac{t'}{\sqrt{1 -\rho^2}}\leq v' \leq t'( 1 +\frac{|\rho|}{\sqrt{1 -\rho^2}} ) \) ]\label{suppthm:phase.large.v} Suppose the conditions of Theorem~\ref{thm:forward-backward} holds. The boundary between Exact Recovery and and Almost Full Recovery is Equation~\ref{suppeq:boundary.middle.v.positive} when the correlation is positive, and Equation~\ref{suppeq:boundary.middle.v.negative} when the correlation is negative. When \( \rho \geq  0 \), 
  \begin{align}\label{eq:boundary.large.v.positive}
    \sqrt{r} = \max \left\{ v'_{\min} + \sqrt{1 -\vt},\, \sqrt{\frac{2(1 -\vt)}{1 -|\rho|}},\, \sqrt{\frac{1 - 2\vt}{1 -\rho^2}} + v'_{\min} \right\}
  \end{align}
  When the actual correlation is negative,
  \begin{align}\label{eq:boundary.large.v.negative}
      \sqrt{r} =  \max \Bigg\{&~ v'_{\min} + \sqrt{1 -\vt},\, \sqrt{\frac{2(1 -\vt)}{1 -|\rho|}},\, \sqrt{\frac{1 - 2\vt}{1 -\rho^2}} + v'_{\min},\sqrt{\frac{1 - 2\vt}{2(1 -|\rho|)}} + \frac{t'_{\min}}{1 -|\rho|} \Bigg\}.
    \end{align}
\end{theorem}

\begin{proof}[Proof of Theorem~\ref{suppthm:phase.large.v}]

We start by considering the case of \( \rho \geq 0 \). The other case of \( \rho < 0 \)  can be prove in a very similar way by adding one more curve. In the discussion below, we do not work with \( |\rho| \), but \( \rho \geq 0 \) itself.

In this particular proof, we do not limit ourselves to discuss the four cases as usual. Instead, we just think about the conditions for \( \min\{v'^2,2t'^2\},\vt + f_1,\vt + f_2,2\vt + f_3 \geq 1 \), and take the smallest \( \sqrt{r} \) possible. Also, remember the important fact that at least two of these requirements should be tight. 

For fixed \( (v',t') \), the necessary and sufficient condition for  \( \min\{v'^2,2t'^2\} \geq 1 \) is \( v' \geq  1 \) and \( t' \geq \frac{\sqrt{2}}{2}  \). The necessary and sufficient condition for \( 2\vt + f_3 \geq 0 \) is \( \sqrt{r} \geq \sqrt{\frac{1 - 2\vt}{1 -\rho^2}} + v' \). (If \( \rho < 0 \), the condition should be \( \sqrt{r} \geq \max\left\{\sqrt{\frac{1 - 2\vt}{1 -\rho^2}} + v', \sqrt{\frac{1 -\vt}{1 -\rho^2}} \right\}  \).) From \( \vt + f_2 \geq 1 \), we know a necessary condition, \( \sqrt{r} \geq v' + \sqrt{1 -\vt} \). From the discussion above, we already know \begin{align*}
  & \sqrt{r} \geq  \max \left\{ v' + \sqrt{1 -\vt},\, \sqrt{\frac{1 - 2\vt}{1 -\rho^2}} + v' \right\} 
\end{align*}

In terms of admissible \( (v',t') \), we already have \( v' \geq 1 \) and \( t' \geq \frac{\sqrt{2}}{2}  \). Additionally, note that \( f_1(\sqrt{r},t',v') \) as a function of \( \sqrt{r} \) takes its maximum when \( f_1 =(1 -\rho^2)v'^2 \), so \( \vt +f_1 \geq 1\) implies \( v' \geq \sqrt{\frac{1 -\vt}{1 -\rho^2}} \). Under the conditions \( v' \geq \max \{1,\sqrt{\frac{1 -\vt}{1 -\rho^2}}\} \), \( t' \geq \frac{\sqrt{2}}{2} \), and  \(  \frac{t'}{\sqrt{1 -\rho^2}}\leq v' \leq t'( 1 +\frac{|\rho|}{\sqrt{1 -\rho^2}} ) \), the smallest admissible \( (v',t') \) are precisely defined by Definition~\ref{def:min.min} as \( (v'_{\min},t'_{\min}) \). 

We consider the two cases: \( \rho \leq 0.576 \) and \( \rho \geq 0.576 \). The point 0.576 is important, because 
when \( \rho \leq 0.576 \), since \( v'_{\min} \geq \max \left\{ 1,\sqrt{\frac{1 -\vt}{1 -\rho^2}} \right\} \), we can prove  
\begin{equation*}
  \sqrt{r} \geq \max \left\{ v'_{\min} + \sqrt{1 -\vt},\,  \sqrt{\frac{1 - 2\vt}{1 -\rho^2}} + v'_{\min} \right\} \geq \sqrt{\frac{2(1 -\vt)}{1 -\rho}}.
\end{equation*}

Since now \( \sqrt{r} \geq \sqrt{\frac{2(1 -\vt)}{1 -\rho}} \) is already implied by other necessary conditions, we only need to check the sufficiency of \( \sqrt{r} = \max \left\{  v'_{\min} + \sqrt{1 -\vt},\,  \sqrt{\frac{1 - 2\vt}{1 -\rho^2}} + v'_{\min}\right\} \) to prove Theorem~\ref{suppthm:phase.large.v} for \( 0 \leq \rho < 0.576 \).  

When \( 0 \leq \rho \leq 0.576 \), we already have \( \min\{v'^2,2t^2\} \geq 1 \), \( 2\vt + f_3 \geq 1 \) and \( \sqrt{r} \geq \sqrt{\frac{2(1 -\vt)}{1 -\rho}} \). we need to check \( \vt + f_1 \geq 1 \) and \( \vt + f_2 \geq 1 \).

For \( \vt + f_1 \geq 1 \), because \( \frac{1}{1 -\rho^2} d^2_{|\rho|}(A, (|\rho| \sqrt{r},\sqrt{r})) \geq \frac{1}{2}(1 -\rho)r \) and \(\frac{1}{1 -\rho^2} d^2_{|\rho|}(B, (|\rho| \sqrt{r},\sqrt{r})) \geq \frac{1}{2}(1 -\rho)r \), we only need to check \( \sqrt{r} \geq \sqrt{\frac{1 -\vt}{1 -\rho^2}} + \frac{t'}{\sqrt{1 -\rho^2}} \) when \( \sqrt{r} \geq \frac{t \sqrt{1 -\rho^2}}{\rho(1 -\rho)} \). Since \( t' \geq \frac{v'}{1 +\frac{\rho}{\sqrt{1 -\rho^2}}} \geq \frac{\sqrt{1 -\vt}}{\rho + \sqrt{1 - \rho^2}} \geq \rho \sqrt{1 -\vt}\) for all \( \rho \leq \frac{\sqrt{2}}{2}  \), we have \( t' \geq  \rho \sqrt{1 -\vt} \), and now actually \( \frac{t \sqrt{1 -\rho^2}}{\rho(1 -\rho)} \geq \sqrt{\frac{1 -\vt}{1 -\rho^2}} + \frac{t'}{\sqrt{1 -\rho^2}}\), so this case is not restrictive at all. 

For \( \vt + f_2 \geq 1 \), when \( \sqrt{r} \leq \sqrt{r_2(v',t')}  \), it is sufficient to have \( \sqrt{r} \geq \sqrt{\frac{2(1 -\vt)}{1 -\rho}} \). We only need to check the rest two cases in which \( \sqrt{r} \) is large:  
\begin{itemize}
  \item When \( \sqrt{r_2(v',t')} \leq \sqrt{r} \leq v' + \frac{t'}{\rho\sqrt{1-\rho^2}} \), we need  \( d^2(D,(\sqrt{r},\rho \sqrt{r})) \geq (1 -\rho^2)(1 -\vt)  \). This trivially holds, because \( \sqrt{r} \geq \sqrt{r_2(v',t')} \). From Equation~\eqref{eq:r_2}:
  \begin{equation*}
    r - 2 \left( v' +\frac{\rho t'}{\sqrt{1 -\rho^2}}\right)  \sqrt{r} + \left( v'^2 + \frac{t'^2}{1 -\rho^2} +\frac{2\rho v' t'}{\sqrt{1 -\rho^2}} \right) \geq \frac{1 -\rho}{2}r \geq 1 -\vt
  \end{equation*}
  whose last inequality is because \( \sqrt{r} \geq \sqrt{\frac{2(1 -\vt)}{1 -\rho}} \).
  \item When \( \sqrt{r} \geq v' + \frac{t'}{\rho\sqrt{1-\rho^2}} \), we need \( \sqrt{r} \geq \sqrt{\frac{1 -\vt}{1 -\rho^2}} + v' \). Since \( t' \geq \rho\sqrt{1 -\vt} \), it trivially holds.
\end{itemize}

We have proved Theorem~\ref{suppthm:phase.large.v} for \( \rho \leq 0.576 \). Now we move on to the next case of \( \rho > 0.576 \).

For \( \rho > 0. 576\), our task is to prove it is necessary and sufficient to have \( \sqrt{r} \geq \sqrt{\frac{2(1 -\vt)}{1 -\rho}} \). 
 Reviewing Theorem~\ref{suppthm:phase.large.v}, as long as \( \sqrt{r} \leq \sqrt{r_2(v',t')} \) and \( \sqrt{r} \leq \frac{2t'}{\sqrt{1-\rho}^2} \) for fixed admissible \( (v',t') \), then \( \sqrt{r} \geq \sqrt{\frac{2(1 -\vt)}{1 -\rho}} \) is also sufficient for \( \sqrt{r} \). In other words, we could simply set
\begin{equation*}
  \sqrt{r} =\max \left\{  v'_{\min} + \sqrt{1 -\vt},\,\sqrt{\frac{2(1 -\vt)}{1 -\rho}},\, \sqrt{\frac{1 - 2\vt}{1 -\rho^2}} + v'_{\min} \right\}.
\end{equation*}

We are left to eliminate other cases, i.e. \( \sqrt{r} \leq \sqrt{r_2(v',t')} \) or \( \sqrt{r} \leq \frac{2t'}{\sqrt{1-\rho}^2} \) for fixed admissible \( (v',t') \). They can either be impossible, or only produce a curve greater than the one in Equation~\ref{eq:boundary.large.v.positive}. The rest of the proof focuses on the elimination of other cases.

To prepare for such work, we take a closer look at the definition of \( f_2 \). We point out that when \( v' \geq \frac{\sqrt{2(1 +\rho)} - 1}{\rho} \cdot \frac{t}{\sqrt{1 -\rho^2}} \), the case of 
\begin{equation*}
  \begin{cases} 
  (\sqrt{r} - v')_ + ^2 \\
  \frac{1}{1 -\rho^2}d^2_{|\rho|} (D, (\sqrt{r},\rho \sqrt{r}))
  \end{cases} \text{ if } \min\{\sqrt{r_2},v' + \frac{t'}{|\rho|\sqrt{1-\rho^2}}\} \leq \sqrt{r} \leq v' + \frac{t'}{|\rho|\sqrt{1-\rho^2}}
\end{equation*}
does not exist, and the degenerated \( f_2 \) is just
\begin{equation*}
  f_2(\sqrt{r},t',v') =
  \begin{cases} 
    \min \{(\sqrt{r} - v')_ + ^2,\frac{1}{2}(1 -\rho) r,t'^2\}
 &\text{ if } \sqrt{r} \leq v' +\rho \frac{t'}{\sqrt{1-\rho^2}} \\
    \min\{(\sqrt{r} - v')_ + ^2,\frac{1}{2}(1 -\rho) r,\frac{1}{1 -\rho^2}d^2_{|\rho|}(D,(\sqrt{r},\rho \sqrt{r}))\}
 & \text{ if } v' +\rho \frac{t'}{\sqrt{1-\rho^2}} \leq \sqrt{r} \leq v' + \frac{t'}{\rho\sqrt{1-\rho^2}} \\
    (1 -\rho^2)\left[ \sqrt{r} - v' \right]^2 & \text{ if }  \sqrt{r} \geq  v' + \frac{t'}{\rho\sqrt{1-\rho^2}}.
  \end{cases}
\end{equation*}
When \( v' < \frac{\sqrt{2(1 +\rho)} - 1}{\rho} \cdot \frac{t}{\sqrt{1 -\rho^2}} \), the case of 
\begin{equation*}
  \begin{cases} 
  (1 -\rho^2)(\sqrt{r} - v')_ + ^2 \\
  d^2 (D, (\sqrt{r},\rho \sqrt{r}))
  \end{cases} \text{ if } \sqrt{r_2(v',t')} \leq \sqrt{r} \leq v' + \frac{t'}{\rho\sqrt{1-\rho^2}}
\end{equation*}
in \( f_2 \) does exist.

Now we are ready to eliminate the unwanted cases.

When the case of \( \frac{2t'}{\sqrt{1-\rho^2}} \leq \sqrt{r} \leq \frac{t \sqrt{1-\rho^2}}{\rho(1-\rho)} \) in \( \vt +f_1 \geq 1 \) is tight and active, we have \( \sqrt{r} =\frac{(1 + \rho)t'}{\sqrt{1 -\rho^2}} - \sqrt{1 -\vt - t'^2}. \) There is one important fact: \( v' +\frac{t'}{\rho \sqrt{1 -\rho^2}} \geq \frac{2}{1 +\rho} v' + \frac{2\rho t'}{(1 +\rho) \sqrt{1-\rho^2}} \geq \frac{2t'}{\sqrt{1-\rho^2}}\). The middle term is the symmetric axis of Equation~\eqref{eq:r_2}, and can be used exchangably with \( \sqrt{r_2(v',t')} \) as we have noted in the proof of Theorem~\ref{suppthm:hamm.large.v}. 
  \begin{itemize}
    \item If \( \sqrt{r} \leq v' +\frac{t'}{\rho \sqrt{1 -\rho^2}} \), we need \( d^2_{|\rho|}(D, (\sqrt{r},\rho \sqrt{r})) \geq (1 -\rho^2)(1 -\vt) \), i.e. \( \sqrt{r} \geq v' + \frac{\rho t'}{\sqrt{1 -\rho^2}} + \sqrt{1 -\vt - t'^2} \), which gives a contradiction. 
    \item If \( \sqrt{r} > v' +\frac{t'}{\rho \sqrt{1 -\rho^2}} \), because \( d^2_{|\rho|}(D, (\sqrt{r},\rho \sqrt{r})) \geq (1 -\rho^2)^2\left[ \sqrt{r} - v' \right]^2 \), it is the same  contradiction.
  \end{itemize} 

  When the case of \( \sqrt{r} > \frac{t \sqrt{1-\rho^2}}{\rho(1-\rho)} \) in \( \vt +f_1 \geq 1 \) is tight and active, we have \(\sqrt{r} = \sqrt{\frac{1-\vt}{1 -\rho^2}} + \frac{t'}{\sqrt{1 -\rho^2}}  \). Then we discuss the conditional expression of \( f_2 \) in \( \vt + f_2 \geq 1 \); given the intractability of \( \sqrt{r_2(\vt,\rho)} \), we work with the alternative \( \frac{2}{1 +\rho} v' + \frac{2\rho t'}{(1 +\rho) \sqrt{1-\rho^2}} \) instead. 

  \begin{itemize}
    \item When \( \sqrt{r} \geq v' +\frac{t'}{\rho \sqrt{1 -\rho^2}}  \) in \( f_2 \), we need \( \sqrt{r} \geq \sqrt{\frac{1-\vt}{1 -\rho^2}} + v' \)  which gives a contradiction.
    \item When  \( \frac{2}{1 +\rho} v' + \frac{2\rho t'}{(1 +\rho) \sqrt{1-\rho^2}} \leq \sqrt{r} < v' +\frac{t'}{\rho \sqrt{1 -\rho^2}}  \) in \( f_2 \), we actually cannot have \( \sqrt{\frac{1-\vt}{1 -\rho^2}} + \frac{t'}{\sqrt{1 -\rho^2}} \geq \frac{2}{1 +\rho} v' + \frac{2\rho t'}{(1 +\rho) \sqrt{1-\rho^2}} \), because it means \( \left( \sqrt{\frac{1-\vt}{1 -\rho^2}} - v' \right) +\frac{1 -\rho}{1 +\rho}\frac{t'}{\sqrt{1 -\rho^2}} \geq \frac{1 -\rho}{1 +\rho}v' \). 
  \end{itemize}

  When the case of \( \sqrt{r} \geq  v' + \frac{t'}{\rho\sqrt{1-\rho^2}} \) in \( FN_1 \) is tight  and active, we have \(  \sqrt{r} = \sqrt{\frac{1-\vt}{1 -\rho^2}} + v'_{\min} \).
  This case does not have any problem or contradiction itself. However, with \( v' \geq \max \{1,\sqrt{\frac{1 -\vt}{1 -\rho^2}} \}\), it is too large, much larger than \( \sqrt{r} = \max \left\{ v'_{\min} + \sqrt{1 -\vt},\, \sqrt{\frac{2(1 -\vt)}{1 -\rho}},\, \sqrt{\frac{1 - 2\vt}{1 -\rho^2}} + v'_{\min} \right\} \) whose sufficiency has been proven.

  When \( v \geq  \frac{\sqrt{2(1 +\rho)} - 1}{\rho} \cdot \frac{t}{\sqrt{1 -\rho^2}}  \), there is no more cases in \( f_2 \), and our discussion is finished. When \( v < \frac{\sqrt{2(1 +\rho)} - 1}{\rho} \cdot \frac{t}{\sqrt{1 -\rho^2}}  \), we need to look at the last case of \( \sqrt{r_2(v',t')} \leq \sqrt{r} \leq v' + \frac{t'}{\rho\sqrt{1-\rho^2}}\) in \( FN_1 \). If this case is tight and active, we have \( \sqrt{r} = v' + \frac{\rho t'}{\sqrt{1 -\rho^2}} + \sqrt{1 -\vt - t'^2} \).
  \begin{itemize}
    \item It can be verified that \( \sqrt{r_2(v',t')} \geq \frac{t \sqrt{1-\rho^2}}{\rho(1-\rho)} \), which is equivalent to 
    \begin{equation*}
      -\frac{1 +\rho}{2}v'^2 + \frac{2v't'}{\sqrt{1 -\rho^2}} \left( 1 +\frac{1 -\rho^3}{2\rho} \right) \geq \frac{t'^2}{1 -\rho^2} \left[ \frac{3}{2} + \frac{1}{\rho} -\frac{\rho}{2} -\rho^2 - \frac{(1 -\rho^2)^2}{4\rho^2} \right]
    \end{equation*}
    \( (RHS-LHS) \) is decreasing in \( \frac{t'}{v'} \), and the inequality holds as \( \frac{t'}{v'} =\frac{\rho \sqrt{1 -\rho^2}}{\sqrt{2(1 +\rho)} - 1} \).
    \item As a result, we need to verify \( v' + \frac{\rho t'}{\sqrt{1 -\rho^2}} + \sqrt{1 -\vt - t'^2} \geq \sqrt{\frac{1 -\vartheta}{1 -\rho^2}} + \frac{t'}{\sqrt{1 -\rho^2}} \) implied by \( \vt + f_1 \geq 1 \). This would give us a contradiction, because actually \( v' + \frac{\rho t'}{\sqrt{1 -\rho^2}} + 1.02 \cdot \sqrt{1 -\vt - t'^2} \leq  \sqrt{\frac{1 -\vartheta}{1 -\rho^2}} + \frac{t'}{\sqrt{1 -\rho^2}} \). In fact, We only need to prove \( \left[ \frac{\sqrt{2(1+\rho)} - 1}{\rho} +\rho - 1 \right]\frac{t'}{\sqrt{1 -\rho^2}} + 1.02 \cdot \sqrt{1 -\vt - t'^2} \leq \sqrt{\frac{1 -\vt}{1 -\rho^2}}\). The coefficient \( 1.02 \) as to make the LHS  decreasing in \( t' \) for \( \rho \geq 0.576 \). Taking \( t =\frac{\rho \sqrt{1 -\vt}}{\sqrt{2(1 +\rho)} - 1} \) proves the inequality.  
  \end{itemize}

So far our discussion is finally finished, and we have proven the phase curve to be 
\begin{equation*}
  \sqrt{r} =\max \left\{  v'_{\min} + \sqrt{1 -\vt},\,\sqrt{\frac{2(1 -\vt)}{1 -\rho}},\, \sqrt{\frac{1 - 2\vt}{1 -\rho^2}} + v'_{\min} \right\}.
\end{equation*}
where \( v'_{\min} \) is defined in Defition~\ref{def:min.min}. 

Reviewing the proof for \( \rho \geq 0 \), we notice that \( 2\vt + f_3(\sqrt{r},t'v') \geq 1 \) is only used at the very start of the proof, and does not change the bulk of the discussion. It can be proved with vitually the same proof, that when \( \rho < 0 \), the phase curve is 
\begin{align*}
  \sqrt{r} =  \max \Bigg\{&~ v'_{\min} + \sqrt{1 -\vt},\, \sqrt{\frac{2(1 -\vt)}{1 -|\rho|}},\, \sqrt{\frac{1 - 2\vt}{1 -\rho^2}} + v'_{\min},\sqrt{\frac{1 - 2\vt}{2(1 -|\rho|)}} + \frac{t'_{\min}}{1 -|\rho|} \Bigg\}.
\end{align*}

\end{proof}

\textit{Summarising the first three cases:} When \( \rho \geq 0 \), among the first three cases, we can take the minimum over Equation~\eqref{eq:boundary.small.v.positive},\eqref{suppeq:boundary.middle.v.positive},\eqref{eq:boundary.large.v.positive}. In fact, the minimum is just Equation~\eqref{suppeq:boundary.middle.v.positive}, which is 
\begin{equation}\label{suppeq:boundary.firstthree.positive}
  \sqrt{r} =\max \left\{ 1 + \sqrt{1 -\vt},\sqrt{\frac{2(1 -\vt)}{1 -\rho}},\sqrt{\frac{1 - 2\vt}{1 -\rho^2}} + \sqrt{\frac{1 -\vt}{1 -\rho^2}} \right\}.
\end{equation}
When \( \rho < 0 \), we also take the minimum over Equation~\eqref{eq:boundary.small.v.negative},\eqref{suppeq:boundary.middle.v.negative},\eqref{eq:boundary.large.v.negative}. In fact, in the region \( \vt\in(\frac{1}{2},1) \), Equation~\eqref{suppeq:boundary.middle.v.negative} is the minimum, but when \( \vt \leq \frac{1}{2} \), Equation~\eqref{eq:boundary.large.v.negative} is the minimum. 
As a result, we have an upper bound on the final phase curve, which can be expressed as:
\begin{align}\label{suppeq:boundary.firstthree.negative}
  \sqrt{r} =  \max \Bigg\{&~ v'_{\min} + \sqrt{1 -\vt},\, \sqrt{\frac{2(1 -\vt)}{1 -|\rho|}},\, \sqrt{\frac{1 - 2\vt}{1 -\rho^2}} + v'_{\min},\sqrt{\frac{1 - 2\vt}{2(1 -|\rho|)}} + \frac{t'_{\min}}{1 -|\rho|} \Bigg\}.
\end{align}
where we define \( v'_{\min} =\max \{1,\sqrt{\frac{1 -\vt}{1 -\rho^2}}\} \) and \( t'_{\min} = \max \left\{ \frac{\sqrt{2}}{2}, \frac{v'_{\min}}{1 +|\rho|/\sqrt{1 -\rho^2}}\right\} \).

\begin{remark}
  Equation~\eqref{suppeq:boundary.firstthree.positive} and \eqref{suppeq:boundary.firstthree.negative} is the result we presented as Theorem~\ref{thm:forward-backward} in the main text.

  When \( \rho \geq 0 \), since we used to define \( v'_{\min} =\max \left\{ 1,\sqrt{\frac{1 -\vt}{1 -\rho^2}},\frac{\sqrt{2}/2}{\sqrt{1 -\rho^2}} \right\} \) in the third case, Equation~\eqref{eq:boundary.large.v.negative} is strictly above Equation~\eqref{suppeq:boundary.middle.v.negative}.

  When \( \rho < 0 \), in Equation~\eqref{eq:boundary.large.v.negative}, we used to define \( v'_{\min} =\max \left\{ 1,\sqrt{\frac{1 -\vt}{1 -\rho^2}},\frac{\sqrt{2}/2}{\sqrt{1 -\rho^2}} \right\} \). When \( \vt \leq  \frac{1}{2}\), it is equivalent to \( v'_{\min} =\max \{1,\sqrt{\frac{1 -\vt}{1 -\rho^2}}\} \), which agrees with Equation~\eqref{suppeq:boundary.firstthree.negative} for \( \vt \leq \frac{1}{2} \) and \( \rho \leq 0 \). 

  When \( \rho < 0 \) and \( \vt > \frac{1}{2} \), Equation~\eqref{eq:boundary.large.v.negative} is not the minimum among the three, mainly because \( \frac{\sqrt{2}/2}{\sqrt{1 -\rho^2}} > 1 \) for \( |\rho| >\frac{\sqrt{2}}{2} \). The lowest phase curve for \( \vt > \frac{1}{2} \) and \( \rho\in( - 1,1) \) should be \( \max \{1 + \sqrt{1 -\vt},\sqrt{\frac{2(1 -\vt)}{1 -|\rho|}}\} \), but we can also write it equivalently as  \( \max \left\{ \max \{1,\sqrt{\frac{1 -\vt}{1 -\rho^2}}\} + \sqrt{1 -\vt},\sqrt{\frac{2(1 -\vt)}{1 -|\rho|}}\right\} \), because \( 1 + \sqrt{1 -\vt} \geq \sqrt{\frac{2(1 -\vt)}{1 -|\rho|}}\) and \( \vt > \frac{1}{2} \) together imply \( \sqrt{\frac{1 -\vt}{1 -\rho^2}} \leq 1 \). 
\end{remark}

\paragraph{The last three cases:}
We are still left to discuss the rest three cases: (i) \( t'( 1 +\frac{|\rho|}{\sqrt{1 -\rho^2}} )\leq v' \leq \frac{t' \sqrt{1 -\rho^2}}{1 -|\rho|} \), (ii) \( \frac{t' \sqrt{1 -\rho^2}}{1 -|\rho|}\leq v' \leq \frac{t'}{1 -|\rho|}  \), (iii) \(  v' \geq \frac{t'}{1 -|\rho|} \). 

When \( \rho \geq  0\), we can prove that Equation~\eqref{suppeq:boundary.firstthree.positive} is already the best, and there is no need to discuss the rest three cases for \( \rho \geq 0 \). This is because in Figure~\ref{subfig:large.plus.v}, \ref{subfig:large.plus.plus.v} and \ref{subfig:large.plus.plus.plus.v}, we need \( v' \geq \max \left\{ 1,\sqrt{\frac{1 -\vt}{1 -\rho^2}} \right\} \). We also need at least that \( \sqrt{r} \geq v' + \sqrt{1 -\vt} \), \( \sqrt{r} \geq \sqrt{\frac{2(1 -\vt)}{1 -\rho}} \) and \( \sqrt{r} \geq \sqrt{\frac{1 - 2\vt}{1 -\rho^2}} + v' \), so it cannot be any better than Equation~\eqref{suppeq:boundary.firstthree.positive}. 

When \( \rho < 0 \), the optimal phase curve for \( \vt \leq  \frac{1}{2} \) may still be one of the last three cases, but the discussion is too difficult. Even the expression of the phase curves is very complicated. We present the phase curves of the rest three cases without proof: 

\textit{Case 4. When \( t'\left( 1 +\frac{|\rho|}{\sqrt{1 -\rho^2}} \right)\leq v' \leq \frac{t' \sqrt{1 -\rho^2}}{1 -|\rho|} \)}, the phase curve is \begin{align}\label{eq:boundary.large.large.v.negative}
  \sqrt{r} =  \max \Bigg\{&~ v'_{\min} + \sqrt{1 -\vt},\, \sqrt{\frac{2(1 -\vt)}{1 -|\rho|}},\, \sqrt{\frac{1 - 2\vt}{1 -\rho^2}} + v'_{\min}(\vt), \sqrt{\frac{1 - 2\vt}{2(1 -|\rho|)}} + \frac{t'_{\min}(\vt)}{1 -|\rho|}\Bigg\}.
\end{align}
and the definition of \( (v'_{\min},t'_{\min}) \) is specific to this case. 
\begin{equation*}
  v_{\min}(\vt) =\max\left\{ 1,\sqrt{\frac{1 -\vt}{1 -\rho^2}},\frac{\sqrt{2}}{2} \left( 1 +\frac{|\rho|}{\sqrt{1 -\rho^2}} \right) \right\},\quad t_{\min}(\vt) =\max \left\{ \frac{1 -|\rho|}{\sqrt{1 -\rho^2}} v_{\min}(\vt),\, f(|\rho|),g(\vt) \right\} 
\end{equation*}
 where 
\begin{equation*}
  f(|\rho|) = \begin{cases} 
  \sqrt{\frac{1 -\rho^2}{2 -\rho^2}} & \text{ if } |\rho| \leq (\sqrt{5}-1)/2 \\
  \frac{1}{1 +|\rho|} & \text{ if } (\sqrt{5}-1)/2 \leq |\rho| \leq \frac{1}{3} \left[ - 2 + (19+3 \sqrt{33})^{1/3} + (19-3 \sqrt{33})^{1/3}\right]\\
  \frac{|\rho|}{\sqrt{1 +(1 +|\rho|)^2 - 2(1 +|\rho|)\sqrt{1 -\rho^2}}} & \text{ if } |\rho| \geq \frac{1}{3} \left[ - 2 + (19+3 \sqrt{33})^{1/3} + (19-3 \sqrt{33})^{1/3}\right]
  \end{cases} 
\end{equation*} 
and
\begin{equation*}
  g(\vt) = \begin{cases} 
  g_1(\vt) & \vartheta \geq \vt^* \\
  g_2(\vt) & \vt < \vt^*
  \end{cases} 
\end{equation*}
where \( \vt^* \), \( g_1(\vt) \) and \( g_2(\vt) \) are respectively the roots of \( \vt \) of the following three equations:
\begin{align*}
  \vt =\vt^* : &~ ~ 
  \frac{|\rho|}{1 -|\rho|} + \sqrt{\frac{1 - 2\vt}{2(1 -|\rho|)}} -\frac{1 +\rho^2}{\sqrt{1 -\rho^2}} - \sqrt{1 -\vt -\rho^2} = 0
  \\
  t =g_1(\vt): & ~~ t \left( \frac{1}{1 -|\rho|} - \frac{2|\rho|}{\sqrt{1 -\rho^2}} \right) + \sqrt{\frac{1 - 2\vt}{2(1 -|\rho|)}} - \sqrt{1 -t^2} - \sqrt{1 -\vt - t^2} = 0\\
  t =g_2(\vt): & ~~  t \left( \frac{1}{1 -|\rho|} - \frac{|\rho|}{\sqrt{1 -\rho^2}} \right) + \sqrt{\frac{1 - 2\vt}{2(1 -|\rho|)}} - \frac{1}{\sqrt{1 -\rho^2}} - \sqrt{1 -\vt - t^2} = 0
\end{align*}
All three equations can be solved easily with bi-section methods.

\textit{Case 5. When \( \frac{t' \sqrt{1 -\rho^2}}{1 -|\rho|}\leq v' \leq \frac{t'}{1 -|\rho|} \)}: We first define \(  v_{\min}(\vt) = \max\left\{ 1, \sqrt{\frac{1 -\vt}{1 -\rho^2}}, \frac{\sqrt{1 -\rho^2}}{1 -|\rho|} \cdot f(|\rho|),\right\} \) in which \( f(|\rho|) \) has the same definition from \textit{Case 4}.

When \( |\rho| \leq \frac{\sqrt{2}}{2}  \), the boundary is \begin{equation*}
  \sqrt{r} =\max \left\{ v_{\min}(\vt) + \sqrt{1 -\vt}, h_1(\vt) \right\}
\end{equation*}
where  \( h_1(\vt) =\min\left\{ \textit{Slope}(\vt) \cdot \sqrt{1 -\vt},\,h_2(\vt),\,  \max \left\{ \sqrt{\frac{1 - 2\vt}{2(1 -|\rho|)}} + \frac{g_1(\vt)}{1 -|\rho|},\,\frac{1+2|\rho|}{\sqrt{2-\rho^2}}+\sqrt{1-\vt-\frac{1-\rho^2}{2-\rho^2}} \right\} \right\} \), in which 
\begin{gather*}
  \textit{Slope}(\vt) = 1 + \frac{\sqrt{1-\rho^2}}{1-|\rho|} \cdot t^*, \text{ where \( t^*\in(0,1) \) solves } \frac{|\rho| t}{\sqrt{1-\rho^2}} + \sqrt{1 -t^2}= 1,
\end{gather*} 
and \begin{equation*}
  h_2(\vt) = \begin{cases} 
    \sqrt{\frac{1 - 2\vt}{1 -\rho^2}} + \frac{1}{\sqrt{1 -\rho^2}} & \text{ if } \vt \leq 1 -\frac{1}{\rho^2(1+|\rho|)^2} \\
    \frac{1 + 2|\rho|}{(1 +|\rho|)\sqrt{1 -\rho^2}} + \sqrt{1 -\vt -\frac{1}{(1 +|\rho|)^2}} & \text{ if } \vt > 1 -\frac{1}{\rho^2(1+|\rho|)^2}
  \end{cases} 
\end{equation*} 
and \( g_1(\vt) \) is the same one in \textit{Case 4}.

We define a numerical special numerical value for \( |\rho| \): \( |\rho| = 0.7544 \). It is the value which makes \( \sqrt{r} = \frac{1 + 2|\rho|}{(1 +|\rho|)\sqrt{1 -\rho^2}} + \sqrt{1 -\vt -\frac{1}{(1 +|\rho|)^2}} \) and \( \sqrt{r} =\left( 1 + \frac{\sqrt{1-\rho^2}}{1-|\rho|}  \right)\sqrt{1 -\vt} \) intersect  at \( \vt = \frac{1}{2} \).

When \( \frac{\sqrt{2}}{2} < |\rho| \leq 0.7544\), the boundary is 
\begin{equation*}
  \sqrt{r} =\max \left\{ v_{\min}(\vt) + \sqrt{1 -\vt}, h_3(\vt) \right\},
\end{equation*} where the definiton of \( v_{\min} \) is unchanged; \( h_3(\vt) =\min \left\{ \left( 1 + \frac{\sqrt{1-\rho^2}}{1-|\rho|}  \right)\sqrt{1 -\vt},\,h_2(\vt) \right\} \).

When \( |\rho| > 0.7544\), the boundary is:
\begin{equation*}
  \sqrt{r} = \begin{cases} 
  \max \left\{ v_{\min}(\vt) + \sqrt{1 -\vt}, \, \left( 1 + \frac{\sqrt{1-\rho^2}}{1-|\rho|}  \right)(1 -\vt) \right\} & \text{ if }\vt \leq \frac{1}{2} \\
  h_2(\vt) & \text{ if }\vt >\frac{1}{2}
  \end{cases} 
\end{equation*}

\textit{Case 6. When \(  v' \geq \frac{t'}{1 -|\rho|} \)}
\begin{equation*}
  \sqrt{r} = \max \left\{ \sqrt{1 -\vt} + \frac{1}{\sqrt{1 -\rho^2}},\,  \sqrt{\frac{1 - 2\vt}{1 -\rho^2}} + \frac{1}{\sqrt{1 -\rho^2}},\,h(\vt) \right\}
\end{equation*}
where the curve \( h(\vt) \) is defined as
\begin{equation*}
  h(\vt) = \begin{cases} 
    \sqrt{\frac{1 - \vt}{1 -\rho^2}} + \frac{1}{\sqrt{1 -\rho^2}} & \text{  if }\vt \leq 1 - \frac{1-|\rho|}{\rho^2(1+|\rho|)} \\
    \frac{1}{\sqrt{1 -\rho^2}} +\frac{|\rho|}{1 +|\rho|} + \sqrt{1 -\vt -\frac{1 -|\rho|}{1 +|\rho|}} & \text{  if }\vt > 1 - \frac{1-|\rho|}{\rho^2(1+|\rho|)}.
  \end{cases} 
\end{equation*}



To sum up all the six cases, the lowest phase curve over the six cases in given in Equation~\eqref{suppeq:boundary.firstthree.positive} when \( \rho \geq  0\). When \( \rho< 0 \), the optimal  curve is too complicated, but an upper bound is given by Equation~\eqref{suppeq:boundary.firstthree.negative}.

\section{Proof of Theorem~\ref{thm:equivalence}}

The key is to analyze the random-design setting and show that its minimax rate of Hamming error is only determined by $\mathbb{E}[X'X]=\Sigma$. Then, when we switch to the fixed-design case of $X'X=\Sigma$, the same minimax rate holds. For the random-design setting, we proceed by deriving a lower bound and an upper bound of the minimax Hamming error separately.

First, we derive a lower bound for the minimax Hamming error. Let $G=X'X$ denote the Gram matrix of the random-design model. Fixing any two subsets $V_0, V_1\subset\{1,2,\ldots,p\}$, we write $V=V_0\cup V_1$. Let $\eta\in \{0,1\}^p$ be an arbitrary binary vector. We consider two binary vectors $\mu^{(0)}, \mu^{(1)}\in\{0,1\}^p$ where $\mu^{(0)}_j=\mu^{(1)}_j=\eta_j$,  for $j\notin V$, and restricted on $V$, $\mathrm{Supp}(\mu^{(0)}_V)=V_0$ and $\mathrm{Supp}(\mu_V^{(1)})=V_1$. Let $\tau_p=\sqrt{2r\log(p)}$. Consider the testing problem
\beq \label{proof-equivalent-1}
H_0: \beta = \tau_p\mu^{(0)}, \qquad v.s. \qquad H_1: \beta=\tau_p \mu^{(1)}. 
\eeq
For a test $T$, let $R(T)$ be the sum of type I and type II errors. Any selector $\hat{\beta}$ can be converted to a test $T(\hat{\beta})$, where we reject the null hypothesis if $\mathrm{Supp}(\hat{\beta})\neq V_0$. It is seen that $R(T(\hat{\beta}))=\mathbb{P}\{\beta=\tau_p\mu^{(0)}, \mathrm{Supp}(\hat{\beta})\neq V_0\} +\mathbb{P}\{\beta=\tau_p\mu^{(1)}, \mathrm{Supp}(\hat{\beta})= V_0) \} \leq \sum_{j\in V}\{\mathbb{P}(\beta_j=0, \hat{\beta}_j\neq \tau_p)+ \mathbb{P}(\beta_j=\tau_p, \hat{\beta}_j= 0)\}$. It follows that
\beq \label{proof-equivalent-2}
\mathbb{E}[H(\hat{\beta}_V, \beta_V)|X]\geq R(T(\hat{\beta}))\geq \inf_{T}R(T)\equiv R^*(V_0, V_1; \eta, X). 
\eeq
We can compute the right hand side using the Neyman-Pearson lemma. Define
\beq\label{proof-equivalent-Def-a}
a=a(V_0, V_1, X) = (\mu^{(0)}-\mu^{(1)})'G(\mu^{(0)}-\mu^{(1)}). 
\eeq
The likelihood ratio test for \eqref{proof-equivalent-1} is equivalent to using the test statistic $Z=a^{-1/2}(\mu^{(1)}-\mu^{(0)})'X'(y-\tau_p X\mu^{(0)})$. 
Then, $Z\sim {\cal N}(0, 1)$ under $H_0$, and $Z\sim {\cal N}(a^{1/2}\tau_p, 1)={\cal N}(\sqrt{2ar\log(p)}, 1)$, under $H_1$. By Neyman-Pearson lemma, 
\begin{align}  \label{proof-equivalent-3}
R^*(V_0, V_1; \eta, X) &=\inf_t\Bigl\{\epsilon_p^{|V_0|}\cdot \mathbb{P}\bigl({\cal N}(0, 1)>t\bigr) +  \epsilon_p^{|V_1|}\cdot \mathbb{P}\bigl( {\cal N}(\sqrt{2ar\log(p)}, 1) <t\bigr)\Bigr\}\cr
&= \inf_{t=\sqrt{2q\log(p)}}\Bigl\{ L_p p^{-|V_0|\vartheta - q} + L_p p^{-|V_1|\vartheta - (\sqrt{ar}-\sqrt{q})_+^2}\Bigr\}\cr
&= L_p p^{-h(V_0, V_1, X)}, 
\end{align}
where
\[
h(V_0,V_1, X)=\max_{q>0}\Bigl(\min\Bigl\{ |V_0|\vartheta+q, \; |V_1|\vartheta +  (\sqrt{ar}-\sqrt{q})_+^2\Bigr\}\Bigr). 
\]
In the second line of \eqref{proof-equivalent-3}, we have used the Mills' ratio of $N(0,1)$ (e.g., see \cite{ke2014covariance} for a similar use of the Mills' ratio). Let $\Sigma$ be the covariance matrix, parameterized by $\rho$. We define the following quantities:
\begin{align} \label{proof-equivalent-h-star}
a^*(V_0, V_1,\rho)& = (\mu^{(0)}-\mu^{(1)})'\Sigma(\mu^{(0)}-\mu^{(1)}), \cr
h^*(V_0, V_1, \rho) &= \max_{q>0}\Bigl(\min\Bigl\{ |V_0|\vartheta+q, \; |V_1|\vartheta +  (\sqrt{a^*r}-\sqrt{q})_+^2\Bigr\}\Bigr). 
\end{align} 
Below, we show that $h(V_0, V_1, X)$ is sufficiently close to $h^*(V_0, V_1, \rho)$. The key is showing that $\Sigma$ and $G$ are sufficiently close on the diagonal block restricted to $V$. We use Theorem 5.39 and Remark 5.40 of \cite{Vershynin} with $t=O(\sqrt{|V|\log(p)})$. It follows that, when $|V|\ll n$, with probability $1-o(p^{-3-|V|})$, 
\[
\|G_{V,V}-\Sigma_{V,V}\|\leq C\|\Sigma_{V,V}\|\sqrt{n^{-1}|V|\log(p)};\quad \mbox{here, $C$ a constant independent of $|V|$}.  
\]
We note that $\|\Sigma_{V,V}\|\leq \|\Sigma\|\leq C$. For any finite integer $m\geq 1$, the total number of size-$m$ subset $V$ is ${p\choose m}=O(p^m)$. We then apply the probability union bound to get that, with probability $1-O(p^{-3})$, 
\beq  \label{proof-equivalent-4}
\max_{V: |V|\leq m} \|G_{V,V}-\Sigma_{V,V}\|\leq C\sqrt{n^{-1}\log(p)}. 
\eeq
Since $|a(V_0, V_1, X)-a^*(V_0, V_1, \rho)|\leq \|G_{V,V}-\Sigma_{V,V}\|\cdot\|\mu^{(1)}-\mu^{(0)}\|^2\leq \|G_{V,V}-\Sigma_{V,V}\|\cdot |V|$, we immediately know that 
\beq \label{proof-equivalent-a-bound}
|a(V_0, V_1, X)-a^*(V_0, V_1, \rho)|\leq C\sqrt{n^{-1}\log(p)} \quad \mbox{here, $C$ depends on $m$}. 
\eeq
Write $h=h(V_0,V_1,X)$ and $h^*=h^*(V_0,V_1,\rho)$ for short, and let $(h^*, a^*)$ be the shorthand notations defined similarly. Then, $h=\max_{q}g(q, a)$ and $h^* =\max_q f(q, a^*)$, for $f(q,a)=\min\{ |V_0|\vartheta+q,  |V_1|\vartheta +  (\sqrt{ar}-\sqrt{q})_+^2\}$. Let $\tilde{q}$ and $\tilde{q}^*$ be the two maximizers. It is seen that $h=f(\tilde{q}, a)\leq f(\tilde{q}, a^*)+\max_{q}|f(q,a)-f(q,a^*)|\leq h^*+\max_{q}|f(q,a)-f(q,a^*)|$. Similarly, we can also derive that $h\leq h^*+\max_q|f(q, a^*)-f(q,a )|$. Combining them gives $|h-h^*|\leq \max_q |f(q, a^*)-f(q,a)|$. We plug in the expression of $f(q,a)$ to get
\[
|h(V_0, V_1, X)-h^*(V_0, V_1, \rho)|\leq |\sqrt{ar}-\sqrt{a^*r}|\leq C\sqrt{n^{-1}\log(p)}. 
\]
We now combine all the results, and note that \eqref{proof-equivalent-4} has a maximum over all $V=V_0\cup V_1$. It follows that, with probability $1-O(p^{-3})$,
\beq  \label{proof-equivalent-5}
\max_{(V_0, V_1): |V_0\cup V_1|\leq m}|h(V_0, V_1, X)-h^*(V_0, V_1, \rho)|\leq C\sqrt{n^{-1}\log(p)}. 
\eeq
We plug it into \eqref{proof-equivalent-3}. Note that $L_p p^{-h}=L_p p^{-h^*}\cdot p^{h^*-h}$. In line of \eqref{proof-equivalent-5}, $p^{h^*-h}$ is a multi-$\log(p)$ term, i.e.,  $L_pp^{-h}=L_pp^{-h^*}$. We then combine it with \eqref{proof-equivalent-2}. It yields that, with probability $1-O(p^{-3})$, 
\beq \label{proof-equivalent-6}
\mathbb{E}[H(\hat{\beta}_V, \beta_V)|X]\geq L_p p^{-h^*(V_0, V_1,\rho)}, \mbox{simultaneously for all $(V_0,V_1)$ with $|V_0\cup V_1|\leq m$}. 
\eeq
Given $V$, we further take a maximum over $(V_0, V_1)$ on the right hand side. It follows that 
\beq \label{proof-equivalent-h-star-star}
\mathbb{E}[H(\hat{\beta}_V, \beta_V)|X]\geq L_p p^{-h^{**}(V,\rho)}, \qquad\mbox{where}\quad h^{**}(V, \rho) = \min_{\substack{(V_0,V_1): V_0\neq V_1,\\V_0\cup V_1=V}}h^*(V_0, V_1, \rho).  
\eeq 
Write $\{1,2,\ldots,p\}=\cup_{j=1}^{\lceil p/2\rceil} V_j$, where $V_j=\{2j-1, 2j\}$ for $j\leq p/2$ and $V_j=\{p\}$ for $j>p/2$ (this happens only if $p$ is odd). It follows that, with probability $1-O(p^{-3})$, 
\[
\mathbb{E}[H(\hat{\beta},\beta)|X]=\sum_{1\leq j\leq \lceil p/2\rceil}\mathbb{E}[H(\hat{\beta}_{V_j}, \beta_{V_j})|X]\geq 
\sum_{1\leq j\leq \lceil p/2\rceil}L_p p^{-h^{**}(V_j,\rho)}. 
\]  
When $p$ is even, $h^{**}(V_j,\rho)$ are all equal. When $p$ is odd, $h^{**}(V_j,\rho)$ are all equal, except for one $V_j$;  but this one has a negligible effect on the right hand side above. Let $h^{**}(\rho)$ be the common value of $h^{**}(V_j,\rho)$. Since $h^{**}(\rho)$ also depends on $(\vartheta,r)$, we write it as $h^{**}(\rho; \vartheta,r)$ to reflect this dependence. We immediately have that, with probability $1-O(p^{-3})$,
\[
\mathbb{E}[H(\hat{\beta},\beta)|X] \geq L_p p^{1-h^{**}(\rho;\vartheta,r)}. 
\] 
On the event that the above inequality does not hold, the Hamming error is at most $p$. The contribution of this event to the expected Hamming error is at most $p\cdot O(p^{-3})=O(p^{-2})$, which is negligible to $L_p p^{1-h^*(\rho;\vartheta,r)}$. It follows that
\beq \label{proof-equivalent-7}
\mathbb{E}[H(\hat{\beta},\beta)]\geq L_p p^{1-h^{**}(\rho;\vartheta,r)}, \qquad\mbox{for any method $\hat{\beta}$}. 
\eeq
This gives a lower bound for the minimax Hamming error. 

Next, we give an upper bound for the minimax Hamming error. We will consider a specific $\hat{\beta}$. Let the partition $\{1,2,\ldots,p\}=\cup_{j=1}^{\lceil p/2\rceil} V_j$ be the same as above. For any subset $U\subset\{1,2,\ldots,p\}$, let ${\bf1}_U$ be the binary vector such that its $j$th entry is $1$ if $j\in U$ and $0$ otherwise. Additionally, let $X_U$ be the submatrix of $X$ restricted to columns in $U$. For each $V_j$, define
\beq \label{proof-equivalent-8}
\hat{U}_j = \argmin_{U\subset V_j}\Bigl\{ \frac{1}{2}\|y-\tau_p X{\bf 1}_U\|^2 + \vartheta \log(p)|U| \Bigr\}.
\eeq
Define $\hat{\mu}\in\{0,1\}^p$ such that for any $i\in V_j$, $\hat{\mu}_i=1$ if $i\in \hat{U}_j$, and $\hat{\mu}_i=0$ otherwise. The estimator is $\hat{\beta}=\tau_p \hat{\mu}$. We now calculate the expected Hamming error of this estimator. Let $S$ be the support of $\beta$. Fix $V_j$ and write $V=V_j$ for short. Given any two subsets $U_0$ and $U_1$ of $V$ such that $U_0\neq U_1$, we consider the event 
\beq \label{proof-equivalent-event}
\mathrm{Supp}(\beta_V)=U_0, \qquad \mathrm{Supp}(\hat{\beta}_V)=U_1, \qquad |S|\leq 2p^{1-\vartheta}. 
\eeq
On this event, it is true that 
\beq \label{proof-equivalent-9}
\frac{1}{2}\|y-\tau_p X{\bf 1}_{U_0}\|^2 + \vartheta \log(p)|U_0|\geq  \frac{1}{2}\|y-\tau_p X {\bf 1}_{U_1}\|^2 + \vartheta \log(p)|U_1|. 
\eeq
Note that $y=X\beta+z=\tau_p X{\bf 1}_{U_0}+\tau_p X{\bf 1}_{S\cap V^c}+z$. We can re-write \eqref{proof-equivalent-9} as
\[
 \frac{1}{2}\|z+\tau_p X{\bf 1}_{S\cap V^c}\|^2 + \vartheta \log(p)|U_0|\geq  \frac{1}{2}\|(z+\tau_p X{\bf 1}_{S\cap V^c})-\tau_p X({\bf 1}_{U_1}-{\bf 1}_{U_0})\|^2 + \vartheta \log(p)|U_1|.
\]
Let $F=U_0\cap U_1$, $E_0=U_0\setminus F$ and $E_1=U_1\setminus F$. Then, $|U_0|=|F|+|E_0|$ and $|U_1|=|F|+|E_1|$. We plug it into the above inequality and re-arrange the terms. It gives
\[
z'X({\bf 1}_{U_1}-{\bf 1}_{U_0})\geq \frac{\tau_p}{2}({\bf 1}_{U_1}-{\bf 1}_{U_0})'G({\bf 1}_{U_1}-{\bf 1}_{U_0}) -  {\bf 1}_{S\cap V^c}'G({\bf 1}_{U_1}-{\bf 1}_{U_0})+\frac{\vartheta\log(p)}{\tau_p}(|E_1|-|E_0|). 
\]
Let $a=a(U_0, U_1, X)=({\bf 1}_{U_1}-{\bf 1}_{U_0})'G({\bf 1}_{U_1}-{\bf 1}_{U_0})$. We note that this definition is indeed the same as that in \eqref{proof-equivalent-Def-a}. Let $\tilde{z}=z'X({\bf 1}_{U_1}-{\bf 1}_{U_0})/\sqrt{a}$. The above can be written equivalently as
\beq \label{proof-equivalent-10}
\tilde{z}\geq \sqrt{2\log(p)} \biggl[ \frac{\sqrt{ra}}{2} - \frac{{\bf 1}_{S\cap V^c}'G({\bf 1}_{U_1}-{\bf 1}_{U_0})}{\sqrt{a}} + \frac{\vartheta(|E_1|-|E_0|)}{2\sqrt{ra}}\biggr], \quad\mbox{where}\;\; \tilde{z}|(X,\beta)\sim N(0,1). 
\eeq
We bound $|{\bf 1}_{S\cap V^c}'G({\bf 1}_{U_1}-{\bf 1}_{U_0})|$. Since $\Sigma_{V^cV}$ is a zero matrix, we immediately have ${\bf 1}_{S\cap V^c}'\Sigma({\bf 1}_{U_1}-{\bf 1}_{U_0})=0$. It follows by the triangle inequality that 
\begin{align}  \label{proof-equivalent-11}
|{\bf 1}_{S\cap V^c}'G({\bf 1}_{U_1}-{\bf 1}_{U_0})| & \leq |{\bf 1}_{S\cap V^c}'(G-\Sigma)({\bf 1}_{U_1}-{\bf 1}_{U_0})|\cr
&\leq |V|\cdot \max_{k\in V}|e_k'(G-\Sigma){\bf 1}_{S\cap V^c}|. 
\end{align}
For $k\in V$, $e_k'(G-\Sigma){\bf 1}_{S\cap V^c}=\sum_{i=1}^n \sum_{\ell \in S\cap V^c}\{X(i,k)X(i,\ell)-\mathbb{E}[X(i,k)X(i,\ell)]\}$. 
We recall that $\{1,2,\ldots,p\}=\cup_{m=1}^{\lceil p/2\rceil}V_m$ is a partition. It induces a partition on $S\cap V^c$, which we denote by $S\cap V^c=\cup_{m=1}^{N_p}S_m$. Each $S_m$ contains at most 2 indices and $ |S\cap V^c|/2\leq N_p\leq |S\cap V^c|$. Write
\[
e_k'(G-\Sigma){\bf 1}_{S\cap V^c} = \sum_{i=1}^n \sum_{m=1}^{N_p}\Bigl[ \sum_{\ell\in S_m}\{X(i,k)X(i,\ell)-\mathbb{E}[X(i,k)X(i,\ell)]\}\Bigr]. 
\]
The right hand side is a sum of $nN_p$ independent variables, where each variable has a zero mean and a sub-exponential norm bounded by $n^{-1}K$, for a constant $K>0$. We apply the Bernstein inequality \citep[Proposition 5.16]{Vershynin} to get that, for every $t>0$, 
\[
\mathbb{P}\bigl(|e_k'(G-\Sigma){\bf 1}_{S\cap V^c}|>t\bigr)\leq 2\exp\Bigl(-c\min\Bigl\{\frac{nt^2}{K^2N_p}, \frac{nt}{K}\Bigr\}\Bigr),
\]
where $c>0$ is a universal constant. By letting $t=C\sqrt{n^{-1}N_p\log(p)}$ for a properly large constant $C$, we have that, with probability $1-O(p^{-3})$, 
\[
|e_k'(G-\Sigma){\bf 1}_{S\cap V^c}|\leq C\sqrt{n^{-1}N_p\log(p)}\leq C\sqrt{n^{-1}|S|\log(p)}. 
\]
We plug it into \eqref{proof-equivalent-10} and apply the probability union bound. We also note that $|S|=O(p^{1-\vartheta})$ on the event \eqref{proof-equivalent-event}; also, $n=p^{\omega}$ with $\omega>1-\vartheta$.  It follows that, on this event, with probability $1-O(p^{-3})$, 
\beq \label{proof-equivalent-12}
|{\bf 1}_{S\cap V^c}'G({\bf 1}_{U_1}-{\bf 1}_{U_0})|\leq Cp^{-(\omega-1+\vartheta)/2}\sqrt{\log(p)}. 
\eeq
We plug \eqref{proof-equivalent-12} into \eqref{proof-equivalent-10} to get: 
\[
\tilde{z}\geq \sqrt{2b_p\log(p)},\quad \mbox{where } b_p= \Bigl[\frac{\sqrt{ra}}{2} + \frac{\vartheta(|E_1|-|E_0|)}{2\sqrt{ra}}-L_p p^{-\frac{\omega-1+\vartheta}{2}}\Bigr]^2. 
\]
Moreover, let $a^*=a^*(U_0, U_1, \rho)=({\bf 1}_{U_1}-{\bf 1}_{U_0})'\Sigma({\bf 1}_{U_1}-{\bf 1}_{U_0})$, which is the same as the definition in \eqref{proof-equivalent-h-star}. By \eqref{proof-equivalent-a-bound}, the replacement of $a$ by $a^*$ only yields a difference of $L_p p^{-\frac{\omega-1+\vartheta}{2}}$ in the expression of $b_p$. We further have: 
\beq \label{proof-equivalent-13}
\tilde{z}\geq \sqrt{2b_p\log(p)},\quad \mbox{where }\tilde{z}|(X,\beta)\sim N(0,1) \mbox{ and } b_p= \Bigl[\frac{\sqrt{ra^*}}{2} + \frac{\vartheta(|E_1|-|E_0|)}{2\sqrt{ra^*}}+L_p p^{-\frac{\omega-1+\vartheta}{2}}\Bigr]^2. 
\eeq
First, what \eqref{proof-equivalent-13} says is that, conditioning on $(X,\beta_{V}, \beta_{V^c})$, if $\|\beta\|_0\leq Cp^{1-\vartheta}$, then except for an event of probability $O(p^{-3})$, $\mathrm{Supp}(\hat{\beta}_V)=U_1$ implies $\tilde{z}>\sqrt{2b_p\log(p)}$. Second, under our model, $(X,\beta_{V^c})$ are independent of $\beta_V$, and $\mathbb{P}(\mathrm{Supp}(\beta_V)=U_0)=L_pp^{-\vartheta |U_0|}=L_pp^{-\vartheta(|F|+|E_0|)}$. Last, $\mathbb{P}(\|\beta_{V^c}\|_0\leq 2p^{1-\vartheta})=1-O(p^{-3})$ (this is seen by noticing that $\|\beta_{V^c}\|_0$ is the sum of independent Bernoulli variables and by applying the Bernstein's inequality). We combine the above to get
\begin{align*}
& \mathbb{P}\bigl(\mathrm{Supp}(\beta_V)=U_0,\, \mathrm{Supp}(\hat{\beta}_V)=U_1, \, |S|\leq 2p^{1-\vartheta}\bigr)\cr
 \leq \;\; &L_pp^{-\vartheta|U_0|}\cdot \mathbb{P}\biggl(\tilde{z}\geq \sqrt{2b_p\log(p)}\biggr) + O(p^{-3})\cr
 \leq \;\; & L_p p^{-\vartheta( |F| +|E_0|)- \bigl[\frac{\sqrt{ra^*}}{2} + \frac{\vartheta(|E_1|-|E_0|)}{2\sqrt{ra^*}}\bigr]_+^2}. 
\end{align*}
By elementary calculations, we have 
\begin{align*}
 & \vartheta(|F|+ |E_0|)+ \Bigl[\frac{\sqrt{ra^*}}{2} + \frac{\vartheta(|E_1|-|E_0|)}{2\sqrt{ra^*}}\Bigr]_+^2\cr
 \geq \quad & \vartheta|F|+\max\{|E_0|, |E_1|\}\vartheta +\frac{1}{4}\Bigl(\sqrt{ra^*}-\frac{|(|E_1|-|E_0|)|}{\sqrt{ra^*}}\Bigr)_+^2\cr
=\quad &  \max\{|U_0|, |U_1|\}\vartheta +\frac{1}{4}\Bigl(\sqrt{ra^*}-\frac{|(|U_1|-|U_0|)|}{\sqrt{ra^*}}\Bigr)_+^2\quad =\quad h^*(U_0, U_1, \rho), 
\end{align*}
where $h^*(U_0, U_1, \rho)$ is the same as that defined in \eqref{proof-equivalent-h-star} (the last equality above follows by solving $q$ in \eqref{proof-equivalent-h-star}). We combine the above to get
\beq \label{proof-equivalent-14}
\mathbb{P}\bigl(\mathrm{Supp}(\beta_V)=U_0,\, \mathrm{Supp}(\hat{\beta}_V)=U_1, \, |S|\leq 2p^{1-\vartheta}\bigr)\leq L_p p^{-h^*(U_0, U_1, \rho)}. 
\eeq 
On the above event, the Hamming error contributed by $\hat{\beta}_V$ is $|E_0|+|E_1|\leq |V|\leq 2$. 
Moreover, $h^*(U_0, U_1,\rho)\geq h^{**}(\rho;\vartheta,r)$, where the latter is defined in \eqref{proof-equivalent-h-star-star}. 
It follows that
\begin{align*}
\mathbb{E}[H(\hat{\beta}_V, \beta_V)] &= \sum_{(U_0, U_1)}2\cdot \mathbb{P}\bigl(\mathrm{Supp}(\beta_V)=U_0,\, \mathrm{Supp}(\hat{\beta}_V)=U_1, \, |S|\leq 2p^{1-\vartheta}\bigr) +O(p^{-3})\cr
&\leq L_p \sum_{(U_0, U_1)} p^{-h^*(U_0, U_1, \rho)} \leq  L_p p^{-h^{**}(\rho;\vartheta,r)}. 
\end{align*}
The above is true for every $V=V_j$ in the partition $\{1,2,\ldots,p\}=\cup_{j=1}^{\lceil p/2\rceil}$ (except for the last $V_j$ in the case that $p$ is odd; but this single $V_j$ has a negligible effect on the rate of the Hamming error). We immediately have
\beq \label{proof-equivalent-15}
\mathbb{E}[H(\hat{\beta}, \beta)]=\sum_{1\leq j\leq \lceil p/2\rceil}\mathbb{E}[H(\hat{\beta}_V, \beta_V)]\leq L_p p^{1-h^{**}(\rho;\vartheta,r)}, \quad\mbox{for the $\hat{\beta}$ in \eqref{proof-equivalent-8}}.  
\eeq
This gives an upper bound for the minimax Hamming error.

Last, we use \eqref{proof-equivalent-7} and \eqref{proof-equivalent-15} to show the claim. By combining these two inequalities, we know that, for the random design, 
\[
\inf_{\hat{\beta}} \mathbb{E}[H(\hat{\beta}, \beta)] = L_p p^{1-h^{**}(\rho;\vartheta,r)}. 
\]
A key observation is that the exponent $h^{**}(\rho; \vartheta,r)$ is only related to $\Sigma$, not the realization of $X'X$. Now, we can mimic all the above derivations to get the same conclusion when the Gram matrix is fixed at $\Sigma$ (the proof is very similar, except that we now have $G=\Sigma$). Therefore, the minimax rates of the Hamming errors under two settings are exactly the same. \qed

\end{document}